\documentclass[12pt]{amsart}
\usepackage{etex} 
\usepackage{amscd}
\usepackage{amssymb}
\usepackage{amsmath}
\usepackage[noadjust]{cite}
\usepackage{booktabs}
\usepackage{url}
\usepackage{hyphenat}
\usepackage{epsf}
\usepackage{mathtools}
\usepackage{mathrsfs, stmaryrd,enumerate,extarrows}
\usepackage{array}
\usepackage[curve]{xypic}
\usepackage{graphicx, color}
\usepackage[lmargin=1.16in,rmargin=1.16in,tmargin=1.17in,bmargin=1.17in]{geometry}
\usepackage[bookmarks=true, bookmarksopen=true,%
bookmarksdepth=3,bookmarksopenlevel=2,%
colorlinks=true,%
linkcolor=blue,%
citecolor=blue,%
filecolor=blue,%
menucolor=blue,%
urlcolor=blue]{hyperref}
\hypersetup{pdftitle={Cornered Floer from Bordered Floer}}
\hypersetup{pdfauthor={Christopher L.~Douglas, Robert Lipshitz, Ciprian Manolescu}}
\usepackage[percent]{overpic}
\usepackage{relsize}
\usepackage{calc}
\usepackage{enumitem}
\usepackage{afterpage}
\usepackage{tikz}
\usetikzlibrary{matrix}
\usetikzlibrary{arrows,backgrounds,patterns,scopes,external,
    decorations.pathreplacing,
    decorations.pathmorphing
}

\newread\testin

\graphicspath{{draws/}{}}

\def\mathcenter#1{%
  \vcenter{\hbox{$#1$}}%
}

\makeatletter
\newcommand\mi@kern[1]{%
  \settowidth\@tempdima{$\mi@obj^{#1}$}
  \kern-\@tempdima
  #1
  \settowidth\@tempdima{$\mi@obj$}
  \kern\@tempdima
}

\newtoks\mi@toksp
\newtoks\mi@toksb
\DeclareRobustCommand{\manyindices}[5]{
  \def\mi@obj{#5}
  \mi@toksp\expandafter{\mi@kern{#2}}
  \mi@toksb\expandafter{\mi@kern{#1}}
  \@mathmeasure4\textstyle{#5_{#1}^{#2}}
  \@mathmeasure6\textstyle{#5_{#3}^{#4}}
  \dimen0-\wd6 \advance\dimen0\wd4
  \@mathmeasure8\textstyle{\hphantom{{}_{#1}^{#2}}#5^{\the\mi@toksp#4}_{\the\mi@toksb#3}}
  \hbox to \dimen0{}{\kern-\dimen0\box8}
}
\makeatother 

\usepackage{ifpdf}
\ifpdf
  
\else
  
\fi

\newcommand{\ZZ}{\mathbb Z}

\newcommand{\FF}{\mathbb F}
\newcommand{\NN}{\mathbb N}

\newcommand{\bD}{\mathbb{D}}

\newcommand{\fG}{{\mathfrak G}}
\newcommand{\co}{\nobreak\mskip2mu\mathpunct{}\nonscript
  \mkern-\thinmuskip{:}\penalty300\mskip6muplus1mu\relax}

\newcommand{\OneHalf}{{\textstyle\frac{1}{2}}}
\newcommand{\Planar}{\mathfrak{P}}

\newcommand{\bdy}{\partial}

\newcommand{\lbracket}{[}
\newcommand{\rbracket}{]}

\newcommand{\spinc}{\mathfrak s}

\DeclareMathOperator{\Hom}{Hom}

\DeclareMathOperator{\spin}{spin}
\newcommand{\SpinC}{\spin^c}

\DeclareMathOperator{\gr}{gr}

\newcommand{\crs}{\operatorname{cr}}

\theoremstyle{plain}
\newtheorem{theorem}{Theorem}
\numberwithin{equation}{section}

\newtheorem{proposition}[equation]{Proposition}
\newtheorem{lemma}[equation]{Lemma}
\newtheorem{corollary}[equation]{Corollary}

\newtheorem{observation}[equation]{Observation}
\newtheorem{convention}[equation]{Convention}
\newtheorem{definition}[equation]{Definition}

\newtheorem{assumption}[equation]{Assumption}

\theoremstyle{definition}

\theoremstyle{remark}
\newtheorem{example}[equation]{Example}
\newtheorem{remark}[equation]{Remark}

\newcommand{\HF}{\mathit{HF}}
\newcommand{\HFa}{\widehat {\HF}}

\newcommand{\CFa}{\widehat {\mathit{CF}}}

\newcommand{\x}{\mathbf x}
\newcommand{\y}{\mathbf y}
\newcommand{\z}{\mathbf z}

\newcommand\HH{\mathit{HH}}

\newcommand\Hochschild\HH

\newcommand{\Ainf}{A_\infty}
\newcommand{\Alg}{\mathcal{A}}

\newcommand\Blg{\mathcal{B}}
\newcommand\Clg{\mathcal{C}}

\newcommand{\Idem}{\mathcal{I}}
\newcommand{\alphas}{{\boldsymbol{\alpha}}}
\newcommand{\betas}{{\boldsymbol{\beta}}}

\newcommand{\rhos}{{\boldsymbol{\rho}}}

\newcommand{\DhAA}{\textit{D-AA}}
\newcommand{\DhAD}{\textit{D-AD}}
\newcommand{\DhDA}{\textit{D-DA}}
\newcommand{\DhDD}{\textit{D-DD}}
\newcommand{\AhDD}{\textit{A-DD}}
\newcommand{\DD}{\textit{DD}}
\newcommand{\DA}{\textit{DA}}
\newcommand{\AD}{\textit{AD}}
\newcommand{\AAm}{\textit{AA}} 
\newcommand{\CFD}{\mathit{CFD}}
\newcommand{\CFDD}{\mathit{CFDD}}
\newcommand{\CFA}{\mathit{CFA}}

\newcommand{\CFDA}{\mathit{CFDA}}
\newcommand{\CFDAa}{\widehat{\CFDA}}

\newcommand{\CFDa}{\widehat{\CFD}}

\newcommand{\CFDDa}{\widehat{\CFDD}}

\newcommand{\CFAa}{\widehat{\CFA}}

\newcommand{\cZ}{\mathcal{Z}}
\newcommand{\PtdMatchCirc}{\cZ}
\newcommand{\PMC}{\PtdMatchCirc}

\newcommand{\CircPts}{{\mathbf{a}}}

\newcommand{\arcz}{\mathbf{z}}

\newcommand{\dg}{\text{dg} }

\newcommand\Id{\mathbb{I}}

\newcommand\DTP{\mathop{\widetilde\otimes}\nolimits}

\newcommand\Gen{\mathfrak{S}}

\newcommand{\Field}{{\FF_2}}
\DeclareMathOperator{\nbd}{nbd}

\newcommand{\Heegaard}{\mathcal{H}}
\newcommand{\HD}{\Heegaard}

\renewcommand{\th}{^\text{th}}

\newcommand{\bigGroup}{G'}

\def\In{\operatorname{in}}
\def\Out{\operatorname{out}}

\DeclareMathOperator{\Mor}{Mor}

\newcommand{\op}{\mathrm{op}}
\newcommand\PunctF{F^\circ}
\newcommand{\Chord}{\mathrm{Chord}}
\newcommand{\bOne}{\mathbf{1}}

\newcommand{\SetS}{\mathbf{s}}
\newcommand{\SetT}{\mathbf{t}}
\DeclareMathOperator{\Coeff}{Coeff}

\makeatletter
\newcommand\honestalg[3]{\bigl\lbracket
\begin{smallmatrix} #1\@ifempty{#3}{}{&#3} \\ #2 \end{smallmatrix}
\bigr\rbracket}

\makeatother
\newcommand{\lab}[1]{$\scriptstyle #1$}

\newcommand{\CT}{\mathsf{Co}}
\newcommand{\trimod}{\mathsf{T}}
\newcommand{\TDDD}{\trimod_{DDD}}
\newcommand{\TDDA}{\trimod_{DDA}}
\newcommand{\THD}[3]{\HD_\trimod(#1,#2,#3)}
\newcommand{\CFDDD}{\mathit{CFDDD}}
\newcommand{\CFDDDa}{\widehat{\CFDDD}}
\newcommand{\CFDDA}{\mathit{CFDDA}}
\newcommand{\CFDDAa}{\widehat{\CFDDA}}
\newcommand{\DDD}{\textit{DDD}}
\newcommand{\DDA}{\textit{DDA}}

\DeclareMathOperator{\chordsplit}{split}

\newcommand{\DiagAlg}{{\boldsymbol{\Delta}}}

\newcommand{\cornmod}{\mathsf{C}}
\newcommand{\cornAA}{\cornmod_{D\{AA\}}}
\newcommand{\cornDA}{\cornmod_{D\{DA\}}}
\newcommand{\cornAD}{\cornmod_{D\{AD\}}}
\newcommand{\cornDD}{\cornmod_{D\{DD\}}}
\newcommand{\AcornDD}{\cornmod_{A\{DD\}}}

\newcommand{\DDa}{\mathit{DD}}

\def\nil {\mathfrak{N}}
\def\A {\mathfrak{A}}
\def\B {\mathfrak{B}}
\def\Av {\mathfrak{A}^v}
\def\Ah {\mathfrak{A}^h}

\newcommand{\Dnil}{\mathfrak{D}}
\newcommand{\DRnil}{\rot{\mathfrak{D}}}
\newcommand{\DLnil}{\rotcc{\mathfrak{D}}}
\newcommand{\sDnil}{\widetilde{\Dnil}} 

\newcommand{\LAlg}{\mathcal{L}}
\newcommand{\RAlg}{\mathcal{R}}
\newcommand{\TAlg}{\mathcal{T}}
\newcommand{\BAlg}{\mathcal{B}}
\newcommand{\LAlgD}{\mathcal{L}}
\newcommand{\RAlgD}{\mathcal{R}}
\newcommand{\TAlgD}{\mathcal{T}}
\newcommand{\BAlgD}{\mathcal{B}}
\newcommand{\seq}{v}
\newcommand{\LAlgS}{\mathcal{L}^{\seq}}
\newcommand{\RAlgS}{\mathcal{R}^{\seq}}
\newcommand{\TAlgS}{\mathcal{T}^{\seq}}
\newcommand{\BAlgS}{\mathcal{B}^{\seq}}

\newcommand{\LAlgv}{\mathcal{L}^v}
\newcommand{\RAlgv}{\mathcal{R}^v}
\newcommand{\TAlgv}{\mathcal{T}^v}
\newcommand{\BAlgv}{\mathcal{B}^v}

\newcommand{\LAlgh}{\mathcal{L}^h}
\newcommand{\RAlgh}{\mathcal{R}^h}
\newcommand{\TAlgh}{\mathcal{T}^h}
\newcommand{\BAlgh}{\mathcal{B}^h}

\newcommand{\Algv}{\Alg^v}

\newcommand{\CDD}{{\CFa\mathit{\{DD\}}}}
\newcommand{\CDA}{{\CFa\mathit{\{DA\}}}}
\newcommand{\CAD}{{\CFa\mathit{\{AD\}}}}
\newcommand{\CAA}{{\CFa\mathit{\{AA\}}}}

\newcommand{\TR}{\mathit{TR}}
\newcommand{\BR}{\mathit{BR}}
\newcommand{\BL}{\mathit{BL}}
\newcommand{\TL}{\mathit{TL}}

\newcommand{\TRv}{\mathit{TR}^v}
\newcommand{\BRv}{\mathit{BR}^v}
\newcommand{\BLv}{\mathit{BL}^v}
\newcommand{\TLv}{\mathit{TL}^v}

\newcommand{\TRh}{\mathit{TR}^h}
\newcommand{\BRh}{\mathit{BR}^h}
\newcommand{\BLh}{\mathit{BL}^h}
\newcommand{\TLh}{\mathit{TL}^h}

\newcommand{\vtp}{\odot}
\newcommand{\obent}{\owedge}
\newcommand{\osmooth}{\circledcirc}
\newcommand{\htp}{\circledast}
\newcommand{\bigoast}{\mathlarger{\circledast}} 
\newcommand{\rvtp}{\odot^v}
\newcommand{\rhtp}{\circledast^h}

\newcommand{\smooth}[1]{\widetilde{#1}}

\newcommand{\del}{\partial}

\def\twoainf {A_{\infty}\text{-}2}

\newcommand{\bs}{\begin{smallmatrix}}
\newcommand{\es}{\end{smallmatrix}}
\newcommand{\pbs}{\left( \begin{smallmatrix}}
\newcommand{\pes}{\end{smallmatrix} \right)}
\newcommand{\bmat}{\renewcommand\arraystretch{.6}\begin{matrix}}
\newcommand{\emat}{\end{matrix}}
\newcommand{\pbmat}{\renewcommand\arraystretch{.6}\begin{pmatrix}}
\newcommand{\pemat}{\end{pmatrix}}
\newcommand{\bnmat}{\renewcommand\arraystretch{1}\begin{matrix}}
\newcommand{\enmat}{\end{matrix}}
\newcommand{\pbnmat}{\renewcommand\arraystretch{1}\begin{pmatrix}}
\newcommand{\penmat}{\end{pmatrix}}

\newcommand{\todot}{\rule[-.33ex]{0pt}{1.66ex}\odot\rule[-.33ex]{0pt}{1.66ex}}

\newenvironment{xsmallmatrix}
  {\renewcommand\thickspace{\kern.05em}\smallmatrix}
  {\endsmallmatrix}
 \newcommand{\pbxs}{\left( \begin{xsmallmatrix}}
 \newcommand{\pexs}{\end{xsmallmatrix} \right)}

\newcommand{\sbull}{
\hspace{.1ex}
\begin{tikzpicture}[baseline=-.3ex] \filldraw (0,0) circle (.15ex); \end{tikzpicture}
\hspace{.1ex}
}

\newcommand{\vertme}[5]{
\begin{matrix}
#2 \hphantom{\scriptscriptstyle#3}
\\[#4]
\mathlarger{\mathlarger{\odot}}{\scriptscriptstyle#3}
\\[#5]
#1 \hphantom{\scriptscriptstyle#3}
\end{matrix}
}

\newcommand{\smallvertme}[5]{
\begin{matrix}
#2 \hphantom{\,\scriptscriptstyle#3}
\\[#4]
\odot{\scriptscriptstyle#3}
\\[#5]
#1 \hphantom{\,\scriptscriptstyle#3}
\end{matrix}
}

\newcommand{\rvertme}[5]{
\begin{matrix}
#2 \hphantom {v \scriptscriptstyle#3}
\\[#4]
\mathlarger{\mathlarger{\odot^{\scriptscriptstyle v}}}{\scriptscriptstyle#3}
\\[#5]
#1 \hphantom{v \scriptscriptstyle#3}
\end{matrix}
}

\newcommand{\rsmallvertme}[5]{
\begin{matrix}
#2 \hphantom{\,\scriptscriptstyle#3}
\\[#4]
\odot^v{\scriptscriptstyle#3}
\\[#5]
#1 \hphantom{\,\scriptscriptstyle#3}
\end{matrix}
}

\newcommand{\horme}[5]{
#1 \hspace{#4} \underset{\scriptscriptstyle#3}{\bigoast} \hspace{#5} #2
}

\newcommand{\rhorme}[5]{
#1 \hspace{#4}
\underset{\scriptscriptstyle#3}{\bigoast^{\scriptscriptstyle h}} \hspace{#5} #2
}

\newcommand{\hormein}[5]{
#1 \hspace{#4} {\bigoast}_{\scriptscriptstyle#3} \hspace{#5} #2
}

\newcommand{\rhormein}[5]{
#1 \hspace{#4}
{\bigoast^{\scriptscriptstyle h}_{\scriptscriptstyle#3}} \hspace{#5} #2
}

\newcommand{\smallhorme}[5]{
#1 \hspace{#4} \underset{\scriptstyle#3}{\htp} \hspace{#5} #2
}

\DeclareMathOperator{\HomH}{Hom^h}
\DeclareMathOperator{\HomV}{Hom^v}

\newcommand{\putaround}[5]{
\begin{tikzpicture}[baseline=(A.base),inner sep=0ex]
\node [label={[label distance=.15ex]below:$\scriptscriptstyle #2$},
 label={[label distance=-.1ex, text depth=0pt, text height=0.5ex]left:$\scriptscriptstyle #3$}, 
 label={[label distance=.15ex]above:$\scriptscriptstyle #4$},
 label={[label distance=-.1ex, text depth=0pt, text height=0.5ex]right:$\scriptscriptstyle #5$}] (A)
{$#1$};
\end{tikzpicture}
}

\newcommand{\putaroundmarg}[5]{
\begin{tikzpicture}[baseline=(A.base),inner sep=0ex]
\node [label={[label distance=.25ex]below:$\scriptscriptstyle #2$},
 label={[label distance=.1ex, text depth=0pt, text height=0.5ex]left:$\scriptscriptstyle #3$}, 
 label={[label distance=.25ex]above:$\scriptscriptstyle #4$},
 label={[label distance=.1ex, text depth=0pt, text height=0.5ex]right:$\scriptscriptstyle #5$}] (A)
{$#1$};
\end{tikzpicture}
}

\newcommand{\pp}{\phantom{'}}

\newcommand{\RAlga}[3]{\putaround{\RAlg}{#1}{}{#2}{#3}}
\newcommand{\LAlga}[3]{\putaround{\LAlg}{#1}{#2}{#3}{}}
\newcommand{\TAlga}[3]{\putaround{\TAlg}{}{#1}{#2}{#3}}
\newcommand{\BAlga}[3]{\putaround{\BAlg}{#1}{#2}{}{#3}}

\newcommand{\RAlgasm}[3]{\putaround{\scriptstyle\RAlg}{#1}{}{#2}{#3}}
\newcommand{\LAlgasm}[3]{\putaround{\scriptstyle\LAlg}{#1}{#2}{#3}{}}
\newcommand{\TAlgasm}[3]{\putaround{\scriptstyle\TAlg}{}{#1}{#2}{#3}}
\newcommand{\BAlgasm}[3]{\putaround{\scriptstyle\BAlg}{#1}{#2}{}{#3}}

\newcommand{\Aa}[4]{\putaround{\A}{#1}{#2}{#3}{#4}}

\newcommand{\Aasm}[4]{\putaround{\scriptstyle\A}{#1}{#2}{#3}{#4}}

\newcommand{\Dnila}[4]{\putaround{\Dnil}{#1}{#2}{#3}{#4}}
\newcommand{\DRnila}[4]{\putaround{\DRnil}{#1}{#2}{#3}{#4}}
\newcommand{\DLnila}[4]{\putaround{\DLnil}{#1}{#2}{#3}{#4}}

\newcommand{\Dnilasm}[4]{\putaround{\scriptstyle\Dnil}{#1}{#2}{#3}{#4}}
\newcommand{\DRnilasm}[4]{\putaround{\scriptstyle
\lcurvearrowup \!
\mathfrak{D}
\! \lcurvearrowdown
}{#1}{#2}{#3}{#4}}

\newcommand{\TRa}[2]{\putaround{\TR}{}{}{#1}{#2}}
\newcommand{\BRa}[2]{\putaround{\BR}{#1}{}{}{#2}}
\newcommand{\TLa}[2]{\putaround{\TL}{}{#1}{#2}{}}
\newcommand{\BLa}[2]{\putaround{\BL}{#1}{#2}{}{}}

\newcommand{\Ava}[2]{\putaround{(\Av)}{#1}{\;}{#2}{\;}}
\newcommand{\Aha}[2]{\putaround{(\Ah)}{\;}{#1}{\;}{#2}}

\newcommand{\RAlgva}[2]{\putaround{(\RAlgv)}{#1}{}{#2}{}}

\newcommand{\TAlgva}[1]{\putaround{(\TAlgv)}{}{}{#1}{}}

\newcommand{\RAlgha}[1]{\putaround{(\RAlgh)}{}{}{}{#1}}

\newcommand{\TAlgha}[2]{\putaround{(\TAlgh)}{}{#1}{}{#2}}

\newcommand{\TRva}[1]{\putaround{(\TRv)}{}{}{#1}{}}

\newcommand{\TRha}[1]{\putaround{(\TRh)}{}{}{}{#1}}

\newcommand{\fa}[2]{\putaround{f}{}{}{#1}{#2}}

\newlength{\heightofbendarg}
\newcommand{\bendme}[5]{
\settoheight{\heightofbendarg}{$\scriptscriptstyle#4$}
\shortstack[r]{$#1\ \underset{#4}{\mathlarger{\mathlarger{\circledast}}}\ #2$\phantom{$\scriptscriptstyle#5$}\vspace{-2.65ex}\\
  \shortstack{$\mathlarger{\mathlarger{\odot}}{\scriptscriptstyle#5}$\\$#3$\phantom{$\scriptscriptstyle#5$}}}
}

\newcommand{\bendyou}[5]{
\shortstack[l]
{$#1$\\
$\mathlarger{\mathlarger{\odot}}{\scriptscriptstyle#4}$\\
$#2\ \underset{#5}{\mathlarger{\mathlarger{\circledast}}}\ #3$
}}

\newcommand{\bendhim}[5]{
\shortstack[r]
{$#1\;\;$\\[3pt]
$\mathlarger{\mathlarger{\odot}}{\scriptscriptstyle#4}$\\[3pt]
$#2\ \underset{\scriptscriptstyle#5}{\mathlarger{\mathlarger{\circledast}}}\ #3\;\;$
}}

\newcommand{\RTModCat}{\putaround{\mathsf{Mod}}{}{}{\RAlg}{\TAlg}}
\newcommand{\RModCat}{\putaround{\mathsf{Mod}}{}{}{\RAlg}{}}
\newcommand{\TModCat}{\putaround{\mathsf{Mod}}{}{}{}{\TAlg}}

\newcommand{\Barbell}{\mathfrak{X}}
\newcommand{\BBh}{\mathfrak{B}_h}
\newcommand{\BBv}{\mathfrak{B}_v}

\newcommand{\Ideal}{\mathfrak{I}}
\newcommand{\otherIdeal}{\mathfrak{K}}

\newcommand{\resp}{respectively }

\newcommand{\aup}{a_\uparrow}
\newcommand{\aleft}{a_\leftarrow}

\newcommand{\aright}{\aleft}
\newcommand{\adown}{\aup}

\DeclareFontFamily{U}  {MnSymbolA}{}
\DeclareFontShape{U}{MnSymbolA}{m}{n}{
    <-6>  MnSymbolA5
   <6-7>  MnSymbolA6
   <7-8>  MnSymbolA7
   <8-9>  MnSymbolA8
   <9-10> MnSymbolA9
  <10-12> MnSymbolA10
  <12->   MnSymbolA12}{}

\DeclareFontShape{U}{MnSymbolA}{b}{n}{
    <-6>  MnSymbolA-Bold5
   <6-7>  MnSymbolA-Bold6
   <7-8>  MnSymbolA-Bold7
   <8-9>  MnSymbolA-Bold8
   <9-10> MnSymbolA-Bold9
  <10-12> MnSymbolA-Bold10
  <12->   MnSymbolA-Bold12}{}

\DeclareSymbolFont{MnSyA}         {U}  {MnSymbolA}{m}{n}
\SetSymbolFont{MnSyA}       {bold}{U}  {MnSymbolA}{b}{n}

\DeclareMathSymbol{\lcurvearrowup}{\mathrel}{MnSyA}{185}
\DeclareMathSymbol{\lcurvearrowdown}{\mathrel}{MnSyA}{187}
\DeclareMathSymbol{\rcurvearrowdown}{\mathrel}{MnSyA}{195}
\DeclareMathSymbol{\rcurvearrowup}{\mathrel}{MnSyA}{193}

\newcommand{\rot}[1]{
\lcurvearrowup \!\!
#1
\!\! \lcurvearrowdown
}

\newcommand{\rotsp}[1]{
\lcurvearrowup
#1
\lcurvearrowdown
}

\newcommand{\rotcc}[1]{
\rcurvearrowdown \!\!
#1
\!\! \rcurvearrowup
}

\newcommand{\rotccsp}[1]{
\rcurvearrowdown
#1
\rcurvearrowup
}

\begin{document}

\title[Cornered Heegaard Floer homology]{Cornered Heegaard Floer homology}

\author[Douglas]{Christopher L. Douglas}
\thanks{CD was partially supported by ESPRC Grant EP/K015478/1.}
\address {Mathematical Institute, University of Oxford\\ Oxford, OX2 6GG, UK}
\email{cdouglas@maths.ox.ac.uk}

\author[Lipshitz]{Robert Lipshitz}
\thanks{RL was partially supported by NSF Grants DMS-0905796 and DMS-1149800, and a Sloan
  Research Fellowship.}
\address{Department of Mathematics, University of North Carolina\\
  Chapel Hill, NC 27599, USA}
\email{lipshitz@math.columbia.edu}

\author[Manolescu]{Ciprian Manolescu}
\thanks{CM was partially supported by NSF Grant number DMS-1104406 and a CNRS Visiting Research Grant at the Jussieu Mathematics Institute (UMR-7586).}
\address {Department of Mathematics, UCLA, 520 Portola Plaza\\ 
Los Angeles, CA 90095, USA}
\email{cm@math.ucla.edu}

\begin{abstract}
Bordered Floer homology assigns invariants to 3-manifolds with boundary, such that the Heegaard Floer homology of a closed 3-manifold, split into two pieces, can be recovered as a tensor product of the bordered invariants of the pieces.  We construct cornered Floer homology invariants of 3-manifolds with codimension-2 corners, and prove that the bordered Floer homology of a 3-manifold with boundary, split into two pieces with corners, can be recovered as a tensor product of the cornered invariants of the pieces.
\end{abstract}

\maketitle

\tableofcontents

\section{Introduction}\label{sec:intro}
Heegaard Floer homology, a holomorphic-curve-based analogue of
Seiberg--Witten Floer homology, was introduced by P.~Ozsv\'ath and
Z.~Szab\'o as a kind of $(3+1)$-dimensional extension of the
Seiberg-Witten
invariant~\cite{OS04:HolomorphicDisks,OS06:HolDiskFour}. One variant
of Heegaard Floer homology associates to each closed, oriented,
connected $3$-manifold $Y$ a chain complex $(\CFa(Y),\bdy)$ over
$\Field$, well-defined up to chain homotopy equivalence, and with homology denoted $\widehat{HF}$, and to each
smooth, oriented, compact cobordism $W$ from $Y_1$ to $Y_2$ a chain
map $\widehat{F}_W\co \CFa(Y_1)\to \CFa(Y_2)$, well-defined up to chain
homotopy. (The variant $\HFa$ does not have enough information to
recover the Seiberg-Witten invariant.)

Bordered Floer homology, introduced by P.~Ozsv\'ath, D.~Thurston, and
the second author is a further downwards extension of
$\HFa$~\cite{LOT1}. To each compact, connected, oriented surface $F$
(plus a little extra data: a representation of $F$ by a ``pointed matched circle'', which encodes a parametrization of the surface), it associates a differential algebra
$\Alg(F)$. To a compact, connected, oriented $3$-manifold $Y$ with boundary $F$ (represented by a ``bordered Heegaard diagram'') it
associates a right $\Ainf$-module $\CFAa(Y)$ over $\Alg(F)$ and a left
dg-module $\CFDa(Y)$ over $\Alg(-F)$. (Here and later, ``$-$" denotes
orientation reversal.) These modules relate to $\CFa(Y)$ via a 
pairing theorem: If $Y=Y_0\cup_F Y_1$ then
\begin{equation}
\label{eq:pair-bordered}
\CFa(Y)\simeq \CFAa(Y_0)\DTP_{\Alg(F)} \CFDa(Y_1),
\end{equation}
where $\DTP$ denotes the derived tensor
product. There are also bimodules associated to $3$-dimensional
cobordisms, satisfying analogous pairing theorems~\cite{LOT2}.

Cornered Floer homology, envisaged by the first and third authors, is
a further extension of Heegaard Floer homology, down to
1-manifolds, surfaces with boundary, and $3$-manifolds with
corners. The first steps in this direction were
taken in \cite{DM:cornered}. To the circle $S^1$ was associated a
$2$-algebra, that is, a vector space with two different
multiplications: a vertical one denoted $\cdot$ and a horizontal one
denoted $*$, which are required to satisfy a certain compatibility
condition. For the case at hand the $2$-algebra used was called the
sequential nilCoxeter $2$-algebra and denoted $\nil$. Next, to a surface $F$ with boundary $S^1$ (represented by a ``pointed matched interval'') 
were associated four kinds of algebra-modules over $\nil$, denoted 
$\LAlgS(F)$, $\RAlgS(F)$, $\TAlgS(F)$, and $\BAlgS(F)$,\footnote{In \cite{DM:cornered}, these were denoted $\LAlg(F)$, $\RAlg(F)$, $\TAlg(F)$, and $\BAlg(F)$. However, here we choose to use the notation without the superscript for a different set of algebra-modules, the ones over the $2$-algebra $\Dnil$, which play a more prominent role in this paper.} and called the left,
right, top, and bottom algebra-modules, respectively. If $F$ is a closed surface decomposed into two pieces as $F=F_0\cup_{S^1} F_1$, there is a ``vertical'' pairing theorem:
\begin{equation}
  \label{eq:alg-pairing-intro}
  \Alg(F)\cong \vertme{\TAlgS(F_0)}{\BAlgS(F_1)}{\nil}{0pt}{0pt},
\end{equation}
where $\odot$ denotes the tensor product with respect to the vertical multiplication. \footnote{The reader may be confused about why the top algebra-module sits at the bottom of the tensor product, and vice versa. The name ``top'' refers to the position of $\nil$ on top of $\TAlgS(F)$. This is consistent with the usual convention that a left (resp. right) module is acted on by an algebra on the left (resp. right).}
The paper \cite{DM:cornered} did not define invariants for
$3$-manifolds with corners, but did give a toy model for the
construction of these invariants, in terms of planar grid diagrams. It
also proved the pairing theorems for this toy model.

The goal of the present paper is to further develop the theory of
cornered Floer homology, by constructing invariants of $3$-manifolds
with codimension-two corners, and proving pairing theorems for these
invariants.

There are two variants of this construction: one is based on the
sequential nilCoxeter $2$-algebra $\nil$ from \cite{DM:cornered},
while the other is based on a related, slightly more complicated
object called the diagonal nilCoxeter $2$-algebra, $\Dnil$. In this
paper we focus on the latter approach, because of two advantages: it
makes the constructions technically easier, and it allows for pairing
theorems in both the horizontal and vertical directions. (By
comparison, there is no ``horizontal'' analogue of formula
\eqref{eq:alg-pairing-intro} over $\nil$: the tensor product of the
algebra-modules $\RAlgS(F_0)$ and $\LAlgS(F_1)$ is not well-defined; see Section~\ref{sec:motility} or \cite[Section 2.4]{DM:cornered} for more details.)

Our setup is as follows. Consider a closed $3$-manifold $Y$, with two decompositions along surfaces $F$ and $F'$:
$$ Y= Y_0 \cup_F Y_1 = Y_0' \cup_{F'} Y_1',$$
where $F$ and $F'$ intersect each other in a circle $S^1$, as in Figure~\ref{fig:Dec1}. The circle cuts the surface $F$ into two pieces $F_0$ and $F_1$, and the surface $F'$ into two other pieces $F'_0$ and $F'_1$. Altogether the surfaces break the $3$-manifold $Y$ into four $3$-manifolds with corners, denoted $Y_{00}, Y_{01}, Y_{10}$ and $Y_{11}$, such that
$$ Y_0 = Y_{00} \cup_{F_0'} Y_{01}, \ \ \ \ Y_1=Y_{10} \cup_{F_1'} Y_{11},$$
$$ Y_0' = Y_{00} \cup_{F_0} Y_{10}, \ \ \ \ Y'_1=Y_{01} \cup_{F_1} Y_{11}.$$

\begin{figure}
\begin{center}
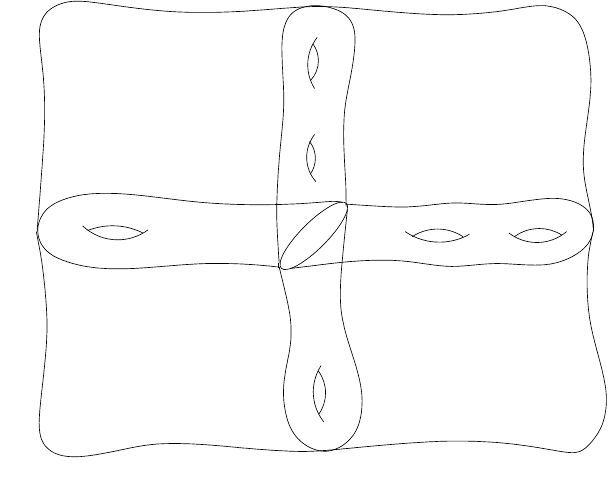
\end{center}
\caption {{\bf Decompositions of a three-manifold $Y$.} 
}
\label{fig:Dec1}
\end{figure}

To the central circle $S^1$ we associate the diagonal nilCoxeter
$2$-algebra $\Dnil$. We choose a representation of each of the surfaces $F_0, F_1,
F_0'$, and $F_1'$ by a pointed matched interval, and then associate to
them algebra-modules over $\Dnil$. These are denoted $\TAlgD(F_0)$,
$\BAlgD(F_1)$, $\RAlgD(F_0')$, and $\LAlgD(F_1')$. We then prove a
pairing theorem similar to formula \eqref{eq:alg-pairing-intro}:
\begin{equation}
  \label{eq:alg-pairing-vert}
  \Alg(F)\cong  \rvertme{\TAlgD(F_0)}{\BAlgD(F_1)}{\Dnil}{0pt}{0pt}.
\end{equation}

Here and later, the symbol $\rvtp$ denotes a variant of the tensor product with respect to vertical multiplication, called the \emph{restricted vertical tensor product}. This is not exactly the tensor product of the two algebra-modules viewed as modules over $\Dnil$. We call that latter construction the {\em full tensor product}
$$\vertme{\TAlgD(F_0)}{\BAlgD(F_1)}{\Dnil}{0pt}{0pt}.$$
The restricted tensor product is a certain summand of the full tensor product. (We refer the reader to sections~\ref{sec:motility} and~\ref{sec:symm-seq} for more details.) In our setting, the full tensor product of $\TAlgD(F_0)$ and $\BAlgD(F_1)$ yields a bigger algebra $\Alg^{\Dnil}(F)$, that depends not only on $F$ but also on its decomposition as $F_1 \cup F_0$.

Unlike in the sequential case, we now also have a horizontal pairing theorem:
\begin{equation}
  \label{eq:alg-pairing-horiz}
  \Alg(F')\cong  \rhormein{\RAlgD(F_0')}{\LAlgD(F_1')}{\Dnil}{1ex}{1ex},
\end{equation}
where $\rhtp$ denotes the restricted horizontal tensor product.  (Here, the restricted horizontal tensor product is naturally an algebra with a vertical multiplication, whereas the algebra $\Alg(F')$ is naturally an algebra with ordinary, horizontal multiplication.  The isomorphism~\eqref{eq:alg-pairing-horiz} is, more precisely, an isomorphism between $\Alg(F')$ and the ordinary algebra associated to the restricted horizontal tensor product by identifying bottom--top multiplication with left--right multiplication; said another way, the algebra $\Alg(F')$ is isomorphic to the ninety-degree clockwise rotation of the restricted horizontal tensor product.)
 
We now turn to $3$-manifolds with corners. Given a $3$-manifold with a
codimension-2 corner, we define four kinds of invariants,
denoted $\CAA, \CDA, \CAD$ and $\CDD$. In our setup, there are four
3-manifolds with a codimension-2 corner, and we will focus on one type of invariant for each of
them: $\CAA$ for $Y_{00}$, $\CDA$ for $Y_{01}$, $\CAD$ for $Y_{10}$,
and $\CDD$ for $Y_{11}$. Each of these invariants is constructed by
starting from a presentation of the $3$-manifold in terms of a
``cornered Heegaard diagram.'' 
 
Our first result is: 

\begin{theorem}\label{thm:invariance}
Suppose we are in the setup from Figure~\ref{fig:Dec1}, and the $3$-manifolds with corners $Y_{ij}$ are represented by cornered Heegaard diagrams $\HD_{ij}, \ i, j\in \{0,1\}$. Associated to these diagrams are differential $2$-modules: 
\begin {align*}
 \CAA(\HD_{00}) &\text{ over } \TAlgD(F_0) \text{ and }\RAlgD(F_0'), &  
 \CAD(\HD_{01}) &\text{ over } \BAlgD(F_1) \text{ and }\RAlgD(F_0'), \\
\CDA(\HD_{10}) &\text{ over } \TAlgD(F_0) \text{ and }\LAlgD(F_1'), &
\CDD(\HD_{11}) &\text{ over } \BAlgD(F_1) \text{ and }\LAlgD(F_1').
 \end{align*}
  If $\HD_{ij}^1$ and $\HD_{ij}^2$ are diagrams representing the same cornered $3$-manifold $Y_{ij}$, then there are quasi-isomorphisms
  \begin{align*}
   \CAA(\HD_{00}^1)&\simeq \CAA(\HD_{00}^2), &
   \CAD(\HD_{01}^1) &\simeq \CAD(\HD_{01}^2), \\
   \CDA(\HD_{10}^1) &\simeq \CDA(\HD_{10}^2), &
   \CDD(\HD_{11}^1) &\simeq \CDD(\HD_{11}^2).
  \end{align*}
\end{theorem}
Because of the invariance statement in Theorem~\ref{thm:invariance},
we are justified in writing $\CAA(Y_{00})$ and so on to denote the cornered
$2$-module computed with respect to any Heegaard diagram for the respective manifold with corners.

Next, we have vertical and horizontal pairing theorems for cornered $3$-manifolds glued along parts of their boundaries:
 
\begin{theorem}\label{thm:pairing}
  With notation as in Theorem~\ref{thm:invariance}, there are
  quasi-isomorphisms of differential modules:
  \begin{align}
    \CFAa(Y_0)&\simeq\ 
    \rvertme{\CAA(Y_{00})}{\CAD(Y_{01})}{\RAlgD(F_0')}{0pt}{0pt}  \label{eq:pair0} \\[6pt]
    \CFDa(Y_1)&\simeq \  \rvertme{\CDA(Y_{10})}{\CDD(Y_{11})}{\LAlgD(F_1')}{0pt}{0pt} \label{eq:pair1}\\[6pt]
     \CFAa(Y_0')&\simeq
     \rhorme{\CAA(Y_{00})}{\CDA(Y_{10})}{\TAlgD(F_0)}{0pt}{0pt}  \label{eq:pair0'}
     \\[3pt]
      \CFDa(Y_1')&\simeq  \rhorme{\CAD(Y_{01})}{\CDD(Y_{11})}{\BAlgD(F_1)}{0pt}{0pt}.  \label{eq:pair1'}
  \end{align}
\end{theorem}

\noindent (Here, the tensor products on the righthand sides of equations~\eqref{eq:pair0'} and~\eqref{eq:pair1'} are naturally top and bottom modules, respectively, over the vertical algebra $\rhormein{\RAlgD(F_0')}{\LAlgD(F_1')}{\Dnil}{1ex}{1ex}$.  As in equation~\eqref{eq:alg-pairing-horiz}, we implicitly rotate that vertical algebra ninety-degrees clockwise so the righthand sides of equations~\eqref{eq:pair0'} and~\eqref{eq:pair1'} become respectively right and left modules---the theorem is that those rotated modules are quasi-isomorphic to $\CFAa(Y_0')$ and $\CFDa(Y_1')$, respectively.)

Combining the decompositions \eqref{eq:pair0} and \eqref{eq:pair1}
with the bordered pairing theorem \eqref{eq:pair-bordered} we obtain a
decomposition of $\CFa(Y)$ into four pieces, corresponding to the four
manifolds-with-corners from Figure~\ref{fig:Dec1}. A similar
decomposition can be obtained by combining formulas \eqref{eq:pair0'} and \eqref{eq:pair1'} with the bordered pairing theorem for $Y$ decomposed as $Y_0' \cup_{F'} Y_1'$.

The proofs of Theorems~\ref{thm:invariance} and \ref{thm:pairing} are based on reduction to the invariance and pairing theorems from bordered Floer homology. The main idea is to replace the manifolds with corners by their smoothings, that is, to smooth the corner in their boundaries. More precisely, for $i, j \in \{0,1\}$, suppose we are given a manifold $Y_{ij}$ with boundaries $(-1)^jF'_i$ and $(-1)^iF_j$ and corner $S^1$. (The signs denote orientations, which are chosen to be consistent with Figure~\ref{fig:Dec1}.) Then, there is a smoothed manifold $\smooth{Y}_{ij}$ with boundary the closed surface
$$F_{ij} := (-1)^j F'_i \cup_{S^1} (-1)^i F_j.$$

Moreover, we can describe the original manifold $Y_{ij}$ as the union of $\smooth{Y}_{ij}$ and a manifold $K_{ij}$ obtained from the cobordism $F_{ij} \times [0,1]$ by introducing the corner $S^1$ into  one of its boundaries:
$$ Y_{ij} = \smooth{Y}_{ij} \cup_{F_{ij}} K_{ij},$$
cf. Figure~\ref{fig:Dec2}.

\begin{figure}
\begin{center}
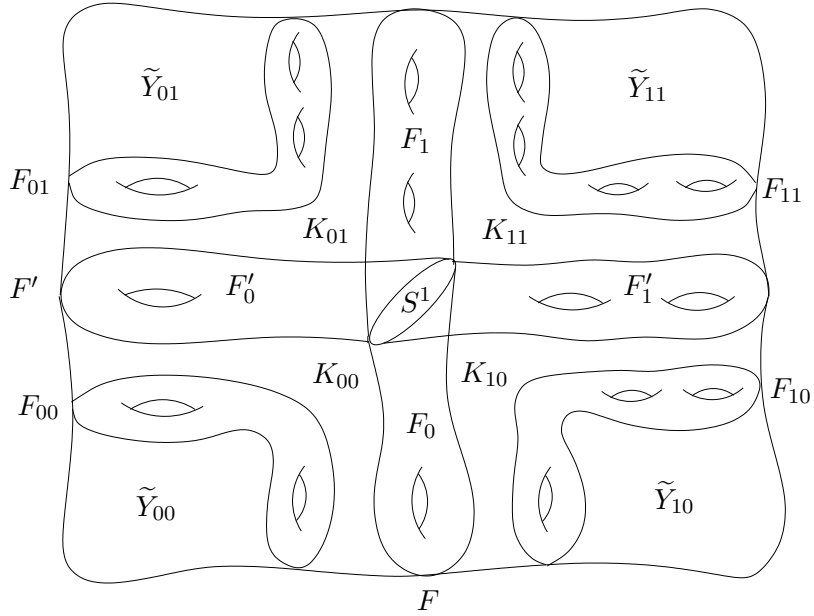
\end{center}
\caption {{\bf Smoothing the corners.} 
}
\label{fig:Dec2}
\end{figure}

The first step in the proof of Theorem~\ref{thm:invariance} is a bit of homological algebra. We show that the notion of $2$-module over two algebra-modules, say $\TAlg = \TAlgD(F_0)$ and $\RAlg = \RAlgD(F_0')$, can be reinterpreted as an ordinary module over a new differential algebra, called the {\em bent tensor product} of $\TAlg$ and $\RAlg$, and denoted $$\TAlg \obent_{\Dnil} \RAlg.$$ Further, the bent tensor product has a distinguished summand $$\TAlg \osmooth_{\Dnil} \RAlg,$$ called the {\em smoothed tensor product}. In our case, the smoothed tensor product is simply the bordered algebra $\Alg(F_{00})$. 

Given a differential module $M$ over the smoothed tensor product, one can obtain
a module over the bent tensor product by tensoring with a canonical
bimodule. If we view the module over the bent tensor product as a
$2$-module, the same operation can be viewed as tensoring with a
module--$2$-module. In the situation at hand, the module--$2$-module is
associated to the cobordism $K_{00}$; we denote it by
$\cornAA(K_{00})$ and call it a {\em cornering module--$2$-module}.  We
now define
\begin{equation}
\label{eq:caa-intro}
 \CAA(Y_{00}) := \CFAa(\smooth{Y}_{00}) \DTP_{\Alg(F_{00})} \cornAA(K_{00}).
 \end{equation}
The module--$2$-module $\cornAA(K_{00})$ is not projective over
$\Alg(F_{00})$ (see Remark~\ref{rem:not-proj}), so
in order to take the derived tensor product $\DTP$ in \eqref{eq:caa-intro},
we need to first choose a projective resolution of one of the two
factors. A convenient resolution for $\CFAa(\smooth{Y}_{00})$ is given
by $\CFDa(\smooth{Y}_{00})$~\cite[Theorem 3]{LOTHomPair}.

Similar constructions can be done for the other cornered $2$-modules---they are obtained
by tensoring the bordered invariants with cornering module--$2$-modules:
\begin{align}
 \CAD(Y_{01}) &:= \CFAa(\smooth{Y}_{01}) \DTP_{\Alg(F_{01})} \cornAD(K_{01}), \label{eq:cad-intro}\\
 \CDA(Y_{10}) &:= \CFAa(\smooth{Y}_{10}) \DTP_{\Alg(F_{10})} \cornDA(K_{10}), \label{eq:cda-intro}\\  
\CDD(Y_{11}) &:= \CFAa(\smooth{Y}_{11}) \DTP_{\Alg(F_{11})} \cornDD(K_{11}).\label{eq:cdd-intro}
\end{align}
It follows easily from this description that the cornered $2$-modules are
well-defined and independent of the diagram, which is the content of Theorem~\ref{thm:invariance}.

Moving on to Theorem~\ref{thm:pairing}, its proof follows from a variant of the pairing theorem in bordered Floer theory. Precisely, let $K_0$ be the union $K_{00} \cup_{F_0'} K_{01}$. This is a $3$-manifold with three boundary components: $F_{00}, F_{01}$, and $F$. By \cite{LOT2}, one can associate to it a bordered trimodule $\CFDDAa(K_0)$. This module can be described explicitly, as it is similar to the bimodule $\CFDD(\Id)$ from \cite{LOT2}. Further, the results of \cite{LOT2} yield a gluing theorem of the form:
$$ \CFAa(Y_0) \simeq \CFAa(\smooth{Y}_{00}) \DTP_{\Alg(F_{00})} \CFDDAa(K_0) \DTP_{\Alg(F_{01})}  \CFAa(Y_{01}).$$

Comparing this with the combination of formulas \eqref{eq:caa-intro} and \eqref{eq:cad-intro}, we see that in order to obtain the cornered pairing formula \eqref{eq:pair0}, it suffices to show that
\begin{equation}
\label{eq:tri-k0}
 \CFDDAa(K_0) \simeq \rvertme{\cornAA(K_{00}).}{\cornAD(K_{01})}{\RAlg(F_0')}{0pt}{0pt}
 \end{equation}

This equation can be checked directly, because all the objects appearing in it admit explicit descriptions. A similar strategy can be applied to deduce the pairing formulas \eqref{eq:pair1}, \eqref{eq:pair0'}, and \eqref{eq:pair1'}.

We remark that the A-type bordered invariants from \cite{LOT1} are $\Ainf$-modules; in contrast, in our construction of the cornered invariants we replace $\Ainf$-modules by their projective resolutions (which are differential modules) and, thus, we avoid working directly with $\twoainf$-structures.
(Nevertheless, we could define an $\twoainf$-module to be an $\Ainf$-module over the bent tensor product; see Section~\ref{sec:2ainf}.) Our method of using bent tensor products may be useful for extending other TQFTs down in dimension as well.

The organization of the paper is as follows. In Section~\ref{sec:2alg}
we give the general definitions of $2$-algebras, algebra-modules, and
$2$-modules, and mention some of their basic properties; we also define the $2$-algebra $\Dnil$. In
Section~\ref{sec:bent} we define the bent tensor product of two
algebra-modules, and establish the correspondence between modules over
the bent tensor product and $2$-modules over the algebra-modules. In
Section~\ref{sec:quasi} we show that the various tensor products used
in this paper respect quasi-isomorphism. In Section~\ref{sec:am} we
turn to the specific examples of algebra-modules that we need, which
are associated to surfaces with boundary. In
Section~\ref{sec:cornering} we define the cornering module--$2$-modules
$\cornAA(K_{00}), \cornAD(K_{01}), \cornDA(K_{10})$, and
$\cornDD(K_{11})$ that appear in
\eqref{eq:caa-intro}--\eqref{eq:cdd-intro}. In
Section~\ref{sec:trimods} we describe the trimodule $\CFDDAa(K_0)$ and
its analogues combinatorially, and also prove equation
\eqref{eq:tri-k0} and its analogues. In Section~\ref{sec:nice} we
define the cornered $2$-modules associated to cornered Heegaard
diagrams and then prove Theorems~\ref{thm:invariance} and
\ref{thm:pairing}. In Section~\ref{sec:gradings} we describe the noncommutative grading structure on the various cornered invariants. In Section~\ref{sec:various} we note that $\Dnil$
is quasi-isomorphic to a finite-dimensional $2$-algebra, give a few concrete
examples of $2$-modules, and show how one can do computations with
these 2-modules. In Section~\ref{sec:conclusion} we suggest some possible
extensions of our results, and describe a relation to planar
algebras.

\subsubsection*{Acknowledgments}
We thank David Nadler, Peter Ozsv\'ath, Rapha\"el Rouquier, and Dylan Thurston for helpful
conversations leading to many of the ideas used here. CM and RL thank the
Simons Center, at which part of this work was undertaken, for its
hospitality. We also thank James Cornish, Ina Petkova, and the referees for helpful comments on previous versions of this paper.

\subsubsection*{Conventions}\label{sec:conventions}
Throughout this paper, our algebras and modules will be defined over the field $\Field$
with two elements. Unless otherwise noted, tensor products are over
$\Field$ as well. 

For clarity of exposition, we work with ungraded complexes for most of the paper; only in Section~\ref{sec:gradings} do we introduce a grading. Thus, in the text before Section~\ref{sec:gradings}, by a \emph{chain complex} we mean an ungraded chain complex, i.e., an $\Field$-vector space $V$ together with a linear map $\bdy\co V\to V$ such that $\bdy^2=0$; similarly, differential algebras and modules are ungraded.


\section{Some abstract 2-algebra}
\label {sec:2alg}

\subsection{Rectangular 2-algebras}
\label{sec:rect2}

\begin {definition}
A (differential) {\em rectangular 2-algebra} $\A=\bigl \{\Aa{m}{n}{p}{q} \; \big \vert \; m, n, p, q \geq 0\bigr \}$ is a collection of chain complexes $\Aa{m}{n}{p}{q}$ over $\Field$, together with chain maps   
$$ \text{\em (horizontal multiplication)} \ \  \Aa{m}{n}{p}{q} \otimes \Aa{\pp m'}{q}{\pp p'}{n'}  \longrightarrow \Aa{\pp m+m'}{n}{\pp p+p'}{n'},  \ \ (a \otimes b) \mapsto a * b $$
for all $m, m', n, n', p, p', q \geq 0,$ and
$$ \text{\em (vertical multiplication)} \ \  \Aa{m}{n}{p}{q} \otimes \Aa{p}{n'}{\pp m'}{q'}  \longrightarrow \Aa{m}{n +n'}{\pp m'}{q+q'},  \ \ (a \otimes b) \mapsto \bmat b \\ \cdot \\ a \emat$$
for all $m, m', n, n', p, q, q' \geq 0.$
These maps are required to satisfy the following associativity conditions:
$$ (a * b) * c = a * (b * c), \ \ \ \ \ 
\bmat c \\ \cdot \\ \pbmat b \\ \cdot \\ a \pemat \emat = \bmat \pbmat c \\ \cdot \\ b \pemat \\ \cdot \\ a \emat,$$
as well as local commutativity:
\begin{equation}
\label{eq:LocComm}
\bmat (c * d) \\ \cdot \\ (a * b) \emat
=
\pbmat c \\ \cdot \\ a \pemat \! * \! \pbmat d \\ \cdot \\ b \pemat,
 \end{equation}
for any $a, b, c, d$ such that the respective operations make sense.

The rectangular 2-algebra $\A$ is called {\em biunital} if there exist elements $e^h_n \in
\Aa{0}{n}{0}{n}$ and $e^v_n \in \Aa{n}{0}{n}{0}$ for all $n \geq 0,$
such that
$$e^h_n * a = a * e^h_m = a, \ \ \ \ \ 
\bmat b \\ \strut \cdot \\ e^v_n \emat = \bmat e^v_m \\ \strut \cdot \\ b \emat = b$$ 
for any $a, b$ such that the respective operations are defined. (Note that we necessarily have $e^h_0=e^v_0$; we denote this element by $e_0$.) Further, we require that for all $m, n \geq 0$, 
$$ e^v_m * e^v_n = e^v_{m+n}, \ \ \ \bmat e^h_n \\[2pt] \cdot \\ e^h_m \emat = e^h_{m+n}.$$
\end {definition}

Henceforth, especially in displayed equations, we often omit the star and dot symbols from products, indicating horizontal multiplication simply by horizontal juxtaposition, and vertical multiplication by vertical juxtaposition.  

Graphically, we can represent an element of $\Aa{m}{n}{p}{q}$ by a box in the
plane, with $m$ marks on its bottom edge, $n$ marks on its left
edge, $p$ marks on its top edge and $q$ marks on its right edge. The
operation $*$ is represented by putting boxes side-by-side, and the
operation $\cdot$ by stacking them vertically (when the number of
marks match). For example, Equation~\eqref{eq:LocComm} can be drawn
as:
\[
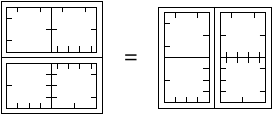
\]

\begin {remark}
  A rectangular 2-algebra is a particular example of a double
  category, in the sense of Ehresmann \cite{Ehresmann63:Double}. A
  \emph{double category} consists of a set of objects $S$; for any
  pair of objects $x,y\in S$ a set $\HomH(x,y)$ of \emph{horizontal
    morphisms} from $x$ to $y$ and a set $\HomV(x,y)$ of
  \emph{vertical morphisms} from $x$ to $y$; for any $x,y,z,w\in S$,
  $f\in\HomH(x,y)$, $g\in\HomV(x,z)$, $h\in\HomV(y,w)$, and $i\in
  \HomH(z,w)$ a set $\Hom(f,g,h,i)$ of \emph{double morphisms} (or
  \emph{squares}), which we visualize as
  \[
  \xymatrix{ z \ar[r]^i  & w\\ \ar @{} [ur] |{\theta\!\!\text{\rotatebox{135}{$\implies$}}} x \ar[u]^g\ar[r]_f & y. \ar[u]_h}
  \]
  A double category also has horizontal (respectively vertical)
  composition maps for horizontal morphisms and double morphisms
  (respectively vertical morphisms and double morphisms) and units for
  these compositions, satisfying various axioms, including an
  interchange law between the vertical and horizontal compositions of
  squares.

  (A $2$-category $\mathcal{C}$ induces a double category with both the
  horizontal and vertical morphisms given by the $1$-morphisms in
  $\mathcal{C}$, and a double morphism for each $2$-morphism $h\circ
  f\Rightarrow i\circ g$ in $\mathcal{C}$. There are also other ways to
  regard a $2$-category as a double category.)

  A rectangular 2-algebra $\A$ is a double category with only one object $s$, such that the set of horizontal maps from $s$ to $s$ is the monoid of nonnegative integers, and the same is true for the set of vertical maps. The elements of $\Aa{m}{n}{p}{q}$ are the squares going between the maps $m$ on the bottom, $n$ on the left, $p$ on the top and $q$ on the right.
\end{remark}

We now give two examples of (biunital) rectangular 2-algebras. 

\begin {example}
\label {ex:one}
Following \cite{DM:cornered}, let us define a {\em sequential 2-algebra} to be a particular kind of rectangular 2-algebra, in which $\Aa{m}{n}{p}{q} = 0$ unless $n=q=0$ and $m=p.$ For example, the {\em sequential nilCoxeter 2-algebra} $\nil$, which played an important role in \cite{DM:cornered}, can be viewed as a rectangular 2-algebra. Let us recall its construction.

Let $n \geq 0.$ The {\em nilCoxeter algebra} $\nil_n$ is defined as the unital $\Field$-algebra generated by elements $\sigma_i, \ i=1, \dots, n-1,$ subject to the relations
\begin {eqnarray}
 \sigma_i^2 &=& 0, \label{eq:nc1} \\
 \sigma_i \sigma_j &=& \sigma_j \sigma_i  \ \ \text{for } \ |i - j | \geq 2,\label{eq:nc2} \\
 \sigma_i \sigma_{i+1} \sigma_i &=& \sigma_{i+1} \sigma_i \sigma_{i+1}. \label{eq:nc3}
\end {eqnarray}
(By definition, $\nil_0=\Field$.)
The nilCoxeter algebra admits an $\Field$-basis $\{\sigma_w\}_{w \in S_n}$, indexed by the elements of the symmetric group. The differential on $\nil_n$ is defined by setting $\del \sigma_i =1,$ and then extending it by the Leibniz rule. We then define the {\em sequential nilCoxeter 2-algebra} $\nil$ to be composed of the pieces $\nil_n$, with their intrinsic product $\cdot$ viewed as the vertical multiplication, and the second product $* : \nil_n \otimes \nil_m \to \nil_{n+m}$ given by concatenation $\sigma_i \otimes \sigma_j \to \sigma_i \sigma_{n+j}.$ Observe that $\nil$ is biunital. 

Graphically, we represent generators of $\nil_n$ as pictures with $n$
strands going up, from the bottom edge to the top edge, such that
$\sigma_i$ corresponds to an interchange between the $i\th$ and the
$(i+1)\th$ strands. The algebra multiplication $\cdot$ is given by
stacking pictures vertically, and the second multiplication $a * b$ is
stacking pictures horizontally. Local versions of the defining relations of the
nilCoxeter algebra are shown in Figure~\ref{fig:nilcoxrel}. 
\end {example} 

\begin{figure}
\begin{center}
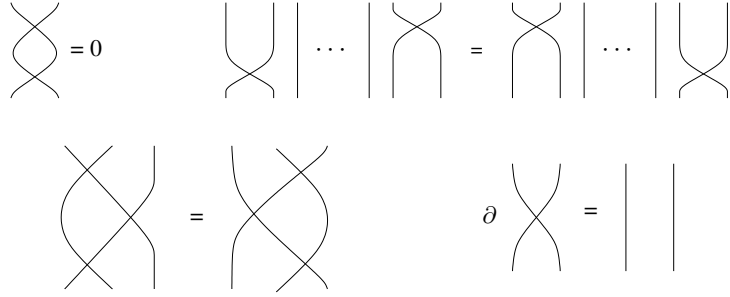
\end{center}
\caption {{\bf Relations in the nilCoxeter algebra $\nil_n$.} 
}
\label{fig:nilcoxrel}
\end{figure}

Our second example will play an important role in this paper, being the invariant associated to the circle:

\begin {definition} \label{def:ex2}
We define the {\em diagonal nilCoxeter 2-algebra}, denoted $\Dnil$, as
follows. The piece $\Dnila{m}{n}{p}{q}$ is zero unless $m+q = n+p.$
If $m+q=n+p=s,$ we let  
$$ \Dnila{m}{n}{p}{q} = \nil_s,$$
with the generators drawn as $s$ strands joining the union of the bottom and right edges to the union of the left and top edges. The strands are oriented in the general upward-leftward direction, as in Figure~\ref{fig:diag}. (In future pictures we will stop indicating the arrows, as they are implicit in the convention that every strand goes from the bottom or right edge to the top or left edge.)  The first $m$ initial points are on the bottom edge, the other $q$ initial points on the right, the first $n$ end points on the left, and the other $p$ end points on the top. The two multiplications are  given by horizontal and vertical concatenation of boxes, with the same relations as in Figure~\ref{fig:nilcoxrel}. In particular, the differential is given by the sum over the intersection points of the oriented resolution at that point.
\end {definition}

\begin{figure}
\begin{center}
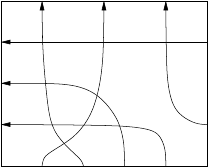
\end{center}
\caption {{\bf An element of the diagonal nilCoxeter $2$-algebra $\Dnil$.}}
\label{fig:diag}
\end{figure}

Let us go back to the case of a general biunital rectangular 2-algebra $\A$. We can formally set the two multiplications to produce zero when the number of marks does not match. This induces  two structures of a dg-algebra ($*$ and $\cdot$) on the infinite direct sum 
\[
\Aa{\sbull}{\sbull}{\sbull}{\sbull} := \bigoplus_{m,n,p,q} \Aa{m}{n}{p}{q}\ .
\] 
The two multiplications on $\Aa{\sbull}{\sbull}{\sbull}{\sbull}$ are associative, but they are not unital, and only commute in a local sense---that is, when the number of marks matches. Otherwise, for example,
\[
\bmat (e_1^v * e_0) \\ \cdot \\ (e_0 * e_1^v) \emat
= \bmat e_1^v \\ \cdot \\ e_1^v \emat = e_1^v
\ \ \ \text{but} \ \ \ 
\pbmat e_1^v \\ \cdot \\ e_0 \pemat \! * \! \pbmat e_0 \\ \cdot \\ e_1^v \pemat = 0 * 0 = 0.
\]

When there is no chance of confusion, we will sometimes write $\A$ to mean $\Aa{\sbull}{\sbull}{\sbull}{\sbull}$.

\begin{remark}
  In \cite{DM:cornered}, where only sequential $2$-algebras were
  considered, direct products were used instead of direct sums. The
  advantage of direct sums is that we avoid potentially infinite sums
  in the product operations; the disadvantage is that many of the (direct sum)
  algebras under consideration will not have units.
\end{remark}

We will sometimes need to take direct sums only over some of the indices in $\A$, with the other indices remaining fixed. When we do that, we will use the bullet $\sbull$ to denote the variable indices, and we will then drop the direct sum symbol from the notation. For example,
$$\Aa{\sbull}{\sbull}{p}{\sbull} := \bigoplus_{m,n,q} \Aa{m}{n}{p}{q}, \ \ \ \ \Aa{m}{n}{p}{\sbull} := \bigoplus_{q} \Aa{m}{n}{p}{q}.$$

\subsection {Rectangular algebra-modules and 2-modules}
\label {sec:rects}

Let $\A$ be a biunital rectangular 2-algebra. 

\begin {definition}
\label{def:amods}
A (differential, rectangular) {\em top algebra-module} $\TAlg$ over $\A$ is a collection of chain complexes $\bigl\{ \TAlga{n}{m}{q} \; \big \vert \; m, n, q \geq 0 \bigr\}$ together with chain maps:
\begin{align*}
 * : \TAlga{n}{m}{q} \otimes \TAlga{q}{\pp m'}{n'} &\to \TAlga{n}{\pp m+m'}{n'},\\
 \cdot : \TAlga{n}{m}{q} \otimes \Aa{m}{n'}{p}{q'} &\to
 \TAlga{n+n'}{p}{q+q'},
\end{align*}
satisfying associativity and local commutation relations:
\begin{align}
\label{eq:alc1}
(\phi \; \psi) \; \zeta &= \phi \; (\psi \; \zeta), \hskip1cm \forall \phi \in \TAlga{n}{m}{q},\ \psi \in \TAlga{q}{\pp m'}{n'},\ \zeta \in \TAlga{n'}{\pp\pp m''}{n''}, \\
\label{eq:alc2}
\bmat \pbmat b \\ a \pemat \\ \phi \emat &= \bmat b \\ \pbmat a \\ \phi \pemat \emat,
\hskip1.2cm \forall  \phi \in \TAlga{n}{m}{q},\ a \in \Aa{m}{n'}{p}{q'},\ b \in \Aa{p}{n''}{\pp p'}{q''},  \\
\label{eq:alc3} 
\bmat \pbmat a \;\, b \pemat \\[.5ex] \pbmat \phi \; \psi \pemat \emat
&= \pbmat a \\ \phi \pemat \!\pbmat b \\ \psi \pemat,
\ \ \forall \phi \in \TAlga{n}{m}{q},\ \psi \in \TAlga{q}{\pp m'}{s},\ a \in \Aa{m}{n'}{p}{q'},\ b \in \Aa{\pp m'}{q'}{\pp p'}{s'}.
\end{align}
Here horizontal juxtaposition of elements indicates an application of
the $*$ product, and vertical juxtaposition indicates an application of
the $\cdot$ product.  The algebra-module $\TAlg$ is called {\em
  unital} if $\bmat e_m^v \\ \cdot \\ \phi \emat = \phi$ for all $\phi \in
\TAlga{n}{m}{q}$, and there is an element $1^h_0 \in \TAlga{0}{0}{0}$ so that for each
$n$, the element $1^h_n\coloneqq \bmat e^h_n\\ \cdot \\1^h_0 \emat \in
\TAlga{n}{0}{n}$ acts as a horizontal unit (i.e., $\phi * 1^h_q = 1^h_n
* \phi= \phi$ for all $\phi \in \TAlga{n}{m}{q}$).

A {\em right algebra-module} $\RAlg$ over $\A$ is a collection of
chain complexes $\bigl\{ \RAlga{m}{p}{n} \; \big \vert \; m, n, p \geq 0 \bigr \}$ together with
chain maps:
\begin{align*}
  * : \RAlga{m}{p}{n} \otimes \Aa{m'}{n}{p'}{q} &\to \RAlga{m+m'}{p+p'}{q},\\
  \cdot : \RAlga{m}{p}{n} \otimes \RAlga{p}{m'}{n'} &\to \RAlga{m}{m'}{n+n'},
\end{align*}
satisfying associativity and local commutation relations similar to
\eqref{eq:alc1}--\eqref{eq:alc3}. $\RAlg$ is called {\em unital} if
$\phi * e_n^h = \phi$ for all $\phi \in \RAlga{m}{n}{p}$, and there is an
element $1^v_0 \in \RAlga{0}{0}{0}$ so that for each $m$, the element
$1^v_m\coloneqq 1^v_0 * e^v_m   \in \RAlga{m}{m}{0}$ acts as a
vertical unit.

A {\em bottom algebra-module} $\BAlg$ over $\A$ is a collection of chain complexes $\bigl \{ \BAlga{p}{n}{q} \; \big \vert \; n, p, q \geq 0\bigr \}$ together with chain maps:
\begin{align*}
* : \BAlga{p}{n}{q} \otimes \BAlga{\pp p'}{q}{n'} &\to \BAlga{\pp p+p'}{n}{n'},\\
 \cdot : \Aa{m}{n}{p}{q} \otimes \BAlga{p}{n'}{q'} &\to
 \BAlga{m}{n+n'}{q+q'},
\end{align*}
satisfying associativity and local commutation relations similar to
\eqref{eq:alc1}--\eqref{eq:alc3}. $\BAlg$ is called {\em unital} if
$\bmat \phi \\ \cdot \\ e_p^v \emat = \phi$ for all $\phi \in
\BAlga{p}{n}{q}$, and there is an element $1^h_0 \in \BAlga{0}{0}{0}$ so that for each
$n$, the element $1^h_n\coloneqq \bmat 1^h_0 \\
\cdot \\ e^h_n \emat  \in \BAlga{0}{n}{n}$ acts as a horizontal unit.

A {\em left algebra-module} $\LAlg$ over $\A$ is a collection of chain complexes $\bigl \{ \LAlga{m}{q}{p} \; \big \vert \; m, p, q \geq 0 \bigr \}$ together with chain maps:
\begin{align*}
 * : \Aa{m}{n}{p}{q} \otimes \LAlga{\pp m'}{q}{\pp p'} &\to \LAlga{\pp m+m'}{n}{\pp p+p'},\\
 \cdot : \LAlga{m}{q}{p} \otimes \LAlga{p}{q'}{\pp m'} &\to \LAlga{m}{q+q'}{\pp m'},
\end{align*}
satisfying associativity and local commutation relations similar to
\eqref{eq:alc1}--\eqref{eq:alc3}. $\LAlg$ is called {\em unital} if
$e_q^h * \phi= \phi$ for all $\phi \in \LAlga{m}{q}{p}$, and there is an
element $1^v_0 \in \LAlga{0}{0}{0}$ so that for each $m$, the element
$1^v_m\coloneqq e^v_m * 1^v_0   \in \LAlga{m}{0}{m}$
acts as a vertical unit.
\end {definition}

To visualize these definitions, it helps to think of the generators of the algebra-modules as boxes, as in Figure~\ref{fig:mods}. The elements of the 2-algebra act on them in one direction (by horizontal or vertical concatenation), and the algebra-modules have their own multiplication (also concatenation) in the other direction.

\begin{figure}
\begin{center}
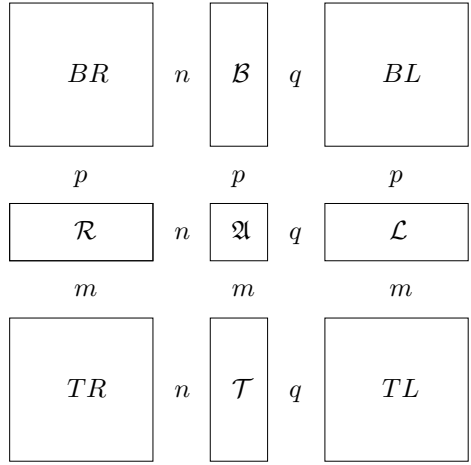
\end{center}
\caption {{\bf A rectangular 2-algebra, algebra-modules, and 2-modules.} The notation refers to the relative position of the 2-algebra $\A$ with respect to the object. For example, the top-right 2-module $\TR$ has $\A$ near its top right corner.
}
\label{fig:mods}
\end{figure}

Let now $\A$ be a biunital rectangular 2-algebra, and $\TAlg, \RAlg, \BAlg, \LAlg$ be top, right, bottom, and left unital algebra-modules over $\A,$ respectively.

\begin {definition} \label {def:2mods}
A (differential) {\em top-right 2-module} $\TR$ over $\RAlg$ and $\TAlg$ is a collection of chain complexes $\bigl\{\TRa{m}{n} \; \big \vert \; m, n \geq 0 \bigr \}$ together with chain maps  
\begin{align*}
 * : \TRa{m}{n} \otimes \TAlga{n}{\pp m'}{q} &\to \TRa{m+m'}{q},\\
 \cdot : \TRa{m}{n} \otimes \RAlga{m}{p}{n'} &\to \TRa{p}{n+n'},
\end{align*}
satisfying associativity and local commutation relations as follows:
\begin {align}
\label {eq:alc1'}
(\x \; \phi) \; \psi &= \x \; (\phi \; \psi), 
\hskip.9cm \forall \x \in \TRa{m}{n},\phi \in \TAlga{n}{\pp m'}{q}, \psi \in \TAlga{q}{\pp\pp m''}{q'}, \\
\label{eq:alc2'}
\bmat \pbmat \psi \\ \phi \pemat \\ \strut \x \emat &= \bmat \psi \\[.5ex] \pbmat \phi \\[.8ex] \x \pemat \emat,
\hskip1.4cm \forall  \x \in \TRa{m}{n},\phi \in \RAlga{m}{p}{n'}, \psi \in \RAlga{p}{\pp p'}{n''},  \\
\label{eq:alc3'} 
\bmat \pbmat \psi \;\, a \pemat \\[.5ex] \pbmat \x \; \phi \pemat \emat
&= \pbmat \psi \\[.5ex] \x \pemat \!\pbmat a \\[.5ex] \phi \pemat,
\hskip .6cm \forall \x \in \TRa{m}{n}, \phi \in \TAlga{n}{\pp m'}{q}, \psi \in \RAlga{m}{p}{n'}, a \in \Aa{\pp m'}{n'}{\pp p'}{q'}.
\end{align}
The 2-module $\TR$ is called {\em unital} if it further satisfies
\begin{equation} 
\label{eq:alc4'} 
(\x \; 1^h_n) =  \pbmat 1^v_m  \\[.7ex] \x \pemat = \x, \hskip.5cm \forall \x \in \TRa{m}{n}.
\end {equation}

A (differential) {\em bottom-right 2-module} $\BR$ over $\RAlg$ and $\BAlg$ is a collection of chain complexes $\bigl \{\BRa{p}{n} \; \big \vert \; n,p \geq 0 \bigr \}$ together with chain maps  
\begin{align*}
 * : \BRa{p}{n} \otimes \BAlga{\pp p'}{n}{q} &\to \BRa{p+p'}{q},\\
 \cdot : \RAlga{m}{p}{n} \otimes \BRa{p}{n'} &\to \BRa{m}{n+n'},
\end{align*}
satisfying associativity and local commutation relations similar to \eqref{eq:alc1'}-\eqref{eq:alc3'}.

A  (differential) {\em bottom-left 2-module} $\BL$ over $\LAlg$ and $\BAlg$ is a collection of chain complexes $\bigl \{\BLa{p}{q} \; \big \vert \; p,q \geq 0 \bigr \}$ together with chain maps  
\begin{align*}
 * : \BAlga{p}{n}{q} \otimes \BLa{p'}{q} &\to \BLa{p+p'}{n},\\
 \cdot : \LAlga{m}{q}{p} \otimes \BLa{p}{q'} &\to \BLa{m}{q+q'},
\end{align*}
satisfying associativity and local commutation relations similar to \eqref{eq:alc1'}-\eqref{eq:alc3'}.

A (differential) {\em top-left 2-module} $\TL$ over $\LAlg$ and $\TAlg$ is a collection of chain complexes $\bigl \{\TLa{q}{m} \; \big \vert \; m, q \geq 0 \bigr \}$ together with chain maps  
\begin{align*}
 * : \TAlga{n}{m}{q} \otimes \TLa{q}{m'} &\to \TLa{n}{m+m'},\\
 \cdot : \TLa{q}{m} \otimes \LAlga{m}{q'}{p} &\to \TLa{q+q'}{p},
\end{align*}
satisfying associativity and local commutation relations similar to \eqref{eq:alc1'}-\eqref{eq:alc3'}.

Unitality for bottom-right, bottom-left, and top-left 2-modules is defined by imposing relations analogous to \eqref{eq:alc4'}. 
\end {definition}

Again, it helps to look at Figure~\ref{fig:mods} to visualize these definitions. 

\subsection{Motility hypotheses and tensor products}
\label{sec:motility}

Let $\A$ be a biunital rectangular 2-algebra, and $\TAlg, \RAlg, \BAlg, \LAlg$ be top, right, bottom, and left unital algebra-modules over $\A,$ respectively. Let also $\TR, \BR, \BL, \TL$ be unital $2$-modules over $\TAlg, \RAlg, \BAlg, \LAlg$, as in Definition~\ref{def:2mods}.

We will need to take direct sums of certain pieces of these algebra-modules and $2$-modules. We then use the same conventions as at the end of Section~\ref{sec:rect2}, with variable indices denoted by bullets. For example,
$$\TAlga{\sbull}{\sbull}{\sbull} := \bigoplus_{m,n,q} \TAlga{m}{n}{q}, \ \ \ \RAlga{m}{\sbull}{n} := \bigoplus_{p} \RAlga{m}{p}{n}, \ \ \  \TLa{q}{\sbull} := \bigoplus_{m} \TLa{q}{m}\ .$$

Observe that the direct sum $\TAlga{\sbull}{\sbull}{\sbull}$ is an
algebra with respect to horizontal multiplication, and a module over
$\Aa{\sbull}{\sbull}{\sbull}{\sbull}$ with respect to vertical
multiplication. (In particular, when indices do not match the
corresponding products are defined to vanish.
Note that these two structures are not fully
compatible, since the horizontal and vertical multiplications only
commute locally.) The same goes for $\BAlga{\sbull}{\sbull}{\sbull}$. 

The tensor product of $\TAlga{\sbull}{\sbull}{\sbull}$ and $\BAlga{\sbull}{\sbull}{\sbull}$ with respect to vertical multiplication is called the {\em vertical tensor product} of the algebra-modules $\TAlg$ and $\BAlg$ over the $2$-algebra $\A$:
$$ \vertme{\TAlg}{\BAlg}{\A}{0ex}{0.5ex} := \vertme{\TAlga{\sbull}{\sbull}{\sbull}}{\BAlga{\sbull}{\sbull}{\sbull}}{\Aasm{\sbull}{\sbull}{\sbull}{\sbull}}{1.5ex}{.5ex}.$$

We have a decomposition
\begin{equation}
\label{eq:algdec}
 \vertme{\TAlg}{\BAlg}{\A}{0ex}{0.5ex} =   \bigoplus_{n, q} \ \scriptstyle{n} \! \left( \vertme{\TAlg}{\BAlg}{\A}{0ex}{0.5ex} \right)  \! \scriptstyle{q}
 \end{equation}
where the local piece 
$$\scriptstyle{n} \! \left( \vertme{\TAlg}{\BAlg}{\A}{0ex}{0.5ex} \right)  \! \scriptstyle{q}$$ is generated by
elements of the form $ \bs \psi \\ \todot \\ \phi \es$ with $\phi \in
\TAlga{n'}{m}{q'},  \psi \in \BAlga{p}{n''}{q''},$ such that $n'+n'' = n$
and $q'+q'' = q.$ Further, we can assume that that $m=p$, for otherwise 
$$  \bnmat \psi \\ \todot \\ \phi \enmat= \bnmat \psi \\ \todot \\ \pbs e^v_m \\[.6ex] \phi  \pes \enmat =\bnmat \pbs \psi \\[.4ex] e^v_m \pes \\ \todot \\ \phi  \enmat  =0.$$

Similarly, $\RAlga{\sbull}{\sbull}{\sbull}$ and $\LAlga{\sbull}{\sbull}{\sbull}$ are modules over $\Aa{\sbull}{\sbull}{\sbull}{\sbull}$ with respect to the horizontal multiplication. We define their {\em horizontal tensor product} to be
$$ \horme{\RAlg}{\LAlg}{\A}{0ex}{0ex}  := \horme{\RAlga{\sbull}{\sbull}{\sbull}}{\LAlga{\sbull}{\sbull}{\sbull}}{\Aasm{\sbull}{\sbull}{\sbull}{\sbull}}{.25ex}{.25ex}.$$

One might naively expect that the two horizontal algebra structures on $\TAlga{\sbull}{\sbull}{\sbull}$ and $\BAlga{\sbull}{\sbull}{\sbull}$ combine to produce an algebra structure on $\smallvertme{\TAlg}{\BAlg}{\A}{-0.6ex}{-0.3ex}$, and that the two vertical algebra structures on $\RAlga{\sbull}{\sbull}{\sbull}$ and $\LAlga{\sbull}{\sbull}{\sbull}$ to combine into an algebra structure on $ \smallhorme{\RAlg}{\LAlg}{\A}{0ex}{0ex}$. As explained in \cite[Section 2.4]{DM:cornered}, this is not in general the case. The trouble is the following. For $\phi \in \TAlga{n}{m}{q},  \psi \in \BAlga{m}{s}{u}, \phi' \in \TAlga{q'}{\pp m'}{n'}$ and $\psi' \in \BAlga{\pp m'}{u'}{s'}$, if we try to set
$$\pbnmat \psi\vphantom{'} \\ \todot \\ \phi\vphantom{'} \penmat \pbnmat \psi' \\ \todot \\ \phi' \penmat
:=
\bnmat (\psi \; \psi') \\ \todot \\ (\phi \; \phi') \enmat,$$
then the product will be zero unless the indices on the right of $\psi$ and $\phi$ match with the indices on the left of $\psi'$ and $\phi'$; that is, unless $q=q'$ and $u=u'$. However, when $(q, u) \neq (q', u')$, we might be able to alternately express $ \bs \psi \\ \todot \\ \phi \es$ as $ \bs \psi'' \\ \todot \vphantom{''} \\ \phi'' \es$, with $\phi'' \in \TAlga{n''}{\pp\pp m''}{q''},  \psi'' \in \BAlga{\pp\pp m''}{s''}{u''}$, where $(q'', u'')=(q', u').$ We would then not expect the product 
$$\pbnmat \psi'' \\ \todot \\ \phi'' \penmat \pbnmat \psi' \\ \todot \\ \phi' \penmat
=
\bnmat (\psi'' \; \psi') \\ \todot \\ (\phi'' \; \phi') \enmat$$
to be zero.

To obtain well-defined algebra structures on these full tensor products, we need some additional hypotheses:

\begin{definition}
\label{def:motileAM}
  We say that a top algebra-module $\TAlg$ over $\A$ satisfies the
  \emph{motility hypothesis} if for any non-negative
  integers $n$, $p$, and $q$, the multiplication map
\begin{equation}
\label{eq:motT}
\cdot \co\bigoplus_{m}\bigl(\TAlga{0}{m}{0} \otimes \Aa{m}{n}{p}{q}\bigr)\to \TAlga{n}{p}{q}
\end{equation}
is surjective. 
  
Similar definitions apply to right, bottom, and top algebra-modules  $\RAlg$, $\BAlg$, and $\TAlg$. Specifically, the motility hypothesis requires surjectivity of the maps:
\begin{align*}
\ast   \co & \bigoplus_{n} \bigl(\RAlga{0}{0}{n} \otimes \Aa{m}{n}{p}{q}\bigr)\to \RAlga{m}{p}{q} \ , \\
\cdot \co & \bigoplus_{p} \bigl(\Aa{m}{n}{p}{q}\otimes \BAlga{p}{0}{0} \bigr)\to \BAlga{m}{p}{q}  \ ,\\
\ast \co & \bigoplus_{q}  \bigl(\Aa{m}{n}{p}{q} \otimes \LAlga{0}{q}{0} \bigr)\to \LAlga{m}{n}{p} \ .
\end{align*}
\end{definition}

Note that surjectivity of the map \eqref{eq:motT} is equivalent to that of the map
\[
\ast \co \TAlga{0}{\sbull}{0} \otimes \Aa{\sbull}{n}{p}{q} \to \TAlga{n}{p}{q},
\]
and similarly for $\RAlg$, $\BAlg$, and $\LAlg$.

\begin{lemma}
\label{lem:AMweak}
  Suppose that $\TAlg$ and $\BAlg$ satisfy the motility hypotheses. Then
  the vertical tensor product $$\Alg= \vertme{\TAlg}{\BAlg}{\A}{0ex}{0.5ex} $$ has an
  induced structure of a differential algebra with respect to
  horizontal multiplication.
  
  Similarly, if $\RAlg$ and $\LAlg$ satisfy the motility hypotheses, then their horizontal tensor product 
  $$\Alg' = \horme{\RAlg}{\LAlg}{\A}{0ex}{0ex}$$ has an induced structure of a differential algebra with respect to vertical multiplication.
\end{lemma}

\begin{proof}
 For $\phi \in \TAlga{n}{m}{q},  \psi \in \BAlga{m}{s}{u}, \phi' \in \TAlga{q}{\pp m'}{n'}$, and $\psi' \in \BAlga{\pp m'}{u}{s'},$ we set
\begin{equation}
\label{eq:locom1}
\pbnmat \psi\vphantom{'} \\ \todot \\ \phi\vphantom{'} \penmat \pbnmat \psi' \\ \todot \\ \phi' \penmat
:=
\bnmat (\psi \; \psi') \\ \todot \\ (\phi \; \phi') \enmat.
 \end{equation}

More generally, suppose $\phi \in \TAlga{n}{m}{q}$,  $\psi \in \BAlga{m}{s}{u}$, $\phi' \in \TAlga{q'}{\pp m'}{n'}$, and $\psi' \in \BAlga{\pp m'}{u'}{s'}$. We set $\pbs \psi \\ \todot \\ \phi \pes  \pbs \psi' \\ \odot \vphantom{'} \\ \phi' \pes=0$ if $q+u \neq q' + u'$. When $q+u = q'+u'$ but $(q, u) \neq (q', u')$, using the formula \eqref{eq:locom1} directly would result in zero, which is not what we want. Rather, we move all nonzero indices away from $\TAlg$, and then use \eqref{eq:locom1}. Specifically, using the motility hypothesis for $\TAlg$, we write 
\begin{align*}
 \phi &= \pbnmat a \\ \zeta \penmat, 
 \ \ \zeta \in \TAlga{0}{\sbull}{0}, \ a \in \Aa{\sbull}{n}{m}{q},\\
\phi' &= \pbnmat a'  \\ \zeta' \penmat, 
\ \ \zeta' \in \TAlga{0}{\sbull}{0}, \ a' \in \Aa{\sbull}{q'}{m'}{n'}
\end{align*}
and define
\begin{equation}
\label{eq:zeta}
\pbnmat \psi\vphantom{'} \\ \todot \\ \phi\vphantom{'} \penmat \pbnmat \psi' \\ \todot \\ \phi' \penmat
:=
\bnmat \left( \pbnmat \psi\vphantom{'} \\ a \penmat \pbnmat \psi' \\ a' \penmat \right) \\ \todot \\ (\zeta \; \zeta') \enmat.
 \end{equation}

Equivalently, we could move all nonzero indices away from $\BAlg$, and then use  \eqref{eq:locom1}. Precisely, we write 
\begin{align*}
 \psi &=  \pbnmat \eta \\ b \penmat, \ \ b \in \Aa{m}{s}{\sbull}{u}, \ \ \eta \in \BAlga{\sbull}{0}{0}, \\
\psi' &= \pbnmat \eta' \\ b' \penmat, \ \ b' \in \Aa{\pp m'}{u'}{\sbull}{s'}, \ \eta' \in \BAlga{\sbull}{0}{0},
\end{align*}
and set
\begin{equation}
\label{eq:eta}
\pbnmat \psi\vphantom{'} \\ \todot \\ \phi\vphantom{'} \penmat \pbnmat \psi' \\ \todot \\ \phi' \penmat
:=
\bnmat (\eta \; \eta') \\ \todot \\ \left( \pbnmat b\vphantom{'} \\ \phi \penmat \pbnmat b' \\ \phi' \penmat \right) \enmat.
 \end{equation}

This definition is the same as \eqref{eq:zeta}, because
\begin{multline*}
\bs
\pbxs \pbs \psi\vphantom{'} \\ a\vphantom{'} \pes & \pbs \psi' \\ a' \pes \pexs \\
\todot \\
\pbs \zeta & \zeta' \pes
\es
= 
\bs
\pbxs \pbs \pbs \eta\vphantom{'} \\ b\vphantom{'} \pes \\ a\vphantom{'} \pes & \pbs \pbs \eta' \\ b' \pes \\ a' \pes \pexs \\
\todot \\
\pbs \zeta & \zeta' \pes
\es
= 
\bs
\pbxs \pbs \eta\vphantom{'} \\ \pbs b\vphantom{'} \\ a\vphantom{'} \pes \pes & \pbs \eta' \\ \pbs b' \\ a' \pes \pes \pexs \\
\todot \\
\pbs \zeta & \zeta' \pes
\es
=
\bs
\pbs \pbs \eta\vphantom{'} & \eta' \pes \\ \pbxs \pbs b\vphantom{'} \\ a\vphantom{'} \pes & \pbs b' \\ a' \pes \pexs \pes \\
\todot \\
\pbs \zeta & \zeta' \pes
\es
\\[12pt]
= 
\bs
\pbs \eta & \eta' \pes \\
\todot \\
\pbs \pbxs \pbs b\vphantom{'} \\ a\vphantom{'} \pes & \pbs b' \\ a' \pes \pexs \\ \pbs \zeta \vphantom{'} & \zeta' \pes \pes
\es
= 
\bs
\pbs \eta & \eta' \pes \\
\todot \\
\pbxs \pbs \pbs b\vphantom{'} \\ a\vphantom{'} \pes \\ \zeta\vphantom{'} \pes & \pbs \pbs b' \\ a' \pes \\ \zeta' \pes \pexs
\es
= 
\bs
\pbs \eta & \eta' \pes \\
\todot \\
\pbxs \pbs b\vphantom{'} \\ \pbs a\vphantom{'} \\ \zeta \vphantom{'} \pes \pes & \pbs b' \\ \pbs a' \\ \zeta' \pes \pes \pexs
\es
= 
\bs
\pbs \eta & \eta' \pes \\
\todot \\
\pbxs \pbs b\vphantom{'} \\ \phi\vphantom{'} \pes & \pbs b' \\ \phi' \pes \pexs
\es
\end{multline*}

The fact that we have the equivalent definition \eqref{eq:eta} shows that the multiplication \eqref{eq:zeta} is well-defined: if we had decompositions $\phi = \pbs \tilde  a \\ \tilde \zeta \pes$ instead of $\pbs a \\ \zeta \pes,$ and $\phi' = \pbs \tilde a' \\ \tilde \zeta' \pes$ instead of $\pbs a' \\ \zeta' \pes$, the formula \eqref{eq:eta} does not change. %

It follows from the definition that the multiplication on $\Alg$ is associative and satisfies
the Leibniz rule.

The vertical multiplication on the horizontal tensor product of $\RAlg$ and $\LAlg$ can be constructed in the same manner, and has similar properties.
\end{proof}

\begin{remark}
\label{rem:algdec}
  The fact that $\Alg$ has a decomposition into pieces
  $$\putaround{\Alg}{}{n}{}{q}=\ \scriptstyle{n} \! \left( \vertme{\TAlg}{\BAlg}{\A}{0ex}{0.5ex} \right)  \! \scriptstyle{q}$$ means that we can view it as a category with objects
  the nonnegative integers, such that the space of morphisms from $n$
  to $q$ is $\Alg_{n, q}.$ The same observation applies to $\Alg'$.
\end{remark}

We can also tensor together $2$-modules and, under suitable hypotheses, obtain ordinary
(differential) modules:

\begin{definition}\label{def:2m-mot}
We say that a top-right $2$-module $\TR$ over algebra-modules $\TAlg$ and $\RAlg$ satisfies the {\em vertical motility hypothesis} if for all non-negative
  integers $n$ and $p$, the multiplication map
\begin{equation}
\label{eq:motTR1}
\cdot \co\bigoplus_{m}\bigl(\TRa{m}{0} \otimes \RAlga{m}{p}{n}\bigr)\to \TRa{p}{n}
\end{equation}
is surjective. We say that $\TR$ satisfies the {\em horizontal motility hypothesis} if for all non-negative integers $m$ and $q$, the map
\begin{equation}
\label{eq:motTR2}
\ast \co\bigoplus_{n}\bigl(\TRa{0}{n} \otimes \TAlga{n}{m}{q}\bigr)\to \TRa{m}{q}
\end{equation}
is surjective.
  
Similar definitions apply to the other types of 2-modules. Precisely, the vertical motility hypothesis requires surjectivity of the maps
\begin{align*}
\cdot \co & \bigoplus_{p}\bigl(\RAlga{m}{p}{n} \otimes \BRa{p}{0}\bigr)\to \BRa{m}{n}, \\
\cdot \co & \bigoplus_{m}\bigl(\TLa{0}{m} \otimes \LAlga{m}{q}{p}\bigr)\to \TLa{q}{p},\\
\cdot \co & \bigoplus_{p}\bigl(\LAlga{m}{q}{p} \otimes \BLa{p}{0}\bigr)\to \BLa{m}{q},
\end{align*}
whereas the horizontal motility hypothesis requires surjectivity of the maps
\begin{align*}
\ast \co & \bigoplus_{n}\bigl(\BRa{0}{n} \otimes \BAlga{p}{n}{q}\bigr)\to \BRa{p}{q}, \\
\ast \co & \bigoplus_{q}\bigl(\TAlga{n}{m}{q} \otimes \TLa{q}{0} \bigr)\to \TLa{n}{m},\\
\ast \co & \bigoplus_{q}\bigl(\BAlga{p}{n}{q} \otimes \BLa{0}{q} \bigr)\to \BLa{p}{n}.
\end{align*}
\end{definition}

\begin{lemma}
\label{lem:weak2m}
Suppose the algebra-modules $\TAlg$ and $\BAlg$ satisfy the motility hypothesis. If, in addition, the $2$-modules $\TR$ and $\BR$ (resp. $\TL$ and $\BL$) satisfy the vertical motility hypotheses, then the tensor products
  $$ \vertme{\TR}{\BR}{\RAlg}{-0.3ex}{0ex} := \vertme{\TRa{\sbull}{\sbull}}{\BRa{\sbull}{\sbull}}{\RAlgasm{\sbull}{\sbull}{\sbull}}{.5ex}{.5ex}, \ \ \ \text{resp.} \ \ \ \vertme{\TL}{\BL}{\LAlg}{-0.3ex}{0ex} := \vertme{\TLa{\sbull}{\sbull}}{\BLa{\sbull}{\sbull}}{\LAlgasm{\sbull}{\sbull}{\sbull}}{.5ex}{.5ex},$$
are naturally right, resp. left, modules over the algebra $ \smallvertme{\TAlg}{\BAlg}{\A}{-0.8ex}{-0.3ex}.$
 
 Similarly, if the algebra-modules $\TAlg$ and $\BAlg$ satisfy the motility hypothesis, and in addition the $2$-modules $\TR$ and $\TL$ (resp. $\BR$ and $\BL$) satisfy the horizontal motility hypotheses, then the tensor products
  $$\horme{\TR}{\TL}{\TAlg}{0.5ex}{0.2ex} := \horme{\TRa{\sbull}{\sbull}}{\TLa{\sbull}{\sbull}}{\TAlgasm{\sbull}{\sbull}{\sbull}}{.5ex}{.5ex}, \ \ \ \text{resp.} \ \ \ \horme{\BR}{\BL}{\BAlg}{0.5ex}{0.2ex} := \horme{\BRa{\sbull}{\sbull}}{\BLa{\sbull}{\sbull}}{\BAlgasm{\sbull}{\sbull}{\sbull}}{.5ex}{.5ex},$$
  are naturally top, resp. bottom, modules over the algebra $ \smallhorme{\RAlg}{\LAlg}{\A}{0ex}{0ex}.$
\end{lemma}

\begin{proof}
The construction of the module structure is completely analogous to the construction of the algebra structure in the proof of Lemma~\ref{lem:AMweak}. For example, if we are given $\x \in \TRa{m}{n}, \y \in \BRa{m}{n'}, \phi \in \TAlga{u}{m'}{q}$ and $\psi \in \BAlga{m'}{u'}{q'}$, we define 
$$  \pbnmat \y \\ \todot \\ \x \penmat \pbnmat \psi \\ \todot \\ \phi \penmat
 $$
to be zero unless $n+n' = u+u'$. If $n+n'=u+u'$, we use the motility hypotheses to write 
\begin{align*}
 \x &= \pbnmat \eta \\ \z \penmat, 
 \ \ \z \in \TRa{\sbull}{0}, \ \eta \in \RAlga{\sbull}{m}{n}, \\
 \phi &= \pbnmat a \\ \zeta \penmat, 
 \ \ \zeta \in \TAlga{0}{\sbull}{0}, \ a \in \Aa{\sbull}{u}{m'}{q}.
\end{align*}

Then, we define
$$\pbnmat  \y \\ \todot \\ \x \penmat \pbnmat \psi \\ \todot \\ \phi \penmat
:=
\bnmat \left( \pbnmat \y \\ \eta \penmat \pbnmat \psi \\ a \penmat \right) \\ \todot \\ (\z \; \zeta) \enmat.$$
Well-definedness follows as in the proof of Lemma~\ref{lem:AMweak}.
\end{proof}

\begin{remark}
Recall from \eqref{eq:algdec} that the algebras $\Alg$ and $\Alg'$ admit direct sum decompositions into pieces indexed by two nonnegative integers. In a similar manner, the tensor products of algebra-modules admit direct sum decompositions into pieces indexed by a single nonnegative integer. Specifically, we have
$$\vertme{\TR}{\BR}{\RAlg}{-0.3ex}{0ex} = \bigoplus_{n}  \putaround{\left( \vertme{\TR}{\BR}{\RAlg}{-0.3ex}{0ex} \hspace{-.5ex}\right)}{}{}{}{n},$$
where $\putaround{\left( \vertme{\TR}{\BR}{\RAlg}{-0.3ex}{0ex} \hspace{-.5ex}\right)}{}{}{}{n}$ is generated by elements of the form $\pbnmat  \y \\ \todot \\ \x \penmat$ with $\y \in \BRa{p}{n'}$, $\x \in \TRa{p}{n''}$, and $n = n' + n''$.  The product $\vertme{\TL}{\BL}{\LAlg}{-0.3ex}{0ex}$ has a similar decomposition, as do the horizontal tensor products.
\end{remark}

\subsection{Sequential objects and restricted tensor products} \label{sec:symm-seq}
Recall from \cite{DM:cornered} and Example~\ref{ex:one} that a {\em sequential $2$-algebra} is a rectangular $2$-algebra $\A$ such that $\Aa{m}{n}{p}{q} = 0$ unless $n=q=0$ and $m=p$. Most of the constructions done with sequential $2$-algebras in \cite{DM:cornered} can also be done if we relax the  requirements a bit, by asking only that $\Aa{m}{n}{p}{q} = 0$ unless $n=q=0$; that is,
$$ \Aa{\sbull}{\sbull}{\sbull}{\sbull} = \Aa{\sbull}{0}{\sbull}{0}.$$

If $\A$ satisfies this condition, we call it a {\em vertically sequential $2$-algebra}. We can make similar definitions for algebra-modules and $2$-modules:

\begin{definition}
An algebra-module or a $2$-module is called {\em vertically sequential} if its constituent pieces vanish whenever they have a nonzero index on the left or the right; that is, 
$$ \TAlga{\sbull}{\sbull}{\sbull} =  \TAlga{0}{\sbull}{0}\ ,\ \BAlga{\sbull}{\sbull}{\sbull} =  \BAlga{\sbull}{0}{0} \ ,\ \RAlga{\sbull}{\sbull}{\sbull} =  \RAlga{\sbull}{\sbull}{0}\ ,\  \LAlga{\sbull}{\sbull}{\sbull} =  \LAlga{\sbull}{0}{\sbull} \ ,$$
$$ \TRa{\sbull}{\sbull} = \TRa{\sbull}{0} \ , \ \BRa{\sbull}{\sbull} = \BRa{\sbull}{0}\ , \ \TLa{\sbull}{\sbull} = \TLa{0}{\sbull} \ , \BLa{\sbull}{\sbull} = \BLa{\sbull}{0}\ .$$
\end{definition}

Note that these notions have already appeared in \cite{DM:cornered}, under the name of sequential algebra-modules and $2$-modules.

In this paper we will also need analogous notions defined with respect to the horizontal rather than the vertical direction:
\begin{definition}
A rectangular $2$-algebra, algebra-module or $2$-module is called {\em horizontally sequential} if its constituent pieces vanish whenever they have a nonzero index on the top or the bottom; that is, 
$$ \Aa{\sbull}{\sbull}{\sbull}{\sbull} = \Aa{0}{\sbull}{0}{\sbull}, \ \TAlga{\sbull}{\sbull}{\sbull} =  \TAlga{\sbull}{0}{\sbull} , \ \TRa{\sbull}{\sbull} = \TRa{0}{\sbull},  \text{etc.}$$
\end{definition}

We can simplify our notation for vertically sequential and horizontally sequential objects by dropping the indices that are supposed to be zero anyway. For example, in the case of a vertically sequential $2$-algebra, we drop the indices $n$ and $q$ from the notation and just write $\Aa{m}{\;}{p}{\;}$ for $\Aa{m}{0}{p}{0}$. Also, we write
$$\Aa{\sbull}{\;}{\sbull}{\;} = \bigoplus_{m,p} \Aa{m}{\;}{p}{\;} = \Aa{\sbull}{\sbull}{\sbull}{\sbull}.$$

Similarly, for vertically sequential algebra-modules and $2$-modules, we set 
$$\TAlga{\;}{m}{\;} = \TAlga{0}{m}{0}, \ \RAlga{m}{p}{\;} = \RAlga{m}{p}{0}, \  \TRa{\;}{m} = \TRa{0}{m}, \ \text { etc.}$$

The following observations are immediate from the definitions:
\begin{proposition}
\label{prop:vh}\ 
\begin{enumerate}[label=(\alph*),ref=(\alph*)]
\item A unital, vertically sequential top or bottom algebra-module satisfies the motility hypothesis;
\item A vertically sequential $2$-module (of any of the four possible types) satisfies the vertical motility hypothesis;
\item A unital, horizontally sequential right or left algebra-module satisfies the motility hypothesis;
\item A horizontally sequential $2$-module (of any of the four possible types) satisfies the horizontal motility hypothesis.
\end{enumerate}
\end{proposition}

Here is a way of constructing vertically sequential and horizontally sequential objects:

\begin{definition}
\label{def:AssocSeq}
Given an arbitrary rectangular $2$-algebra (resp. algebra-module or $2$-module), we define its {\em associated vertically sequential $2$-algebra (resp. algebra-module or $2$-module)} to consist of those pieces indexed by $0$ on the left and the right. The associated vertically sequential object is denoted by a superscript $v$. For example,
$$ \Ava{m}{p} = \Aa{m}{0}{p}{0}, \ \TAlgva{m}=\TAlga{0}{m}{0}, \ \RAlgva{m}{p} = \RAlga{m}{p}{0}, \ \TRva{m}=\TRa{m}{0}, \ \text{etc.}$$

Similarly, we define the {\em associated horizontally sequential $2$-algebra (resp. algebra-module or $2$-module)} to consist of those pieces indexed by $0$ on the top and bottom. The associated horizontally sequential object is denoted by a superscript $h$. For example,
$$ \Aha{n}{q} = \Aa{0}{n}{0}{q}, \ \TAlgha{n}{q}=\TAlga{n}{0}{q}, \ \RAlgha{n} = \RAlga{0}{0}{n}, \ \TRha{n}=\TRa{0}{n}, \ \text{etc.}$$
\end{definition}

Let $\A$ be an arbitrary biunital rectangular 2-algebra, and $\TAlg, \RAlg, \BAlg, \LAlg$ be top, right, bottom, and left unital algebra-modules over $\A,$ respectively. Let also $\TR, \BR, \BL, \TL$ be unital $2$-modules over $\TAlg, \RAlg, \BAlg, \LAlg$.

Note that $\TAlgv, \RAlgv, \BAlgv, \LAlgv$ are algebra-modules over $\Av$, and similarly $\TAlgh, \RAlgh, \BAlgh, \LAlgh$ are algebra-modules over $\Ah$. Also, $\TRv$ is a $2$-module over $\TAlgv$ and $\RAlgv$, $\TRh$ is a $2$-module over $\TAlgh$ and $\RAlgh$, and so on.

In view of Proposition~\ref{prop:vh} (a) and Lemma~\ref{lem:AMweak}, the vertically sequential algebra-modules $\TAlgv$ and $\BAlgv$ can be tensored together vertically, yielding a differential algebra
$$\Alg^v = \vertme{\TAlgv}{\BAlgv}{\Av}{0ex}{0.5ex} = \vertme{\TAlga{0}{\sbull}{0}}{\BAlga{\sbull}{0}{0}}{\Aasm{\sbull}{0}{\sbull}{0}}{.5ex}{.5ex}. $$

Moreover, in view of Proposition~\ref{prop:vh} (b) and Lemma~\ref{lem:weak2m}, the vertically sequential $2$-modules $\TRv$ and $\BRv$ can be tensored together vertically, yielding a differential module
$$ \vertme{\TRv}{\BRv}{\RAlgv}{-0.3ex}{0ex} =   \vertme{\TRa{\sbull}{0}}{\BRa{\sbull}{0}}{\RAlgasm{\sbull}{\sbull}{0}}{.5ex}{.5ex}$$
over $\Alg^v$. By the same token, we can tensor together $\TLv$ and $\BLv$ to get another differential module
$$ \vertme{\TLv}{\BLv}{\RAlgv}{-0.3ex}{0ex} =  \vertme{\TLa{0}{\sbull}}{\BLa{\sbull}{0}}{\LAlgasm{\sbull}{0}{\sbull}}{.5ex}{.5ex}$$
over $\Alg^v$. 

Similarly, we can tensor the horizontally sequential algebra-modules associated to $\RAlg$ and $\LAlg$
to get a differential algebra
$$  (\Alg')^h = \horme{\RAlgh}{\LAlgh}{\Ah}{0.5ex}{0.2ex}.$$
We can then tensor the associated horizontally sequential $2$-modules to get two differential modules over $(\Alg')^h$: 
$$  \horme{\TRh}{\TLh}{\TAlgh}{0.5ex}{0.2ex} \ \ \text{and} \ \ \horme{\BRh}{\BLh}{\BAlgh}{0.5ex}{0.2ex}.$$

One may ask if there is a relation between the tensor products of the associated vertical (or horizontal) sequential objects and those of the original objects. Indeed, we have:

\pagebreak
\begin{proposition}
\label{prop:zeros}
\
\begin{enumerate}[label=(\alph*),ref=(\alph*)]
\item The differential algebras $\Algv = \vertme{\TAlgv}{\BAlgv}{\Av}{0ex}{0.2ex}$ and $(\Alg')^h = \horme{\RAlgh}{\LAlgh}{\Ah}{0.2ex}{0.2ex}$ are the $(0,0)$-pieces of the differential modules $\Alg = \vertme{\TAlg}{\BAlg}{\A}{0ex}{0.2ex}$ and $\Alg'= \horme{\RAlg}{\LAlg}{\A}{0.2ex}{0.2ex}$, respectively:
$$ \Algv = \putaround{\Alg}{}{0}{}{0} \ , \ \  (\Alg')^h= \putaround{(\Alg')}{0}{}{0}{} \ .$$

\item The differential modules obtained by tensoring the associated
  sequential $2$-modules are the
 pieces indexed by $0$ in the corresponding differential modules obtained by tensoring the original $2$-modules:
$$   \vertme{\TRv}{\BRv}{\RAlgv}{-0.3ex}{0ex} = \putaround{ \left( \vertme{\TR}{\BR}{\RAlg}{-0.3ex}{0ex} \hspace{-1ex} \right)}{}{}{}{0} \ , \ \  \vertme{\TLv}{\BLv}{\LAlgv}{-0.3ex}{0ex}= \putaround{\left( \vertme{\TL}{\BL}{\LAlg}{-0.3ex}{0ex} \hspace{-1ex}\right )}{}{0}{}{} \ ,$$
and similarly for  $\horme{\TRh}{\TLh}{\TAlgh}{0.5ex}{0.2ex} $ and $ \horme{\BRh}{\BLh}{\BAlgh}{0.5ex}{0.2ex}$.
  \end{enumerate}
\end{proposition}

\begin{proof}
(a) By definition, the tensor product
$$\Alg = \vertme{\TAlga{\sbull}{\sbull}{\sbull}}{\BAlga{\sbull}{\sbull}{\sbull}}{\Aasm{\sbull}{\sbull}{\sbull}{\sbull}}{1.4ex}{1ex}$$
is generated by elements of the form $\bs \psi \\ \todot \\ \phi \es$ subject to relations
$$ \bnmat \psi \\ \todot \\ \pbmat a \\ \phi \pemat \enmat =  \bnmat \pbmat \psi \\ a \pemat  \\ \todot \\  \phi  \enmat$$
It follows that the $(0,0)$-part of $\Alg$ is generated by elements as above such that $\phi \in \TAlga{0}{\sbull}{0}$ and $\psi \in \BAlga{\sbull}{0}{0}$. Furthermore, the only relations between such objects appear from multiplying in the middle with an element $a \in \Aa{\sbull}{0}{\sbull}{0}$. This gives exactly the description of $\Alg^v$.

The argument for $(\Alg')^h$ is completely analogous, and part (b) is similar to part (a).
\end{proof}  

\begin{remark}
We know from subsection~\ref{sec:motility} that in order for $\Alg$ to be an algebra, we need to assume the motility hypotheses on $\TAlg$ and $\BAlg$. Nevertheless, Proposition~\ref{prop:zeros} (a) holds even in the absence of these hypotheses. Thus, the $(0,0)$-part of $\Alg$ is always an algebra, even when $\Alg$ is not. Similar remarks apply to the other tensor products.
\end{remark}

We have seen that starting from arbitrary rectangular objects we can take the associated vertically sequential objects, and then tensor them vertically. We will refer to the combined effect of these operations as the {\em restricted vertical tensor product} of the original objects, denoted by the symbol $\rvtp$. For example, the restricted vertical tensor product of $\TAlg$ and $\BAlg$ over $\A$ is the algebra
$$ \Alg^v = \rvertme{\TAlg}{\BAlg}{\A}{-0.3ex}{0ex} =\vertme{\TAlgv}{\BAlgv}{\Av}{-0.3ex}{0ex}.$$
We also have restricted vertical tensor products of $2$-modules:
$$\rvertme{\TR}{\BR}{\RAlg}{-0.3ex}{0ex} =\vertme{\TRv}{\BRv}{\RAlgv}{-0.3ex}{0ex} \ \text{ and } \ \rvertme{\TL}{\BL}{\RAlg}{-0.3ex}{0ex} =\vertme{\TRv}{\BRv}{\RAlgv}{-0.3ex}{0ex}$$

We define {\em restricted horizontal tensor products} in a similar way, and denote them by the symbol $\rhtp$. For example, we shall write $\rhormein{\RAlg}{\LAlg}{\A}{0.5ex}{0.2ex}$ for $(\Alg')^h = \hormein{\RAlgh}{\LAlgh}{\Ah}{0.5ex}{0.2ex}$.

\subsection{Module--2-modules, algebra-bimodules, and bimodule-modules}\label{sec:bestiary}
We will consider a few
straightforward generalizations of algebra-modules and $2$-modules. We
outline these generalizations here, discussing only the cases we
will use in constructing cornered Floer homology.

We start with the notion that will capture the cornering object $\cornAA$ mentioned in the
introduction:
\begin{definition}\label{def:module-2-module}
  Let $\Alg$ be a differential algebra, $\A$ a rectangular $2$-algebra, $\RAlg$ a right
  algebra-module over $\A$, and $\TAlg$ a top algebra-module over
  $\A$. A (left-top-right) \emph{module--$2$-module} over $\Alg$,
  $\RAlg$, and $\TAlg$ is a top-right $2$-module $\TR$ over $\RAlg$ and
  $\TAlg$ such that each $\TRa{m}{n}$ is a left differential $\Alg$-module, and such that
  the $2$-algebra actions $*$ and $\cdot$ commute with the action of $\Alg$.

  Module--$2$-modules over left algebra-modules and/or bottom
  algebra-modules are defined similarly.
\end{definition}

To construct the other cornering module--$2$-modules $\cornDA$,
$\cornAD$ and $\cornDD$ in Section~\ref{sec:cornering} we will make
use of a couple of auxiliary objects: the ``barbell algebra-bimodule"
(Definition~\ref{def:barbell}) and the ``\DD\ identity bimodule-module"
(Definition~\ref{def:DD-id}). These fit into the following
algebraic framework:

\begin{definition}\label{def:algebra-bimodule}
  Fix biunital, rectangular $2$-algebras $\A_1$ and $\A_2$. A unital top-bottom
  \emph{algebra-bimodule} $\Barbell$ over $\A_1$ and $\A_2$ consists of a
  collection of chain complexes $\putaroundmarg{\Barbell}{m}{\substack{n_1\\n_2}}{p}{\substack{q_1\\q_2}}$
  together with chain maps
  \begin{align*}
    *\co \putaroundmarg{\Barbell}{m}{\substack{n_1\\n_2}}{p}{\substack{q_1\\q_2}}\otimes\putaroundmarg{\Barbell}{m'}{\substack{q_1\\q_2}}{p'}{\substack{q'_1\\q'_2}} &\to\putaroundmarg{\Barbell}{m+m'}{\substack{n_1\\n_2}\:\:}{p+p'}{\:\:\:\substack{q'_1\\q'_2}} \\
    \cdot_1\co \putaroundmarg{\Barbell}{m}{\substack{n_1\\n_2}}{p}{\substack{q_1\\q_2}}\otimes\putaround{(\A_1)}{p}{n'}{p'}{q'} &\to \putaroundmarg{\Barbell}{m}{\substack{n_1+n'\\n_2}}{p'}{\substack{q_1+q'\\q_2}}\\
    \cdot_2\co
    \putaround{(\A_2)}{m'}{n'}{m}{q'}\otimes\putaroundmarg{\Barbell}{m}{\substack{n_1\\n_2}}{p}{\substack{q_1\\q_2}}
    &\to \putaroundmarg{\Barbell}{m'}{\substack{n_1\\n_2+n'}}{p}{\substack{q_1\\q_2+q'}},
  \end{align*}
  which we view as horizontal, vertical, and vertical multiplications,
  respectively, such that these three operations are associative, and for any $a\in
  \putaround{(\A_1)}{p}{\sbull}{\sbull}{r}$,
  $b\in\putaround{(\A_1)}{u}{r}{\sbull}{\sbull}$,
  $\phi\in\putaroundmarg{\Barbell}{m}{\substack{n_1\\n_2}}{p}{\substack{q_1\\q_2}}$,
  $\psi\in\putaroundmarg{\Barbell}{t}{\substack{q_1\\q_2}}{u}{\substack{v_1\\v_2}}$,
  $c\in\putaround{(\A_2)}{\sbull}{\sbull}{m}{s}$ and
  $d\in\putaround{(\A_2)}{\sbull}{s}{t}{\sbull}$, any way of
  associating the following product gives the same result:
  \[
  \pbmat
  a & b\\
  \phi & \psi\\
  c & d
  \pemat .
  \]
 Further, we ask that $\pbs e_p^v \\ \x \pes = \x = \pbs \x  \\ e^v_m \pes$ for all $\x \in
\putaroundmarg{\Barbell}{m}{\substack{n_1\\n_2}}{p}{\substack{q_1\\q_2}}$,
and that there exist horizontal units $\varepsilon^h_{n_1,n_2} \in
\putaroundmarg{\Barbell}{0}{\substack{n_1\\n_2}}{0}{\substack{n_1\\n_2}}$
so that $\varepsilon^h_{n_1+n',n_2}  = \pbs e^h_{n'}  \\
\varepsilon^h_{n_1,n_2} \pes$ and $ \varepsilon^h_{n_1,n_2+n'}= \pbs
\varepsilon^h_{n_1,n_2}  \\  e^h_{n'} \pes$ for all $n_1, n_2, n' \geq 0$. 

 Unital left-right algebra-bimodules are defined similarly.
\end{definition}

\begin{remark}
The reader may wonder why our top-bottom algebra-bimodules have only two indices on the left and two on the right. The reason is that the only example considered in the paper (the vertical barbell algebra-bimodule) has this property. Similarly, bimodule-modules are defined below with two indices on the right, because we had in mind the \DD\ identity bimodule-module.
\end{remark}

\begin{definition}\label{def:bimodule-module}
  Fix biunital, rectangular $2$-algebras $\A_1$ and $\A_2$, a unital top-bottom algebra-bimodule
  $\Barbell$ over $\A_1$ and $\A_2$, and unital right algebra-modules
  $\RAlg_1$ and $\RAlg_2$ over $\A_1$ and $\A_2$, respectively. A
  \emph{bimodule-module} $M$ over $\Barbell$, $\RAlg_1$, and $\RAlg_2$
  consists of a collection of chain complexes
  $\putaroundmarg{M}{p}{}{q}{\substack{r_1\\r_2}}$ together with maps 
  \begin{align*}
    *\co \putaroundmarg{M}{p}{}{q}{\substack{r_1\\r_2}}\otimes
    \putaroundmarg{\Barbell}{p'}{\substack{r_1\\r_2}}{q'}{\substack{r'_1\\r'_2}}&\to
    \putaroundmarg{M}{p+p'}{}{q+q'}{\:\:\substack{r'_1\\r'_2}}\\
    \cdot_1\co \putaroundmarg{M}{p}{}{q}{\substack{r_1\\r_2}}\otimes
    \putaround{(\RAlg_1)}{q}{}{q'}{r'}&\to\putaroundmarg{M}{p}{}{q'}{\substack{r_1+r'\\r_2}}\\    
    \cdot_2\co \putaround{(\RAlg_2)}{p'}{}{p}{r'}\otimes \putaroundmarg{M}{p}{}{q}{\substack{r_1\\r_2}} &\to\putaroundmarg{M}{p'}{}{q}{\substack{r_1\\r_2+r'}}
  \end{align*}
  which we view as horizontal, vertical, and vertical multiplications,
  respectively, such that these operations are associative, and for any
  $\phi\in\putaround{(\RAlg_1)}{q}{}{\sbull}{m}$,
  $a\in\putaround{(\A_1)}{s}{m}{\sbull}{\sbull}$,
  $\x\in\putaroundmarg{M}{p}{}{q}{\substack{r_1\\r_2}}$,
  $\xi\in\putaroundmarg{\Barbell}{t}{\substack{r_1\\r_2}}{s}{\substack{\sbull\\
   \sbull}}$,
  $\psi\in\putaround{(\RAlg_2)}{\sbull}{}{p}{n}$ and
  $b\in\putaround{(\A_2)}{\sbull}{n}{t}{\sbull}$, 
  any way of associating the following product gives the same result:
  \[
  \pbmat
  \phi & a\\
  \x & \xi\\
  \psi & b
  \pemat.
  \]
Further, we ask that the vertical units $1^v_{q} \in \RAlg_1$ and $1^v_{p} \in \RAlg_2$ and the horizontal units  $\varepsilon^h_{n_1,n_2} \in \Barbell$ act by the identity on $\putaroundmarg{M}{p}{}{q}{\substack{r_1\\r_2}}$.

 Other kinds of bimodule-modules are defined similarly.
\end{definition}

The main point, of course, is that one can take tensor products of
these objects. For instance:
\begin{itemize}
\item Given a right differential $\Alg$-module $M$ and a left-top-right
  module--$2$-module $N$ over $\Alg$, $\RAlg$, and $\TAlg$, the tensor product 
  $M\otimes_\Alg N$ is naturally a $2$-module over $\RAlg$ and
  $\TAlg$.
\item Given a left-top-right module--$2$-module $M$ over $\Alg_1$,
  $\RAlg$, and $\TAlg$ and a left-bottom-right module--$2$-module $N$
  over $\Alg_2$, $\RAlg$, and $\BAlg$ (where $\TAlg$, $\RAlg$, and
  $\BAlg$ are algebra-modules over the same $2$-algebra $\A$), one
  can form the restricted vertical tensor product of $M$ and $N$, which is
  an ordinary trimodule over $\Alg_1$, $\Alg_2$, and the restricted
  tensor product of $\BAlg$ and $\TAlg$:
  \[
\left( \rvertme{M}{N}{\RAlg}{2pt}{2pt} \!\!\!\ \right) \otimes
  \left( \rvertme{\TAlg}{\BAlg}{\A}{2pt}{2pt} \!\!\!\ \right)
 \to
\left( \rvertme{M}{N}{\RAlg}{2pt}{2pt} \!\!\!\ \right)   .
  \]
\item Given $2$-algebras $\A_1$ and $\A_2$, a bottom algebra-module
  $\BAlg$ over $\A_1$ and an algebra-bimodule $\Barbell$ over
  $\A_1$ and $\A_2$ we can form the restricted tensor product
  of $\BAlg$ and $\Barbell$ as the direct sum of components
  \[
  \putaround{\left( \rvertme{\Barbell}{\BAlg}{\A_1}{2pt}{2pt} \!\!\!\ \right)}{m}{n}{}{p}
  = \vertme{\putaround{\Barbell}{m}{\substack{0\\n}}{\sbull}{\substack{0\\p}}}{\putaround{\BAlg}{\sbull}{0}{}{0}}{\putaround{\scriptstyle(\A_1)}{\sbull}{0}{\sbull}{0}}{2pt}{2pt}.
  \]
  This tensor product is a bottom  algebra-module over
  $\A_2$:
  \[
  \A_2 
  \otimes
  \left( \rvertme{\Barbell}{\BAlg}{\A_1}{2pt}{2pt} \!\!\!\ \right)
 \to
   \left( \rvertme{\Barbell}{\BAlg}{\A_1}{2pt}{2pt} \!\!\!\ \right) .
  \]
\item Fix $2$-algebras $\A_1$ and $\A_2$, right algebra-modules
  $\RAlg_1$ and $\RAlg_2$ over $\A_1$ and $\A_2$, respectively,
  and a bottom algebra-module $\BAlg$ over $\A_1$. Fix also an algebra-bimodule
  $\Barbell$ over $\A_1$ and $\A_2$. 
  Given a left-bottom-right module--$2$-module $M$ over $\Alg$,
  $\RAlg_1$, and $\BAlg$ and a bimodule-module $N$ over $\RAlg_1$,
  $\RAlg_2$, and $\Barbell$, we can form the restricted vertical tensor product
  of $M$ and $N$ as the direct sum of components
  \[
  \putaround{\left( \rvertme{N}{M}{\RAlg_1}{2pt}{2pt} \!\!\!\ \right)}{m}{}{}{n}
  =
  \vertme{\putaround{N}{m}{}{\sbull}{\substack{0\\n}}}{\putaround{M}{\sbull}{}{}{0}}{\putaround{\scriptstyle(\RAlg)_1}{\sbull}{}{\sbull}{0}}{2pt}{2pt}.
  \]
  This is a module--$2$-module over $\Alg$,
  $\RAlg_2$, and the restricted tensor product of $\BAlg$ and
  $\Barbell$:
  \[
  \left( \rvertme{N}{M}{\RAlg_1}{2pt}{2pt} \!\!\!\ \right) \otimes
   \left( \rvertme{\Barbell}{\BAlg}{\A_1}{2pt}{2pt} \!\!\!\ \right)
\to
 \left( \rvertme{N}{M}{\RAlg_1}{2pt}{2pt} \!\!\!\ \right) . 
 \]
 \end{itemize}

 There are many obvious variants of these operations, some of which we will also use. For example, in the first bullet point, $M$ could be a left module and $N$ a right-bottom-left module--$2$-module.


\section{More 2-algebra: bending and smoothing}
\label{sec:bent}
Let $\A$ be a biunital rectangular $2$-algebra, and let $\TAlg, \RAlg,
\BAlg, \LAlg$ be unital top, right, bottom, and left algebra-modules
over $\A,$ respectively. Our main goal in this section is to
reinterpret the top-right and bottom-left $2$-modules over these
algebra-modules as ordinary modules over some differential graded
algebras. The differential graded algebras in question are called bent
tensor products: they are obtained by tensoring algebra-modules
together with respect to actions of the 2-algebra, as described
below.

The reformulation of $2$-modules in terms of ordinary modules is possible only in the presence of an additional assumption, called the bent motility hypothesis; see Definition~\ref{def:bentMH} below. The reader may notice an asymmetry (between the two diagonal directions) in that definition. This is the reason why the reformulation only works for top-right and bottom-left $2$-modules. Of course, one could 
impose a different bent motility hypothesis, which would allow a similar reformulation for the top-left and bottom-right $2$-modules. However, this second hypothesis does not hold for the diagonal nilCoxeter $2$-algebra considered in this paper, so we shall not discuss it here.

\subsection{The top-right bent tensor product} \label{sec:benttens}
We let the \emph{bent tensor product} of $\TAlg$ and $\RAlg$ be
\begin{equation}
\label{eq:bentdef}
\TAlg\obent \RAlg := \mathcenter{ \bendme{\RAlga{\sbull}{\sbull}{\sbull}}{\Aa{\sbull}{\sbull}{\sbull}{\sbull}}{\TAlga{\sbull}{\sbull}{\sbull}}{\Aasm{0}{\sbull}{\sbull}{\sbull}}{\Aasm{\sbull}{0}{\sbull}{\sbull}}}
\end{equation}
Concretely, the bent tensor product is generated by elements of the form
\[
\begin{pmatrix}
  \zeta & a\\
  & \phi
\end{pmatrix}, \ \ \zeta \in \RAlga{\sbull}{\sbull}{\sbull}, \ a \in \Aa{\sbull}{\sbull}{\sbull}{\sbull}, \ \phi \in \TAlga{\sbull}{\sbull}{\sbull}
\]
subject to relations
\begin{equation}
\label{eq:defBentG}
\begin{pmatrix}
  (\zeta a') & a\\
  & \phi
\end{pmatrix} =
\begin{pmatrix}
  \zeta  & (a' a)\\
  & \phi
\end{pmatrix}
\ \text{ and } \
\begin{pmatrix}
  \zeta  & a\hphantom{''}\\
  & \pbmat a'' \\ \phi\hphantom{''} \pemat
\end{pmatrix}
= 
\begin{pmatrix}
  \zeta  & \pbmat a\hphantom{''} \\ a'' \pemat \\
  & \phi \hphantom{''}
\end{pmatrix}
\end{equation}
for $a' \in \Aa{0}{\sbull}{\sbull}{\sbull}$, $a'' \in \Aa{\sbull}{0}{\sbull}{\sbull}$. 

The bent tensor product has components
\[
\putaround{(\TAlg\obent \RAlg)}{m}{n}{\sbull}{\sbull}=
\mathcenter{\bendme{\RAlga{m}{\sbull}{\sbull}}{\Aa{\sbull}{\sbull}{\sbull}{\sbull}}{\TAlga{n}{\sbull}{\sbull}}{\Aasm{0}{\sbull}{\sbull}{\sbull}}{\Aasm{\sbull}{0}{\sbull}{\sbull}}}
\]
for $m,n \in\NN$, and these further decompose as direct sums of pieces $$\putaround{(\TAlg\obent \RAlg)}{m}{n}{p}{q} $$ where $p$ is the number of marks on the top (coming from the top two factors) and $q$ is
the number of marks on the right (coming from the right two factors).

\begin{definition}
\label{def:bentMH}
  We say a $2$-algebra $\A$ satisfies the \emph{bent motility hypothesis} if the homomorphisms
 \[   \horme{ \Aa{0}{\sbull}{\sbull}{\sbull}}{
    \Aa{\sbull}{\sbull}{0}{\sbull}}{\Aasm{0}{\sbull}{0}{\sbull}}{1ex}{1ex}  \to
    \Aa{\sbull}{\sbull}{\sbull}{\sbull} \
\hskip1cm    \text{and} \
  \hskip1cm   \vertme
       {\Aa{\sbull}{0}{\sbull}{\sbull}} {\Aa{\sbull}{\sbull}{\sbull}{0}}{\Aasm{\sbull}{0}{\sbull}{0}}{.6ex}{.6ex}
        \to
         \Aa{\sbull}{\sbull}{\sbull}{\sbull}
  \]  are bijective.
\end{definition}

\begin{example}
The bent motility hypothesis is satisfied by the diagonal nilCoxeter $2$-algebra $\Dnil$ from Definition~\ref{def:ex2}. However, it is not satisfied by the sequential nilCoxeter $2$-algebra $\nil$ from Example~\ref{ex:one}.
\end{example}

Note that under the bent motility hypothesis, we can combine the two isomorphisms above into a single one:
\begin{equation}
\label{eq:bentA}
\mathcenter{\bendme{\Aa{0}{\sbull}{\sbull}{\sbull}}{\Aa{\sbull}{\sbull}{0}{0}}{\Aa{\sbull}{0}{\sbull}{\sbull}}{\Aasm{0}{\sbull}{0}{\sbull}}{\Aasm{\sbull}{0}{\sbull}{0}}} \xlongrightarrow{\cong}  \Aa{\sbull}{\sbull}{\sbull}{\sbull}
\end{equation}

With this in mind, we obtain:

\begin{proposition}
\label{prop:quasar}
Assuming that $\A$ satisfies the bent motility hypothesis, the bent tensor product $\TAlg\obent \RAlg$ is naturally isomorphic to 
\begin{equation}
\label{eq:newbent}
 \mathcenter{\bendme{\RAlga{\sbull}{\sbull}{\sbull}}{\Aa{\sbull}{\sbull}{0}{0}}{\TAlga{\sbull}{\sbull}{\sbull}}{\Aasm{0}{\sbull}{0}{\sbull}}{\Aasm{\sbull}{0}{\sbull}{0}}}
 \end{equation}
Moreover, we can re-write the components of the bent tensor product as 
\[
\putaround{(\TAlg\obent \RAlg)}{m}{n}{p}{q} \cong
\mathcenter{
\bendme{\RAlga{m}{p}{\sbull}}{\Aa{\sbull}{\sbull}{0}{0}}{\TAlga{n}{\sbull}{q}}{\Aasm{0}{\sbull}{0}{\sbull}}{\Aasm{\sbull}{0}{\sbull}{0}}
}
\]
\end{proposition}

\noindent This follows immediately by substituting~\eqref{eq:bentA} into the definition of the bent tensor product.

Under the bent motility hypothesis, we can use the representation \eqref{eq:newbent} to define an algebra structure on the bent tensor product by
\begin{equation}
\label{eq:bentp}
\begin{pmatrix}
  \zeta & a\\
  & \phi
\end{pmatrix}
\times
\begin{pmatrix}
  \zeta' & a'\\
  & \phi'
\end{pmatrix}
= \sum_{m, n}
\setlength{\arraycolsep}{1pt}
\pbmat 
\pbmat \zeta' \\[1ex] \zeta \pemat
&
\pbmat e^h_n & a' \\[1ex] a & e^v_m \pemat \\[3ex]
& (\phi \;\, \phi')
\pemat.
\end{equation}
for 
$$\zeta, \zeta' \in \RAlga{\sbull}{\sbull}{\sbull}, \ \phi, \phi' \in \TAlga{\sbull}{\sbull}{\sbull}, \ a, a' \in \Aa{\sbull}{\sbull}{0}{0}.$$
We also define a differential on the bent tensor product in the
obvious way, i.e.,
\[
\bdy\begin{pmatrix}
  \zeta & a\\
  & \phi
\end{pmatrix}
=\begin{pmatrix}
  \bdy \zeta & a\\
  & \phi
\end{pmatrix}
+\begin{pmatrix}
  \zeta & \bdy a\\
  & \phi
\end{pmatrix}
+\begin{pmatrix}
  \zeta & a\\
  & \bdy \phi
\end{pmatrix}.
\]

\begin{proposition}
\label{prop:BentAlgebra}
  Assuming that $\A$ satisfies the bent motility hypothesis, the operations above give the bent
  tensor product $\TAlg\obent \RAlg$ the structure of a differential algebra.
\end{proposition}
\begin{proof}
 Let us check associativity. Using \eqref{eq:bentp}, we see that 
\begin{equation}
\label{eq:bentassoc1}
\left(
\begin{pmatrix}
  \zeta & a\\
  & \phi
\end{pmatrix}
\times
\begin{pmatrix}
  \zeta' & a'\\
  & \phi'
\end{pmatrix}
\right)
\times
\begin{pmatrix}
  \zeta'' & a''\\
  & \phi''
\end{pmatrix}
= \sum_{m, n, m', n' \geq 0}
\setlength{\arraycolsep}{1pt}
\pbmat 
\pbmat \zeta'' \\[2ex]
\pbmat \zeta' \\[1ex] \zeta \pemat
\pemat
&
\pbmat
e^h_{n'} & a'' \\[2ex]
\pbmat e^h_n & a' \\[1ex] a & e^v_m \pemat 
& e^v_{m'} 
\pemat
\\[6ex]
& ((\phi \;\, \phi') \;\, \phi'')
\pemat
 \end{equation}
 and
 \begin{equation}
\label{eq:bentassoc2}
\begin{pmatrix}
  \zeta & a\\
  & \phi
\end{pmatrix}
\times
\left ( 
\begin{pmatrix}
  \zeta' & a'\\
  & \phi'
\end{pmatrix}
\times
\begin{pmatrix}
  \zeta'' & a''\\
  & \phi''
\end{pmatrix}
\right)
= \sum_{m', n', m'', n'' \geq 0}
\setlength{\arraycolsep}{1pt}
\pbmat 
\pbmat 
\pbmat \zeta'' \\[1ex] \zeta' \pemat \\[4ex]
\zeta
\pemat
&
\pbmat
e^h_{n''} & \pbmat e^h_{n'} & a'' \\[1ex] a' & e^v_{m'} \pemat  \\[4ex]
a & e^v_{m''} 
\pemat
\\[6ex]
& (\phi \;\, ( \phi' \;\, \phi''))
\pemat.
 \end{equation}
 
In the summands on the right hand sides of \eqref{eq:bentassoc1} and \eqref{eq:bentassoc2}, associativity of the horizontal and vertical multiplications suffices to identify the expressions involving $\zeta, \zeta', \zeta''$ and $\phi, \phi', \phi''$. We are left to show that:
\begin{equation}
\label{eq:bentassoc3}
\sum_{m, n, m', n' \geq 0}
\pbmat
e^h_{n'} & a'' \\[2ex]
\pbmat e^h_n & a' \\[1ex] a & e^v_m \pemat 
& e^v_{m'} 
\pemat
= \sum_{m', n', m'', n'' \geq 0} \pbmat
e^h_{n''} & \pbmat e^h_{n'} & a'' \\[1ex] a' & e^v_{m'} \pemat  \\[4ex]
a & e^v_{m''} 
\pemat.
\end{equation}
Without loss of generality, we can assume that $a, a', a''$ have fixed indices in the $2$-algebra, say 
$$a \in  \Aa{p}{q}{0}{0}, \ \ a' \in  \Aa{p'}{q'}{0}{0}, \ \ a'' \in  \Aa{p''}{q''}{0}{0}.$$ Under this assumption, the two sums in \eqref{eq:bentassoc3} consist of at most one term each, corresponding to 
$$n = q', \ n' = q'', \ m =p', \ m' = p'', \ n'' = n+n', \ m'' = m+m'.$$  
In this case, after writing 
$$e^h_{n'} = e^h_{n'} * e^h_{n'}, \ e^h_{n''} = \pbmat e^h_{n'} \\[1ex] e^h_n \pemat, \ e^v_{m'} = \pbmat e^v_{m'} \\[1ex] e^v_{m'} \pemat, \
e^v_{m''} = (e^v_m \; e^v_{m'})
$$ and applying local commutation we have
$$ \pbmat
e^h_{n'} & a'' \\[2ex]
\pbmat e^h_n & a' \\[1ex] a & e^v_m \pemat 
& e^v_{m'} 
\pemat
= 
\pbmat
e^h_{n'} & e^h_{n'} & a'' \\[1ex]
e^h_n & a' & e^v_{m'} \\[1ex] 
a & e^v_m & e^v_{m'} 
\pemat
=\pbmat
e^h_{n''} & \pbmat e^h_{n'} & a'' \\[1ex] a' & e^v_{m'} \pemat  \\[4ex]
a & e^v_{m''} 
\pemat.
$$
Compatibility of the multiplication on $\TAlg\obent \RAlg$ with the differential is easy to check. 
\end{proof}

Note that the algebra $\TAlg\obent\RAlg$ is not, strictly speaking, unital, as the element $\sum_{n,m} \pbmat 1^v_n & e_0 \\ & 1^h_m \pemat$ is not in $\TAlg\obent\RAlg$.

\subsection{2-modules as bent modules}
\label{sec:newperspective}
We keep the notation from the previous subsection and maintain the assumption that $\A$ satisfies the bent motility hypothesis. Suppose we have a differential module $M$ over $\TAlg \obent \RAlg$. Recall that the bent tensor product decomposes as a direct sum of pieces of the  form $\putaround{(\TAlg\obent \RAlg)}{m}{n}{p}{q}$. We say that the $\TAlg\obent\RAlg$-module $M$ is {\em indexed} if it admits a direct sum decomposition
$$ M = \bigoplus_{m,n \geq 0} \putaround{M}{}{}{m}{n}$$
such that the action of $\putaround{(\TAlg\obent \RAlg)}{m}{n}{p}{q}$ takes $\putaround{M}{}{}{m}{n}$ to $\putaround{M}{}{}{p}{q}$.  We say that an indexed $\TAlg\obent\RAlg$-module $M$ is \emph{unital} if the element $\pbmat 1^v_m & e_0 \\ & 1^h_n \pemat$ of $\TAlg\obent\RAlg$ acts as the identity on the summand $\putaround{M}{}{}{m}{n}$.

\begin {proposition}
\label{prop:2modsNew}
There is a one-to-one correspondence between unital (differential top-right) 2-modules over $\TAlg$ and $\RAlg$ and unital indexed (differential right) modules over $\TAlg \obent \RAlg$. 
\end {proposition}

\begin{proof}
Let $\TR$ be a unital $2$-module over $\TAlg$ and $\RAlg$. We define an indexed module $M$ over $\TAlg \obent \RAlg$ as follows. As a chain complex, each $\putaround{M}{}{}{m}{n}$ is the same as $\TRa{m}{n}$. We use Proposition~\ref{prop:quasar} to write a generator of $\putaround{(\TAlg\obent \RAlg)}{m}{n}{p}{q}$ as
\[
\begin{pmatrix}
  \zeta & a\\
  & \phi
\end{pmatrix} , 
\zeta \in \RAlga{m}{p}{\sbull}, \ a \in \Aa{\sbull}{\sbull}{0}{0}, \ \phi \in \TAlga{n}{\sbull}{q}.
\]
For $\x \in \putaround{M}{}{}{m}{n}$, we set:
\begin{equation}
\label{eq:Action2ModBent}
\x \ \begin{pmatrix}
  \zeta & a\\
  & \phi
\end{pmatrix}
 := \begin{pmatrix}
  \zeta & a\\
\x  & \phi
\end{pmatrix}
\end{equation}
Note that the matrix on the right hand side is well-defined by local commutation, because the indices match. 

The formula~\eqref{eq:Action2ModBent} is compatible with the two relations coming from the tensor product structures in \eqref{eq:newbent}. Furthermore, the action \eqref{eq:Action2ModBent} satisfies associativity with respect to the multiplication \eqref{eq:bentp} on the bent tensor product. The proof of this fact is entirely similar to the proof of associativity for \eqref{eq:bentp} in Proposition~\eqref{prop:BentAlgebra}.  The unitality relations \eqref{eq:alc4'} of the 2-module $\TR$ imply that the $\TAlg\obent\RAlg$-module $M$ is unital.

Conversely, suppose we have a unital indexed module $M$ over $\TAlg \obent \RAlg$. We define a $2$-module $\TR$ over $\TAlg$ and $\RAlg$ by letting $\TRa{m}{n}=\putaround{M}{}{}{m}{n}$ as chain complexes, and with the actions of the two algebra-modules given by:
\begin{eqnarray}
\label {eq:BentTo2}
 \x * \phi &:=& \sum_{m, n \geq 0} \x \pbmat 1_n^v & e^v_m \\[1ex] & \phi \pemat \\
 \label{eq:BentTo2'}
\bmat \zeta \\[.6ex] \cdot \\ \x \emat &:=&\sum_{m, n \geq 0} \x \pbmat \zeta & e^h_m \\[1ex] & 1^h_n \pemat.
 \end{eqnarray}
Note that the two summations on the right hand side are always finite, and that elements of the bent tensor product are here represented using the definition \eqref{eq:bentdef}, not decomposition \eqref{eq:newbent}.

We need to check that $\TR$ is indeed a $2$-module. To verify the associativity of the $\TAlg$-action, it suffices to consider elements $\x \in \TRa{p}{\sbull}, \phi \in \TAlga{\sbull}{m}{\sbull}, \phi' \in \TAlga{\sbull}{m'}{\sbull}$ and compare the expression
\begin{equation}
\label{eq:lunar}
 \x * (\phi * \phi') = \x \ \pbmat 1^v_p & e^v_{m + m'} \\[1ex] &(\phi \ \phi')  \pemat 
 \end{equation}
with the expression
\begin{equation}
\label{eq:solar}
 (\x * \phi) * \phi' = \x \ \pbmat 1^v_p & e^v_{m } \\[1ex] &\phi \pemat \  \pbmat 1^v_{p+m} & e^v_{m'} \\[1ex] & \phi' \pemat.
 \end{equation}

In order to multiply the two matrices on the righthand side of \eqref{eq:solar}, we need to write them in the form given by \eqref{eq:newbent}. To do this, we use the bent motility hypothesis to write
\begin{equation}
\label{eq:evm}
 e^v_{m} = \sum_{i \in I} f(i)^{m}  * g(i)_{m} \ \text{with} \ f(i)^m \in \Aa{0}{0}{m}{\sbull} \ \text{and} \ g(i)_{m} = \Aa{m}{\sbull}{0}{0}
 \end{equation}
\begin{equation}
\label{eq:evm'}
  e^v_{m'} = \sum_{j \in J} f(j)^{m'} * g(j)_{m'} \ \text{with} \ f(j)^{m'} \in \Aa{0}{0}{m'}{\sbull} \ \text{and} \ g(j)_{m'} = \Aa{m'}{\sbull}{0}{0}
  \end{equation}
where $I$ and $J$ are finite index sets. We obtain
\[
\x  \pbmat 1^v_p & e^v_{m } \\[1ex] &\phi \pemat \  \pbmat 1^v_{p+m} & e^v_{m'} \\[1ex] & \phi' \pemat = \sum_{i, j} \x \pbmat \pbmat (1^v_{p+m} \; f(j)^{m'}) \\[1ex] (1^v_p  \; f(i)^m) \pemat & \pbmat e^h_k & g(j)_{m'} \\[1ex] 
 g(i)_m & e^v_{m'} \pemat 
\\[3ex] 
& (\phi \ \phi')
\pemat.
\]
After decomposing $1^v_{p+m}$ as $(1^v_p \ e^v_m)$ and using local commutation in $\RAlg$, the above expression becomes
\[
 \sum_{i, j} \x \pbmat 
 \pbmat 1^v_p \\[1ex] 1^v_p \pemat
 \pbmat (e^v_{m} \; f(j)^{m'}) \\[1ex]  f(i)^m \pemat 
 & 
 \pbmat e^h_k & g(j)_{m'} \\[1ex] 
 g(i)_m & e^v_{m'} \pemat 
\\[3ex] 
& (\phi \ \phi')
\pemat.
\]
Furthermore, we can move the term $\pbs (e^v_{m} \; f(j)^{m'}) \\[1ex]  f(i)^m \pes$ to the $2$-algebra side of the bent tensor product using \eqref{eq:defBentG}, and then observe that
\[
\sum_{i,j,k}
  \pbmat (e^v_{m} \; f(j)^{m'}) \\[1ex]  f(i)^m \pemat 
 \pbmat e^h_k & g(j)_{m'} \\[1ex] 
 g(i)_m & e^v_{m'} \pemat
= \pbmat e^v_m & e^v_{m'} \\[1ex] e^v_{m} & e^v_{m'} \pemat = e^v_{m+m'}. 
\]
Thus, we arrive at the desired expression \eqref{eq:lunar}.

Associativity of the $\RAlg$-action on $\TR$ is similar. 

To prove local commutation in $\TR$, suppose we are given $\x \in \TRa{p}{q}, \phi \in \TAlga{q}{m}{q'}, \zeta \in \RAlga{p}{p'}{n}$ and $a \in \Aa{m}{n}{m'}{n'}$. Decompose $e^v_{m'}$ as in \eqref{eq:evm'} and write also
$$ e^h_n = \sum_{k \in K} \pbmat \;_n\hat g(k) \\[1ex]   \hat f(k)_n \pemat \ \text{with} \ \hat f(k)_n\in \Aa{0}{0}{\sbull}{n} \ \text{and} \ \;_n\hat g(k) \in \Aa{\sbull}{n}{0}{0}$$
for some index set $K$. Then:
\begin{align}
\label{eq:agree}
 \pbmat \zeta \\[1ex] \x \pemat \pbmat a \\[1ex] \phi \pemat &= \x \pbmat \zeta & e^h_n \\[1ex] & 1^h_q \pemat \pbmat 1^v_{p'} & e^v_{m'} \\[1ex] & \pbmat a \\ \phi \pemat \pemat\\[1ex] \notag
 & = 
\sum_{j, k}  \x \pbmat \zeta & \;_n\hat g(k) \\[.5ex] & \pbmat \hat f(k)_n \\[1ex] 1^h_q \pemat \pemat \pbmat (1^v_{p'} \ f(j)^{m'}) & g(j)_{m'} \\[2ex] & \pbmat a \\[1ex] \phi \pemat \pemat 
\\[1ex] \notag
&=  \sum_{j, k, l} \x \pbmat \pbmat 1^v_{p'} & f(j)^{m'} \\[.5ex] \zeta & e^h_n \pemat & \pbmat e^h_l & g(j)_{m'} \\[.5ex]  \;_n\hat g(k) & e^v_{m'} \pemat \\[3ex] & \pbmat \hat f(k)_n & a \\[.3ex] 1^h_q & \phi \pemat  \pemat
\\[1ex] \notag
&= \sum_{j, k, l} \pbmat \pbmat1^v_{p'} \\[2.5ex] \zeta \pemat & \pbmat 
f(j)^{m'} & e^h_l & g(j)_{m'}\\[.7ex] 
e^h_n & \;_n\hat g(k) & e^v_{m'} \\[.7ex] 
e_0 &  \hat f(k)_n  & a
\pemat \\[4ex] & (\ 1^h_q \hspace{1.4cm}  \phi \ ) \pemat = \x \pbmat \zeta & a \\ & \phi \pemat. \notag
\end{align}
A symmetric argument shows that
$$ \bmat (\zeta \ a)\\[1ex] (\x \ \phi)\emat  = \x \pbmat \zeta & a \\ & \phi \pemat.$$
This implies local commutation. The 2-module unitality relations \eqref{eq:alc4'} follow from the unitality condition on the $\TAlg\obent\RAlg$-module.  We conclude that $\TR$ is a $2$-module.

Finally, we need to check that the two constructions defined above (from a $2$-module to a module over the bent tensor product, and vice versa) are inverse to each other. If we start from $M$, construct $\TR$ according to ~\eqref{eq:BentTo2} and \eqref{eq:BentTo2'}, and then construct a module $M'$ according to ~\eqref{eq:Action2ModBent}, clearly the chain complexes $\putaround{M}{}{}{m}{n}$ agree with $\putaround{M'}{}{}{m}{n}$. The fact that the actions of the bent tensor product on $M$ and $M'$ agree is exactly what was proved in \eqref{eq:agree}.
 
Going the other way, from a $2$-module $\TR$ to $M$ and then to another 2-module $\TR'$, clearly $\TRa{m}{n} = \putaround{TR'}{}{}{m}{n}$ as chain complexes. Let us denote the actions of the algebra-modules on $\TR'$ by adding a prime to the corresponding multiplication symbol. We check that the horizontal $\TAlg$-action on $\TR'$ agrees with the one on $\TR$:
$$\x *' \phi = \sum_{m, n \geq 0} \x \pbmat 1_n^v & e^v_m \\[1ex] & \phi \pemat =\sum_{m, n \geq 0} \pbmat 1_n^v & e^v_m \\[1ex] \x & \phi \pemat = \x * \phi.$$
 Similarly, the vertical $\RAlg$-actions on $\TR$ and $\TR'$ agree. \end{proof}

In particular, we can apply Proposition~\ref{prop:2modsNew} to the bent tensor product $\TAlg \obent \RAlg$ viewed as a right module over itself. This equips it with the structure of a top-right $2$-module over $\TAlg$ and $\RAlg$. Combining this with the structure of $\TAlg \obent \RAlg$ as a left module over itself, we find that $\TAlg \obent \RAlg$ is a left-top-right module--$2$-module in the sense of Definition~\ref{def:module-2-module}. Let us write $\widetilde{\TAlg \obent \RAlg}$ for $\TAlg \obent \RAlg$ when viewed as a module--$2$-module. 

One direction of the identification in Proposition~\ref{prop:2modsNew} can now be expressed in the following language: given a module $M$ over $\TAlg \obent \RAlg$, the corresponding $2$-module is given by
\begin{equation}
\label{eq:tensorWithBent}
 \TR = M \underset{\scriptscriptstyle \TAlg \obent \RAlg}\otimes \bigl( \widetilde{\TAlg \obent \RAlg} \bigr).
 \end{equation}

\subsection{The smoothed tensor product}\label{sec:smoothed}
We define the \emph{smoothed tensor
  product} of $\TAlg$ and $\RAlg$ to be
$$\TAlg\osmooth\RAlg := \putaround{(\TAlg\obent \RAlg)}{0}{0}{0}{0}.$$  
This is a
subalgebra of the bent tensor product, so we can view $\TAlg\obent\RAlg$
as a
$(\TAlg\osmooth\RAlg,\TAlg\obent\RAlg)$-bimodule. The actions preserve 
$\putaround{\TAlg\obent\RAlg}{0}{0}{\sbull}{\sbull}$, so we can view
$\putaround{\TAlg\obent\RAlg}{0}{0}{\sbull}{\sbull}$  as a
$(\TAlg\osmooth\RAlg,\TAlg\obent\RAlg)$-bimodule, as well.

A key construction in this paper (going from bordered modules to cornered $2$-modules) has the following algebraic underpinning. If we have a module $M$ over the smoothed tensor product $\TAlg\osmooth\RAlg$, we can use an induction functor to get a module $M'$ over the bent tensor product:
$$ M' =  M \underset{\scriptscriptstyle \TAlg \osmooth \RAlg}\otimes  \bigl( \putaround{\TAlg\obent \RAlg}{0}{0}{\sbull}{\sbull} \bigr).$$
We then combine this with \eqref{eq:tensorWithBent} to obtain a $2$-module $\TR$ over $\TAlg$ and $\RAlg$:
 $$ TR = \left( M \underset{\scriptscriptstyle \TAlg \osmooth \RAlg}\otimes \bigl( \putaround{\TAlg\obent \RAlg}{0}{0}{\sbull}{\sbull}\bigr) \right)  \underset{\scriptscriptstyle \TAlg \obent \RAlg}\otimes \bigl( \widetilde{\TAlg \obent \RAlg} \bigr).$$

\subsection{The bottom-left bent tensor product}
\label{sec:blbent}
Everything in sections \ref{sec:benttens}, \ref{sec:newperspective}, and \ref{sec:smoothed} applies equally well to bottom-left $2$-modules. We can define the bent tensor product of $\BAlg$ and $\LAlg$ as
$$\BAlg\obent \LAlg := \mathcenter{\bendyou{\BAlga{\sbull}{\sbull}{\sbull}}{\Aa{\sbull}{\sbull}{\sbull}{\sbull}}{\LAlga{\sbull}{\sbull}{\sbull}}{\Aasm{\sbull}{\sbull}{\sbull}{0}}{\Aasm{\sbull}{\sbull}{0}{\sbull}}}
$$
Under the same bent motility hypothesis from Definition~\ref{def:bentMH}, the analogues of Propositions~\ref{prop:quasar} and \ref{prop:BentAlgebra} hold, so we can write
$$ \BAlg\obent \LAlg \cong \mathcenter{\bendyou{\BAlga{\sbull}{\sbull}{\sbull}}{\Aa{0}{0}{\sbull}{\sbull}}{\LAlga{\sbull}{\sbull}{\sbull}}{\Aasm{\sbull}{0}{\sbull}{0}}{\Aasm{0}{\sbull}{0}{\sbull}}}
$$
and use this to define an algebra structure on $\BAlg \obent \LAlg$. Moreover, we have a one-to-one correspondence between unital bottom-left $2$-modules and unital indexed modules over $\BAlg \obent \LAlg$, similar to the correspondence in Proposition~\ref{prop:2modsNew}.  Also, we can define a smoothed tensor product $\BAlg \osmooth \LAlg$ as the summand of $\BAlg \obent \LAlg$ indexed by four zeros.

\subsection{\texorpdfstring{$\Ainf$}{A-infinity}-2-modules}
\label{sec:2ainf}
In bordered Floer homology, the invariant $\CFAa$ is typically an
$\Ainf$-module. While the cornered invariants in this paper are always
honest \dg modules (remember that one can always resolve an
$\Ainf$-module by an honest module), it is natural to try to define a
notion of $\Ainf$-$2$-module. 

Combining a $2$-categorical
structure with an $\Ainf$-structure is a subtle task; the best proposals take the form of $E_2$-algebras \cite{LurieTQFT, LurieEk} or variants of these \cite[Section 6.1]{MorrisonWalker}. 

In our setting, recall that in Proposition~\ref{prop:2modsNew} we established an identification between $2$-modules and modules over the bent tensor product. We could therefore define an $\twoainf$-module over $\TAlg$ and $\RAlg$ to be an $\Ainf$-module over $\TAlg \obent \RAlg$; similarly, for $\twoainf$-modules over $\BAlg$ and $\LAlg$.

The tensor products $\htp$ and $\vtp$ used in
Theorem~\ref{thm:pairing} are well-defined for $2$-modules. We do not
know a direct definition of these tensor products if, for example, the
top-right $2$-module is replaced by an $\twoainf$-module. However, one
can resolve the $\Ainf$-$2$-module by an honest $2$-module and then
tensor.


\section{Some homological algebra of 2-modules}
\label{sec:quasi}
To prove Theorem~\ref{thm:invariance}, we need a little homological
algebra of $2$-modules. Specifically, we will define $\CAA(\HD_{00})$,
say, by tensoring the bordered module $\CFAa(\widetilde{\HD}_{00})$ of
the smoothing of $\HD_{00}$ with the cornering module--$2$-module
$\cornAA(K_{00})$ (compare Equation~\eqref{eq:caa-intro}). Changing
$\HD_{00}$ to a different Heegaard diagram $\HD'_{00}$ representing
the same cornered $3$-manifold gives a quasi-isomorphic module
$\CFAa(\widetilde{\HD}'_{00})$, and we need to know that the derived tensor
product
$\CFAa(\widetilde{\HD}_{00})\tilde\otimes_{\Alg(F_{00})}\cornAA(K_{00})$ is
quasi-isomorphic to the derived tensor product
$\CFAa(\widetilde{\HD}'_{00})\tilde\otimes_{\Alg(F_{00})}\cornAA(K_{00})$. 
To construct the derived tensor product we need to first choose a projective 
resolution of one of the factors, and then tensor. In the present section we will
prove that as long as either side is projective, tensoring
respects quasi-isomorphism. Similarly, in the
pairing theorem, Theorem~\ref{thm:pairing}, since the $2$-modules
$\CAA(Y_{00})$, $\CAD(Y_{01})$, $\CDA(Y_{10})$, and $\CDD(Y_{11})$ are
only well-defined up to quasi-isomorphism, we ought to verify that the
tensor products respect quasi-isomorphism. In this case, the $D$ modules in the tensor products are already flat; this is
verified in Section~\ref{sec:nice}
(Proposition~\ref{prop:D-proj}), using
Proposition~\ref{prop:biproj-tensor}. That the tensor products respect
quasi-isomorphism will then follow from Lemma~\ref{lem:tensor-resp-q-i}.

Most of the discussion in this section is inspired by the efficient
treatment of homological algebra for differential modules
in~\cite{BernsteinLunts94:EquivariantSheaves}.

For brevity, we will generally state definitions and lemmas for
top-right $2$-modules only; the statements for other kinds of $2$-modules
are exactly analogous. Fix a $2$-algebra $\A$, a right
algebra-module $\RAlg$ over $\A$ and a top algebra-module $\TAlg$
over $\A$. 

\begin{assumption}
  Throughout this section, we will assume that the algebra-modules and
  $2$-modules satisfy the appropriate motility
  hypotheses, so that tensor products or module structures are well-defined.
\end{assumption}

\begin{definition}
  Let $M$ and $N$ be top-right $2$-modules over $\RAlg$ and
  $\TAlg$. A \emph{morphism} from $M$ to $N$ is a family of maps 
  \[
  \fa{m}{n}\co \putaround{M}{}{}{m}{n}\to\putaround{N}{}{}{m}{n} 
  \]
  such that $f(\x*\phi)=f(\x)*\phi$ and $f \pbs \zeta \\ \cdot \\ \x \pes=\bs \zeta \\ \cdot \\ f(\x) \es$
  for any $\phi\in \TAlga{n}{m'}{q}$ and $\zeta\in \RAlga{m}{p}{n'}$.

  Let $\Mor(M,N)$ denote the vector space of morphisms from $M$
  to $N$. Composition of morphisms is defined term-by-term: given $f\in \Mor(M,N)$ and $g\in\Mor(N,P)$ we define $g\circ f$ by $\putaround{(g\circ f)}{}{}{m}{n}=\putaround{g}{}{}{m}{n}\circ\putaround{f}{}{}{m}{n}$. Further, there is a differential on $\Mor(M,N)$ defined by 
  \[
  \bdy(f)(\x)=\bdy_N(f(\x))+f(\bdy_M(\x)).
  \]
  (Again, this definition is implicitly term-by-term.) Let $\RTModCat$ denote the differential category of top-right modules over $\RAlg$ and $\TAlg$.

  If $\bdy(f)=0$ then $f$ is called a \emph{homomorphism}. If $f=\bdy(h)$ then $f$ is called \emph{null-homotopic}; more generally, if $f-g=\bdy(h)$ then $f$ and $g$ are called \emph{homotopic}.  The \emph{homotopy category} of top-right $2$-modules is obtained by replacing each morphism space $\Mor(M,N)$ by its homology. A homomorphism $f$ is a \emph{homotopy equivalence} if its image in the homotopy category is an isomorphism.

  Given a top-right $2$-module $M$, the homology of $M$ is the family of vector spaces $\{\putaround{H(M)}{}{}{m}{n}=H_*\bigl(\putaround{M}{}{}{m}{n}\bigr)\mid m,n\in\NN\}$. (This is naturally a $2$-module over $H(\TAlg)$, $H(\RAlg)$, and $H(\A)$, but we will not use this fact.) Given a homomorphism $f\co M\to N$ there are induced maps $\putaround{(f_*)}{}{}{m}{n}\co \putaround{H(M)}{}{}{m}{n}\to \putaround{H(N)}{}{}{m}{n}$. We say $f$ is a \emph{quasi-isomorphism} if  $\putaround{(f_*)}{}{}{m}{n}$ is an isomorphism for each $m$ and $n$.
\end{definition}

There is a forgetful functor from $\RTModCat$ to the category $\RModCat$ of top $\RAlga{\bullet}{\bullet}{\bullet}$-modules (respectively $\TModCat$ of right $\TAlga{\bullet}{\bullet}{\bullet}$-modules). 

\begin{definition}
  We say that $M$ is \emph{right-projective} (respectively
  \emph{top-projective}) if $M$ is projective as a right
  $\TAlga{\bullet}{\bullet}{\bullet}$-module (respectively top
  $\RAlga{\bullet}{\bullet}{\bullet}$-module).\footnote{By projective,
    we always mean categorically projective or
    $\mathcal{K}$-projective as a differential module, in the spirit of~\cite{Spaltenstein88:resolutions,BernsteinLunts94:EquivariantSheaves}, as opposed to merely projective as a module (without the differential).} Similarly, $M$ is \emph{right-flat} (respectively \emph{top-flat}) if $M$ is flat as a right $\TAlga{\bullet}{\bullet}{\bullet}$-module (respectively top $\RAlga{\bullet}{\bullet}{\bullet}$-module).
\end{definition}

From now on, we will generally state results about right-projectivity or right-flatness; the corresponding results for top-projectivity or top-flatness are exactly the same.

\begin{lemma}\label{lem:proj-is-flat}
  If $M$ is right-projective then $M$ is right-flat.
\end{lemma}
\noindent  This is immediate from the corresponding result for (ordinary) differential
  modules; see, for instance,~\cite[Corollary
  10.12.4.4]{BernsteinLunts94:EquivariantSheaves}.

Suppose that $M$ is a top-right module over $\RAlg$ and $\TAlg$, that $N$ and
$P$ are top-left modules over $\LAlg$ and $\TAlg$, and $f\co N\to
P$ is a morphism. Then we can form the tensor products $M\htp_{\TAlg}N$ and
$M\htp_{\TAlg}P$. The morphism $f$ gives a morphism $\Id\otimes
f\co M\htp_{\TAlg}N\to M\htp_{\TAlg}P$ of top $(\RAlg\htp_\A
\LAlg)$-modules, defined by $(\Id\otimes f)(m\htp n)=m\htp f(n)$.
The map $\Id\otimes f$ also restricts to a morphism of restricted tensor products.

\begin{lemma}\label{lem:tensor-resp-q-i}
  Let $M$ be a top-right module over $\RAlg$ and $\TAlg$, let
  $N$ and $P$ be top-left modules over $\LAlg$ and $\TAlg$, and let $f\co
  N\to P$ be a quasi-isomorphism. Suppose that either $M$ is right flat or $N$ and $P$ are both left flat. Then
  $\Id\htp f$ is a quasi-isomorphism, as is its restriction to the
  restricted tensor product.
\end{lemma}
\begin{proof}
  As a chain complex, $M\htp N$ is just
  $\putaround{M}{}{}{\bullet}{\bullet}\otimes_{\TAlgasm{\bullet}{\bullet}{\bullet}}\putaround{N}{}{\bullet}{\bullet}{}$,
  so this follows from the analogous result for ordinary differential modules.
\end{proof}

By the same argument, the three rotated versions of
Lemma~\ref{lem:tensor-resp-q-i} also hold, as does its analogue, say,
if $M$ is a module--$2$-module over $\RAlg_2$, $\TAlg$, and a
differential algebra $\Alg$ (with the $\Alg$ action on the left) and
$N$ and $P$ are (right) modules over $\Alg$.

We will also need some results about projectivity of the more exotic
$2$-objects of Section~\ref{sec:bestiary}. Let $\Barbell$ be a
top-bottom algebra-bimodule over $\A_1$ and $\A_2$, and $N$ a
bimodule-module over $\Barbell$, $\RAlg_1$, and $\RAlg_2$. Forgetting
the action of $\Barbell$ makes $N$ into an ordinary bimodule over
$\RAlg_1$ and $\RAlg_2$. We say that $N$ is \emph{biprojective} over
$\RAlg_1$ and $\RAlg_2$ if $N$ is biprojective as an ordinary
bimodule, i.e., is a projective module over $\RAlg_1\otimes\RAlg_2^\op$.

Before our next result, we recall a fact
from homological algebra:
\begin{lemma}\label{lem:biproj-tensor-baby}
  If $A$ and $B$ are differential algebras, $M$ is a biprojective bimodule over
  $A$ and $B$, and $N$ is a left $B$-module then $M\otimes_B N$ is a
  flat left $A$-module.
\end{lemma}
\begin{proof}
  Since $M$ is biprojective, $M$ is flat over $A\otimes B^\op$, i.e.,
  the functor $\mathcal{F}(-)=- \otimes_{A\otimes B^\op} M$ is
  exact. But the functor $\mathcal{G}(-)=-\otimes_A
  (M\otimes_B N)$ satisfies
  $\mathcal{G}(-)=\mathcal{F}(-\otimes N)$. Since $N$ is certainly flat as an $\Field$-module, it follows
  that $\mathcal{G}$ is exact, as well.
\end{proof}

The fancy version of this result, in our setting is:

\begin{proposition}\label{prop:biproj-tensor} Let $N$ be a
  biprojective bimodule-module over $\Barbell$, $\RAlg_1$, and
  $\RAlg_2$ and $M$ a $2$-module over $\RAlg_2$ and $\TAlg$. Then the
  tensor product  (respectively, the restricted tensor product) of $M$ and
  $N$ over $\RAlg_2$ is flat over $\RAlg_1$ (respectively, over $\RAlg_1^v$).

  The same statement holds if $M$ is a module--$2$-module over
  $\RAlg_2$, $\TAlg$, and a differential algebra $\Alg$.
\end{proposition}
\begin{proof}
  For the unrestricted tensor product, this is immediate from the definitions and
  Lemma~\ref{lem:biproj-tensor-baby}, by forgetting the actions of $\Barbell$ and
  $\TAlg$. 
  
  The restricted tensor product of $M$ and $N$ is the tensor product of the vertically sequential objects $M^v$ and $N^v$, over $\RAlg_2^v$. The claim would follow, as in the proof of Lemma~\ref{lem:biproj-tensor-baby}, if we can show that $N^v$ is projective (hence flat) over $\RAlg_1^v \otimes \RAlg_2^{v, \op}$; it suffices to show that the functor $$\Hom_{\RAlg_1^v \otimes \RAlg_2^{v, \op}} (N^v, -)$$ is exact. If that functor were not exact, then there would exist an acyclic $(\RAlg_1^v, \RAlg_2^v)$-bimodule $Q$ such that  $\Hom_{\RAlg_1^v \otimes \RAlg_2^{v, \op}} (N^v, Q)$ is not acyclic (compare~\cite[Definition 10.12.2.1]{BernsteinLunts94:EquivariantSheaves}). We can view $Q$ as a bimodule $Q'$ over $\RAlg_1$ and $\RAlg_2$ by
setting $\putaround{Q'}{m}{}{n}{0}=\putaround{Q}{m}{}{n}{}$ and $\putaround{Q'}{m}{}{n}{p}=0$ for $p>0$, and
 letting the elements of $\putaround{(\RAlg_i)}{\sbull}{}{\sbull}{n}$ ($i=1,2$) act by zero when $n > 0$. Observe that
$$\Hom_{\RAlg_1 \otimes \RAlg_2^{ \op}} (N, Q) \cong  \Hom_{\RAlg_1^v \otimes \RAlg_2^{v, \op}} (N^v, Q).$$
 The left hand side is acyclic because $N$ is biprojective. We conclude that $Q$ could not have been acyclic, as required.
\end{proof}

Finally, we observe that the correspondence of 2-modules and modules over the bent tensor product, from
Section~\ref{sec:newperspective}, descends to the derived category:
\begin{proposition}\label{prop:bent-der-cat} The identification in
  Proposition~\ref{prop:2modsNew} respects quasi-isomorphism.
\end{proposition}

\begin{proof}
This is clear from the construction: as chain complexes, the $2$-modules are identical to the corresponding modules over the bent tensor product.
\end{proof}


\section{The algebras and algebra-modules} 
\label{sec:am}
\subsection{The algebra associated to a matched circle}
\label{sec:AZ}
We start by briefly reviewing the definition of the algebra associated to a surface in bordered Heegaard Floer theory, following~\cite{LOT1, LOT4}.

For $n \geq 1$, we will write $[n]$ for the finite set $\{1, 2, \dots, n\}$.

\begin{definition}\cite[Definition 3.9]{LOT1}
A \emph{pointed matched circle} is a quadruple $\cZ=(Z, \CircPts, M, z) $ consisting  of an oriented circle $Z$, a basepoint $z \in Z$, a collection $\CircPts$ of $4k$ distinct points on $Z \setminus \{z\}$, and a  $2$-to-$1$ matching function $M: \CircPts \to [2k]$. The matching describes that the points $i$ and $j$ are paired with each other, whenever $M(i)=M(j)$. We require that surgery along these $2k$ pairs of points yields a single circle.
\end{definition}

Given a pointed matched circle, surgery on the pairs of points provides a cobordism from a circle to a circle; capping this cobordism with two discs gives a closed surface $F: = F(\cZ)$.  Conversely, any surface can be represented by a pointed matched circle.

Deleting the basepoint from $Z$ yields an oriented interval $Z'$. We define a {\em strand diagram} $v$ to be a collection of arcs in $[0,1] \times Z'$, where each arc is the graph of a smooth, non-decreasing function $f:[0,1] \to Z'$, such that both the initial point $f(0)$ and the final point $f(1)$ are in the distinguished set $\CircPts$. Further, we require the following:
\begin{itemize}
\item If $v$ contains a horizontal strand starting at some point $i \in \CircPts$, then it also contains the horizontal strand starting at $j$, where  $j \neq i$ is the point matched with $i$, that is, $M(i) = M(j)$. (Graphically, the horizontal strands will be drawn as dashed lines.)
\item If there is some non-horizontal strand in $v$ starting at some $i \in \CircPts$, then there is no other strand in $v$ starting at either $i$ or the point matched with $i$.
\item If there is some non-horizontal strand in $v$ ending at some $i \in \CircPts$, then there is no other strand in $v$ ending at either $i$ or the point matched with $i$.
\end{itemize}
Given a strand diagram $v$, let $\In(v), \Out(v) \subseteq \CircPts$ be the sets of initial and final points for the arcs in $v$. We define the {\em initial (\resp final) idempotent} of $v$ to be the strand diagram consisting of the horizontal strands starting at the points in $M^{-1}(M(\In(v))$, \resp $M^{-1}(M(\Out(v))$.

The matching algebra $\Alg(\cZ)$ is generated (over $\Field$) by strand diagrams, modulo the following relations, which are similar to those shown in Figure~\ref{fig:nilcoxrel} for the nilCoxeter algebra:
\begin{itemize}
\item If two strands in a diagram $v$ intersect each other more than once, we set $v = 0$.
\item Two diagrams related by isotopy of $[0,1] \times Z'$ (rel its boundary) are set to be equal to each other. 
\item If in a diagram $v$ we push a strand past a crossing between two other strands, we get a new diagram $v'$, and we set $v=v'$.
\end{itemize}

Multiplication in $\Alg(\cZ)$ is defined by concatenating strand diagrams horizontally, with the product $v * v'$ being set to zero unless the final idempotent of $v$ coincides with the initial idempotent of $v'$; further, in the concatenation $v * v'$ we delete any horizontal strands that do not go all the way across.
Here are some examples (with the matching drawn in red):
$$\includegraphics[scale=0.6]{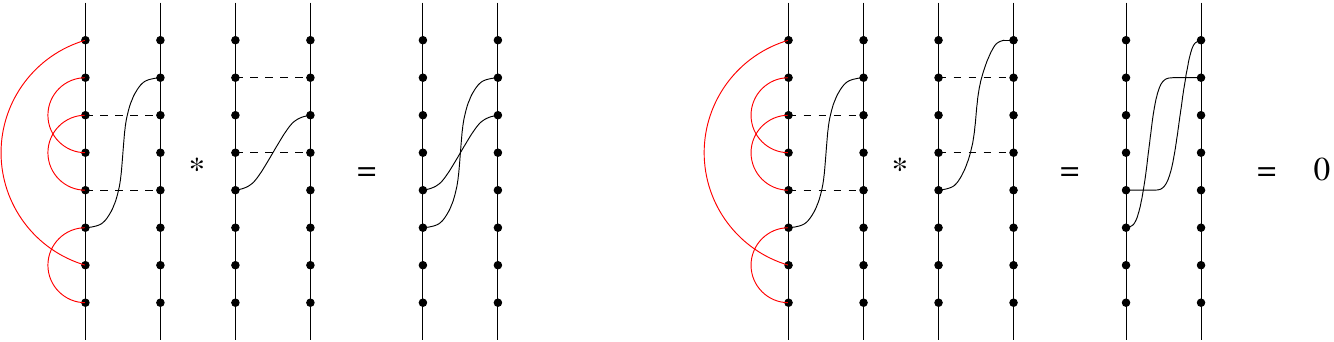}$$

We turn $\Alg(\cZ)$ into a differential algebra by defining $\del v$ to be the sum (over all crossings in $v$) of the oriented resolutions of $v$ at that crossing, and then deleting the horizontal strands that are no longer matched with a horizontal strand. For example,
\begin{center}
\begin{picture}(0,0)%
\includegraphics{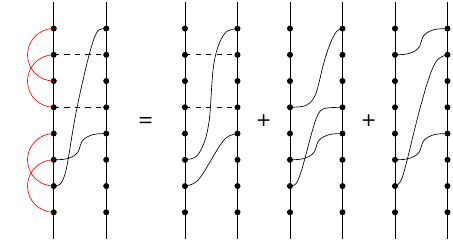}%
\end{picture}%
\setlength{\unitlength}{2763sp}%
\begingroup\makeatletter\ifx\SetFigFont\undefined%
\gdef\SetFigFont#1#2#3#4#5{%
  \reset@font\fontsize{#1}{#2pt}%
  \fontfamily{#3}\fontseries{#4}\fontshape{#5}%
  \selectfont}%
\fi\endgroup%
\begin{picture}(5153,2724)(886,-973)
\put(901,314){\makebox(0,0)[lb]{\smash{{\SetFigFont{11}{13.2}{\rmdefault}{\mddefault}{\updefault}{\color[rgb]{0,0,0}$\del$}%
}}}}
\end{picture}%

\end{center}

We conclude this section by introducing some notation for particular
elements of $\Alg(\PMC)$.

Given a subset $\SetS$ of the set of matched pairs in $\PMC$ (or
equivalently, a subset of $[2k]$), there is a corresponding idempotent
$I(\SetS)\in\Alg(\PMC)$, which has a pair of horizontal strands at each element
of $\SetS$. The idempotents $I(\SetS)$ (for varying $\SetS$) form an
orthogonal (with respect to multiplication) basis for the vector space of
idempotents $\Idem(\PMC)$ in $\Alg(\PMC)$.

Next, by a \emph{chord} $\rho$ in $\PMC$ we mean a sub-arc of $Z\setminus z$ with boundary in $\CircPts$. Given a chord $\rho$ in $\PMC$ there is an associated algebra element $a(\rho) \in \Alg(\PMC)$ gotten as follows:
\begin{enumerate}
\item View $\rho$ as the image of an orientation-preserving map $\rho\co [0,1]\to Z$.
\item Consider the graph $\Gamma_\rho\subset Z\times[0,1]$ of $\rho$.
\item Take the sum over all ways of adding horizontal strands to $\Gamma_\rho$ to obtain a valid strand diagram. The result is $a(\rho)$.
\end{enumerate}
Equivalently, given a strand diagram $a$, projecting to $Z$ gives an element $[a]\in H_1(Z\setminus\{z\},\CircPts)$; we call this the \emph{support} of $a$. Then $a(\rho)$ is the sum of all strand diagrams with a single non-horizontal strand and with support $[\rho]$.

The same construction works if one replaces $\rho$ with a set of chords $\rhos=\{\rho_1,\dots,\rho_n\}$. In this case, the element $a(\rhos)$ will be a sum of terms, each with $i$ non-horizontal strands and some number of horizontal strands. (The element $a(\rhos)$ may be zero, for instance if two of the $\rho_i$'s have the same initial endpoint in $\CircPts$.) Any strand diagram is of the form $I(\SetS)a(\rhos)$ for some subset $\SetS$ of the matched pairs and set of chords $\rhos$.

\subsection{The algebra-modules associated to matched intervals} \label{sec:amint} Just as closed surfaces can be represented by matched circles, surfaces with circle boundary can be represented by matched intervals:

\begin{definition}\cite[Definition 4.2]{DM:cornered}
A \emph{matched interval} is a triple $\PMC=(Z, \CircPts, M) $ consisting of an oriented compact interval $Z$, a collection $\CircPts$ of $4k$ points in the interior of $Z$, and a {\em matching}, i.e., a $2$-to-$1$ function $M: \CircPts \to [2k]$. We require that performing surgery along the $2k$ matched pairs of points yields a single interval. 
\end {definition}

\begin{definition}
Given a pointed matched circle $\PMC=(Z,\CircPts,M,z)$ (respectively matched interval $\PMC=(Z,\CircPts,M)$), let $-Z$
  denote $Z$ with its orientation reversed and let
  $-\PMC=(-Z,\CircPts,M,z)$ (respectively $-\PMC=(-Z,\CircPts,M)$). Let $r\co \PMC\to -\PMC$ denote the identity map (which is orientation-reversing).
\end{definition}

Performing surgery on all the pairs of points in a matched interval $\PMC$ provides a cobordism from an interval to an interval. We can view that cobordism as a surface $F$ with circle boundary. It is important to identify the boundary of $F$ to the standard circle $S^1$ (up to canonical diffeomorphism). In order to do that, it suffices to pick a basepoint on $\del F$ and map it to $1 \in S^1$.
The two endpoints of the interval $\PMC$ provide two natural choices of basepoint. If we pick the initial (negatively oriented) endpoint, we write the based surface $F$ as $F_{\circ}(\PMC)$. If we pick the final (positively oriented) endpoint, we write the based surface $F$ as $F^{\circ}(\PMC)$. Note that there is a canonical (based, orientation preserving) diffeomorphism from $F_{\circ}(\PMC)$ to $-F^{\circ}(-\PMC)$.

Now, fix a matched interval $\PMC=(Z,\CircPts,M)$. We will
associate to $\PMC$ four algebraic objects:
\begin{itemize}
\item A top algebra-module $\TAlg=\TAlg(\PMC)$,
\item A bottom algebra-module $\BAlg=\BAlg(\PMC)$,
\item A right algebra-module $\RAlg=\RAlg(\PMC)$,
\item A left algebra-module $\LAlg=\LAlg(\PMC)$,
\end{itemize}
all over the diagonal nilCoxeter $2$-algebra $\Dnil$ from Definition~\ref{def:ex2}.

We start by defining the top algebra-module $\TAlg = \TAlg(\PMC)$. In \cite[Section 5.2]{DM:cornered} there was a sequential top algebra-module\footnote{In \cite{DM:cornered} the algebra-module $\TAlgv$ was denoted $\TAlg$. Here, we choose to reserve the notation $\TAlg$ for the algebra-module over $\Dnil$, whose associated vertically sequential algebra-module is $\TAlgv$.}
$$\TAlgv(\PMC) = \TAlgv = \{ \ \TAlgva{m} \ \mid m\geq 0\}$$  
over the sequential nilCoxeter 2-algebra $\nil$ from Example~\ref{ex:one}. The component $\TAlgva{m}$ is   generated (over $\Field$) by strand diagrams just as in the definition of $\Alg(\PMC)$, except that now we require $m$ of the strands to end on the top edge of $[0,1] \times Z$, rather than on the right edge. Here is an example of a strand diagram with $m=2$:
$$\includegraphics[scale=0.7]{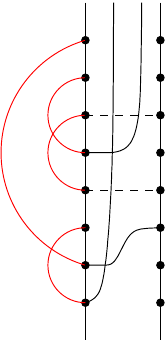}$$

To define $\TAlgv$, we impose the same relations on strand diagrams as in the case of $\Alg(\PMC)$. Horizontal multiplication is still given by concatenation, and the differential by resolving crossings. Further, we now have a (top) vertical action of the $2$-algebra $\nil$ on $\TAlgv,$ given by vertical concatenation.

The algebra-module $\TAlg$ that we use in this paper has components $\TAlga{n}{m}{q}$ given by:
\begin{equation}
\label{eq:talgva}
 \TAlga{n}{m}{q} :=  \vertme{\TAlgva{p}}{\Dnila{p}{n}{m}{q}}{\nil_{p}}{.6ex}{.6ex}
 \end{equation}
where $p=n+m-q$.

Pictorially, a generator of $\TAlga{n}{m}{q}$ is a strand diagram as in $\TAlgva{p}$, but with its $p$ strands at the top matched with the bottom strands in an attached rectangle.  The rectangle (drawn with dotted edges) represents a generator of the diagonal $2$-algebra component $\Dnila{p}{n}{m}{q}$. The subscripts $m$, $n$, and $q$ refer to the number of strands coming out of the top rectangle, through its top, left and right edge, respectively. For example, 
$$\includegraphics[scale=0.7]{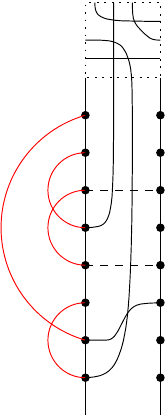}$$
is an element of $\TAlga{2}{3}{3}$. The tensor product in Formula~\eqref{eq:talgva} indicates that some crossings can move freely between the top rectangle and the strand diagram at the bottom. As before, double crossings are set to zero, and the differential consists in summing over all ways of resolving a single crossing (and for crossings in the $\TAlgv$ portion of $\TAlg$, deleting any horizontal strands unmatched with horizontal strands). The diagonal nilCoxeter 2-algebra $\Dnil$ acts on $\TAlg$ by vertically concatenating rectangles at the top. Further, $\TAlg(N)$ has an algebra structure given by horizontal concatenation of diagrams.

The algebra-modules $\RAlg = \RAlg(\PMC)$, $\BAlg=\BAlg(\PMC)$ and $\LAlg=\LAlg(\PMC)$ are constructed similarly to $\TAlg$, except that the part of the strand diagram without the dotted rectangle is rotated clockwise by $90^\circ$, $180^\circ$, and $270^\circ$ respectively, and in the case of $\RAlg$ and $\BAlg$ the orientation of the strands is changed. Furthermore, the dotted rectangle in $\RAlg, \BAlg$ and $\LAlg$ is attached at the right, bottom and left, respectively. For example, here are generators for $\RAlga{2}{3}{3}$, $\BAlga{2}{2}{2}$, and $\LAlga{3}{3}{2}$ (in that order):
$$\includegraphics[scale=0.7]{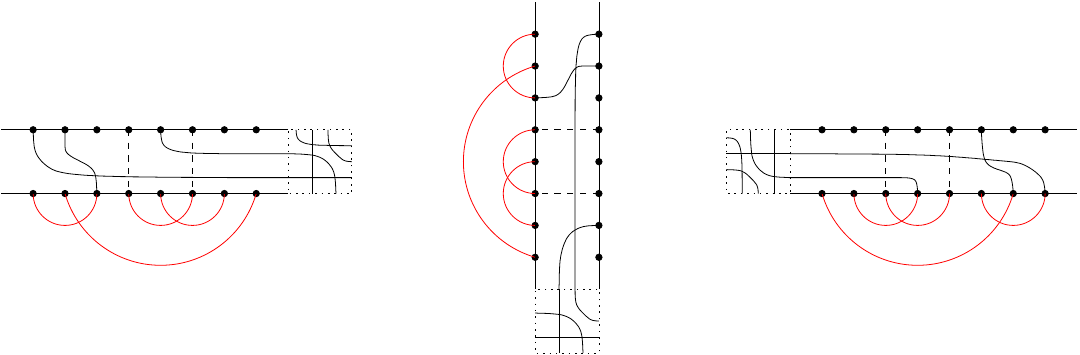}$$

\begin{lemma}
  The algebra-modules $\RAlg(\PMC)$, $\BAlg(\PMC)$, $\LAlg(\PMC)$, and
  $\TAlg(\PMC)$ satisfy the motility hypothesis
  (Definition~\ref{def:motileAM}).
\end{lemma}
\noindent  This is immediate from the definitions of the algebra-modules as
  tensor products with $\Dnil$.

Again, we conclude with some notation for algebra elements. Given a
subset $\SetS$ of the matched pairs in $\PMC$ there is a corresponding
idempotent $I(\SetS)$ in $\TAlg(\PMC)$ (with respect to the horizontal
multiplication), consisting of horizontal strands at the points in
$\SetS$. (Again, we let $\Idem(\PMC)$ be the subring of $\TAlg(\PMC)$
spanned by the idempotents.)
Given a chord $\rho$ in $Z$ with boundary in $\CircPts$, there is a corresponding element $a(\rho)$, defined just as for $\Alg(\PMC)$.

Given a chord $\rho$ whose initial endpoint is in $\CircPts$ and whose terminal endpoint is in $\bdy Z$---colloquially, a chord $\rho$ running off the top of $\PMC$---we can also define elements $\aup(\rho)$ and $\aleft(\rho)$ in $\TAlg(\PMC)$. The element $\aup(\rho)$ is the sum of all strand diagrams $a\in \putaround{\TAlg}{}{0}{1}{0}$ with $[a]=[\rho]\in H_1(Z,\CircPts\cup\bdy Z)$, while the element $\aleft(\rho)$ is the sum of all strand diagrams $a\in\putaround{\TAlg}{}{1}{0}{0}$ with $[a]=[\rho]\in H_1(Z,\CircPts\cup\bdy Z)$. Colloquially, the terms in $\aup(\rho)$ run off the top of $\PMC$, while the terms in $\aleft(\rho)$ bend left in the nilCoxeter region:
\begin{center}
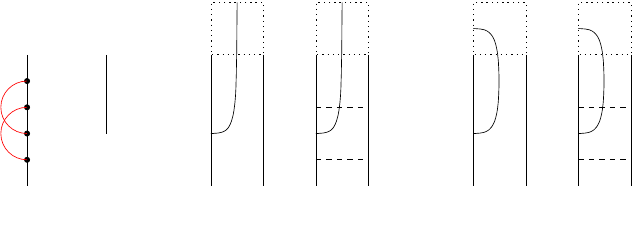
\end{center}
Similarly, the algebra-modules $\RAlg(\PMC)$, $\BAlg(\PMC)$ and $\LAlg(\PMC)$ also have elements denoted $\aleft(\rho)$ and $\adown(\rho)$.

\begin{remark}
The vertically sequential algebra-modules associated to $\TAlg$ and $\BAlg$ are the top and bottom algebra-modules from \cite[Section 5.2]{DM:cornered}. By contrast, $\RAlg^v$ and $\BAlg^v$ are not the right and left algebra-modules constructed in \cite[Section 5.2]{DM:cornered}. One origin of this discrepancy is that in the current paper we will use cornered Heegaard diagrams where only the alpha curves intersect the boundary; by contrast \cite{DM:cornered} worked with planar grid diagrams, which are closely related to split  cornered Heegaard diagrams (where the alpha curves intersect the vertical boundary, and the beta curves intersect the horizontal boundary).
\end{remark}

\subsection{Gluing surfaces with boundary}\label{sec:glue-surf}
Given matched intervals $\PMC_0=(Z_0, \CircPts_0, M_0)$ and $\PMC_1=(Z_1, \CircPts_1, M_1)$, gluing them at their
endpoints gives a pointed matched circle 
$$\PMC=\PMC_0 \cup \PMC_1.$$
Precisely, recall that $Z_0$ and $Z_1$ are oriented; for the gluing, we identify the final point of $Z_0$ with the initial point of $Z_1$, and also identify the initial point of $Z_0$ with the final point of $Z_1$; this latter point is taken to be the basepoint $z$ on $Z = Z_0 \cup Z_1$. We then set $\CircPts = \CircPts_0 \cup \CircPts_1$ and obtain a matching $M$ on $\CircPts$ by combining $M_0$ and $M_1$. We let $\PMC=(Z, \CircPts, M,z)$.

If $F^{\circ} (\PMC_0)$ and $F_{\circ}(\PMC_1)$ are surfaces (with circle boundary) associated to $\PMC_0$ and $\PMC_1$ as in Section~\ref{sec:amint}, observe that 
$$ F(\PMC) \cong F^{\circ} (\PMC_0)\cup_{S^1} F_{\circ} (\PMC_1).$$
(More precisely, the right hand side is obtained from $F(\PMC)$ by collapsing the two discs in the construction of $F(\PMC)$ to intervals. This collapse can be modified to produce a diffeomorphism.)

\begin{proposition}
\label{prop:rtp01}
Let $\PMC_0$ and $\PMC_1$ be matched intervals, and set $\PMC = \PMC_0 \cup \PMC_1$.
Then, the restricted tensor product 
$$ \rvertme{\TAlgD(\PMC_0)}{\BAlgD(\PMC_1)}{\Dnil}{0pt}{0pt}$$
(with its horizontal multiplication) is isomorphic to the matching algebra $\Alg(\PMC)$. Similarly, the restricted tensor product
$$\rhormein{\RAlgD(\PMC_1)}{\LAlgD(\PMC_0)}{\Dnil}{1ex}{1ex}$$ 
(with its vertical multiplication) is also isomorphic (after a ninety-degree clockwise rotation) to $\Alg(\PMC)$.
\end{proposition}

\begin{proof}
The vertically sequential algebra-modules associated to $\TAlg(\PMC_0)$ and $\BAlg(\PMC_1)$ are the top and bottom algebra-modules from~\cite{DM:cornered}. Their tensor product (which is the restricted tensor product of $\TAlg(\PMC_0)$ and $\BAlg(\PMC_1)$) was identified with $\Alg(\PMC)$ in \cite[Theorem 5.1]{DM:cornered}. 

The corresponding statement for $\RAlgD(\PMC_1)$ and $\LAlgD(\PMC_0)$ is obtained by a counterclockwise ninety-degree rotation of all the objects involved.
\end{proof}

\begin{proposition}
\label{prop:bent01}
Let $\PMC_0$ and $\PMC_1$ be matched intervals, and set $\PMC = \PMC_0 \cup \PMC_1$.
Then, the smoothed tensor products (as defined in Section~\ref{sec:smoothed}):
$$ \TAlg(\PMC_0) \osmooth \RAlg(\PMC_1) \ \ \ \text{and} \ \ \ \BAlg(\PMC_1) \osmooth \LAlg(\PMC_0)$$
are both isomorphic to the matching algebra $\Alg(\PMC)$. 
\end{proposition}

\begin{proof}
A typical generator of the {\em bent} tensor product of $\TAlg(\PMC_0)$ and $\RAlg(\PMC_0)$ appears in Figure~\ref{fig:benttensorelt}.
\begin{figure}
\centering
\[
\includegraphics[scale=0.7]{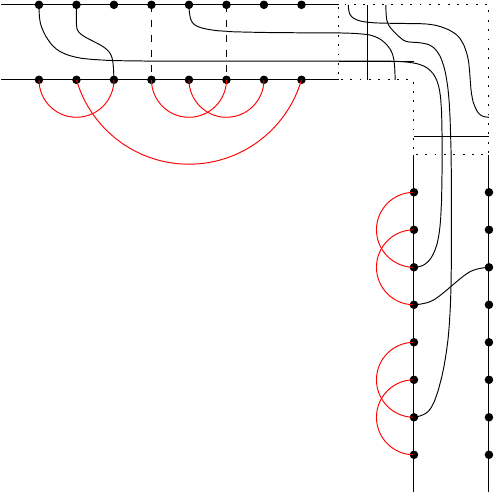}
\]
\caption{\textbf{A generator of the bent tensor product.}}
\label{fig:benttensorelt}
\end{figure}
The smoothed tensor product is generated by pictures of that form where no strands escape to the unattached edges of the dotted region. Straightening the corner in such a picture produces an ordinary strand diagram, representing an element of $\Alg(\PMC)$. This gives the desired identification, which preserves the multiplication and the differential.

The proof for $\BAlg(\PMC_1) \osmooth \LAlg(\PMC_0)$ is similar.
\end{proof}

\begin{remark}
As mentioned, the algebra-modules $\TAlg, \BAlg, \RAlg$, and $\LAlg$ all satisfy the corresponding motility hypotheses. Although this will not be needed in the rest of the paper, it is worth noting that we could construct a full tensor product 
$$ \Alg^{\Dnil}(\PMC) := \vertme{\TAlgD(\PMC_0)}{\BAlgD(\PMC_1)}{\Dnil}{0pt}{0pt}$$
This is a direct sum of pieces $\putaround{\Alg^{\Dnil}(\PMC)}{}{n}{}{q}$, where $\putaround{\Alg^{\Dnil}(\PMC)}{}{0}{}{0} = \Alg(\PMC)$. In general, a generator of  $\putaround{\Alg^{\Dnil}(\PMC)}{}{n}{}{q}$ consists of strand diagrams as in $\Alg(\PMC)$, but where some strands can escape to the left or right inside a middle dotted rectangle, as shown in Figure~\ref{fig:fulltensor}.
\begin{figure}
\centering
\[
\includegraphics[scale=0.7]{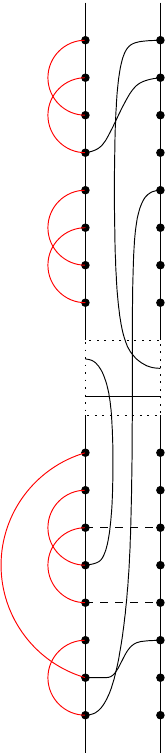}
\]
\caption{\textbf{A generator of the full vertical tensor product.}}
\label{fig:fulltensor}
\end{figure}
\end{remark}


\section{The cornering module--2-modules}
\label{sec:cornering}

\subsection{The \DD\ identity module-bimodule}
To start, recall from~\cite{LOT4} the bordered \DD\ bimodule
associated to the identity cobordism of a surface $F(\PMC)$:
\begin{definition}\label{def:bord-dd-id}
  Let $\PMC=(Z,\CircPts,M,z)$ be a pointed matched circle, with
  $|\CircPts|=4k$.  
  Recall that basic idempotents for $\Alg(\PMC)$ correspond to subsets
  of $M(\CircPts)=[2k]$.  Call basic idempotents $I\in\Alg(\PMC)$ and
  $J\in\Alg(-\PMC)$ \emph{complementary} if the corresponding subsets
  $s,t\subset [2k]$ satisfy $s\cap t=\emptyset$ and $s\cup t=[2k]$.

  Let $X(\Id_\PMC)=\Field\langle\{I\otimes
  J\in\Alg(\PMC)\otimes\Alg(-\PMC)\mid (I,J)\text{
    complementary}\}\rangle$. Then, as a module,
  \[
  \CFDDa(\Id_\PMC)=\bigl(\Alg(\PMC)\otimes_{\Field}\Alg(-\PMC)\bigr)\otimes_{\Idem(\PMC)\otimes \Idem(-\PMC)}X(\Id_\PMC).
  \]
  (That is, $\CFDDa(\Id_\PMC)$ is \emph{projectively generated} as a
  (left-left) bimodule over $(\Alg(\PMC),\Alg(-\PMC))$ by the pairs of
  complementary idempotents $I\otimes J\in
  \Alg(\PMC)\otimes\Alg(-\PMC)$.)
  Let $\Chord(\PMC)$ denote the set of chords in $\PMC$. Given a chord
  $\xi\in\Chord(\PMC)$ let $-\xi\in\Chord(-\PMC)$ denote the
  orientation-reverse of $\xi$. The differential on
  $\CFDDa(\Id_\PMC)$ is defined by
  \[
  \bdy(I\otimes J)=\sum_{\substack{(K,L)\\\text{
      complementary}}}\sum_{\xi\in\Chord(\PMC)}\bigl((I\otimes J)
  (a(\xi)\otimes a(-\xi)) (K\otimes L)\bigr)\otimes (K\otimes L)
  \]
  and the Leibniz rule.

  Via the isomorphism $\Alg(-\PMC)=\Alg(\PMC)^\op$, we can also view
  $\CFDDa(\Id_{\PMC})$ as a left-right bimodule over $\Alg(\PMC)$ and
  $\Alg(\PMC)$. Viewed this way, the differential has the form
  \[
  \bdy(I\otimes J)=\sum_{\substack{(K,L)\\\text{
      complementary}}}\sum_{\xi\in\Chord(\PMC)}(I a(\xi)
K)\otimes (K\otimes L) \otimes (L a(\xi) J).
  \]
\end{definition}

The goal of this section is to generalize Definition~\ref{def:bord-dd-id} to the cornered setting. To start, we need some auxiliary concepts.

Recall that $\Dnil$ denotes the diagonal nilCoxeter $2$-algebra. There
is a related $2$-algebra obtained by rotating all of the pictures in
$\Dnil$ clockwise by $90^\circ$. We call the result the \emph{rotated nilCoxeter
  algebra} and denote it $\DRnil$. 

Given a right algebra-module $\RAlg$ over $\Dnil$ there is an
associated bottom algebra-module $\rot{\RAlg}$ over $\DRnil$. Similarly,
the rotation operation takes bottom algebra-modules to left algebra-modules; left
algebra-modules to top algebra-modules; and top algebra-modules
to right algebra-modules. Analogous statements apply to
$2$-modules. For instance, given a top-right $2$-module $M$ over
$\RAlg$ and $\TAlg$ there is an associated bottom-right $2$-module $\rot{M}$
over $\rot{\RAlg}$ and $\rot{\TAlg}$. Also, if $a$ is an element of one of the $2$-objects that get rotated, we write $\rot{a}$ for the rotated element.

We also write $\rotcc{A}$ for the counterclockwise rotation of an object $A$ by $90^\circ$, and $\rotsp{\rot{A}}$ for the rotation of $A$ by $180^\circ$. For example, we have $\rotcc{\Dnil} \cong \rot{\Dnil}^{\op}$ and $\rotsp{\rot{\Dnil}} \cong \Dnil^\op$, where the isomorphisms reverse the orientations of the strands. (Here, when we have a $2$-algebra $\A$, we let its opposite $\A^\op$ be the $2$-algebra with $\putaround{(\A^\op)}{m}{n}{p}{q} = \Aa{p}{q}{m}{n}$, and the order of the terms flipped in both the horizontal and the vertical multiplication. One can define the opposite of an algebra-module or a $2$-module similarly.)

We can now define a structure needed in the construction of the identity module-bimodules.

\begin{definition}\label{def:barbell}
  The \emph{horizontal barbell algebra-bimodule} is defined to be 
  \[
\BBh = \horme{ \Dnila{\bullet}{\bullet}{\bullet}{\bullet}}{\DLnila{\bullet}{\bullet}{\bullet}{\bullet}}{\Dnilasm{0}{\bullet}{0}{\bullet}}{0.5ex}{0.5ex}
  \]
  This has a left action by $\Dnil$, a right action by $\DLnil$, and
  a vertical multiplication defined by
  \[
  \bmat
  \pbmat
  a &\htp & b
  \pemat\\
  \cdot\\
  \pbmat
  c & \htp &d
  \pemat
  \emat
  =
  \pbmat
  a\\
  c
  \pemat
  \htp
  \pbmat
  b\\d
  \pemat.
  \]
  It decomposes into pieces
  \[
  \putaroundmarg{(\BBh)}{m_1\ m_2}{n}{p_1\ p_2}{q}=\horme{ \Dnila{m_1}{n}{p_1}{\bullet}}{\DLnila{m_2}{\bullet}{p_2}{q}}{\Dnilasm{0}{\bullet}{0}{\bullet}}{0.5ex}{0.5ex}.
  \]
  
\newsavebox{\capone}
\newsavebox{\captwo}
\newsavebox{\capthree}
\newsavebox{\capfour}
\newsavebox{\capfive}
\savebox{\capone}{$\putaroundmarg{(\BBh)}{2\ 1}{4}{1\ 2}{2}$}
\savebox{\captwo}{$\Dnila{2}{4}{1}{3}$}
\savebox{\capthree}{$\DLnila{1}{3}{2}{2}$}
\savebox{\capfour}{$ \Dnila{0}{3}{0}{3}$}
\savebox{\capfive}{$\putaroundmarg{(\BBv)}{2}{\substack{\scriptscriptstyle 1\\ \scriptscriptstyle 1}}{2}{\substack{\scriptscriptstyle 2\\ \scriptscriptstyle 2}}$}

  \begin{figure}
    \centering
    \includegraphics[scale=.5]{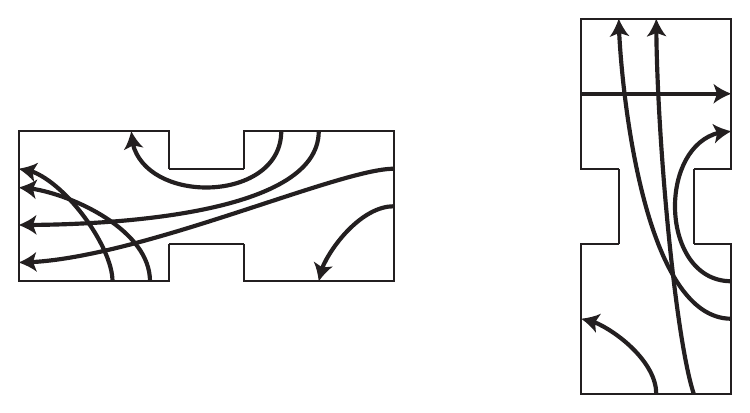}
    \caption{\textbf{The barbell algebra-bimodules.} On the left we show an element of~\usebox{\capone}, the tensor product of an element in~\usebox{\captwo} with one in~\usebox{\capthree}. Each of the factors in the tensor product is represented by a square, and the two squares are connected by a region that illustrates the fact that the tensor product is over~\usebox{\capfour}. Similarly, on the right an element of~\usebox{\capfive}. For clarity, in this figure we added arrows to indicate the orientations.} \label{fig:barbell}
  \end{figure}
 
  Similarly, the \emph{vertical barbell algebra-bimodule} is defined to be 
  \[
 \BBv= \vertme{\Dnila{\bullet}{\bullet}{\bullet}{\bullet}}{\DRnila{\bullet}{\bullet}{\bullet}{\bullet}}{\Dnilasm{\bullet}{0}{\bullet}{0}}{1.5ex}{.5ex}
  \]
  This has a bottom action by $\Dnil$, a top action by $\DRnil$, and
  a horizontal multiplication defined by
  \[
  \pbmat
  a\\
  \vtp\\
  b
  \pemat
  *
  \pbmat
  c\\
  \vtp\\
  d
  \pemat=
  \bmat
  \pbmat
  a & c
  \pemat\\
  \vtp\\
  \pbmat
  b & d
  \pemat
  \emat
  \]
  It decomposes into pieces
  \[
  \putaroundmarg{(\BBv)}{m}{\substack{n_1\\n_2}}{p}{\substack{q_1\\q_2}}=
  \vertme{\Dnila{m}{n_1}{\bullet}{q_1}}{\DRnila{\bullet}{n_2}{p}{q_2}}{\Dnilasm{\bullet}{0}{\bullet}{0}}{1.5ex}{.5ex}.
  \]
\end{definition}

\noindent  Figure~\ref{fig:barbell} illustrates the barbell algebras.

\begin{lemma}
  The barbell constructions in Definition~\ref{def:barbell} produce algebra-bimodules.
\end{lemma}
\begin{proof}
  We prove the result for $\BBv$; the proof of $\BBh$ is similar. First, to see that the horizontal multiplication on $\BBv$ respects the relations in the tensor product, observe that:
 \[
\pbmat
b \\ \odot \\  \pbmat x \\ a \pemat 
\pemat
\pbmat d \\ \odot \\ c \pemat
=
\pbmat 
(b \ \ d) \\ \odot \\ \pbmat x & e^v_n \\[5pt] a & c \pemat
\pemat
= \pbmat
\pbmat b & d \\ x & e^v_n
\pemat \\
\odot \\
(a \ \ c)
\pemat
= 
\pbmat
\pbmat b \\ x \pemat \\ \odot \\ a
\pemat
\pbmat d \\ \odot \\ c \pemat.
\]
where $x \in \Dnila{\sbull}{0}{\sbull}{0}$, $c\in \Dnila{\sbull}{\sbull}{n}{\sbull}$, and $d \in \Dnila{n}{\sbull}{\sbull}{\sbull}$  If the indices of $c$ and $d$ do not match, then the whole expression is zero. Second, the associativity and local commutation relations for $\BBv$ follow from corresponding relations for $\Dnil$ and $\DRnil$.
\end{proof}

\begin{lemma}\label{lem:barbell-tensor} Fix a matched interval $\PMC$. The vertical restricted tensor product of $\rot{\RAlg(\PMC)}$ and the vertical barbell
  algebra $\BBv$ is isomorphic to $\BAlg(\PMC)$.
\end{lemma}
\begin{proof} For $m=n+p-q$, we have
  \[
  \rvertme{\putaroundmarg{(\BBv)}{p}{\substack{\sbull\\ q}}{\sbull}{\substack{\sbull\\ n}}}{\putaroundmarg{\rot{\RAlg(\PMC)}}{\sbull}{\sbull}{}{\sbull}}{\DRnilasm{\sbull}{0}{\sbull}{0}}{1.2ex}{.8ex}
  = \hspace{.5ex}
  \bmat
  \putaround{\rot{\RAlg(\PMC)}}{m}{0}{}{0}\\[1.2ex]
  \hphantom{\DRnilasm{m}{0}{m}{0}}
  \vtp \ {\DRnilasm{m}{0}{m}{0}} \\[.8ex]
  \DRnila{m}{0}{m}{0}\\[1.2ex]
  \hphantom{\Dnilasm{m}{0}{m}{0}}
  \vtp \ {\Dnilasm{m}{0}{m}{0}} \\[.8ex]
  \Dnila{p}{q}{m}{n}
  \emat
  \hspace{1.3ex} = \hspace{1.1ex}
  \vertme{\Dnila{p}{q}{m}{n}}{\putaround{\rot{\RAlg(\PMC)}}{m}{0}{}{0}}{\Dnilasm{m}{0}{m}{0}}{1.5ex}{.8ex}.
  \]
  But $\putaround{\rot{\RAlg(\PMC)}}{m}{0}{}{0}=
  \putaround{\BAlg(\PMC)}{m}{0}{}{0}$, so this is exactly
  $\putaround{\BAlg(\PMC)}{p}{q}{}{n}$.
\end{proof}

\begin{definition}\label{def:DD-id}
  Fix a matched interval $\PMC$.  The (right)
  \emph{\DD\ identity module-bimodule} $\DDa(\Id)$ has a bottom action
  by $\RAlg(\PMC)$, a top action by
  $\rot{\TAlg(-\PMC)}$, 
  and a right action by the vertical barbell algebra $\BBv$. The module
  $\DDa(\Id)$ is projectively generated
  by pairs of complementary idempotents of
  $\RAlg(\PMC)$ and $\rot{\TAlg(-\PMC)}$. That is, let
  \[
  X=\Field\left\langle\Big\{\bmat J\\\odot \\I\emat\in \left(\vertme{\putaround{\RAlg(\PMC)}{0}{}{0}{0}}{\putaround{\rot{\TAlg(-\PMC)}}{0}{}{0}{0}}{\Field}{1.5ex}{.5ex}\right) \Big| \hspace{.5ex}
  (I,J)\text{ complementary}\Big\}\right\rangle.
  \]
  Then 
  \[
  \DDa(\Id) \coloneqq
  \bmat
    \rot{\TAlg(-\PMC)} \hphantom{\scriptscriptstyle \Idem(-\PMC)}\\[1ex]
    \vtp {\scriptscriptstyle \Idem(-\PMC)}\\[1ex]
    X\hphantom{\scriptscriptstyle \Idem(\PMC)}\\[1ex]
    \vtp {\scriptscriptstyle \Idem(\PMC)}\\[1ex]
    \RAlg(\PMC)\hphantom{\scriptscriptstyle \Idem(\PMC)}
  \emat.
  \]
  The top and bottom actions of  
  $\rot{\TAlg(-\PMC)}$
  and $\RAlg(\PMC)$,
  respectively, are clear. It remains to define the differential and
  the right action of the vertical barbell algebra.

  Let $\Chord_0(\PMC)$ denote the set of chords in $\PMC$ contained
  entirely in the interior of $\PMC$, and let $\Chord_1(\PMC)$ denote
  the set of chords ending in $\PMC$ and beginning off the right side
  of $\PMC$.

  (For compactness, we will now write some vertical products in a horizontal line---the multiplication symbols specify the product direction. To turn a vertical product into a horizontal one, we rotate it clockwise by $90^\circ$.)  The differential on $\DDa(\Id)$ is defined by 
  \[
  \bdy(I\odot J)=\sum_{\substack{(K, L)\\\text{complementary}}}\sum_{\xi\in\Chord_0(\PMC)}
  \bigl( I\cdot a(\xi)\cdot K\bigr) \odot (K\odot
  L)\odot \bigl( L\cdot \rot{a(-\xi)} \cdot J \bigr).
  \]

  Define a map $f_1\co X\to
  \RAlg(\PMC)\odot_{\Idem(\PMC)}X\odot_{\Idem(\PMC)}
  \rot{\TAlg(-\PMC)}
  $ by 
  \[
  f_1(I\odot J) = \sum_{\substack{(K, L)\\\text{complementary}}}\sum_{\xi\in\Chord_1(\PMC)}
 \bigl( I\cdot \adown(\xi)\cdot K\bigr) \odot (K\odot
  L)\odot  \bigl( L\cdot \rot{\aleft(-\xi)} \cdot J \bigr).
  \]
  Let $f_n\co X\to
  \RAlg(\PMC)\odot_{\Idem(\PMC)}X\odot_{\Idem(\PMC)}
  \rot{\TAlg(-\PMC)}
  $ be
  the map induced by composing $f_1$ with itself $n$ times and
  multiplying the outputs on the two sides.

  The action of the barbell algebra is given as follows. Suppose
  $a\odot b$ is an element of the barbell algebra, where
  $a \in \Dnila{\bullet}{\bullet}{n}{\bullet}$ and $b \in \DRnila{n}{\bullet}{\bullet}{\bullet}$. Define
  \[
 \bigl( \phi \odot (I\odot J) \odot \psi \bigr ) \ast (a\odot b) = (\phi * a) \cdot f_n(I \odot  J)\cdot (\psi * b).
  \]
\end{definition}

Graphically, we represent an element $\phi \odot (I \odot J) \odot \psi$ of $\DDa(\Id)$ by a rectangle, with $\phi$ placed at the bottom of the rectangle, and $\psi$ at the top. (We usually do not draw the left edge of the rectangle, to indicate that there may be more chords to the left.) Thus, the action of the differential can be drawn as:
\begin{center}
  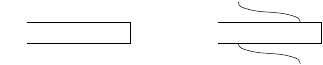
\end{center}
Here is an example, with the idempotents shown explicitly:
\begin{center}
  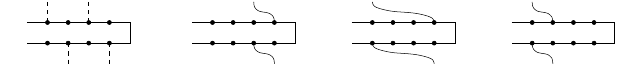
\end{center}

The action of a single-strand element in the barbell algebra is:
\begin{center}
  \begin{picture}(0,0)%
\includegraphics{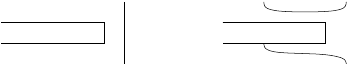}%
\end{picture}%
\setlength{\unitlength}{2171sp}%
\begingroup\makeatletter\ifx\SetFigFont\undefined%
\gdef\SetFigFont#1#2#3#4#5{%
  \reset@font\fontsize{#1}{#2pt}%
  \fontfamily{#3}\fontseries{#4}\fontshape{#5}%
  \selectfont}%
\fi\endgroup%
\begin{picture}(5049,924)(1189,-1573)
\put(3301,-1186){\makebox(0,0)[lb]{\smash{{\SetFigFont{12}{14.4}{\rmdefault}{\mddefault}{\updefault}{\color[rgb]{0,0,0}$= \ \sum$}%
}}}}
\end{picture}%

\end{center}
In our example, we have
\begin{center}
  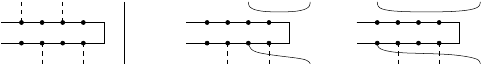
\end{center}

\begin{proposition}\label{prop:interval-dd-id-defined}
  The definitions above make $\DDa(\Id)$ into a well-defined
  module-bimodule.
\end{proposition}
\begin{proof}
  We must check:
  \begin{enumerate}
  \item That $\bdy^2=0$.
  \item Associativity of the actions of $\RAlg(\PMC)$ and $\rot{\TAlg(-\PMC)}$.
  \item The local commutation relations.
  \item That the action of the barbell algebra respects the relations
    in the barbell algebra.
    \item Associativity of the action of the vertical barbell algebra.
   \item The Leibniz rule.
  \end{enumerate}  
  The fact that $\bdy^2=0$ follows from the fact that $\bdy^2=0$ for
  $\CFDDa(\Id_{\PMC})$ (a direct, combinatorial proof of which is
  given in~\cite[Proposition 3.4]{LOT4}). Associativity of the actions of $\RAlg(\PMC)$ and $\rot{\TAlg(-\PMC)}$ is immediate from the definitions. So will be the associativity of the action of the vertical barbell algebra, once we check in (4) that this action is well-defined. The
  local commutation relations for $\DDa(\Id)$ follow from the ones for $\RAlg(\PMC)$ and $\TAlg(-\PMC)$.
  
It remains to check (4) and (6). Compatibility with the relations in the barbell algebra means that:
\begin{equation}
\label{eq:compBarbell}
 \bigl( \phi \odot (I\odot J) \odot \psi \bigr ) \ast \bigl ((a\cdot x) \odot b \bigr) =  \bigl( \phi \odot (I\odot J) \odot \psi \bigr ) \ast \bigl (a\odot (x \cdot b) \bigr),
\end{equation}
where $x \in \Dnila{m}{0}{m}{0} = \nil_m$ for some $m$. Using local commutation and associativity, we can reduce the verification of \eqref{eq:compBarbell} to the case when $\phi = I$ and $\psi = J$, and then further to $a$ and $b$ both being the vertical identity, and $x$ being the generator $\sigma_1$ of $\nil_2$ consisting of one crossing. To verify \eqref{eq:compBarbell} in this case, let us first consider the action of the identity of $\nil_2$ (viewed as an element of the barbell algebra) on $\DDa(\Id)$. In the resulting sum, the terms come in pairs, according to whether the chords from the action of the first strand are to the left or the right of those from the action of the second strand:
\begin{center}
  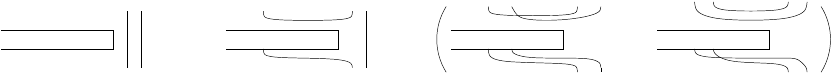
\end{center}
We now check \eqref{eq:compBarbell} for $\phi=I, \psi=J$ and $x=\sigma_1$. By canceling terms with double crossings, the left hand side of \eqref{eq:compBarbell} becomes:
\begin{center}
  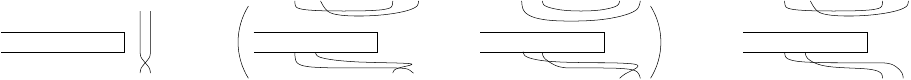
\end{center}
We get the same answer for the right hand side:
\begin{center}
  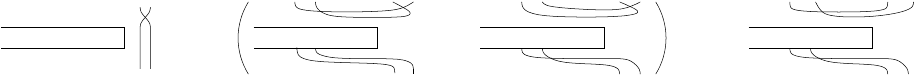
\end{center}

Finally, for the Leibniz rule, using local commutation and associativity, we reduce to the case when we act on an element of the form $I \odot J$ with a single strand (the identity) in $\nil_1 \subset \BBv$. That is, we need to check that
\begin{center}
  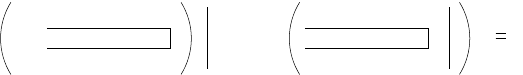
\end{center}

The left hand side is a sum of various terms, which depend on the exact positions in the idempotents $I$ and $J$. Each term has one or two strands attached at the top, and one or two strands attached at the bottom. At the top, if we apply the differential first, we get one strand $\rho$, and then when we apply the multiplication we get another strand $\rho'$. There are three cases:
\begin{enumerate}[label=(\alph*)]
\item The strand $\rho'$ ends at the initial point of $\rho$, and therefore the two combine to produce a single strand;
\item The strand $\rho'$ is entirely to the right of $\rho$;
\item The strand $\rho'$ crosses $\rho$.
\end{enumerate}
These cases are shown in Figure~\ref{fig:cancellation}. In each situation the final terms cancel in pairs. Note that in Case (b), there are also contributions exactly as in (a), but these are not shown because they cancel each other as shown in (a). The same thing happens in Case (c); however, there we have several terms of different types, and for completeness we illustrate them all.
\end{proof}

\begin{figure}
  \centering
  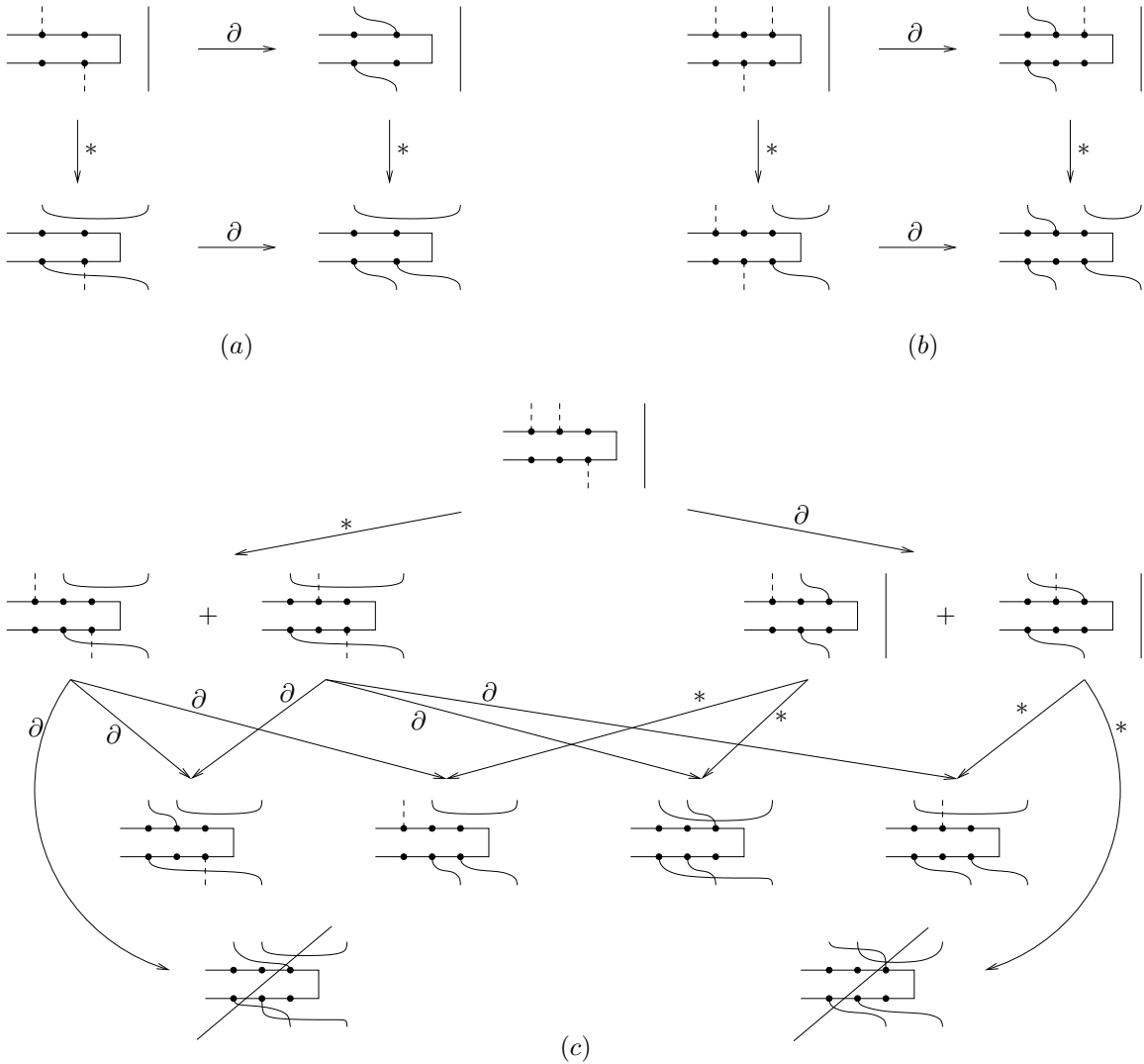
  \caption{\textbf{Cancellation of terms in the Leibniz rule for $\DDa(\Id)$.} The arrows marked with $*$  indicate multiplication by the strand on the right. Note that there may be additional occupied positions in the idempotents, but these are not involved in the respective operations, and hence are not shown in the picture.}
  \label{fig:cancellation}
\end{figure}

\begin{proposition}\label{prop:interval-dd-id-proj}
  The module-bimodule $\DDa(\Id)$ is biprojective over $\RAlg(\PMC)$
  and $\rot{\TAlg(-\PMC)}$.
\end{proposition}
\begin{proof}
  Let $\PMC^\dagger$ be the pointed matched circle corresponding to the
  matched interval $\PMC$. Then as an ($\RAlg(\PMC)$,$\rot{\TAlg(-\PMC)}$)
  bimodule,
  \[
  \DDa(\Id)\cong \RAlg(\PMC)\odot_{\Alg(\PMC^\dagger)}\CFDDa(\Id_{\PMC^\dagger})\odot_{\Alg(\PMC^\dagger)}\rot{\TAlg(-\PMC)}.
  \]
  The desired result follows from the fact that the induction functor takes projective modules to projective
  modules.
\end{proof}

\subsection{The \DhAA- and \AhDD-cornering modules}\label{sec:DAA-AAD}
Fix matched intervals $\PMC_0$ and $\PMC_1$. In this subsection and the next we will construct five module--2-modules associated to $\PMC_0$ and $\PMC_1$. Although the definitions are algebraic, it is helpful to think of these objects as associated to ``cornering surfaces'' of the form $(\PMC_0 \cup \PMC_1) \times [0,1]$, where one of the two boundaries is viewed as a smooth interval and the other as having a corner at the point $\PMC_0 \cap \PMC_1$. See Figure~\ref{fig:cornerings} for a preview of the 
five module--2-modules. In particular, when reading the definitions below, one can keep track of the orientations on $\PMC_0$ and $\PMC_1$ by comparing them with the arrows in Figure~\ref{fig:cornerings}.

\begin{figure}
  \centering
  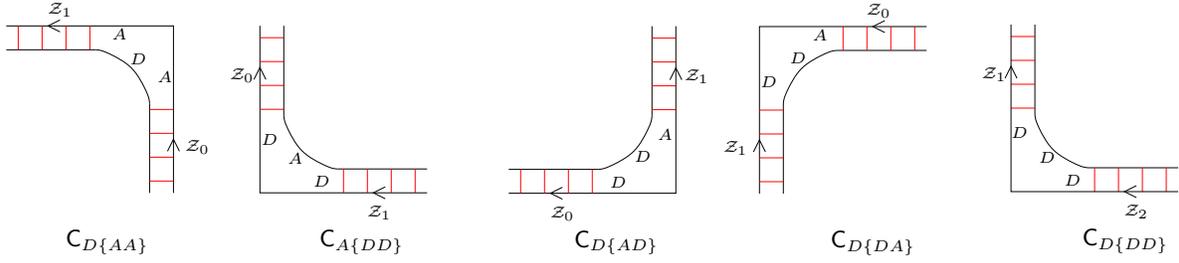
  \caption{\textbf{Graphical representations of the five cornering module--2-modules.} The red lines indicate the position of the marked points. These lines could also be thought of as $\alpha$ arcs in the respective cornering surfaces. (Compare the definition of cornered Heegaard diagrams in Section~\ref{sec:CHD}.)}
  \label{fig:cornerings}
\end{figure}

\begin{definition}\label{def:corn-AA}
  The \emph{\DhAA-cornering module--2-module} is
  $\putaround{(\TAlg(\PMC_0)\obent\RAlg(\PMC_1))}{0}{0}{\sbull}{\sbull}$ as a
  $(\TAlg(\PMC_0)\osmooth\RAlg(\PMC_1),\TAlg(\PMC_0)\obent\RAlg(\PMC_1))$-bimodule. (Recall from
  Section~\ref{sec:newperspective} that a module for the bent tensor product can be interpreted as a 2-module.)  We denote this module--2-module by
  $\cornAA=\cornAA(\PMC_0,\PMC_1)$.
\end{definition}

We represent $\cornAA$ graphically as in Figure~\ref{fig:cornAA}.

\begin{figure}
  \centering
  \includegraphics[scale=.75]{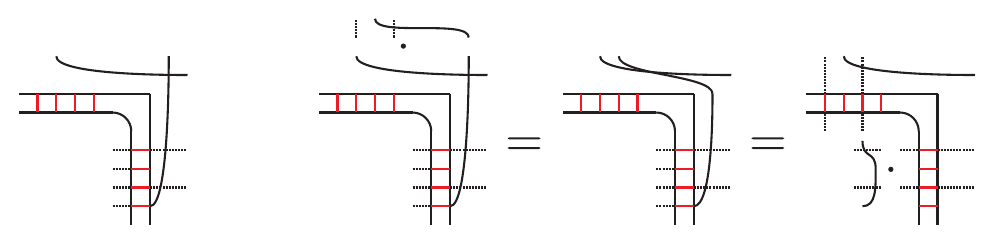}
  \caption{\textbf{Graphical representation of $\cornAA$.} Left: a
    reasonably generic element of $\cornAA$. Right: an action on the top and an action on the left, with the same result, illustrating how strands may pass through $\cornAA$.}
  \label{fig:cornAA}
\end{figure}

The \AhDD-cornering module--2-module is defined similarly:
\begin{definition}
  The \emph{\AhDD-cornering module--2-module} $\AcornDD(\PMC_0,\PMC_1)$ is
  $$\rotsp{\rot{\cornAA(-\PMC_0,-\PMC_1)}}^{\op},$$ that is, the result of rotating 
  $\cornAA(-\PMC_0,-\PMC_1)$ by $180^\circ$ and then taking the opposite. This has a left action by  $\rotsp{\rot{\TAlg(-\PMC_0)}}^\op=\BAlg(\PMC_0)$, a bottom action by
  $\rotsp{\rot{\RAlg(-\PMC_1)}}^\op=\LAlg(\PMC_1)$, and a right action by
  $\BAlg(\PMC_0)\osmooth \LAlg(\PMC_1)=\Alg(\PMC_1\cup\PMC_0)$.
\end{definition}

\begin{remark}\label{rem:not-proj}
  The module--2-module $\cornAA$ is, unfortunately, not projective (or
  even flat) over $\TAlg\osmooth\RAlg=\Alg$. For example, let $\PMC$
  be the genus-one matched interval. Let $\Alg_+$ denote the
  ideal in $\Alg$ spanned by the non-idempotent elements (or
  equivalently, the elements with non-trivial support). Consider the
  following elements of $\TAlg(\PMC)\obent\RAlg(\PMC)$: \vspace{5pt}
  \[
  \includegraphics{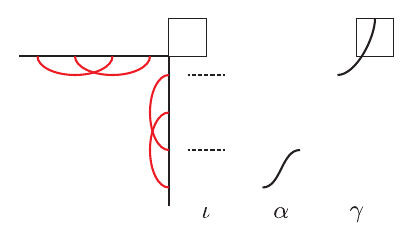}
  \] \vspace{5pt}
  The elements $\iota$ and $\alpha$ lie in $\Alg(\PMC\cup\PMC)\subset
  \TAlg(\PMC)\obent\RAlg(\PMC)$. Consider the following quotients of $\iota\Alg(\PMC\cup\PMC)$: \vspace{5pt}
  \begin{align*}
    M&=  \left(\iota\Alg(\PMC\cup\PMC)\right)\big/\left(\iota\Alg_+(\PMC\cup\PMC)\right)\cong\Field\\
    N&= \left(\iota\Alg(\PMC\cup\PMC)\right)\big/\left(\alpha\Alg_+(\PMC\cup\PMC)\right).
  \end{align*} \vspace{8pt}
  We view $M$ and $N$ as right modules over $\Alg(\PMC\cup\PMC)$. There
  is an injective map $f\co M\to N$ defined by $f(\iota)=\alpha$. But
  tensoring over $\Alg(\PMC \cup \PMC)$ with $\TAlg(\PMC)\obent\RAlg(\PMC)$, we have $\iota\otimes
  \gamma\neq 0$ but $f(\iota)\otimes \gamma=\alpha\otimes
  \gamma=1\otimes \alpha\gamma=0$.

  The non-projectivity of $\cornAA$ will force us to take projective
  resolutions in the definition of the cornered $2$-modules.
\end{remark} 

\subsection{The other cornering modules}\label{sec:other-corners}
\begin{definition}\label{def:corn-DA-AD}
  Let $\PMC_0$ and $\PMC_1$ be matched intervals.
  The \emph{\DhAD-cornering module--2-module} $\cornAD(\PMC_0,\PMC_1)$
  associated to $\PMC_0$ and $\PMC_1$ is defined to be 
  \[
  \cornAD(\PMC_0,\PMC_1)=\left( \ 
  \rvertme{\DDa(\Id_{\PMC_0})}{\rot{\cornAA(-\PMC_0,\PMC_1)}}{\rotsp{\TAlg(-\PMC_0)}}{0pt}{0pt} \!\!\!\!  \!\!\!\!  \!\!\!\! \!\!\!\!  \!\!\! \! \right).
  \]
  This has a bottom action by $\RAlg(\PMC_0)$ (coming from the bottom
  action on $\DDa(\Id_{\PMC_0})$), a right action by
  \[
  \left(\rvertme{\BBv}{\rot{\RAlg(\PMC_1)}}{\Dnil}{0pt}{0pt} \!\!\!\!  \!\!\! \right)\cong
  \BAlg(\PMC_1)
  \]
  (where the isomorphism comes from Lemma~\ref{lem:barbell-tensor}), and a left action by $\Alg((-\PMC_0) \cup \PMC_1)$.

  The \emph{\DhDA-cornering module--2-module} $\cornDA(\PMC_0,\PMC_1)$
  associated to $\PMC_0$ and $\PMC_1$ is defined to be %
  \[
  \cornDA(\PMC_0,\PMC_1)=\left(    \rhorme{\rotcc{\DDa(\Id_{-\PMC_1})}}{\rotcc{\cornAA(\PMC_0,-\PMC_1)}}{\rotccsp{\RAlg(-\PMC_1)}}{0pt}{0pt}
  \right).
  \]
  This has a left action by $\TAlg(\PMC_1)$ (coming from the top
  action of $\rot{\TAlg(\PMC_1)}$ on $\DDa(\Id_{-\PMC_1})$), a top action by 
  \[
  \rhorme{\BBh \ }{\rotcc{\TAlg(\PMC_0)}}{\Dnil}{0pt}{0pt} \ \cong \  \LAlg(\PMC_0).
  \]
  and a left action by $\Alg(\PMC_0 \cup (-\PMC_1))$.
\end{definition}

\noindent See Figure~\ref{fig:corn-AD}.

\begin{figure}
  \centering
   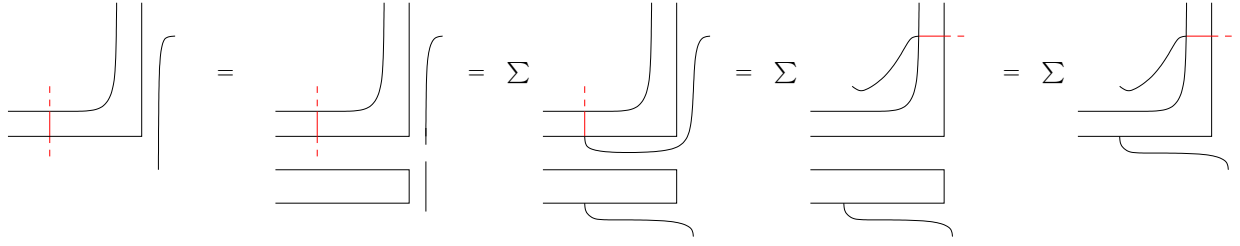  
  \caption{\textbf{The \DhAD-cornering module--2-module.} We may or may
    not explicitly include the \DD\ identity piece, as illustrated in
    the first and last pictures. The remaining pictures indicate the action
    of an element of $\BAlg(\PMC_1)$. The sums are taken over all idempotent marks on the bottom (only one of which is shown here).}
  \label{fig:corn-AD}
\end{figure}

We spell out some special cases of these definitions for $\cornAD$:
\begin{observation}\label{obs:AD}
  \begin{enumerate}[label=(AD-\arabic*),ref=(AD-\arabic*)]
  \item As a vector space, $\cornAD$ is given by
    $$\putaround{\RAlg(\PMC_0)}{\sbull}{}{0}{\sbull}\odot_{\Idem(\PMC_0)}
    \Alg((-\PMC_0)\cup\PMC_1),$$ where the tensor over $\Idem(\PMC_0)$
    means that the parts of the right idempotent in
    $\Alg((-\PMC_0)\cup\PMC_1)$ and the top idempotent in $\RAlg(\PMC_0)$
    lying in $\Alg(\PMC_0)$ should be complementary.
  \item\label{item:cornAD-easy-actions} The actions of $\RAlg(\PMC_0)$
    and $\Alg((-\PMC_0)\cup\PMC_1)$ are the obvious ones.
  \item\label{item:cornAD-differential} In terms of this basis, the
    differential on $\cornAD$ is given by
    \[
    \bdy
    \pbmat
    x\\
    \vtp\\
    y
    \pemat
    =
    \pbmat
    \bdy(x)\\
    \vtp
    \\
    y
    \pemat
    +
    \pbmat
    x\\
    \vtp
    \\
    \bdy (y)
    \pemat
    +\sum_{\xi\in\Chord_0(\PMC_1)}
    \pbmat
    x\\
    a(\xi)\\[.7ex]
    \vtp\\
    a(\xi)\\[.3ex]
    y
    \pemat,
    \]
    where $x\in \Alg((-\PMC_0)\cup\PMC_1)$ and $y\in \RAlg(\PMC_0)$ (and
    the differentials in the first two terms on the right are the
    differentials on the corresponding algebras).
  \item\label{item:cornAD-long-chords} Elements of
    $\putaround{\BAlg(\PMC_1)}{0}{\sbull}{}{\sbull}$ act via the inclusion
    $\putaround{\BAlg(\PMC_1)}{0}{\sbull}{}{\sbull}\hookrightarrow
    \Alg((-\PMC_0)\cup\PMC_1)$, together with the obvious right action of
    $\Alg((-\PMC_0)\cup\PMC_1)$ on $\cornAD$.
  \item Given a chord $\xi\in \PMC_0$ which goes off the bottom of
    $\PMC_0$, we have
    \[
    \pbmat x\\[.5ex] \vtp\\y\pemat
    \adown(\xi)=\sum_{\substack{\eta\in\Chord((-\PMC_0)\cup\PMC_1)\\ \eta\cap
        \PMC_1=\xi}} 
    \pbmat
    x\\ a(\eta)\\[.5ex] \vtp \\ \adown(\eta\cap\PMC_1)\\ y
    \pemat.
    \]
    (See Figure~\ref{fig:corn-AD}.)
  \end{enumerate}
\end{observation}

Finally:
\begin{definition}\label{def:corn-DD}
  Let $\PMC_0$ and $\PMC_1$ be matched intervals, and $\PMC$ the
  pointed matched circle obtained by gluing $\PMC_1$ and $\PMC_0$.
  The \emph{\DhDD-cornering module--2-module} $\cornDD(\PMC_0,\PMC_1)$
  associated to $\PMC_0$ and $\PMC_1$ is defined to be 
  \[
  \CFDDa(\Id_\PMC)\otimes_{\Alg(\PMC)}\AcornDD(\PMC_0,\PMC_1).
  \]
  (Here, $\CFDDa(\Id_\PMC)$ is the bimodule of Definition~\ref{def:bord-dd-id}.)
\end{definition}

\begin{proposition}\label{prop:cornering-mods-defined} The objects
  $\cornAA(\PMC_0,\PMC_1)$, $\cornAD(\PMC_0,\PMC_1)$,
  $\cornDA(\PMC_0,\PMC_1)$ and $\cornDD(\PMC_0,\PMC_1)$ are
  well-defined module--$2$-modules.  
\end{proposition}
\begin{proof}
  For $\cornAA(\PMC_0,\PMC_1)$ this is immediate from the
  definition. For the other three modules it follows from
  Proposition~\ref{prop:interval-dd-id-defined} and the fact that the
  cornered tensor product operations are well-defined.
\end{proof}


\section{The trimodules \texorpdfstring{$\TDDD$}{TDDD} and \texorpdfstring{$\TDDA$}{TDDA}}
\label{sec:trimods}

The goal of this section is to describe explicitly trimodules $\TDDD$
and $\TDDA$ which allow one to do cornered-type gluings of bordered
modules. To explain this precisely, we fix some notation:

\begin{definition}\label{def:connect-sum-pmc}
  Fix matched intervals $\PMC_1$ and $\PMC_2$. As described in
  Section~\ref{sec:glue-surf}, we can glue $\PMC_1$ and $\PMC_2$ to
  get a pointed matched circle $\PMC=\PMC_1\cup\PMC_2$.

  Writing $\PMC_i=(Z_i,\CircPts_i,M_i,z_i)$, $i=1,2$, there are
  projection maps $p_*^{i}\co H_1(Z_1\cup Z_2\setminus z,
  \CircPts_1\cup\CircPts_2)\to H_1(Z_i,\CircPts_i)$.
\end{definition}

For the rest of this section, fix matched intervals $\PMC_1$, $\PMC_2$, and
$\PMC_3$. Write $\PMC_i=(Z_i,\CircPts_i,M_i)$, where $\CircPts_i$ is
a subset of $Z_i$ of cardinality $4k_i$, and $M_i\co \CircPts_i\to [2k]$
is a $2$-to-$1$ map. Let:
\begin{equation}\label{eq:PMC-ij}
  \begin{aligned}
    \PMC_{12}&= \PMC_1\cup(-\PMC_2) & \PMC_{23}&= \PMC_2\cup(-\PMC_3)
    &
    \PMC_{31}&= \PMC_3\cup(-\PMC_1)\\
    \PMC_{21}&= \PMC_2\cup(-\PMC_1)=-\PMC_{12} & \PMC_{32}&=
    \PMC_3\cup(-\PMC_2)=-\PMC_{23} & \PMC_{13}&=
    \PMC_1\cup(-\PMC_3)=-\PMC_{31}.
  \end{aligned}
\end{equation}
(See also Figure~\ref{fig:THD}.)
There are associated surfaces $\PunctF(\PMC_{ij})$ with boundary
$S^1$, and closed surfaces $F(\PMC_{ij})=\PunctF(\PMC_{ij})\cup_\bdy
\bD^2$ (see Section~\ref{sec:am} or, e.g.,~\cite[Construction 3.2]{LOT2}). These surfaces satisfy $-F(\PMC_{ij})=F(-\PMC_{ij})=F(\PMC_{ji}).$

\begin{definition}
  There is a $3$-dimensional cobordism $\CT(\PMC_1,\PMC_2,\PMC_3)$ from
  $F(\PMC_{12})\amalg F(\PMC_{23})$ to $F(\PMC_{13})$ defined as
  follows. Let $\Delta$ denote a 2-simplex, with boundary edges $e_1$, $e_2$, and
  $e_3$ (in clockwise order). Consider the surfaces $\PunctF(\PMC_i)$, each of which has boundary $S^1$. Then $\CT(\PMC_1,\PMC_2,\PMC_3)$ is obtained from
  \[
  \bigl([0,1]\times \PunctF(\PMC_1))\amalg
  \bigl([0,1]\times \PunctF(\PMC_2))\amalg
  \bigl([0,1]\times \PunctF(\PMC_3))\amalg
  \bigl( \Delta\times S^1\bigr)
  \]
  by gluing $[0,1]\times \bdy \PunctF(\PMC_i)$ to $e_i\times S^1$.
\end{definition}

The manifold $\CT(\PMC_1,\PMC_2,\PMC_3)$ has boundary
\[
\bdy \CT(\PMC_1,\PMC_2,\PMC_3)=F(-\PMC_{12})\amalg F(-\PMC_{23})\amalg F(-\PMC_{31}).
\]
In particular, given bordered $3$-manifolds $Y_{12}$ and $Y_{23}$ with
boundaries $F(\PMC_{12})$ and $F(\PMC_{23})$, respectively, we can
glue $Y_{12}$ and $Y_{23}$ to $\CT(\PMC_1,\PMC_2,\PMC_3)$ to give a
$3$-manifold $(Y_{12}\amalg Y_{23})\cup \CT(\PMC_1,\PMC_2,\PMC_3)$
with boundary $F(\PMC_{13})$.

\begin{lemma}\label{lem:Co-glue}
  Let $Y'_{12}$ (respectively $Y'_{23}$) be a cornered $3$-manifold
  with vertical boundary $\PunctF(\PMC_{1})$ (respectively $\PunctF(\PMC_3)$) and
  horizontal boundary $\PunctF(\PMC_2)$ (respectively $-\PunctF(\PMC_2)$). Let
  $Y_{ij}$ be the smoothing of $Y'_{ij}$. Then
  \[
  Y'_{12}\cup_{\PunctF(\PMC_2)}Y'_{23}\cong (Y_{12}\amalg Y_{23})\cup \CT(\PMC_1,\PMC_2,\PMC_3)
  \]
  as bordered $3$-manifolds.
\end{lemma}
\noindent  This is immediate from the definitions.

\begin{corollary}\label{cor:corn-pair-via-bord}
  Let $Y'_{12}$ (respectively $Y'_{23}$) be a cornered $3$-manifold
  with vertical boundary $F(\PMC_{1})$ (respectively $F(\PMC_3)$) and
  horizontal boundary $F(\PMC_2)$ (respectively $-F(\PMC_2)$). Let
  $Y_{ij}$ be the smoothing of $Y'_{ij}$. Then  
  \begin{align*}
    \CFAa(Y'_{12}\cup_{F(\PMC_2)}Y'_{23})&\simeq \CFAa(Y_{12})\DTP_{\Alg(\PMC_{12})} \bigl(\CFAa(Y_{23})\DTP_{\Alg(\PMC_{23})} \CFDDAa(\CT(\PMC_1,\PMC_2,\PMC_3))\bigr)\\
    \CFDa(Y'_{12}\cup_{F(\PMC_2)}Y'_{23})&\simeq \CFAa(Y_{12})\DTP_{\Alg(\PMC_{12})} \bigl(\CFAa(Y_{23})\DTP_{\Alg(\PMC_{23})} \CFDDDa(\CT(\PMC_1,\PMC_2,\PMC_3))\bigr).
  \end{align*}
\end{corollary}
\noindent  This is immediate from Lemma~\ref{lem:Co-glue} and the pairing
  theorem for bordered Floer homology. (The notations $\CFDDDa$ and $\CFDDAa$ denote type \DDD\ and \DDA\ trimodules, respectively. See~\cite{LOT2} for the definitions of type \DD, \DA, and \AAm\ bimodules in bordered Floer homology. The extension from bimodules to trimodules is obvious; compare~\cite[Remark 5.7]{LOT2}.)

\begin{definition} \label{def:THD}
  Let $\THD{\PMC_1}{\PMC_2}{\PMC_3}$ be the Heegaard diagram shown in
  Figure~\ref{fig:THD}, with boundary
  \begin{equation}\label{eq:bdy-THD}
  -\bdy\THD{\PMC_1}{\PMC_2}{\PMC_3}=\PMC_{12}\amalg \PMC_{23}\amalg\PMC_{31},
  \end{equation}
  constructed as follows. Start with the
  canonical arced, bordered Heegaard diagrams $\HD_i$ for the identity
  map of $\PMC_i$~\cite[Definition 5.35]{LOT2}. Let $\arcz_i$ denote
  the arc in $\HD_i$ connecting the two boundary components. Then
  $\THD{\PMC_1}{\PMC_2}{\PMC_3}$ is obtained from $\Delta\amalg \Delta\amalg \coprod_i
  (\HD_i\setminus \nbd(\arcz_i))$ by gluing the arcs in
  $\bdy\Delta\amalg \bdy\Delta\amalg \coprod_i
  (\bdy\nbd\arcz_i)$ together, in such a way that the boundary of
  $\THD{\PMC_1}{\PMC_2}{\PMC_3}$ is given by
  Formula~\eqref{eq:bdy-THD}. (We place the basepoint in one of the two triangles $\Delta$.)
\end{definition}
\begin{figure}
  \centering
  \includegraphics[scale=1.2]{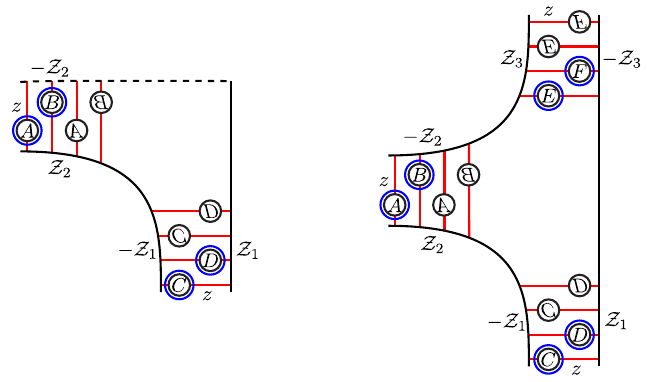}
  \caption{\textbf{Diagrams for cornering and cornered-type gluing.}
    Left:  The cornered Heegaard diagram for
    turning bordered invariants into cornered invariants. 
    Right: The bordered Heegaard diagram
    $\THD{\PMC_1}{\PMC_2}{\PMC_3}$ for cornered-type gluing. The black circles indicate handles attached, according to the letter pairs. The red lines are $\alpha$-arcs and the blue circles are $\beta$-circles.}
  \label{fig:THD}
\end{figure}

\begin{lemma}
  The bordered Heegaard diagram $\THD{\PMC_1}{\PMC_2}{\PMC_3}$
  represents the bordered $3$-manifold $\CT(\PMC_1,\PMC_2,\PMC_3)$.
\end{lemma}
\noindent  This is immediate from the definitions.

The goal of the rest of the section is to compute explicitly the
invariants
\[\CFDDDa(\CT(\PMC_1,\PMC_2,\PMC_3))=\CFDDDa(\THD{\PMC_1}{\PMC_2}{\PMC_3})
\]
and
\[
\CFDDAa(\CT(\PMC_1,\PMC_2,\PMC_3))=\CFDDAa(\THD{\PMC_1}{\PMC_2}{\PMC_3}).
\]
The trimodule $\CFDDDa(\CT(\PMC_1,\PMC_2,\PMC_3))$ is similar to the
bimodule $\CFDDa(\Id)$, and the computation of
$\CFDDDa(\CT(\PMC_1,\PMC_2,\PMC_3))$ is essentially the same as the
computation of $\CFDDa(\Id)$ given in~\cite[Section 3]{LOT4}. The
trimodule $\CFDDAa(\CT(\PMC_1,\PMC_2,\PMC_3))$ shares features with
both $\CFDDa(\Id)$ and $\CFDAa(\Id)\simeq \Alg(F)$.
The
combinatorial answers for $\CFDDDa(\CT(\PMC_1,\PMC_2,\PMC_3))$ and
$\CFDDAa(\CT(\PMC_1,\PMC_2,\PMC_3))$
are given in
Section~\ref{sec:describe-trimodules}, under the names
$\TDDD(\PMC_1,\PMC_2,\PMC_3)$ and $\TDDA(\PMC_1,\PMC_2,\PMC_3)$, respectively. We
prove that the answers are correct in Section~\ref{sec:compute-TDDD}
(for $\CFDDDa$) and Section~\ref{sec:compute-TDDA} (for $\CFDDAa$).

\subsection{Combinatorial descriptions of the
  trimodules}\label{sec:describe-trimodules}

\subsubsection{Description of \texorpdfstring{$\TDDD$}{TDDD}}
\begin{definition}\label{def:complementary-idems}
  Choose subsets $\SetS_i\subset [2k_i]$ for $i=1,2,3$. The
  sets $\SetS_i$ specify idempotents
  \begin{align*}
    I_{12}&=I(\SetS_1\cup ([2k_2]\setminus \SetS_2))\in \Alg(\PMC_{12})\\
    I_{23}&=I(\SetS_2\cup ([2k_3]\setminus \SetS_3))\in \Alg(\PMC_{23})\\
    I_{31}&=I(\SetS_3\cup ([2k_1]\setminus \SetS_1))\in \Alg(\PMC_{31}).
  \end{align*}
  We call the idempotents $I_{12}, I_{23}, I_{31}$ a \emph{complementary
    idempotent triple} for $\PMC_1, \PMC_2, \PMC_3$. We will also
  sometimes write $I_{12}\otimes I_{23}\otimes I_{31}$ for the
  complementary idempotent triple.

  Let $\Idem_{\DDD}$ denote the set of complementary idempotent triples.
\end{definition}
A complementary idempotent triple is illustrated in
Figure~\ref{fig:DDD-chord-triple}.

\begin{figure}
  \centering
  \begin{overpic}[tics=5, scale=.5]{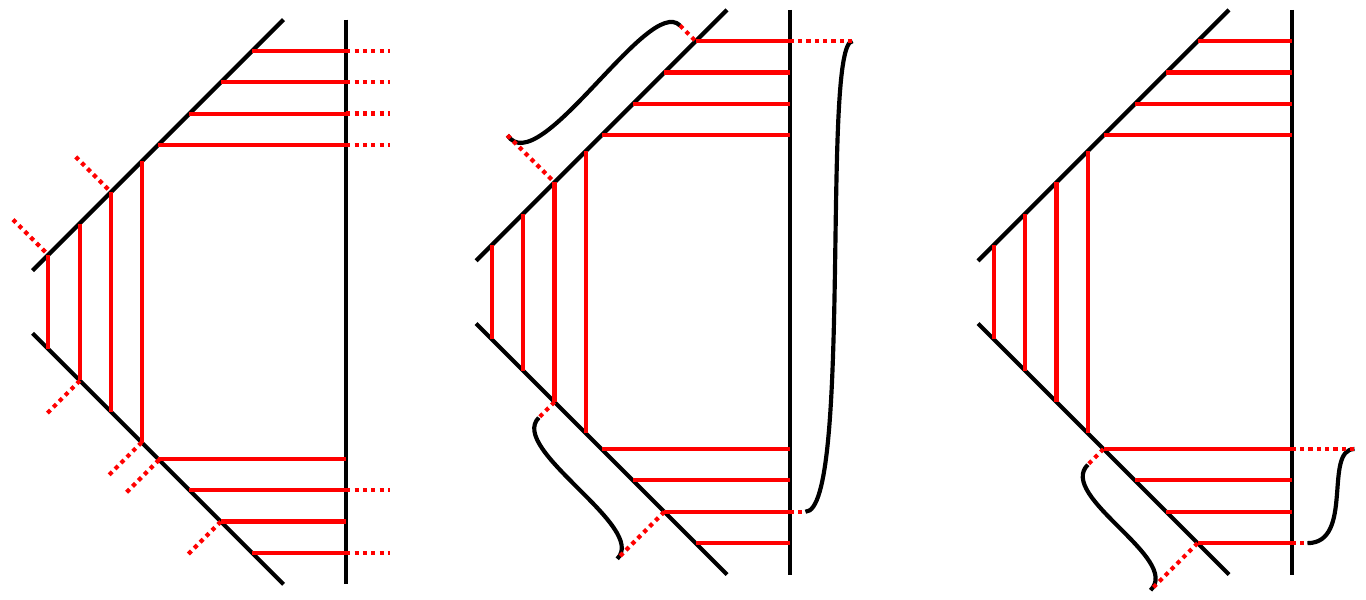}
    \put(15, -2){$-\PMC_{12}$}
    \put(17, 20){$-\PMC_{31}$}
    \put(4, 35){$-\PMC_{23}$}
    \put(37,7){$\xi_{12}$}
    \put(62,20){$\xi_{31}$}
    \put(37,37){$\xi_{23}$}
  \end{overpic}
  \caption{\textbf{Generators and differential on $\TDDD$.} 
    Left: a complementary idempotent triple.
    Center: a \DDD\ chord
    triple. Right: a chord contributing the term $a(\xi_1)\otimes \bOne_{23}\otimes
    a(-\xi_1)$ to the differential.}
  \label{fig:DDD-chord-triple}
\end{figure}
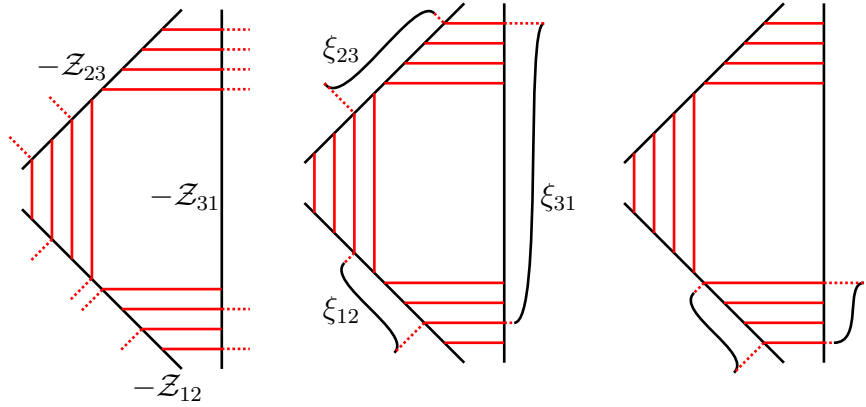

\begin{definition}\label{def:DDD-chord-triple}
  Given a pointed matched circle $\PMC$, let $\Chord(\PMC)$ denote the
  set of chords in $\PMC$. 
  A \emph{\DDD\ chord triple} for $\PMC_1, \PMC_2, \PMC_3$ consists of
  chords $\xi_{12}\in\Chord(\PMC_{12})$, $\xi_{23}\in\Chord(\PMC_{23})$ and
  $\xi_{31}\in\Chord(\PMC_{31})$ such that
  \begin{align*}
    p_*^2([\xi_{23}])&=r_*p_*^2([\xi_{12}]) \qquad
    p_*^1([\xi_{12}])=r_*p_*^1([\xi_{31}]) \qquad
    p_*^3([\xi_{31}])=r_*p_*^3([\xi_{23}]) \\
    \shortintertext{and}
    n_{z'_{23}}(\xi_{23})&=n_{z'_{12}}(\xi_{12})=n_{z'_{13}}(\xi_{13}),
  \end{align*}
  where $z'_{ij}$ is the point where $\PMC_i$ and $-\PMC_j$ are glued
  together, $n_{z'}$ denotes the local multiplicity at $z'$, and
  $[\xi]$ denotes the relative homology class represented by $\xi$.
  (See Definition~\ref{def:connect-sum-pmc} for the maps $p_*$ and
  $r_*$.)

  Let $\Chord_{\DDD}$ denote the set of \DDD\ chord triples. 
\end{definition}

A \DDD\ chord triple is illustrated in
Figure~\ref{fig:DDD-chord-triple}.
Abusing terminology, we will not distinguish between a chord triple
$(\xi_{12},\xi_{23},\xi_{31})$ and the associated algebra element $a(\xi_{12})\otimes a(\xi_{23})\otimes a(\xi_{31})$ of
$\Alg(\PMC_{12})\otimes\Alg(\PMC_{23})\otimes\Alg(\PMC_{31})$.

\begin{definition} \label{def:TDDD}
  The left-left-left trimodule
  $\TDDD(\PMC_1,\PMC_2,\PMC_3)$ is projectively generated by the set of complementary idempotent triples
  $I_{12}\otimes I_{23}\otimes I_{31}$ for $\PMC_1, \PMC_2,
  \PMC_3$, i.e., 
  \[
  \TDDD(\PMC_1,\PMC_2,\PMC_3)=\bigoplus_{I_{12}\otimes I_{23}\otimes I_{31}}\Alg(\PMC_{12})I_{12}\otimes \Alg(\PMC_{23})I_{23}\otimes \Alg(\PMC_{31})I_{31}.
  \]
  Define an element $A\in \Alg(\PMC_{12})\otimes
  \Alg(\PMC_{23})\otimes\Alg(\PMC_{31})$ by
  \begin{multline}\label{eq:A}
  A=\sum_{\xi_1\in \Chord(\PMC_1)}a(\xi_1)\otimes \bOne_{23}\otimes
  a(r(\xi_1))
  +\sum_{\xi_2\in \Chord(\PMC_2)}a(r(\xi_2))\otimes a(\xi_2)\otimes
  \bOne_{31}\\
  +\sum_{\xi_3\in \Chord(\PMC_3)}\bOne_{12}\otimes a(r(\xi_3))\otimes
  a(\xi_3)
  +\sum_{(\xi_{12},\xi_{23},\xi_{31})\in\Chord_{\DDD}}
  a(\xi_{12})\otimes a(\xi_{23})\otimes a(\xi_{31}).
  \end{multline}
  (Here, $\bOne_{ij}$ denotes the unit in $\Alg(\PMC_{ij})$.)
  The differential on $\TDDD(\PMC_1,\PMC_2,\PMC_3)$ is
  defined by
  \[
  \bdy(I_{12}\otimes I_{23}\otimes I_{31}) = \sum_{ (J_{12},
    J_{23}, J_{31})\in\Idem_{\DDD}} \bigl((I_{12}\otimes I_{23}\otimes
  I_{31})A (J_{12}\otimes J_{23}\otimes J_{31})\bigr)
  \]
  and the Leibniz rule.
\end{definition}

\noindent Two terms in the differential on $\TDDD(\PMC_1,\PMC_2,\PMC_3)$ are
illustrated in Figure~\ref{fig:DDD-chord-triple}.

Section~\ref{sec:compute-TDDD} will be devoted to proving:
\begin{theorem}\label{thm:compute-TDDD} There is an isomorphism
  $\CFDDDa(\THD{\PMC_1}{\PMC_2}{\PMC_3})\cong
  \TDDD(\PMC_1,\PMC_2,\PMC_3)$ of $\Field$-vector spaces intertwining the trimodule structures and the operators $\bdy$.
\end{theorem}

\noindent The reason for the convoluted phrasing in Theorem~\ref{thm:compute-TDDD} is that we will not verify directly that $\bdy^2=0$ on $\TDDD$; however, this follows from Theorem~\ref{thm:compute-TDDD}:
\begin{corollary}\label{cor:TDDD-is-trimodule}
  Definition~\ref{def:TDDD} defines a differential trimodule.
\end{corollary}
\begin{remark}
  It is not hard to prove Corollary~\ref{cor:TDDD-is-trimodule}
  directly; compare~\cite[Proposition 3.4]{LOT4}.
\end{remark}
\subsubsection{Description of \texorpdfstring{$\TDDA$}{TDDA}}
\begin{lemma}\label{lem:DDA-factor}
  Any basic generator $a=a(\rhos)\in\Alg(\PMC_1\cup(-\PMC_3))$ can be factored
  uniquely as a product of basic generators
  $a=b\cdot a(\xi_1)\cdots a(\xi_k) \cdot c$ where $b\in
  \Alg(-\PMC_3)\subset\Alg(\PMC_1\cup(-\PMC_3))$,
  $c\in\Alg(\PMC_1)\subset\Alg(\PMC_1\cup(-\PMC_3))$, and each
  $\xi_i$ is a chord in $\PMC_1\cup(-\PMC_3)$ with initial point in
  $\PMC_1$ and terminal point in $-\PMC_3$. 
\end{lemma}
\begin{proof}
  Suppose that $\rhos=\{\rho_1,\dots,\rho_n\}$ is a collection of
  chords in $\PMC$. Write $\rhos=\rhos'\cup \rhos''$ where
  $\rhos'=\{\rho_1,\dots,\rho_m\}$ and
  $\rhos''=\{\rho_{m+1},\dots,\rho_n\}$. Then $a(\rhos)=a(\rhos')\cdot
  a(\rhos'')$ if the following condition is met: 
  \begin{itemize}
  \item For each pair $(i,j)$ with $1\leq i \leq m < j\leq n$, the
    terminal point of $\rho_i$ is not the initial point of $\rho_j$
    and is not matched to the initial point of $\rho_j$.
  \end{itemize}

  Now, consider a basic generator $a(\rhos)\in
  \Alg(\PMC_1\cup(-\PMC_3))$. Write $\rhos=\rhos_1\cup
  \rhos_2\cup\rhos_3$ where:
  \begin{itemize}
  \item Each of the chords in $\rhos_1$ is completely contained in $\PMC_3$.
  \item Each of the chords in $\rhos_2$ has its initial point in
    $\PMC_1$ and its terminal point in $-\PMC_3$.
  \item Each of the chords in $\rhos_3$ is completely contained in $-\PMC_3$.
  \end{itemize}
  Then, by the observation above,
  \[
  a(\rhos_1\cup\rhos_2\cup\rhos_3)=a(\rhos_2\cup\rhos_3)\cdot
  a(\rhos_1)=a(\rhos_3)\cdot a(\rhos_2)\cdot a(\rhos_1).
  \]
  This proves existence of the factorization. The uniqueness statement
  is clear.
\end{proof}

The following are the analogues of
Definitions~\ref{def:complementary-idems}
and~\ref{def:DDD-chord-triple} for the \DDA\ case:
\begin{definition}
  Choose subsets $\SetS_i\subset [2k_i]$ for $i=1,2,3$. The
  sets $\SetS_i$ specify idempotents
  \begin{align*}
    I_{12}&=I(\SetS_1\cup ([2k_2]\setminus \SetS_2))\in \Alg(\PMC_{12})\\
    I_{23}&=I(\SetS_2\cup ([2k_3]\setminus \SetS_3))\in \Alg(\PMC_{23})\\
    I_{31}&=I(\SetS_1\cup ([2k_3]\setminus \SetS_3))\in \Alg(-\PMC_{31}).
  \end{align*}
  We call the idempotents $I_{12}, I_{23}, I_{31}$ a \emph{\DDA\
    idempotent triple} for $\PMC_1, \PMC_2, \PMC_3$. We will also
  sometimes write $I_{12}\otimes I_{23}\otimes I_{31}$ for the \DDA\
  idempotent triple.

  Let $\Idem_{\DDA}$ denote the $\Field$-vector space spanned by the
  set of \DDA\ idempotent triples. The vector space $\Idem_{\DDA}$ has
  obvious left actions of the sub-rings of idempotents
  $\Idem(\PMC_{12})\subset\Alg(\PMC_{12})$ and
  $\Idem(\PMC_{23})\subset\Alg(\PMC_{23})$ and an obvious right
  action by $\Idem(-\PMC_{31})\subset\Alg(-\PMC_{31})$.
\end{definition}

\begin{definition}
  A \emph{\DDA\ chord triple} for $\PMC_1, \PMC_2, \PMC_3$ consists of
  chords $\xi_{12}$ for $\PMC_{12}$, $\xi_{23}$ for
  $\PMC_{23}$ and $\xi_{31}$ for $-\PMC_{31}$ such that
  the following holds:
  \begin{align*}
    p_*^2([\xi_{23}])&=r_*p_*^2([\xi_{12}])\qquad
    p_*^1([\xi_{12}])=p_*^1([\xi_{31}])\qquad
    p_*^3([\xi_{31}])=p_*^3([\xi_{23}])\\
    \shortintertext{and}
    n_{z'_{23}}(\xi_{23})&=n_{z'_{12}}(\xi_{12})=n_{z'_{13}}(\xi_{13}).
  \end{align*}
 
  Let $\Chord_{\DDA}$ denote the set of \DDA\ chord triples. 
\end{definition}

Graphically, \DDA\ chord triples look like \DDD\ chord triples
(Definition~\ref{def:DDD-chord-triple}); the difference is merely in
how we are interpreting one of the boundary components.

\begin{figure}
  \centering
  \begin{overpic}[tics=5]{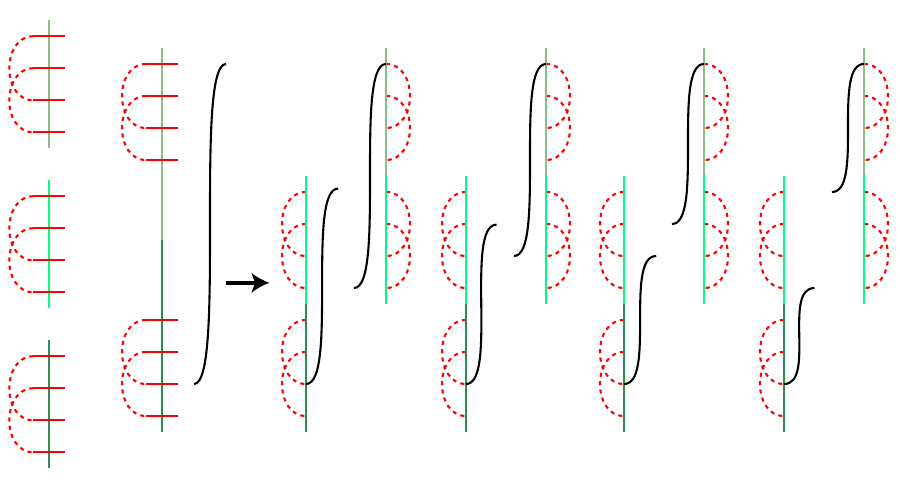}
    \put(4,2){$\PMC_1$}
    \put(4,20){$\PMC_2$}
    \put(4,38){$\PMC_3$}
    \put(16,3){$-\PMC_{31}$}
    \put(30,3){$\PMC_{12}$}
    \put(39, 17){$\PMC_{23}$}
    \put(37,27){$\otimes$}
    \put(46,27){$+$}
    \put(55,27){$\otimes$}
    \put(64,27){$+$}
    \put(73,27){$\otimes$}
    \put(82,27){$+$}
    \put(91,27){$\otimes$}    
    \put(26,46){$\xi$}
    \put(36,11){$\xi_{12}$}
    \put(38,46){$\xi_{23}$}
    \put(25,24){\tiny split}
  \end{overpic}
  \caption{\textbf{The splitting operation.} We have suppressed the
    sum over the idempotents: the output should in fact include a sum
    over all sensible ways of adding horizontal lines so that the result is appropriately complementary.}
  \label{fig:split-chord}
\end{figure}
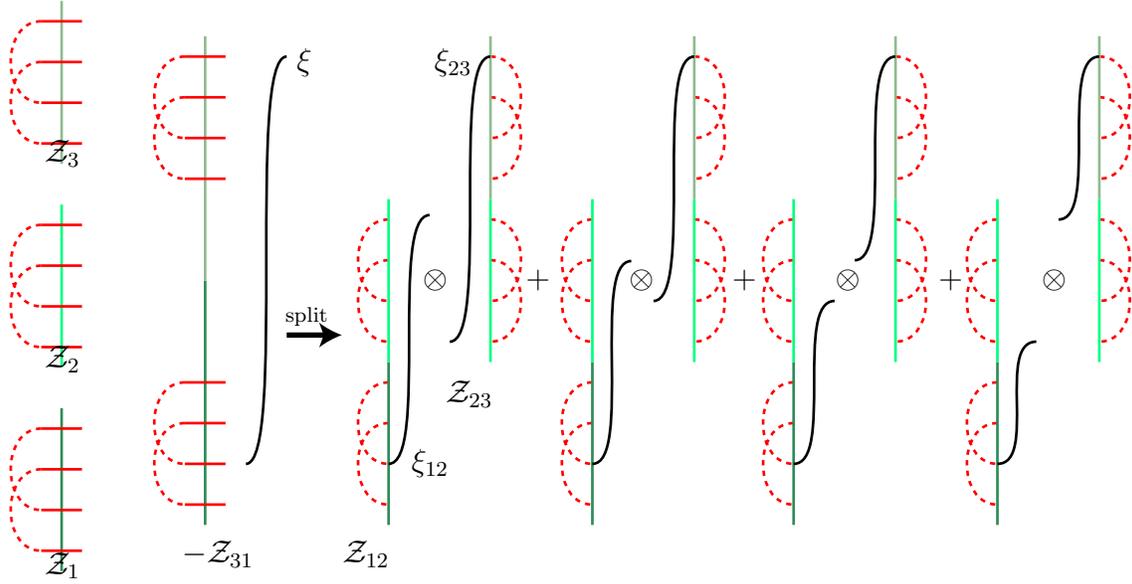
\begin{definition}
  Let $\xi$ be a chord in $\PMC_1\cup(-\PMC_3)$. Define an element
  $\chordsplit(\xi)\in \Alg(\PMC_{12})\otimes \Alg(\PMC_{23})$ by:
  \begin{align*}
    \chordsplit(\xi)=
    \sum_{\substack{(\xi_{12},\xi_{23},\xi)\in\Chord_{\DDA}\\
        (I_{12},I_{23},I_{31})\in\Idem_{\DDA}\\ (J_{12},J_{23},J_{31})\in\Idem_{\DDA}\\
        I_{31}a(\xi)J_{31}\neq 0}} (I_{12}a(\xi_{12})J_{12})\otimes(I_{23}a(\xi_{23})J_{23}).
  \end{align*}
  (See Figure~\ref{fig:split-chord}.)
\end{definition}

\begin{figure}
  \centering
  \includegraphics{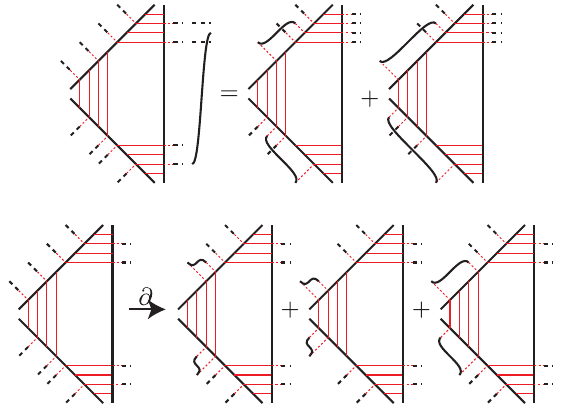}
  \caption{\textbf{Action and differential on $\TDDA$.} Top: the
    action by an element of the algebra $\Alg(\PMC_{31})$. Bottom: the
    differential of a generator of $\TDDA$.}
  \label{fig:DDA-action}
\end{figure}

Our next goal is to define the trimodule
$\TDDA(\PMC_1,\PMC_2,\PMC_3)$, which we can view as a left-right
$(\Alg(\PMC_{12})\otimes\Alg(\PMC_{23}),\
\Alg(-\PMC_{31}))$-bimodule.
\begin{definition}\label{def:module-TDDA}
  As a left module,
  $\TDDA(\PMC_1,\PMC_2,\PMC_3)$
  is just 
  \[ \bigl(\Alg(\PMC_{12})\otimes_{\Field}\Alg(\PMC_{23})\bigr)\otimes_{\Idem(\PMC_{12})\otimes\Idem(\PMC_{23})}\Idem_{\DDA}
  =\Alg(\PMC_{12})\otimes_{\Idem(\PMC_2)}\Alg(\PMC_{23}).
  \]
  The idempotents of $\Alg(-\PMC_{31})$ act on the right via their
  obvious action on $\Idem_{\DDA}$. It remains to define the
  differential and the right action of non-idempotent elements.

  The differential on $\TDDA(\PMC_1,\PMC_2,\PMC_3)$ is defined by
  \[
  \bdy(I_{12},I_{23},I_{31})=\sum_{\xi_2\in\Chord(\PMC_{2})}\sum_{(J_{12},J_{23},J_{31}) \in\Idem_{\DDA}}
  \bigl(I_{12}\cdot a(r(\xi_2))\cdot J_{12} \otimes I_{23}\cdot
  a(\xi_2)\cdot J_{23}\bigr)\otimes (J_{12},J_{23},J_{31}).
  \]
  and the Leibniz rule.

  By Lemma~\ref{lem:DDA-factor}, to define the right module structure
  on $\TDDA$ it suffices to define the actions of $\Alg(\PMC_1)$,
  $\Alg(-\PMC_3)$ and elements $a(\xi)$ where $\xi$ is a chord
  starting in $\PMC_1$ and terminating in $-\PMC_3$.
  
  Given an element $x\otimes y\in
  \TDDA=\Alg(\PMC_{12})\otimes\Alg(\PMC_{23})$
  and elements $a\in \Alg(\PMC_1)\subset\Alg(\PMC_{12})$, $b\in\Alg(-\PMC_3)\subset\Alg(\PMC_{23})$, define
  \begin{align*}
    (x\otimes y)\cdot a&:=(xa)\otimes y\\
    (x\otimes y)\cdot b&:=x\otimes (yb).
  \end{align*}
  Given a chord $\xi$ starting in $\PMC_1$ and ending in
  $-\PMC_3$, define 
  \[
  (x\otimes y)\cdot a(\xi)=(x\otimes y)\chordsplit(\xi).
  \]
\end{definition}

We will verify in Proposition~\ref{prop:TDDA-is-module} that
Definition~\ref{def:module-TDDA} defines a
trimodule. This, in turn, will be the main work in proving:

\begin{theorem}\label{thm:compute-TDDA} There is a quasi-isomorphism of $\Ainf$-trimodules
  $\CFDDAa(\THD{\PMC_1}{\PMC_2}{\PMC_3})\cong
  \TDDA(\PMC_1,\PMC_2,\PMC_3)$.
\end{theorem}

\noindent Note here that though $\TDDA(\PMC_1,\PMC_2,\PMC_3)$ is an honest differential trimodule, $\CFDDAa(\THD{\PMC_1}{\PMC_2}{\PMC_3})$ may be only an $\Ainf$-trimodule.

\subsection{Computation of \texorpdfstring{$\TDDD$}{TDDD}}\label{sec:compute-TDDD}
This section is devoted to proving Theorem~\ref{thm:compute-TDDD}. The
proof is an adaptation of techniques
from~\cite{LOT4}, and in this section we assume familiarity with that paper. As there, the proof has two components. First, one
uses the grading to restrict what terms can occur in the
differential. Second, one uses the fact that $\bdy^2=0$ on $\CFDDDa(\CT(\PMC_1,\PMC_2,\PMC_3)$ and a few
simple computations to show that all terms in the correct grading do,
in fact, appear. 

For the first half of the argument, the paper~\cite{LOT4} has two
different approaches. When computing $\CFDDa$ of the identity
cobordism, it uses a factorization argument; for $\CFDDa$ of an
arc-slide it uses the notion of a coefficient algebra (see
Definition~\ref{def:coeff-alg}, below). (In fact, both arguments can
be made to work for both computations, but the factorization argument
for arc-slides involves a massive case analysis.) Here, we will use
the coefficient algebra approach; so, even though the diagrams look
closer to the diagram for the identity cobordism, the proof is more in
the spirit of the arc-slide argument from~\cite{LOT4}.

The following is analogous to~\cite[Definitions 3.1 and 4.3]{LOT4}:
\begin{definition}
  The \emph{diagonal subalgebra $\DiagAlg$} of
  $\Alg(\PMC_{12})\otimes_\Field \Alg(\PMC_{23})\otimes_\Field
  \Alg(\PMC_{31})$ is the subalgebra of
  $\Alg(\PMC_{12})\otimes_\Field \Alg(\PMC_{23})\otimes_\Field
  \Alg(\PMC_{31})$ generated by
  $\{(I_{12}\otimes I_{23}\otimes I_{31})\cdot (a_{12}\otimes a_{23}\otimes a_{31})\cdot (J_{12}\otimes
  J_{23}\otimes J_{31})\}$
  where
  $(I_{12},I_{23},I_{31})$ and $(J_{12},J_{23},J_{31})$ are complementary idempotent triples and
   $a_{12}\otimes a_{23}\otimes a_{31}$ is a triple of strand
  diagrams such that
  \begin{align*}
    p_*^2([a_{23}])&=r_*p_*^2([a_{12}]) \qquad\qquad
    p_*^1([a_{12}])=r_*p_*^1([a_{31}]) \qquad\qquad
    p_*^3([a_{31}])=r_*p_*^3([a_{23}])\\
    \shortintertext{and}
    n_{z'_{23}}(a_{23})&=n_{z'_{12}}(a_{12})=n_{z'_{13}}(a_{13}).
  \end{align*}
  (Here, $[a]$ denotes the relative homology class represented by $a$.)
\end{definition}

The relevance of the diagonal subalgebra comes from the following:
\begin{lemma}\label{lem:diff-in-diag-alg}
  For each complementary idempotent triple $I_{12}\otimes
  I_{23}\otimes I_{31}$ there is a unique generator
  $\x=\x_{I_{12}\otimes I_{23}\otimes I_{31}}$ of
  $\CFDDDa(\THD{\PMC_1}{\PMC_2}{\PMC_3})$ so that $(I_{12}\otimes
  I_{23}\otimes I_{31})\cdot \x_{I_{12}\otimes I_{23}\otimes
    I_{31}}=\x_{I_{12}\otimes I_{23}\otimes I_{31}}$, and
  every generator of
  $\CFDDDa(\THD{\PMC_1}{\PMC_2}{\PMC_3})$ arises this way.
  Moreover, for
  any generator $\x=\x_{I_{12}\otimes I_{23}\otimes I_{31}}$ of
  $\CFDDDa(\THD{\PMC_1}{\PMC_2}{\PMC_3})$, the differential $\bdy \x$
  has the form
  \[
  \bdy\x = \sum_\y a^{\x,\y}\otimes \y
  \]
  where each $a^{\x,\y}$ is an element of $\DiagAlg$.
\end{lemma}
\begin{proof}
  The statement about generators is clear. For the statement about the differential, 
  let $B$ be a domain in $\THD{\PMC_1}{\PMC_2}{\PMC_3}$ and let
  $\bdy_{ij}B$ denote the part of $\bdy B$ lying in $\PMC_{ij}$.  Then
  \begin{align*}
    p_*^2(\bdy_{23}B)&=r_*p_*^2(\bdy_{12}B) \qquad\qquad
    p_*^1(\bdy_{12}B)=r_*p_*^1(\bdy_{31}B) \qquad\qquad
    p_*^3(\bdy_{31}B)=r_*p_*^3(\bdy_{23}B) \\
    \shortintertext{and}
    n_{z'_{23}}(B)&=n_{z'_{12}}(B)=n_{z'_{13}}(B).
  \end{align*}  
  The result follows.
\end{proof}

To prove Theorem~\ref{thm:compute-TDDD}, we need some further
properties of $\DiagAlg$. First we explain gradings. There will be a detailed discussion of gradings in Section~\ref{sec:gradings}; for now, we need a $\ZZ$-grading $\gr_{\DiagAlg}$ on $\DiagAlg$ defined as follows. Recall that the algebra $\Alg(\PMC_{ij})$ is graded by a
group $\bigGroup(\PMC_{ij})$, which is a $\ZZ$ central extension of
$H_1(Z_{ij}\setminus\{z_{ij}\},\CircPts_{ij})$; see \cite[Section 3.3]{LOT1} or the summary in Section~\ref{sec:GradingBordered} below.  We can write elements
of this group as pairs $(m;x)$ where $m\in\OneHalf\ZZ$ and $x\in
H_1(Z_{ij}\setminus\{z_{ij}\},\CircPts_{ij})$. Suppose $a_{12}\otimes
a_{23}\otimes a_{31}\in\DiagAlg$, with
$\gr(a_{ij})=(m_{ij};x_{ij})$. Then define
\begin{equation}\label{eq:gr-diag-1}
\gr_{\DiagAlg}(a_{12}\otimes a_{23}\otimes
a_{31})=m_{12}+m_{23}+m_{31}+\frac{1}{2}n_{z'_{23}}(a_{23}).
\end{equation}
(Since $n_{z'_{23}}(a_{23})=n_{z'_{12}}(a_{12})=n_{z'_{13}}(a_{13})$,
this expression is symmetric in the $ij$'s.)

The grading $\gr_{\DiagAlg}$ has a more invariant description as
follows. First, we recall~\cite[Definition 2.15]{LOT4}:
\begin{definition}\label{def:coeff-alg}
  Let $M$ be a type $D$ module over a differential algebra $\Alg$, graded by a $G$-set $S$ (where $G$ is a group with distinguished central element $\lambda$).
  The \emph{coefficient algebra} of $M$ is generated
  over $\Field$ by triples $(\x,a,\y)$ with $\x,\y$ generators of $M$
  and ``$a$" a generator of $\Alg$ satisfying:
  \begin{enumerate}
  \item If $I\cdot \x=\x$ and $J\cdot\y =\y$ for basic idempotents $I$
    and $J$ then $a=I\cdot a\cdot J$; and
  \item\label{item:coeff-alg-2} There is a $k\in\ZZ$ so that $\lambda^k\gr(\x)=\gr(a)\gr(\y)$.
  \end{enumerate}
  Addition is formal except that $(\x,a_1+a_2,\y)=(\x,a_1,\y)+(\x,a_2,\y)$.
  The differential is given by $\bdy(\x,a,\y)=(\x,\bdy(a),\y)$ and the
  product is given by
  \[
  (\x_1,a_1,\y_1)\cdot(\x_2,a_2,\y_2)=
  \begin{cases}
  (\x_1,a_1\cdot a_2,\y_2)  & \y_1=\x_2\\
  0 &\text{otherwise}.
  \end{cases}
  \]
  The grading on the coefficient algebra is given by $\gr(\x,a,\y)=k$
  where $\lambda^k\gr(\x)=\gr(a)\gr(\y)$. This is well-defined if
  $\lambda$ acts freely on $S$~\cite[Lemma 2.16]{LOT4}.
  
  This extends to type left-left \DD\ bimodules $M$ over $\Alg$ and
  $\Blg$ (respectively left-left-left type \DDD\ trimodules $M$ over
  $\Alg$, $\Blg$, and $\Clg$) by viewing $M$ as a module over
  $\Alg\otimes \Blg$ (respectively $\Alg\otimes\Blg\otimes\Clg$).
\end{definition}

The following lemma is analogous to~\cite[Lemma 4.12]{LOT4}.
\begin{lemma}
  The coefficient algebra of $\CFDDDa(\THD{\PMC_1}{\PMC_2}{\PMC_3})$
  is exactly the diagonal subalgebra $\DiagAlg$.
\end{lemma}
\begin{proof}
  Recall that $\Alg(\PMC)$ has a canonical grading by a group
  $G'(\PMC)$ consisting of pairs $(m;a)$ with $m\in\frac{1}{2}\ZZ$ and
  $a\in H_1(Z\setminus \{z\},\CircPts)$; we refer to $m$ as the
  \emph{Maslov component} of the grading and $a$ as the
  \emph{$\SpinC$-component} of the grading. The central element
  $\lambda$ is $(1;0)$.

  The trimodule
  $\CFDDDa(\THD{\PMC_1}{\PMC_2}{\PMC_3})$ is graded by the
  $\bigl(G'(\PMC_{12})\times_\ZZ G'(\PMC_{23})\times_\ZZ
  G'(\PMC_{31})\bigr)$-set
  \[
  \bigl(G'(\PMC_{12})\times_\ZZ G'(\PMC_{23})\times_\ZZ
  G'(\PMC_{31})\bigr)/\langle g(B)\mid B\in\pi_2(\x_0,\x_0)\rangle.
  \]
  Here, $g(B)=(m(B);\bdy^\bdy(B))$ where the $\SpinC$-component $\bdy^\bdy(B)$ is given by the multiplicities of
  $B$ at $\bdy\THD{\PMC_1}{\PMC_2}{\PMC_3}$.  Hence, the
  condition~(\ref{item:coeff-alg-2}) in the definition of the
  coefficient algebra is equivalent to $\gr(\x)$ and $\gr(a)\gr(\y)$
  having the same $\SpinC$-component, up to adding the boundaries of
  periodic domains.

  By Lemma~\ref{lem:diff-in-diag-alg}, the generators of
  $\CFDDDa(\THD{\PMC_1}{\PMC_2}{\PMC_3})$ correspond to the
  complementary idempotent triples in the diagonal subalgebra.  It
  remains to show that the triples of the form
  $(\x_{I_{12},I_{23},I_{31}},a,\y_{J_{12},J_{23},J_{31}})$ in the
  coefficient algebra correspond to the elements of $(I_{12}\otimes
  I_{23}\otimes I_{31})\cdot \DiagAlg\cdot (J_{12}\otimes
  J_{23}\otimes J_{31})$.  Given generators
  $\x=\x_{I_{12},I_{23},I_{31}}$ and $\y=\y_{J_{12},J_{23},J_{31}}$ in
  $\CFDDDa(\THD{\PMC_1}{\PMC_2}{\PMC_3})$, the grading satisfies
  \[
  \gr(\y)=g(B)\gr(\x)
  \]
  for any $B\in\pi_2(\x,\y)$.  In particular, $\gr(\x)$ and
  $\gr(a)\gr(\y)$ have the same $\SpinC$-component if and only if the
  support of $a$ is the boundary of some domain connecting $\x$ and
  $\y$. Inspecting the diagram, this occurs if and only if
  $(I_{12}\otimes I_{23}\otimes I_{31})\cdot a\cdot(J_{12}\otimes
  J_{23}\otimes J_{31})$ lies in the diagonal subalgebra.
\end{proof}

Turning to the grading on the coefficient algebra, the
following lemma and corollary are analogous to~\cite[Proposition 4.15]{LOT4}:
\begin{lemma}\label{lem:g-of-domains} Let $\x$ and $\y$ be generators
  of $\CFDDDa(\THD{\PMC_1}{\PMC_2}{\PMC_3})$, and let $B\in
  \pi_2(\x,\y)$ be a domain. Let $e(B)$ denote the Euler measure of
  $B$ and let $n_\x(B)$ denote the point measure of $B$ with respect
  to $\x$. Then
  \begin{equation}\label{eq:TDDD-index}
  e(B)+n_\x(B)+n_\y(B)=-\frac{1}{2}n_{z'_{23}}(B).
  \end{equation}
\end{lemma}
\begin{proof}
  The diagram $\THD{\PMC_1}{\PMC_2}{\PMC_3}$ has two kinds of regions:
  $8$-sided regions $R_i$ running between two boundary components of
  the diagram and a single $12$-sided region $T$ in the middle
  touching all three boundary components. (See
  Figure~\ref{fig:THD-region-labels}.) For any generators $\x$ and
  $\y$, each $R_i$ has $n_\x(R_i)=n_\y(R_i)=1/2$, while
  $e(R_i)=-1$. Thus, $R_i$ does not contribute to the left side of
  Formula~(\ref{eq:TDDD-index}). Similarly, for any generators $\x$
  and $\y$, the region $T$ has $n_\x(T)=n_\y(T)=3/4$ and
  $e(T)=-2$. Thus, $T$ contributes $-1/2$ to the left side of
  Formula~(\ref{eq:TDDD-index}). Of course, $n_{z'_{23}}(R_i)=0$ while
  $n_{z'_{23}}(T)=1$. This proves the result.
\end{proof}

\begin{figure}
  \centering
  \includegraphics{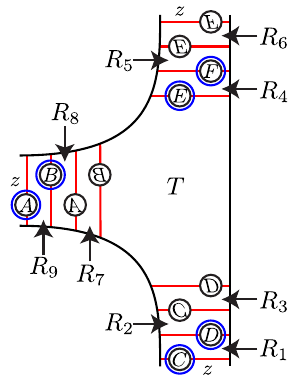}
  \caption{\textbf{Labeling of regions in
      $\THD{\PMC_1}{\PMC_2}{\PMC_3}$.} Each $R_i$ has $8$ sides, and
    $T$ has $12$ sides. The ordering of the $R_i$ is not important.}
  \label{fig:THD-region-labels}
\end{figure}

\begin{corollary}\label{cor:gr-on-diag}
  The $\ZZ$-grading on the coefficient algebra
  $\Coeff(\CFDDDa(\THD{\PMC_1}{\PMC_2}{\PMC_3}))$ is given by
  $\gr_\DiagAlg$.
\end{corollary}
\begin{proof}
  Associated to each generator $a_{12}\otimes a_{23}\otimes a_{31}$ of
  the diagonal algebra is a domain $B(a_{12}\otimes a_{23}\otimes
  a_{31})$ in $\THD{\PMC_1}{\PMC_2}{\PMC_3}$ so that the boundary of
  $B(a_{12}\otimes a_{23}\otimes a_{31})$ is the same as the support
  of $a_{12}\otimes a_{23}\otimes a_{31}$. The grading of
  $a_{12}\otimes a_{23}\otimes a_{31}$, viewed as an element of the
  coefficient algebra, is the Maslov component of
  \begin{equation}\label{eq:gr-diag-proof}
  (\gr(a_{12})\times \gr(a_{23})\times \gr(a_{31}))\cdot g(B)^{-1}\in
  G'(\PMC_{12})\times_\ZZ G'(\PMC_{23})\times_\ZZ G'(\PMC_{31}).
  \end{equation}
  (The $\SpinC$-component of this product is zero.) 
  By Lemma~\ref{lem:g-of-domains}, $g(B)=(-\frac{1}{2}n_{z'_{23}}(B);\bdy^\bdy(B))$.
  So, by Formula~\ref{eq:gr-diag-proof}, writing $m_{ij}$ for the
  Maslov component of the grading of $a_{ij}$ we have
  \[
  \gr_\DiagAlg(a_{12}\otimes a_{23}\otimes a_{31})=m_{12}+m_{23}+m_{31}+\frac{1}{2}n_{z'_{23}}(B),
  \]
  in agreement with Formula~\ref{eq:gr-diag-1}.
\end{proof}

The following lemma is analogous to~\cite[Lemmas 4.20 and 4.36]{LOT4}.
\begin{lemma}\label{lem:DDD-gr}
  If $a_{12}\otimes a_{23}\otimes a_{31}$ is a basic, non-idempotent element of
  $\DiagAlg$, then $\gr_\DiagAlg(a_{12}\otimes a_{23}\otimes
  a_{31})=-1$ if and only if either
  \begin{itemize}
  \item $a_{12}\otimes a_{23}\otimes a_{31}$ is a chord triple, or
  \item $a_{12}\otimes a_{23}\otimes a_{31}$ has the form
    $I_{12}a(\xi_1)\otimes I_{23} \otimes I_{31}a(r(\xi_1))$, $I_{12}
    a(\xi_2)\otimes I_{23}a(\xi_2)\otimes I_{31}$, or $I_{12}\otimes I_{23}
    a(r(\xi_3))\otimes I_{31}a(\xi_3)$.
  \end{itemize}
\end{lemma}
\begin{proof}
  If $a_{12}\otimes a_{23}\otimes a_{31}$ is a chord triple then the
  Maslov components of the gradings are $m_{ij}=-1/2$, and we have
  \[
  \gr_{\DiagAlg}(a_{12}\otimes a_{23}\otimes a_{31})=-1/2 -1/2 -1/2 +
  1/2=-1.
  \]
  In the second case, given a chord $\xi_1$, say, we have
  \[
  \gr_{\DiagAlg}(I_{12}a(\xi_1)\otimes I_{23} \otimes
  I_{31}a(r(\xi_1))) = -1/2 +0 -1/2 + 0=-1.
  \]

  Conversely, if $a_{ij}$ has $n_{ij}$ moving strands then, 
  by~\cite[Lemma 3.6]{LOT2}, the Maslov component of the grading of
  $a_{12}\otimes a_{23}\otimes a_{31}$ is at most
  $-(n_{12}+n_{23}+n_{31})/2$; and
  $n_{z'_{23}}(a_{23})\leq\min\{n_{12}, n_{23}, n_{31}\}.$
  Thus,
  \[
  \gr_{\DiagAlg}(a_{12}\otimes a_{23}\otimes a_{31})\leq \frac{1}{2}\bigl(-n_{12}-n_{23}-n_{31}+\min\{n_{12}, n_{23}, n_{31}\}\bigr).
  \]
  At least two of the $n_{ij}$'s are $1$ or larger. So, the only two
  cases in which $\gr_{\DiagAlg}(a_{12}\otimes a_{23}\otimes
  a_{31})\geq -1$ are when $\{n_{12},n_{23},n_{31}\}=\{1,1,0\}$ or
  when $\{n_{12},n_{23},n_{31}\}=\{1,1,1\}$ (and
  $n_{z'_{23}}(a_{23})=1$). This proves the result.
\end{proof}

\begin{proof}[Proof of Theorem~\ref{thm:compute-TDDD}]
  The isomorphism of modules
  \[
  \CFDDDa(\THD{\PMC_1}{\PMC_2}{\PMC_3})\cong
  \TDDD(\PMC_1,\PMC_2,\PMC_3);
  \]
  is clear (compare
  Lemma~\ref{lem:diff-in-diag-alg}).
  It remains to show that this
  isomorphism entwines the differentials on the two sides. Again by
  Lemma~\ref{lem:diff-in-diag-alg}, the coefficients occurring in the
  differential on $\CFDDDa(\THD{\PMC_1}{\PMC_2}{\PMC_3})$ lie in the
  diagonal algebra.  Write $\bdy \x = \sum_\y a^{\x,\y}\otimes \y$. By
  Lemma~\ref{lem:DDD-gr} (and the definition of
  the coefficient algebra~\cite[Definition 2.15]{LOT4}), the basic
  elements of $\DiagAlg$ occurring in $a^{\x,\y}$ are a subset of the terms in
  the element $A$ (Formula~\eqref{eq:A}).
  
  It remains to show that every term in $I_\x A I_\y$ occurs in
  $a^{\x,\y}$. To keep terminology simple, we will say that a term $a$
  in $A$ occurs in $a^{\x,\y}$ if either $I_\x a I_\y=0$ or $a$
  has a non-zero coefficient in the sum $a^{\x,\y}$.
  Then, since $\THD{\PMC_1}{\PMC_2}{\PMC_3}$ contains the
  identity Heegaard diagram for $\PMC_i$ as a sub-diagram, it follows
  from~\cite[Theorem 1]{LOT4} that all of the terms of the form
  $a(\xi_1)\otimes \bOne_{23}\otimes a(r(\xi_1))$, $a(r(\xi_2))\otimes
  a(\xi_2)\otimes \bOne_{31}$, and $\bOne_{12}\otimes
  a(r(\xi_3))\otimes a(\xi_3)$, where $\xi_i$ is a chord in $\PMC_i$,
  occur in $a^{\x,\y}$.

  Any chord $\xi$ in a pointed matched circle $\PMC$ has a length
  $|\xi|\in\NN$. To prove that the remaining chord triples occur in the
  differential  we proceed by induction on
  $|\xi_{12}|+|\xi_{23}|+|\xi_{31}|$. The base case is the unique
  chord triple $(\xi_{12}^1,\xi_{23}^1,\xi_{31}^1)$ for which each
  $\xi_{ij}^1$ has length $1$. The corresponding domain in
  $\THD{\PMC_1}{\PMC_2}{\PMC_3}$ is a polygon (with 12
  sides). Consequently, for any compatible generators (generators whose  idempotents $I_{12}\otimes I_{23}\otimes I_{31}$ and $J_{12}\otimes J_{23}\otimes J_{31}$ satisfy $(I_{12}a(\xi_{12})J_{12})\otimes (I_{23}a(\xi_{23})J_{23})\otimes (I_{31}a(\xi_{31})J_{31})\neq 0$)
  this domain has a unique holomorphic representative. Thus,
  $a^{\x,\y}$ contains $a(\xi_{12}^1)\otimes a(\xi_{23}^1)\otimes
  a(\xi_{31}^1)$.

  \begin{figure}
    \centering
    \includegraphics{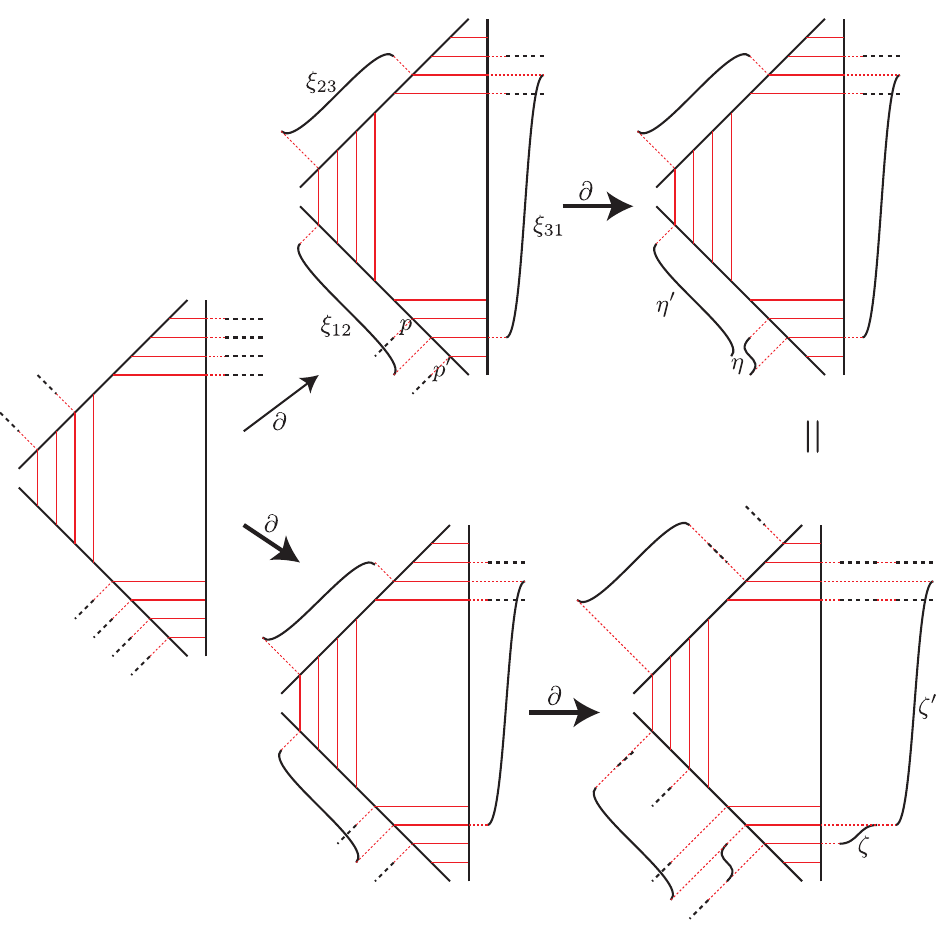}
    \caption{\textbf{Inductive argument used to prove
        Theorem~\ref{thm:compute-TDDD}.} The four thick arrows exist by
    induction (or, in one case, the definition of the
    algebra). Combined with $\bdy^2=0$, this forces the thin arrow to
    exist as well.}
    \label{fig:prove-thm-1}
  \end{figure}
  
  The rest of the argument is outlined in
  Figure~\ref{fig:prove-thm-1}.  Suppose that
  $(\xi_{12},\xi_{23},\xi_{31})$ is a chord triple with
  $|\xi_{12}|+|\xi_{23}|+|\xi_{31}|>3$. Without loss of generality,
  assume that $\xi_{12}$ has length greater than $1$. Then there is a
  point $p$ in either $\PMC_1$ or $\PMC_2$ so that:
  \begin{itemize}
  \item $p$ is in the interior of $\xi_{12}$ and
  \item $p$ is not matched to an endpoint of $\xi_{12}$.
  \end{itemize}
  For definiteness, suppose $p\in \PMC_1$. Let $p'$ be the point
  matched to $p$.

  Suppose that $\x$ (respectively $\y$) corresponds to the
  complementary idempotent triple $I=I_{12}\otimes I_{23}\otimes
  I_{31}$ (respectively $J=J_{12}\otimes J_{23}\otimes J_{31}$) and
  that $I\cdot (a(\xi_{12})\otimes a(\xi_{23})\otimes
  a(\xi_{31}))\cdot J\neq 0$. We must show that $a^{\x,\y}$ contains
  the term $I\cdot (a(\xi_{12})\otimes a(\xi_{23})\otimes
  a(\xi_{31}))\cdot J$.

  Either $I_{12}$ contains the matched pair $\{p,p'\}$ or $I_{31}$
  contains $\{p,p'\}$. For definiteness, suppose that $I_{12}$
  contains $\{p,p'\}$; the other case is similar.

  The element $I_{12}a(\xi_{12})J_{12}$ has a horizontal
  strand at $p$, which crosses the strand $\xi_{12}$. Let
  $\{\eta,\eta'\}$ denote the two chords obtained by smoothing this
  crossing, with $\eta\subset \PMC_1$ ending at $p$, and $\eta'$
  running from $p$ into $\PMC_2$. Write $\xi_{31}=\zeta\cup\zeta'$
  where $\zeta\cap\zeta' = p$ and $\zeta\subset\PMC_1$. Consider the algebra element
  \[
  I\cdot (a(\{\eta,\eta'\})\otimes a(\xi_{23})\otimes
  a(\xi_{31}))\cdot J.
  \]
  There are exactly two ways this element might occur in $\bdy^2\x$:
  \begin{align}
    \x &\stackrel{\bdy}{\longrightarrow} I\cdot (a(\eta')\otimes
    a(\xi_{23}) \otimes a(\zeta')) \otimes \z
    \stackrel{\bdy}{\longrightarrow} I\cdot (a(\eta')\otimes
    a(\xi_{23}) \otimes a(\zeta'))\cdot (a(\eta)\otimes
    \bOne_{23}\otimes a(\zeta))\otimes \y\label{eq:DDD-diff-exists}\\
    \x & \stackrel{\bdy?}{\longrightarrow} I\cdot (a(\xi_{12})\otimes
    a(\xi_{23})\otimes a(\xi_{31}))\otimes \y \stackrel{\bdy}{\longrightarrow} I\cdot (\bdy(a(\xi_{12}))\otimes
    a(\xi_{23})\otimes a(\xi_{31}))\otimes \y.\label{eq:DDD-diff-want}
  \end{align}
  (Here, the generator $\z$ is determined uniquely.) We want to show
  that the first arrow in Formula~\eqref{eq:DDD-diff-want} actually exists.
  In Formula~\eqref{eq:DDD-diff-exists}, both differentials use pairs of chords
  that we already proved contribute to the differential (the first by
  induction and the second using~\cite[Theorem 1]{LOT4}). Thus, 
  $a(\xi_{12})\otimes a(\xi_{23})\otimes a(\xi_{31})$
  occurs in the differential as well. This completes the proof.
\end{proof}

\begin{remark}
  It is immediate from the gradings that the isomorphism of
  Theorem~\ref{thm:compute-TDDD} is the only graded isomorphism
  between $\CFDDDa(\THD{\PMC_1}{\PMC_2}{\PMC_3})$ and
  $\TDDD(\PMC_1,\PMC_2,\PMC_3)$.
\end{remark}

\subsection{Computation of \texorpdfstring{$\TDDA$}{TDDA}}\label{sec:compute-TDDA}
\begin{proposition}\label{prop:TDDA-is-module}
  $\TDDA(\PMC_1,\PMC_2,\PMC_3)$ is a differential trimodule.
\end{proposition}
\begin{proof}
  This is immediate from the facts that
  \[
  \TDDA(\PMC_1,\PMC_2,\PMC_3)\cong \rvertme{\cornAA(\PMC_1,-\PMC_2)}{\cornAD(-\PMC_2,-\PMC_3)}{\RAlg(-\PMC_2)}{0pt}{0pt}
  \]
  (Proposition~\ref{prop:corn-gives-trimod}) and that $\cornAA$ and
  $\cornDA$ are well-defined module--$2$-modules
  (Proposition~\ref{prop:cornering-mods-defined}). 
  \end{proof}
\noindent (Even though
  Proposition~\ref{prop:corn-gives-trimod} appears later in the text,
  its proof does not depend on Proposition~\ref{prop:TDDA-is-module}.)

\begin{proposition}\label{prop:TDDA-DT-Id}
  $\TDDA(\PMC_1,\PMC_2,\PMC_3)\otimes_{\Alg(\PMC_{31})} \CFDDa(\Id_{\PMC_{31}})\cong
  \TDDD(\PMC_1,\PMC_2,\PMC_3).$
\end{proposition}
\begin{proof}
  Recall that $\CFDDa(\Id_{\PMC_{31}})$ has one generator
  $\x=\x_{I,I'}$ for each pair of complementary idempotents $I\otimes
  I'$ in $\Alg(-\PMC_{31})\otimes \Alg(\PMC_{31})$, and the
  differential on $\CFDDa(\Id_{\PMC_{31}})$ is given by 
  \[
  \bdy(\x_{I,I'})=\sum_{\substack{J\otimes J'\\\text{complementary}\\\text{idempotents}}}\sum_{\xi\in\Chord(\PMC_{31})} (I\otimes I')\cdot
  (a(r(\xi))\otimes a(\xi))\otimes \x_{J,J'}.
  \]
  So, the generators of $\TDDA(\PMC_1,\PMC_2,\PMC_3)\otimes_{\Alg(\PMC_{31})}
  \CFDDa(\Id_{\PMC_{31}})$ are in bijection with complementary
  idempotent triples, via the correspondence
  \begin{equation}\label{eq:idem-corresp}
  (I_{12}\otimes I_{23}\otimes I_{31})\otimes\x_{I_{31},I'_{31}} \longleftrightarrow
  (I_{12}\otimes I_{23}\otimes I'_{31})
  \end{equation}
  where $I'_{31}$ is the unique idempotent so that $I_{31}\otimes
  I'_{31}$ is a pair of complementary idempotents.

  The differential on $\TDDA(\PMC_1,\PMC_2,\PMC_3)\otimes_{\Alg(\PMC_{31})}
  \CFDDa(\Id_{\PMC_{31}})$ has four kinds of terms:
  \begin{itemize}
  \item Terms coming from the differential on
    $\TDDA(\PMC_1,\PMC_2,\PMC_3)$. These correspond exactly to the
    second sum in Formula~\eqref{eq:A}.
  \item Terms in the differential on $\CFDDa(\Id_{\PMC_{31}})$ in
    which the chord $\xi$ is entirely contained in $\PMC_1$. These
    correspond exactly to the first sum in Formula~\eqref{eq:A}.
  \item Terms in the differential on $\CFDDa(\Id_{\PMC_{31}})$ in
    which the chord $\xi$ is entirely contained in $\PMC_3$. These
    correspond exactly to the third sum in Formula~\eqref{eq:A}.
  \item Terms in the differential on $\CFDDa(\Id_{\PMC_{31}})$ in
    which the chord $\xi$ runs between $-\PMC_1$ and $\PMC_3$. With
    respect to the correspondence~\eqref{eq:idem-corresp}, such terms
    contribute terms $(I_{12}\otimes I_{23}\otimes I'_{31})\cdot(
    \chordsplit(\xi)\otimes a(r(\xi)))\otimes (J_{12}\otimes J_{23}\otimes
    J'_{31})$. But 
    \[
    \sum_{\xi_{12}\otimes\xi_{23}\in\chordsplit(\xi)}\hspace{-1.25em}
    a(\xi_{12})\otimes a(\xi_{23})\otimes a(r(\xi)) =
    \sum_{\substack{\text{chord triples}\\ (\xi_{12},\xi_{23},r(\xi))}}
    a(\xi_{12})\otimes a(\xi_{23})\otimes a(r(\xi)),
    \]
    so these terms correspond exactly to the fourth sum in
    Formula~\eqref{eq:A}.
  \end{itemize}
  Thus, the correspondence~\eqref{eq:idem-corresp} intertwines the
  differentials. This proves the result.
\end{proof}

\begin{proof}[Proof of Theorem~\ref{thm:compute-TDDA}]
  This follows from Theorem~\ref{thm:compute-TDDD},
  Proposition~\ref{prop:TDDA-is-module},
  Proposition~\ref{prop:TDDA-DT-Id} and the fact that tensoring with
  $\CFDDa(\Id)$ gives an equivalence of derived categories of
  $\Ainf$-trimodules (cf.~\cite[Section 9]{LOT2}). (Recall that any zig-zag of $\Ainf$ quasi-isomorphisms can be replaced by a single $\Ainf$ quasi-isomorphism; see, e.g.,~\cite[Section 2.4.1]{LOT2}.)
\end{proof}


\section{Cornered 2-modules for cornered
  Heegaard diagrams}
\label{sec:nice}
In this section, we define the $2$-modules associated to cornered
$3$-manifolds, and prove our two main theorems: invariance
(Theorem~\ref{thm:invariance}) and pairing (Theorem~\ref{thm:pairing}).

\subsection{Cornered Heegaard diagrams} \label{sec:CHD}
The cornered $2$-modules will be associated to cornered Heegaard
diagrams. Cornered Heegaard diagrams were introduced in~\cite[Definition 4.3]{DM:cornered}. They should not be confused with the split cornered Heegaard diagrams from \cite[Definition 4.5]{DM:cornered}, in which both $\alpha$- and
$\beta$-arcs intersect the boundary. In this paper we will use the ordinary notion of cornered Heegaard diagrams, with only
$\alpha$-arcs intersecting the boundary.

Recall that a \emph{bordered Heegaard diagram} consists of a surface
$\Sigma$ of some genus $g$, with connected boundary, a $g$-tuple of
pairwise-disjoint circles $\betas = \{\beta_1,\dots,\beta_g\}\subset\Sigma$, a
$(g-k)$-tuple of pairwise-disjoint circles
$\alphas^c=\{\alpha_1^c,\dots,\alpha_{g-k}^c\}\subset \Sigma$, a $2k$-tuple of arcs
$\alphas^a=\{\alpha_1^a,\dots,\alpha_{2k}^a\}\subset \Sigma$, and a basepoint $z\in
\bdy\Sigma$, disjoint from the $\alpha$-arcs. We require that the
$\alpha$-circles and $\alpha$-arcs are disjoint, and the
$\beta$-curves (respectively $\alpha$-curves) are linearly independent
in $H_1(\Sigma)$ (respectively $H_1(\Sigma,\bdy\Sigma)$). The boundary
of a Heegaard diagram $\widetilde{\HD}=(\Sigma,\alphas^c \cup \alphas^a,\betas,z)$ is a pointed
matched circle.

The following is a rephrasing of \cite[Definition 4.3]{DM:cornered}:
\begin{definition}\label{def:cornered-HD}
  Let $\PMC_0$ and $\PMC_1$ be matched intervals.  A \emph{cornered
    Heegaard diagram with boundary $\PMC_0$ and $\PMC_1$} consists of
  a bordered Heegaard diagram $\widetilde{\HD}$ together with an identification of
  $\bdy\HD$ with the pointed matched circle $\PMC_0\cup\PMC_1$.
\end{definition}
\noindent (Of course, if such an identification exists, it is unique.)

Given a cornered Heegaard diagram $\HD$ there is an underlying
bordered Heegaard diagram $\widetilde{\HD}$.

A \emph{bordered $3$-manifold} is a pair $(Y,\phi)$ where $Y$ is a
(compact, oriented) $3$-manifold with connected boundary and $\phi$ is
an orientation-preserving diffeomorphism from $F(\PMC)$ to $\bdy Y$,
for some pointed matched circle $\PMC$.  Recall that there is a
bordered $3$-manifold $(Y(\widetilde{\HD}),\phi\co
F(\PMC)\stackrel{\cong}{\to}\bdy Y(\widetilde{\HD}))$ associated to
any bordered Heegaard diagram $\widetilde{\HD}$~\cite[Construction
4.6]{LOT1}. We will abuse notation and write $Y(\HD)$ to denote the
pair $(Y(\HD),\phi)$.

Similarly, a \emph{cornered $3$-manifold} is a triple
$(Y,\phi_0,\phi_1)$ where
\begin{itemize}
\item $Y$ is a $3$-manifold with a connected, codimension-2 corner
  $C\cong S^1$. The corner $C$ divides $\bdy Y$ into two compact
  surfaces with boundary. Denote these surfaces $F_0$ and $F_1$.
\item $\phi_0\co F^\circ(\PMC_0)\stackrel{\cong}{\to} F_0$ and
  $\phi_1\co F_\circ(\PMC_1)\stackrel{\cong}{\to} F_1$ are
  diffeomorphisms, for some matched intervals $\PMC_0$ and $\PMC_1$.
\end{itemize}
We require that the diagram
\[
\xymatrix{
 & F_\circ(\PMC_1)\ar[dr]^{\phi_1} &  \\
S^1\ar[ur]\ar[dr]  & & C\\
 & F^\circ(\PMC_0)\ar[ur]^{\phi_0} &  \\
}
\]
induced by the parametrizations of $\bdy F^\circ(\PMC_0)$ and $\bdy
F_\circ(\PMC_1)$ commutes.

\begin{definition}
  Given a cornered Heegaard diagram $\HD$ with boundary $\PMC_0$ and
  $\PMC_1$, there is an associated cornered $3$-manifold obtained from
  $Y(\widetilde{\HD})$ by identifying $\bdy Y(\widetilde{\HD})$ with
  $F(\PMC_0)\cup_{S^1}F(\PMC_1)$ (as in Section~\ref{sec:glue-surf})
  and viewing the $S^1$ as a corner. Let $Y(\HD)$ denote this cornered
  $3$-manifold.
\end{definition}

\begin{lemma}
  Any cornered $3$-manifold is represented by some cornered Heegaard diagram.
\end{lemma}
\begin{proof}
  This follows from the fact that any bordered
  $3$-manifold is represented by some bordered Heegaard
  diagram~\cite[Lemma 4.9]{LOT1}; alternatively see~\cite[Section 4.3]{DM:cornered}.
\end{proof}

\begin{lemma}\label{lem:corn-HD-same}
  Let $\HD$ and $\HD'$ be cornered Heegaard diagrams, each with boundary
  $\PMC_0$ and $\PMC_1$. Then $\HD$ and $\HD'$ represent diffeomorphic
  cornered $3$-manifolds if and only if $\widetilde{\HD}$ and $\widetilde{\HD'}$
  represent diffeomorphic bordered $3$-manifolds.
\end{lemma}
\noindent  This is immediate from the definitions.

As our proof of invariance of the cornered $2$-modules will be
indirect, we will not need the following corollary. Still, it seems
worth recording:
\begin{corollary}
  With notation as in Lemma~\ref{lem:corn-HD-same}, $\HD$ and
  $\HD'$ represent diffeomorphic cornered $3$-manifolds if and only
  if they become diffeomorphic after a sequence of isotopies, rel
  boundary, of the set of $\alpha$- and $\beta$-curves (keeping the
  $\alpha$-curves (respectively $\beta$-curves) disjoint);
  handleslides among the $\beta$-circles or of an $\alpha$-arc or
  $\alpha$-circle over an $\alpha$-circle; and stabilizations.
\end{corollary}
\begin{proof}
  This is a consequence of Lemma~\ref{lem:corn-HD-same}
  and the corresponding statement for bordered $3$-manifolds,~\cite[Proposition 4.10]{LOT1}. 
\end{proof}

\subsection{Definition of the cornered 2-modules}
\begin{definition}\label{def:corn-2-mods}
  Fix a cornered Heegaard diagram $\HD$ with boundary $\PMC_0$ and
  $\PMC_1$, and let $\widetilde{\HD}$ be the associated bordered Heegaard
  diagram. Define
  \begin{align*}
    \CAA(\HD)&:=\CFAa(\widetilde{\HD})\DTP_{\Alg(\PMC_0\cup\PMC_1)} \cornAA(\PMC_0,\PMC_1)\\
    \CAD(\HD)&:=\CFAa(\widetilde{\HD})\DTP_{\Alg(\PMC_0\cup\PMC_1)} \cornAD(-\PMC_0,\PMC_1)\\
    \CDA(\HD)&:=\CFAa(\widetilde{\HD})\DTP_{\Alg(\PMC_0\cup\PMC_1)} \cornDA(\PMC_0,-\PMC_1)\\
    \CDD(\HD)&:=\CFAa(\widetilde{\HD})\DTP_{\Alg(\PMC_0\cup\PMC_1)} \cornDD(-\PMC_0,-\PMC_1),
  \end{align*}
  where $\DTP$ denotes the derived tensor product. 

  In other words, let $P$ be a projective differential module over
  $\Alg(\PMC_0\cup\PMC_1)$ that is quasi-isomorphic
  to $\CFAa(\widetilde{\HD})$, and let
  \begin{align*}
    \CAA(\HD)&=P\otimes_{\Alg(\PMC_0\cup\PMC_1)} \cornAA(\PMC_0,\PMC_1)\\
    \CAD(\HD)&=P\otimes_{\Alg(\PMC_0\cup\PMC_1)} \cornAD(-\PMC_0,\PMC_1)\\
    \CDA(\HD)&=P\otimes_{\Alg(\PMC_0\cup\PMC_1)} \cornDA(\PMC_0,-\PMC_1)\\
    \CDD(\HD)&=P\otimes_{\Alg(\PMC_0\cup\PMC_1)} \cornDD(-\PMC_0,-\PMC_1).
  \end{align*}
\end{definition}

\begin{lemma}
  The modules $\CAA(\HD)$, $\CAD(\HD)$, $\CDA(\HD)$ and $\CDD(\HD)$
  satisfy the horizontal and vertical motility hypotheses
  (Definition~\ref{def:2m-mot}).
\end{lemma}
\begin{proof}
  This follows from the fact that $\cornAA(\PMC_0,\PMC_1)$,
  $\cornDD(\PMC_0,\PMC_1)$, and the \DD\ identity module-bimodule
  $\DDa(\Id_{\PMC})$ used in defining $\cornAD$ and $\cornDA$ satisfy the motility hypotheses, as can be verified directly.
\end{proof}

\subsection{Tensor products of cornering module--2-modules}
In this section we prove a workhorse proposition, from which
Theorem~\ref{thm:pairing} will follow. Fix matched intervals $\PMC_1, \PMC_2, \PMC_3$ as in Section~\ref{sec:trimods}.
\begin{proposition}\label{prop:corn-gives-trimod}
  There are isomorphisms
  \begin{align*}
    \rvertme{\cornAA(\PMC_1,-\PMC_2)}{\cornAD(-\PMC_2,-\PMC_3)}{\RAlg(-\PMC_2)}{0pt}{0pt}&\hspace{-3.5ex}\cong\hspace{1ex}
    \TDDA(\PMC_1,\PMC_2,\PMC_3)\\
    \shortintertext{and}
    \rvertme{\cornDA(\PMC_2,\PMC_3)}{\cornDD(-\PMC_1,\PMC_2)}{\LAlg(\PMC_2)}{0pt}{0pt}&\hspace{-3.5ex}\cong\hspace{1ex} \TDDD(\PMC_1,\PMC_2,\PMC_3).
  \end{align*}
\end{proposition}
\begin{proof}
  We start with the tensor product of $\cornAA$ and $\cornAD$.

  On the level of vector spaces, the isomorphism is given as follows.
  Basic elements of the tensor product have the form $%
  \pbmat x\\ y\\z\pemat $%
  where $$
  x  \in \putaround{\rot{\cornAA(\PMC_2,-\PMC_3)}}{\sbull}{}{}{0} \ , \
  y  \in \putaround{\DDa(\Id_{-\PMC_2})}{\sbull}{}{\sbull}{\substack{0\\0}} \ , \
  z  \in \putaround{\cornAA(\PMC_1,-\PMC_2)}{}{}{\sbull}{0}.
  $$
   Since
  the tensor products are over
  $\putaround{\rot{\TAlg(\PMC_2)}}{\sbull}{}{\sbull}{0}$ and
  $\putaround{\RAlg(-\PMC_2)}{\sbull}{}{\sbull}{0}$, we can absorb most of $y$
  into $x$ and $z$ and assume that $y\in X$. Then, the element $%
  \pbmat x\\ y\\z\pemat $%
  is non-zero only when 
  \begin{align*}
  x & \in \putaround{\rot{\TAlg(\PMC_2)}\obent
    \BAlg(-\PMC_3)}{0}{0}{0}{0}=\rot{\TAlg(\PMC_2)}\osmooth
    \BAlg(-\PMC_3)=\Alg(\PMC_2\cup(-\PMC_3)), \\
    z & \in\putaround{\TAlg(\PMC_1)\obent
    \RAlg(-\PMC_2)}{0}{0}{0}{0}=\TAlg(\PMC_1)\osmooth
    \RAlg(-\PMC_2)=\Alg(\PMC_1\cup(-\PMC_2)),
    \end{align*}
     and the top idempotent $I$
  of $z$ is complementary to the
  bottom idempotent $J$ of $x$. These elements correspond exactly to
  the basic elements of $\TDDA(\PMC_1,\PMC_2,\PMC_3)$
  (Definition~\ref{def:module-TDDA}).
  
  It is clear that this isomorphism intertwines the actions of
  $\Alg(\PMC_{12})$ and intertwines the actions of
  $\Alg(\PMC_{23})$.  It follows from the form of the
  differential on $\cornAD(\PMC_2,\PMC_3)$ that the isomorphism
  intertwines the differentials. (See
  also point~\ref{item:cornAD-differential} in
  Observation~\ref{obs:AD}.)  It remains to see that the
  isomorphism respects the action of $\Alg(-\PMC_{31})$.

  By Lemma~\ref{lem:DDA-factor}, to prove that the isomorphism respect
  the action of $\Alg(-\PMC_{31})=\Alg(\PMC_1 \cup (-\PMC_3))$ it suffices to check three kinds of
  elements:
  \begin{itemize}
  \item Elements of $\Alg(\PMC_1)$.
  \item Elements of $\Alg(-\PMC_3)$.
  \item Elements of the form $a(\xi)$ where the initial endpoint of $
    \xi$ lies in $\PMC_1$ and the terminal endpoint of $\xi$ lies in
    $-\PMC_3$.
  \end{itemize}
  The first two cases are clear---see
  point~\ref{item:cornAD-easy-actions} in Observation~\ref{obs:AD} and
  the definition of $\cornAA$, respectively. For the third, write
  $a(\xi)=\pbmat \adown(\xi_1)\\a(\xi_2)\\ \aup(\xi_3)\pemat$ where
  $\xi_1\in\putaround{\BAlg(-\PMC_3)}{1}{0}{}{0}$, $\xi_2$ is a single
  vertical strand in the vertical barbell algebra, and $\xi_3\in
  \putaround{\TAlg(\PMC_1)}{}{0}{1}{0}$. Then 
  \[
  \pbmat
  x\\y\\z
  \pemat
  \pbmat
  \adown(\xi_1)\\a(\xi_2)\\ \aup(\xi_3)
  \pemat
  =
  \bmat
  (x \adown(\xi_1))\\
  (ya(\xi_2))\\
  (z \aup(\xi_3))
  \emat
  =
  \sum_{\substack{\eta\in\Chord_1(\PMC_2)\\y'\in X}}
  \bmat
  (x \adown(\xi_1))\\
  \rot{\aleft(\eta)}\\
  (y')\\
  \adown(\eta)\\
  (z \aup (\xi_3))
  \emat
  =
  \sum_{\substack{\eta\in\Chord_1(\PMC_2)\\y'\in X}}
  \bmat
  (xa(\xi_1\cup \rot{\eta}))\\
  (y')\\
  (za(\eta\cup \xi_3))
  \emat
  \]
  (where we are borrowing notation from
  Definition~\ref{def:DD-id}). But this is exactly the action on
  $\TDDA(\PMC_1,\PMC_2,\PMC_3)$ as specified in
  Definition~\ref{def:module-TDDA}.  See
  Figure~\ref{fig:act-on-AA-AD}.

  \begin{figure}
    \centering
    \includegraphics[scale = .5]{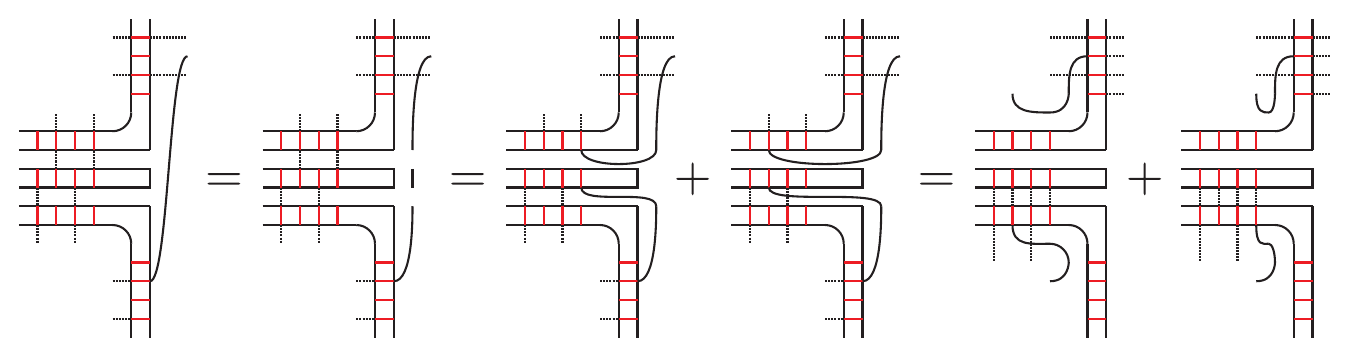}
    \caption{\textbf{Acting on $\cornAA\vtp\cornAD$ by a chord running from $\PMC_1$ to $-\PMC_3$.}}
    \label{fig:act-on-AA-AD}
  \end{figure}

  The argument for
  $\rsmallvertme{\cornDA(\PMC_2,\PMC_3)}{\cornDD(-\PMC_1,\PMC_2)}{\LAlg(\PMC_2)}{0pt}{0pt}$
  is similar. Basic elements of the tensor product have the form
  \[
  \pbmat
  & x & w\\
  (z & y)
  \pemat
  \in
  \pbmat
  & \putaround{\rotsp{\rot{\cornAA(\PMC_1,-\PMC_2)}}^\op}{\sbull}{0}{}{} &
  \CFDDa(\Id_{(-\PMC_1)\cup\PMC_2})\\[10pt]
  \bigl(\putaround{\rotcc{\DDa(\Id_{\PMC_3})}}{}{0}{{\sbull\ 0}}{\sbull} &
  \putaround{\rotcc{\cornAA(\PMC_2,-\PMC_3)}}{}{\sbull}{0}{}\bigr) &
  \pemat.
  \]
  As in the previous case, we can assume that $w$ has the form
  $I\otimes a$ where $I$ is an idempotent and $a\in
  \Alg(\PMC_1\cup(-\PMC_2))$, and that $z$ has the form $b\otimes J$ where
  $b\in \TAlg(\PMC_3)$ and $J$ is an idempotent. It follows that $y\in
  \rotcc{\TAlg(\PMC_2)\osmooth
  \RAlg(-\PMC_3)}=\Alg(\PMC_2\cup(-\PMC_3))$. Without loss of generality, we can assume that $x$ is a basic element of $\putaround{\rotsp{\rot{\cornAA(\PMC_1,-\PMC_2)}}^\op}{\sbull}{0}{}{}$
and write
  $x=x_b\obent x_l$ where $x_b\in
  \putaround{\BAlg(-\PMC_1)}{\sbull}{0}{}{0}$ and $x_l\in
  \putaround{\LAlg(\PMC_2)}{\sbull}{0}{\sbull}{}$. We can absorb the element $x_l$
  into $\rotcc{\DDa(\Id_{\PMC_3})}\htp^h \rotcc{\cornAA(\PMC_2,-\PMC_3)}$, and
  so assume that $x_l$ is an idempotent. Then $$\pbmat x_b\\ b
  \pemat\in
  \rsmallvertme{\TAlg(\PMC_3)}{\BAlg(-\PMC_1)}{\Dnil}{0pt}{0pt} \cong \Alg(\PMC_3\cup(-\PMC_1)).$$ Tracing
  through the idempotents, $\left(a, y, \pbmat x_b\\ b
  \pemat\right)$ is an element of $\TDDD(\PMC_1,\PMC_2,\PMC_3)$. On
the level of vector spaces, this gives the identification between the
two trimodules.

  It is clear that the identification respects the algebra
  actions. It remains to see that it intertwines the
  differentials. The differential on $\TDDD$ comes from five places:
  \begin{enumerate}[label=(Tri-\arabic*),ref=Tri-\arabic*]
  \item\label{item:tri-1} The differentials on the algebras $\Alg(\PMC_{12})$,
    $\Alg(\PMC_{23})$, and $\Alg(\PMC_{31})$.
  \item\label{item:tri-2} Chords in $\PMC_1$.
  \item\label{item:tri-3} Chords in $\PMC_2$.
  \item\label{item:tri-4} Chords in $\PMC_3$.
  \item\label{item:tri-5} Chord triples.
  \end{enumerate}

  The differential on
  $\rsmallvertme{\cornDA(\PMC_1,\PMC_2)}{\cornDD(-\PMC_2,\PMC_3)}{\LAlg(\PMC_2)}{0pt}{0pt}$
  comes from three places:
  \begin{enumerate}[label=(Tens-\arabic*),ref=Tens-\arabic*]
  \item\label{item:tens-1} The differentials on the algebras $\Alg(\PMC_{12})$,
    $\Alg(\PMC_{23})$ and $\Alg(\PMC_{31})$.
  \item\label{item:tens-2} Chords in $\PMC_3$, which contribute to the differential on
    $\rotcc{\DDa(\Id_{\PMC_3})}$.
  \item\label{item:tens-3} Chords in $(-\PMC_1)\cup\PMC_2$ which contribute to the
    differential on $\CFDDa(\Id_{(-\PMC_1)\cup\PMC_2})$. These chords
    come in three kinds:
    \begin{enumerate}
    \item\label{item:tens-3a} Chords entirely contained in $-\PMC_1$.
    \item\label{item:tens-3b} Chords entirely contained in $\PMC_2$.
    \item\label{item:tens-3c} Chords which run between $-\PMC_1$ and $\PMC_2$.
    \end{enumerate}
  \end{enumerate}

  It is easy to see that the contributions~(\ref{item:tri-1})
  and~(\ref{item:tens-1}) correspond, as do the contributions~(\ref{item:tri-2})
  and~(\ref{item:tens-3a}),~(\ref{item:tri-3}) and~(\ref{item:tens-3b}),
  and~(\ref{item:tri-4}) and~(\ref{item:tens-2}). The
  contributions~(\ref{item:tri-5}) and~(\ref{item:tens-3c}) are
  more interesting, so we spell out their correspondence in more
  detail. (See also Figure~\ref{fig:DA_DD_diff}.)

  \begin{figure}
    \centering
    \includegraphics[scale = .5]{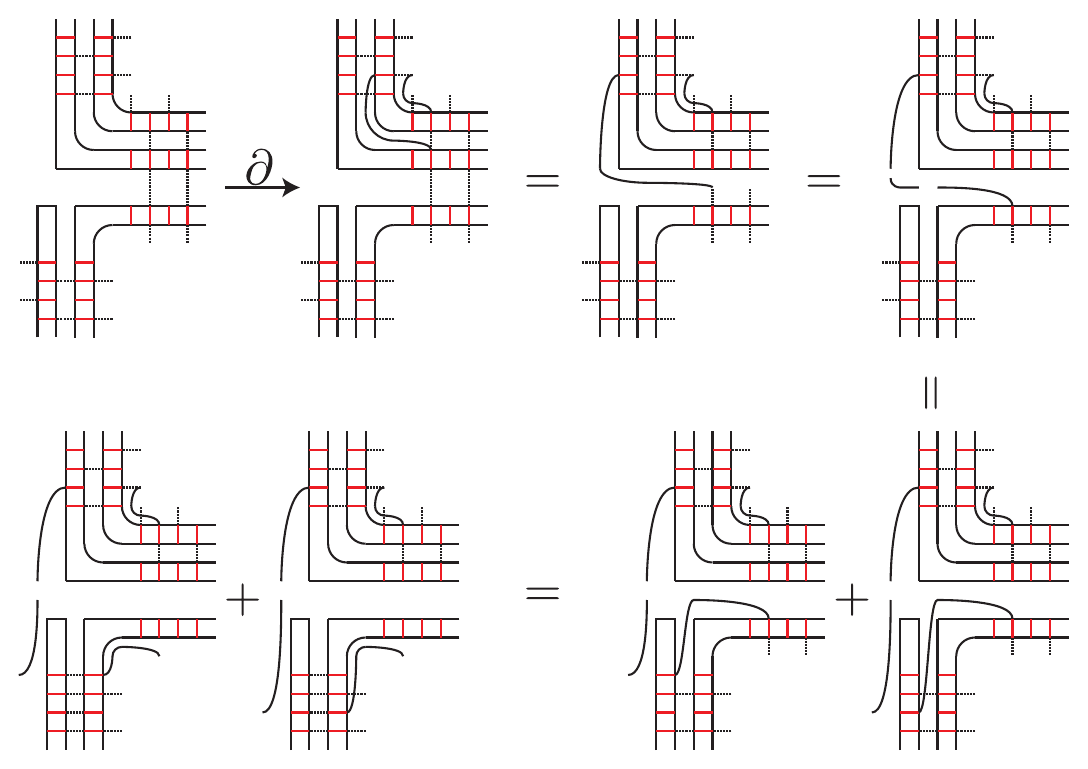}
    \caption{\textbf{The differential on $\cornDA\vtp\cornDD$
        contributing chord triples.} Only two terms in the differential
    are shown.}
    \label{fig:DA_DD_diff}
  \end{figure}

  So, consider a chord $\xi$ running between $-\PMC_1$ and $\PMC_2$,
  contributing $a(\xi)\otimes a(-\xi)$ to the differential on
  $\CFDDa(\Id_{(-\PMC_1)\cup\PMC_2})$. Move the first $a(\xi)$ through
  the tensor product to view it as an element of
  $\rotsp{\rot{\cornAA(\PMC_1,-\PMC_2)}}^\op=\putaround{\BAlg(\PMC_1)\obent
    \LAlg(\PMC_2)}{\sbull}{\sbull}{0}{0}$.
  Write $a(\xi)=\adown(\xi_1)\obent \aup(\xi_2)$ where $\xi_1$
  (respectively $\xi_2$) is a chord in $-\PMC_1$ (respectively
  $\PMC_2$) running off the end.

  Now, pass $\aup(\xi_2)$ through the vertical tensor product with
  $\rotcc{\DDa(\Id_{\PMC_3})}\htp \rotcc{\cornAA(\PMC_2,-\PMC_3)}$. To see how
  $\aup(\xi_2)$ acts on $\rotcc{\DDa(\Id_{\PMC_1})}\htp \rotcc{ \cornAA(\PMC_1,\PMC_2)}$, write $\aup(\xi_2)=c\htp \aleft(\xi_2)$,
  where $c\in\putaround{(\BBh)}{0\ 0}{0}{1\ 0}{1}$ (and
  $\aleft(\xi_2)\in \putaround{\LAlg(\PMC_2)}{0}{1}{0}{0}$). The
  element $c$ acts on $\rotcc{\DDa(\Id_{\PMC_3})}$ to give the sum of
  chords 
  \[
  \sum_{\eta\in\Chord_0(\PMC_3)}\aup(\eta)\otimes \aright(\eta).
  \]
  Now, move the terms $\aright(\eta)$ through the horizontal tensor
  product with $\rotcc{\cornAA(\PMC_2,-\PMC_3)}$. Each combines with
  $\aleft(\xi_2)$ to produce a chord $a(\eta\cup\xi_2)\in
  \Alg(\PMC_2\cup(-\PMC_3))$. The
  chords $\aup(\eta)$ combine with the chord $\adown(\xi_1)$ to give
  chords $a(\eta\cup \xi_1)\in
  \rsmallvertme{\TAlg(\PMC_3)}{\BAlg(-\PMC_1)}{\Dnil}{0pt}{0pt}=
  \Alg(\PMC_3\cup(-\PMC_1))$. Thus, we have a chord triple $(a(\xi),
  a(\eta\cup\xi_2), a(\eta\cup\xi_1))$. It is clear that every chord
  triple arises uniquely in this way.
\end{proof}

\subsection{Proofs of the invariance and gluing theorems}
\begin{proof}[Proof of Theorem~\ref{thm:invariance}]
  The $2$-modules are defined in Definition~\ref{def:corn-2-mods}.
  It remains to verify that they are well-defined, with respect to changes of the Heegaard diagram, up to
  quasi-isomorphism. We prove this for $\CAA$; the other cases are
  essentially the same. By definition,
  \begin{align*}
    \CAA(\HD_{00}^1)&= \CFAa(\widetilde{\HD}_{00}^1)\DTP_{\Alg(\PMC_0\cup\PMC'_0)}\cornAA(\PMC_0,\PMC'_0)\\
    \CAA(\HD_{00}^2)&= \CFAa(\widetilde{\HD}_{00}^2)\DTP_{\Alg(\PMC_0\cup\PMC'_0)}\cornAA(\PMC_0,\PMC'_0).
  \end{align*}
  By Lemma~\ref{lem:corn-HD-same}, the bordered Heegaard diagrams
  $\widetilde{\HD}_{00}^1$ and $\widetilde{\HD}_{00}^2$ represent the
  same bordered $3$-manifold. By the invariance theorem for
  bordered Floer homology,~\cite[Theorem 1.2]{LOT1},
  $\CFAa(\widetilde{\HD}_{00}^1)$ is quasi-isomorphic to
  $\CFAa(\widetilde{\HD}_{00}^2)$. Since the derived tensor product
  respects quasi-isomorphism of differential modules, $\CAA(\HD_{00}^1)$ is
  quasi-isomorphic to $\CAA(\HD_{00}^2)$ as differential modules over
  $\TAlg(\PMC_0)\obent \RAlg(\PMC'_0)$ and hence, by
  Proposition~\ref{prop:bent-der-cat}, $\CAA(\HD_{00}^1)$ is
  quasi-isomorphic to $\CAA(\HD_{00}^2)$ as $2$-modules.
\end{proof}

Note that the cornered modules appearing in Theorem~\ref{thm:pairing}
are only well-defined up to quasi-isomorphism
(cf.~Theorem~\ref{thm:invariance}). So, for Theorem~\ref{thm:pairing}
to make sense we need to know that the tensor products appearing in it
respect quasi-isomorphism; this is the content of the next proposition.

\begin{proposition}\label{prop:D-proj} With notation as in Theorem~\ref{thm:pairing},
  the module $\CAD(\HD_{01})$ is projective over $\RAlgD(\PMC_0')$; the
  module $\CDD(\HD_{11})$ is projective over $\LAlgD(\PMC_1')$; the module
  $\CDA(\HD_{10})$ is projective over $\TAlgD(\PMC_0)$; and the module
  $\CDD(\HD_{11})$ is projective over $\BAlgD(\PMC_1)$. 
\end{proposition}
\begin{proof}
  We focus on the case of $\CAD(\HD_{01})$: the other cases are
  similar. We have
  \begin{align*}
    \CAD(\HD_{01})&= \CFAa(\widetilde{\HD}_{01})\DTP_{\Alg(-\PMC_0'\cup
      \PMC_1)}\cornAD(\PMC'_0,\PMC_1)\\
    &\cong \CFAa(\widetilde{\HD}_{01})\DTP_{\Alg(-\PMC_0'\cup
      \PMC_1)}\left(\rvertme{\DDa(\Id_{\PMC_0'})}{\rot{\cornAA(-\PMC'_0,\PMC_1)}}{\rotsp{\TAlg(-\PMC_0')}}{0pt}{0pt}\hspace{-1cm} \right)\\
    &\cong \rvertme{\DDa(\Id_{\PMC_0'})}{\left(\CFAa(\widetilde{\HD}_{01})\DTP_{\Alg(-\PMC_0'\cup
      \PMC_1)}\rot{\cornAA(-\PMC'_0,\PMC_1)}\right)}{\rotsp{\TAlg(-\PMC_0])}}{0pt}{0pt}.
  \end{align*}
  By Proposition~\ref{prop:interval-dd-id-proj}, $\DDa(\Id_{\PMC_0'})$ is
  biprojective, so the first part of the proposition follows from
  Proposition~\ref{prop:biproj-tensor}.  The second part of the
  proposition follows from the first part and
  Lemma~\ref{lem:tensor-resp-q-i}.
\end{proof}

\begin{proof}[Proof of Theorem~\ref{thm:pairing}]
  We will prove Formula~\eqref{eq:pair0}; the other cases are
  similar. We have
  \begin{align*}
  \left(\rvertme{\CAA(Y_{00})}{\CAD(Y_{01})}{\RAlgD(F_0')}{0pt}{0pt} \hspace{-1cm} \right)
  &\cong
  \rvertme{\left(\CFAa(\widetilde{Y}_{00})\DTP_{\Alg(F_0\cup F'_0)}\cornAA(F_0,F'_0)\right)}{\left(\CFAa(\widetilde{Y}_{01})\DTP_{\Alg(-F'_0\cup F_1)}\cornAD(F'_0,F_1)\right)}{\RAlgD(F_0')}{0pt}{0pt}\\
  &\cong 
  \CFAa(\widetilde{Y}_{00})\DTP_{\Alg(F_0\cup
    F'_0)}\left(\CFAa(\widetilde{Y}_{01})\DTP_{\Alg(-F'_0\cup
      F_1)}\left(\rvertme{\cornAA(F_0,F'_0)}{\cornAD(F'_0,F_1)}{\RAlgD(F_0')}{0pt}{0pt} \hspace{-1cm} \right)\right)\\
  &\cong  \CFAa(\widetilde{Y}_{00})\DTP_{\Alg(F_0\cup
    F'_0)}\left(\CFAa(\widetilde{Y}_{01})\DTP_{\Alg(-F'_0\cup
      F_1)}\TDDA(-F_0,F'_0,F_1)\right)\\
  &\cong  \CFAa(\widetilde{Y}_{00})\otimes_{\Alg(F_0\cup
    F'_0)}\left(\CFAa(\widetilde{Y}_{01})\otimes_{\Alg(-F'_0\cup
      F_1)}\TDDA(-F_0,F'_0,F_1)\right)\\
  &\simeq \CFAa(Y_{00}\cup_{F'_0}Y_{01}),
  \end{align*}
  where the first isomorphism simply expands the definitions of $\CAA$
  and $\CAD$;
  the second isomorphism is just reparenthesizing; the third uses
  Proposition~\ref{prop:corn-gives-trimod}; the fourth uses the fact
  that $\TDDA$, as a type $\mathit{DDA}$ trimodule, is projective over
  $\Alg(F_0\cup F'_0)$ and $\Alg(-F'_0\cup F_1)$ (see~\cite[Corollary
2.3.25]{LOT2}), so the tensor product
  and derived tensor product agree; and the last uses
  Corollary~\ref{cor:corn-pair-via-bord} and
  Theorem~\ref{thm:compute-TDDA}. (Note that we are abusing notation, to be consistent with the discussion in the introduction: the $F$'s are really matched intervals, not surfaces, and the $Y$'s are really Heegaard diagrams, not $3$-manifolds.)
\end{proof}


\section{Gradings}
\label{sec:gradings}
In this section we explain how to add gradings to our various cornered invariants.

\subsection{Noncommutative gradings}
We first review the definitions of noncommutative gradings on algebras and modules, following 
 \cite[Section 2.5]{LOT1}. 
 
A vector space $V$ is said to be graded by a set $S$ if it is equipped with a direct sum decomposition $V = \oplus_{s \in S} V_s$. The elements in each $V_s$ are called homogeneous, and for $v \in V_s$ we write $\gr(v) = s$. (In particular, we consider the relation $\gr(0)=s$ to be true for all $s$.)
 
  Let $G$ be a group and $\lambda \in G$ a distinguished element in the center of $G$.  A {\em $(G, \lambda)$-graded differential algebra}  is defined to be a differential algebra $(\Alg, \del) $ together with a grading of $\Alg$ by the set $G$, such that, for any homogeneous elements $a, b \in \Alg$, we have
 $$ \gr(ab) = \gr(a) \gr(b) \ \ \ \text{and} \ \ \ \gr(\del a) = \lambda^{-1}\gr(a).$$
 (When $G=\ZZ$ and $\lambda=1$ we recover the usual notion of a $\ZZ$-grading.)

Let $\Alg$ be a $(G,\lambda)$-graded differential algebra and let $S$ be a set with a right $G$-action. A right differential $\Alg$-module $M$ is said to be {\em graded by $S$} if it is equipped  
with a grading $\gr$ by the set $S$, such that, for any homogeneous elements $a \in \Alg$ and $x \in M$, we have %
$$ \gr(xa) = \gr(x) \gr(a) \ \ \ \text{and} \ \ \ \gr(\del x) = \gr(x)\lambda^{-1}.$$
There is an obvious analog of this definition for left modules, graded by left $G$-sets. Also, there are similar notions of gradings on $\Ainf$-algebras and $\Ainf$-modules; see \cite[Section 2.5]{LOT1}. 

Let $S$ be a right $G$-set and let $T$ be a left $G$-set. As usual, we denote the balanced product by 
$$ S \times_G T := (S \times T)/(s,gt) \sim (sg,t)\text{ for any } s\in S, t \in T,g \in G.$$
Given a differential algebra $\Alg$ graded by $(G, \lambda)$, a right $\Alg$-module $M$ graded by $S$, and a left $\Alg$-module $N$ graded by $T$, the tensor product $M \otimes_{\Alg} N$ is a chain complex with a grading by $S \times_G T$, a set with a $\ZZ$-action (by multiplication by $\lambda$).

We now describe gradings on $2$-algebras, algebra-modules, and 2-modules.  When $G$ is a commutative group with a distinguished element $\lambda$, a {\em $(G,\lambda)$-grading} on a 2-algebra $\A$ is a grading of each vector space $\Aa{m}{n}{p}{q}$ by the set $G$ such that, for homogeneous elements $a, b \in \A$, we have
 $$ \gr \pbs b \\ a \pes = \gr(a) \gr(b) = \gr(ab) \ \ \ \text{and} \ \ \ \gr(\del a) = \lambda^{-1}\gr(a).$$
We will only be concerned with $(\ZZ,1)$-graded 2-algebras (which we refer to simply as $\ZZ$-graded 2-algebras), and therefore restrict attention to that case.

Suppose $\A$ is a $\ZZ$-graded 2-algebra, and $\TAlg$ is a top algebra-module over $\A$.  For any noncommutative group $G$ with a distinguished central element $\lambda$, a {\em $(G,\lambda)$-grading} on $\TAlg$ is a grading of each vector space $\TAlga{n}{m}{q}$ by the set $G$ such that, for homogeneous elements $a \in \A$, and $\phi,\psi \in \TAlg$, we have
 $$ \gr \pbs a \\ \phi \pes = \lambda^{\gr(a)} \gr(\phi) \ \text{,} \ \gr(\phi\psi) = \gr(\phi) \gr(\psi) \ \text{, and} \ \gr(\del \phi) = \lambda^{-1}\gr(\phi).$$
(Compare \cite[Definition 3.2]{DM:cornered}.)  The definition of a grading on a bottom algebra-module is similar.

A right algebra-module $\RAlg$ over a $\ZZ$-graded 2-algebra $\A$ can be graded, not by an ordinary (horizontal) group, but by a vertical group.  A vertical group is the same structure as a group, but with the multiplication written vertically.  (Of course, one can associate an ordinary group to any vertical group, but this association is not canonical, as the vertical multiplication can equally well be rotated ninety-degrees clockwise or counterclockwise.)  For any noncommutative vertical group $H$ with a distinguished central element $\mu$, an {\em $(H,\mu)$-grading} on $\RAlg$ is a grading of each vector space $\RAlga{m}{p}{n}$ by the set $H$ such that, for homogeneous elements $a \in \A$, and $\phi, \psi \in \RAlg$, we have
$$ \gr(\phi a) = \mu^{\gr(a)} \gr(\phi) \ \text{,} \ \gr \pbs \psi \\ \phi \pes = \bmat \gr(\psi) \\ \gr(\phi) \emat \ \text{, and} \ \gr(\del \phi) = \mu^{-1}\gr(\phi).$$
The definition of a grading on a left algebra-module is similar.

As before, assume the 2-algebra $\A$ is $\ZZ$-graded.  Suppose the top algebra-module $\TAlg$ is $(G_1,\lambda_1)$-graded, and the bottom algebra-module $\BAlg$ is $(G_2,\lambda_2)$-graded.  Recall that the {\em amalgamated direct product} of $G_1$ and $G_2$ is, as in \cite[Definition 2.5.9]{LOT2},
$$ G_1 \times_\ZZ G_2 := G_1 \times G_2/ \langle \lambda_1 \lambda_2^{-1} \rangle.$$
The vertical tensor product of $\TAlg$ and $\BAlg$ inherits a grading by this amalgamated product---this is true for both the full and restricted tensor products, but will we only need the restricted case:
\begin {lemma}\cite[Lemma 3.3]{DM:cornered}
Suppose $\A$ is a differential $\ZZ$-graded $2$-algebra.  If we have a $(G_1, \lambda_1)$-graded differential top algebra-module $\TAlg$ over $\A$ and a  $(G_2, \lambda_2)$-graded bottom algebra-module $\BAlg$ over $\A$, then the restricted vertical tensor product of $\TAlg$ and $\BAlg$ has an induced structure of a $(G, \lambda)$-graded differential algebra, where $G = G_1 \times_\ZZ G_2$ and $\lambda = [\lambda_1]=[\lambda_2] \in G.$ 
\end {lemma}
If $H_1$ and $H_2$ are vertical groups with distinguished central elements $\mu_1$ and $\mu_2$, respectively, the amalgamated direct product $H_1 \times_\ZZ H_2$, defined exactly as for ordinary groups, is a vertical group.  The horizontal tensor product of an $(H_1, \mu_1)$-graded right algebra-module $\RAlg$ and an $(H_2, \mu_2)$-graded left algebra-module $\LAlg$ is naturally a $H_1 \times_\ZZ H_2$-graded vertical algebra.

Let $\A$ be a $\ZZ$-graded 2-algebra, and let $\TAlg$, $\BAlg$, $\RAlg$, and $\LAlg$ be top, bottom, right, and left algebra-modules over $\A$, graded respectively by the groups $(G_1, \lambda_1)$ and $(G_2,\lambda_2)$ and the vertical groups $(H_1,\mu_1)$ and $(H_2,\mu_2)$.  One would expect a top-right 2-module over $\TAlg$ and $\RAlg$ to admit a grading by a set with an action of the amalgamated product $H_1 \times_\ZZ G_1$. 
Unfortunately this product is neither a group nor a vertical group, but a more complicated structure with four distinct associative operations.  However, for our particular purposes, we can artificially simplify the situation by making particular choices of rotations to ensure all the amalgamated products that arise are ordinary groups.  

As before, let $\TAlg$ be a $(G_1,\lambda_1)$-graded top algebra-module and let $\RAlg$ be an $(H_1,\mu_1)$-graded right algebra-module; let $S$ be a set with commuting right actions by $G_1$ and by $\rot{H_1}$ such that the actions of $\lambda_1$ and $\mu_1$ agree---in other words, $S$ is a set with a right $\rot{H_1} \times_\ZZ \,G_1$ action.  Here $\rot{H_1}$ is $H_1$ as a set, but with bottom--top multiplication reinterpreted as left--right multiplication.  We will refer to the distinguished central element of the product $\rot{H_1} \times_\ZZ \,G_1$ as $\kappa$.  We define an {\em $S$-graded top-right 2-module $\TR$} over $\TAlg$ and $\RAlg$, to be a 2-module with a grading of each vector space $\TRa{m}{n}$ by the set $S$, such that, for homogeneous elements $\x \in \TR$, $\phi \in \TAlg$, and $\psi \in \RAlg$, we have
$$ \gr(\x \phi) = \gr(\x) \gr(\phi) \ \text{,} \ \gr \pbs \psi \\ \x \pes = \gr(\x) \gr(\psi) \ \text{, and} \ \gr(\del \x) = \gr(\x) \kappa^{-1}.$$
The definitions of gradings on the other types of 2-modules are similar.  Specifically: 
\begin{itemize}
\item for a bottom-right 2-module $\BR$, the grading is by a set $S$ with a right $\rotcc{H_1} \times_\ZZ \,G_2$ action, such that
$\gr(\x \phi) = \gr(\x) \gr(\phi)$ and $\gr \pbs \x \\ \psi \pes = \gr(\x) \gr(\psi)$;  
\item for a top-left 2-module $\TL$, the grading is by a set $S$ with a left $G_1 \times_\ZZ \rotcc{H_2}$ action, such that
$\gr(\phi \x) = \gr(\phi) \gr(\x)$ and $\gr \pbs \psi \\ \x \pes = \gr(\psi) \gr(\x)$; and  
\item for a bottom-left 2-module $\BL$, the grading is by a set $S$ with a left $G_2 \times_\ZZ \rot{H_2}$ action, such that
$\gr(\phi \x) = \gr(\phi) \gr(\x)$ and $\gr \pbs \x \\ \psi \pes = \gr(\psi) \gr(\x)$.
\end{itemize}

Equipped with these notions, we observe that tensor products of graded 2-modules inherit gradings by balanced products:
\begin{lemma}
\label{lem:TensorGraded2M}
Let $\A$ be a $\ZZ$-graded 2-algebra, and let $\TAlg$, $\BAlg$, and $\RAlg$ be top, bottom, and right algebra-modules over $\A$, graded respectively by the groups $(G_1,\lambda_1)$ and $(G_2, \lambda_2)$ and the vertical group $(H_1,\mu_1)$.  Let $\TR$ and $\BR$ be top-right and bottom-right 2-modules over these algebra modules, graded respectively by the right $(\rot{H_1} \times_\ZZ \,G_1)$-set $S$ and the right $(\rotcc{H_1} \times_\ZZ \,G_2)$-set $T$.  In this situation, the restricted vertical tensor product of $\TR$ and $\BR$ has an induced grading by the right $(G_1 \times_\ZZ G_2)$-set $S \times_{\rot{\:H_1\:}} T$.
\end{lemma}
\noindent The proof is immediate from the definitions, and analogous statements hold for the other restricted tensor products of 2-modules.  Specifically, the restricted vertical tensor product of the $S$-graded $\TL$ and the $T$-graded $\BL$ has a grading by the left $(G_1 \times_\ZZ G_2)$-set $S \times_{\rot{\:H_2\:}} T$.  The restricted horizontal tensor product of the $S$-graded $\TR$ and the $T$-graded $\TL$ has a grading by the {\em top} $(H_1 \times_\ZZ H_2)$-set $S \times_{G_1} T$.  The restricted horizontal tensor product of the $S$-graded $\BR$ and the $T$-graded $\BL$ has a grading by the {\em bottom} $(H_1 \times_\ZZ H_2)$-set $S \times_{G_2} T$.

Gradings can be defined similarly on the various other $2$-objects introduced in Section~\ref{sec:bestiary}.

\subsection{Gradings on the cornered 2-algebras, algebra-modules, and 2-modules}

We now define particular noncommutative gradings on the cornered invariants.

\subsubsection{Grading on the nilCoxeter 2-algebra}
There is a $\ZZ$-grading on the sequential nilCoxeter $2$-alegbra $\nil$ from Example~\ref{ex:one}, given, as in
\cite[Section 2.2]{DM:cornered}, by the number of crossings in the corresponding diagram. 

Similarly, the diagonal nilCoxeter $2$-algebra $\Dnil$ considered in this paper is graded by $\ZZ$. A generator of $\Dnil$, represented by a rectangular diagram with strands as in Figure~\ref{fig:diag}, is homogeneous of degree equal to the number of crossings in the diagram.

\subsubsection{Review of the grading on the bordered algebra} \label{sec:GradingBordered} Let $\PMC = (Z, \CircPts, M, z)$ be a pointed matched circle, where $\CircPts$ consists of $4k$ points. We recall the two gradings on $\Alg(\PMC)$ constructed in \cite[Section 3.3]{LOT1}.

View $\CircPts =\{1, \dots, 4k \}$ as a subset of the open interval $Z' =Z - \{z\} \cong (\tfrac{1}{2}, 4k+\tfrac{1}{2})$. For $p \in \CircPts$ and $\alpha \in H_1(Z', \CircPts) \cong \ZZ^{4k-1},$ we define the multiplicity $m(\alpha, p)$ of $\alpha$ at $p$ to be the average of the local multiplicity of $\alpha$ just above $p$ and just below $p$. We extend $m$ to a bilinear map $m: H_1(Z', \CircPts) \times H_0(\CircPts) \to \frac{1}{2}\ZZ.$ Further, denote by $\delta: H_1(Z', \CircPts) \to H_0(\CircPts)$ the boundary map. Consider the group $\tfrac{1}{2}\ZZ \times H_1(Z', \CircPts)$ with multiplication
\begin{equation}
\label{eq:MultG} 
(j_1, \alpha_1) \ast (j_2, \alpha_2) = (j_1 + j_2 + m(\alpha_2, \delta \alpha_1), \alpha_1 + \alpha_2).\end{equation}
We let $G'(\PMC)$ be the index-two subgroup of $\tfrac{1}{2}\ZZ \times H_1(Z', \CircPts)$ generated by 
$\lambda=(1,0)$ and the elements $(\tfrac{1}{2}, [i, i+1])$ for all $i=1, \dots, 4k-1$.  We say an element $g'=(j, \alpha) \in G'(\PMC)$ has {\em Maslov component} $\mu(g') = j \in \tfrac{1}{2}\ZZ$ and  $\SpinC$-{\em component} $[g']=\alpha \in H_1(Z', \CircPts)$. 

We can now define a $(G'(\PMC),\lambda)$-grading on $\Alg(\PMC)$.  Let $a=I(\SetS)a(\rhos) \in \Alg(\PMC)$ be a basis element, represented by a strand diagram; here, $\SetS$ is the subset of matched pairs that describe the initial idempotent of $a$. Recall from Section~\ref{sec:AZ} that to $a$ we can associate its support $[a] \in H_1(Z', \CircPts)$. Let us also consider a diagram $\tilde a \subset [0,1] \times Z'$, obtained from the strand diagram for $a$ by deleting exactly one dashed horizontal line from each pair of such lines that appears in $a$. Let $\crs(\tilde a)$ denote the total number of crossings between the strands in $\tilde a$, and let $\tilde \SetS \subset Z'$ consist of the initial points of all the strands in $\tilde a$; note that $M(\tilde \SetS) = \SetS$. 

Set  
\begin {equation}
\label {eq:gr'}
  \gr'(a) = (\crs(\tilde a) - m([a], \tilde \SetS), [a]) \in G'(\PMC).
  \end {equation}
It is shown in \cite[Proposition 3.40]{LOT1} that $\gr'$ is independent of the choice of $\tilde a$, and that it defines a $(G'(\PMC), \lambda)$-grading on $\Alg(\PMC)$. 

It turns out that $\Alg(\PMC)$ is also graded by a smaller group $G(\PMC) \subset G'(\PMC)$, defined as follows.  Let $H$ denote the kernel of $M_{\ast} \circ \delta : H_1(Z', \CircPts) \to
H_0([2k])$, and let $F = F(\PMC)$ be the surface associated to the matched circle.
Each pair of matched points corresponds to a handle in $F$, and hence to a basis element for $H_1(F)$. Identify this basis element with the interval connecting the matched points (in the complement of $z$).
This provides an inclusion $i_*: H_1(F) \to H_1(Z', \CircPts)$ with image $H$, and we identify $H$ and $H_1(F)$ by this inclusion. We let $G(\PMC)$ be the subgroup of $G'(\PMC)$ of elements $g'$ whose $\SpinC$-component $[g']$ is in $H$.
Observe that $G(\PMC)$ is a central extension of $H_1(F)$ by $\ZZ$. We let $\lambda = (1,0)$ be the distinguished central element. 

To define a $(G(\PMC), \lambda)$ grading $\gr_\psi$ on $\Alg(\PMC)$, we pick additional {\em grading refinement data} $\psi$, as in \cite[Definition 3.8]{LOT2}. Grading refinement data for $\Alg(\PMC)$ consists of a function
$$ \psi: \{\SetS \subset [2k]\}  \to G'(\PMC)$$ 
such that $\psi(\SetS) g' \psi(\SetT)^{-1} \in G(\PMC)$ whenever $g' =(j, \alpha) \in G'(\PMC)$ satisfies $M_*\delta(\alpha) = \SetT - \SetS.$ Given a function $\psi$ like this, and a nonzero generator $a = I(\SetS) a  I(\SetT) \in \Alg(\PMC)$, we set
\begin{equation}
\label{eq:grpsi}
 \gr_\psi(a) = \psi(\SetS) \ast \gr'(a) \ast \psi(\SetT)^{-1}.
 \end{equation}
A concrete way to specify grading refinement data is the following.  Choose for each $i=1, \dots, 2k$, a subset $\SetT_i \subset [2k]$ with $|\SetT_i| = i$, and then for any $\SetS \subset [2k]$ with $|i|=i$, choose elements $\psi(\SetS) \in G'(\PMC)$ such that $M_*\delta([\psi(\SetS)]) = \SetS-\SetT_i$. We then say that $\SetT_i$ are the {\em base subsets} and $I(\SetT_i)$ are the {\em base idempotents}. We refer the reader to \cite[Section 3.3.2]{LOT1} and  \cite[Section 3.1]{LOT2} for more details.  Later on, we will sometimes write $\psi(I)$ for $\psi(\SetS)$, when $I$ is the idempotent $I(\SetS)$.

\subsubsection{Gradings on the algebra-modules}
Let $\PMC=(Z, \CircPts, M)$ be a matched interval. We identify $Z$ with the interval $[\tfrac{1}{2}, 4k+\tfrac{1}{2}]$, and let $w=\{\tfrac{1}{2}\}$ and $z=\{4k+\tfrac{1}{2}\}$ be the initial and final points. There is an associated pointed matched circle $\PMC^\dagger=(Z^\dagger, \CircPts, M, z)$ obtained from $\PMC$ by identifying the two endpoints. 

Set $G(\PMC) := G(\PMC^\dagger)$. We define a $(G(\PMC), \lambda)$-grading on the algebra-module $\TAlg(\PMC)$. (In the sequential setting, this was done in \cite[Sections 3.3 and 5.2]{DM:cornered}.) The role of $G'(\PMC)$ from Section~\ref{sec:GradingBordered} will be played by a different group, $G'_{\TAlg}(\PMC)$. Consider the relative homology group $H_1(Z, \CircPts \cup \{z\}) \cong \ZZ^{4k}$, with a basis consisting of the intervals $[i, i+1]$ for $i=1, \dots, 4k-1$, together with the interval $[4k, 4k+\tfrac{1}{2}]$. We have a boundary map 
$$\delta: H_1(Z, \CircPts \cup \{z\}) \to H_0(\CircPts \cup \{z\})$$ 
which can be written as $\delta=(\delta', \tau)$, where $\delta'$ is the $H_0(\CircPts)$ component and $\tau$ is the $H_0(\{z\}) \cong \ZZ$ component.  We define $G'_{\TAlg}(\PMC)$ to be the index-two subgroup of $\tfrac{1}{2}\ZZ \times H_1(Z, \CircPts \cup \{z\})$ generated by $\lambda=(1,0)$, $(\tfrac{1}{2}, [i, i+1])$ for $i=1, \dots, 4k-1$, and $(\tfrac{1}{2}, [4k, 4k+\tfrac{1}{2}])$; the multiplication is given by Formula~\eqref{eq:MultG}, except using $\delta'$ in place of $\delta$:
\[
(j_1, \alpha_1) \ast (j_2, \alpha_2) = (j_1 + j_2 + m(\alpha_2, \delta' \alpha_1), \alpha_1 + \alpha_2).
\]

Consider a generator $a$ of $\TAlg(\PMC)$ represented by a strand diagram, with a dotted rectangle at the top containing a nilCoxeter element. The support $[a] \in H_1(Z, \CircPts \cup \{z\})$ is obtained by horizontally projecting the diagram of $a$ without the dotted rectangle. We let $\tilde a$ be obtained from $a$ by deleting one strand from each pair of horizontal dashed lines, as before. Moreover, we let $\crs(\tilde a)$ be the number of crossings in $\tilde a$, including those in the dotted rectangle. Finally, we let $\tilde \SetS \subset Z \backslash \{w,z\}$ consist of the initial points of all the strands in $\tilde a$, ignoring the dotted rectangle. If $M(\tilde \SetS) = \SetS$, then the left idempotent corresponding to $a$ is of the form $\pbs e^h_n \\ I(\SetS) \pes$ for some $n \geq 0$. 

With these definitions in place, we let $\gr'(a) \in G'_{\TAlg}(\PMC)$ be given by Formula~\eqref{eq:gr'}. 

\begin{lemma}
  The function $\gr'(a)$ is a $G'_{\TAlg}(\PMC)$-grading on $\TAlg$.
\end{lemma}
\begin{proof}
  We must check that:
  \begin{enumerate}
  \item The element $\gr'(a)$ is independent of the choice of $\tilde{a}$.
  \item The element $\gr'(a)$ lies in the index $2$ subgroup $G'_{\TAlg}(\PMC)$.
  \item The function $\gr'(a)$ respects the differential, $\gr'(\bdy(a))=\lambda^{-1}\gr'(a)$.
  \item The function $\gr'(a)$ respects vertical multiplication, i.e., $\gr'\pbs b \\ a \pes=\lambda^{\gr(b)}\ast \gr'(a)$ for $b\in\Dnil$ and $\gr(b)$ the grading of $b$.
  \item The function $\gr'(a)$ respects horizontal multiplication, i.e., $\gr'(ab)=\gr'(a)\ast\gr'(b)$ for any $a,b\in\TAlg$ with $ab\neq 0$.
  \end{enumerate}
 
 For the first point, suppose $p$ is a position where no strand of $a$ starts or ends.  Adding a horizontal strand to $\tilde{a}$ at $p$ changes $\crs(\tilde{a})$ and $m([a],\tilde \SetS)$ by the same amount, so has no effect on $\gr'(a)$. Thus, changing which horizontal strands represent the idempotents does not change $\gr'(a)$.
 
  For the second point, given an element $\alpha\in H_1(Z,\CircPts\cup\{z\})=\ZZ^{4k}$, define $\epsilon(\alpha)$ to be $1/4$ the number of times $\alpha$ changes parity plus $\tau(\alpha)/4$; compare~\cite[Section 3.3.1]{LOT1}. We claim that
  \[
    G'_{\TAlg}(\PMC)=\{(m,\alpha)\in \tfrac{1}{2}\ZZ \times H_1(Z, \CircPts \cup \{z\})\mid m\equiv \epsilon(\alpha)\pmod{1}\}.
  \]
  To see this, note that $\epsilon(\alpha+\beta)\equiv \epsilon(\alpha)+\epsilon(\beta)+m(\beta,\delta'(\alpha))\pmod{1}$ (compare~\cite[Lemma 3.36]{LOT1}), so the condition $m\equiv \epsilon(\alpha)\pmod{1}$ does define a subgroup, which clearly has index $2$. Further, each of the generators of $G'_{\TAlg}(\PMC)$ satisfies the congruence condition $m\equiv \epsilon(\alpha)\pmod{1}$, so the two index $2$ subgroups agree. The elements $(\crs(\tilde a) - m([a], \tilde \SetS), [a])$ also clearly satisfy the congruence condition, hence lie in the subgroup $G'_{\TAlg}$.
  
  The third and fourth points are immediate from the definitions.
  
  For the fifth point, given elements $a$ and $b$ of $\TAlg$ with
  $ab\neq 0$, we can choose $\tilde{a}$ and $\tilde{b}$ so that the
  right endpoints of $\tilde{a}$ agree with the left endpoints of
  $\tilde{b}$ (i.e., $\tilde{a}\tilde{b}\neq 0$). Let $\tilde{\SetS}$ and
  $\tilde{\SetT}$ denote the left and right endpoints of $\tilde{a}$,
  respectively. Notice that $\delta'[a]=\tilde{\SetT}-\tilde{\SetS}$. So,
  \begin{align*}
    \gr(a)\gr(b)&=(\crs(\tilde{a})-m([a],\tilde{\SetS}),[a])*(\crs(\tilde{b})-m([b],\tilde{\SetT}),[b])\\
    &=(\crs(\tilde{a})+\crs(\tilde{b})-m([a],\tilde{\SetS})-m([b],\tilde{\SetT})+m([b],\delta'[a]),[a]+[b])\\
    &=(\crs(\tilde{a})+\crs(\tilde{b})-m([a],\tilde{\SetS})-m([b],\tilde{\SetS}),[a]+[b])\\
    &=\gr(ab).
  \end{align*}
  This concludes the proof.
\end{proof}

The $G'_{\TAlg}(\PMC)$-grading can be refined to a $G(\PMC)$-grading as follows. Note that the kernel of $M_{\ast} \circ \delta' : H_1(Z, \CircPts \cup \{z\}) \to H_0([2k])$ is still $i_*(H_1(F))$. Grading refinement data $\psi$ for $\TAlg(\PMC)$ can be defined just as for $\Alg(\PMC)$: it consists of a function
$$  \psi: \{ \SetS \subset [2k] \}  \to G'_{\TAlg}(\PMC)$$ 
such that $\psi(\SetS) g' \psi(\SetT)^{-1} \in G(\PMC)$ whenever $g' =(j, \alpha) \in G'_{\TAlg}(\PMC)$ satisfies 
\begin{equation}
M_*\delta'(\alpha) = \SetT - \SetS.\label{eq:gr-ref-dat}
\end{equation}
 Given any choice of grading refinement data, Formula~\eqref{eq:grpsi} defines a $G(\PMC)$-grading $\gr_\psi$ on $\TAlg(\PMC)$. 

We can construct grading refinement data by choosing a single base subset $\SetT$. Indeed, because we can use strands going off the top of the diagram, the map $M_* \circ \delta': H_1(Z, \CircPts \cup \{z\}) \to H_0([2k])$ is surjective, so for any $\SetS$ we can find an element $\psi(\SetS)$ with $M_*\delta'([\psi(\SetS)]) = \SetS-\SetT$. Observe that the only way in which the dotted rectangle contributes to the grading is through the number of crossings that appears in $\crs(\tilde a)$. If the base subset $\SetT \subset [2k]$ has cardinality $k+i$, where $i \in \{-k, \dots, k\}$, we say that we have chosen {\em $i$-based grading refinement data}.

The gradings for the other three algebra-modules are similar:
\begin{itemize}
\item The algebra-module $\LAlg(\PMC)$ is graded by $(\rotcc{G'_{\TAlg}(\PMC)}, \lambda)$, which is a vertical group. The grading of an element $a$ is given by 
   $\gr'(a) = \rotcc{(\crs(\rot{\tilde{a}})- m([\rot{a}], \tilde \SetS), [\rot{a}])}$.
   A choice of grading refinement data reduces the grading on $\LAlg(\PMC)$ to a grading by $(\rotcc{G(\PMC)}, \lambda)$.
\item The algebra-module $\BAlg(\PMC)$ is graded by $G'_{\BAlg}(\PMC)$, an index $2$ subgroup of $\tfrac{1}{2}\ZZ \times H_1(Z, \CircPts \cup \{w\})$, defined analogously to $G'_{\TAlg}(\PMC)$. The grading $\gr'$ on $\BAlg(\PMC)$ is given by the same formula as the grading on $\TAlg(\PMC)$. Again, a choice of grading refinement data reduces the grading on $\BAlg(\PMC)$ to a grading by $(G(\PMC),\lambda)$, although this is really the opposite group of the group which grades $\TAlg(\PMC)$.
\item The algebra-module $\RAlg(\PMC)$ is graded by $(\rotcc{G'_{\BAlg}(\PMC)}, \lambda)$, which is a vertical group. The grading of an element $a$ is given by 
   $\gr'(a) = \rotcc{(\crs(\rot{\tilde{a}})- m([\rot{a}], \tilde \SetS), [\rot{a}])}$.
   A choice of grading refinement data reduces the grading on $\RAlg(\PMC)$ to a grading by $(\rotcc{G(\PMC)}, \lambda)$, though again this is really the opposite vertical group of the vertical group that grades $\LAlg(\PMC)$.
\end{itemize}

\subsubsection{Graded pairing of algebra-modules}
Let $\PMC_0=(Z_0, \CircPts_0, M_0)$ and $\PMC_1=(Z_1, \CircPts_1, M_1)$ be matched intervals. Set $\PMC = \PMC_0 \cup \PMC_1 = (Z, \CircPts, M)$, where to obtain $Z$ we identify the final endpoint $z_0$ of $Z_0$ with the initial endpoint $w_1$ of $Z_1$. Let $4k_i$ be the number of marked points on the interval $\PMC_i$, for $i=0,1$, and set $k=k_0 + k_1$. For $j \in \CircPts := \CircPts_0 \cup \CircPts_1$, we set
$$ M(j) = \begin{cases}
M_0(j) & \text{if } j \in \CircPts_0,\\
M_1(j) + 2k_0 & \text{if } j \in \CircPts_1.
\end{cases}$$

\noindent Recall from Proposition~\ref{prop:rtp01} that the restricted tensor products
$$  \rvertme{\TAlgD(\PMC_0)}{\BAlgD(\PMC_1)}{\!\!\Dnil}{0pt}{0pt} \ \ \text{and} \ \ \ \rot{\rhormein{\RAlgD(\PMC_1)}{\LAlgD(\PMC_0)}{\Dnil}{1ex}{1ex}}$$
are both isomorphic to $\Alg(\PMC)$. We now establish graded versions of these isomorphisms.

There are restriction maps $H_1(Z,\CircPts)\to H_1(Z_0,\CircPts_0 \cup \{z_0\})$ and $H_1(Z,\CircPts)\to H_1(Z_1,\CircPts_1 \cup \{w_1\})$, and maps $\tau_\TAlg\co H_1(Z_0,\CircPts_0 \cup \{z_0\})\to\ZZ$ and $\tau_{\BAlg}\co H_1(Z_1,\CircPts_1 \cup \{w_1\}) \to \ZZ$ recording the boundary map to $\{z_0\}$ and $\{w_1\}$ respectively. Let $G''(\PMC)$ denote the kernel of the map
\begin{equation}
\tau_\TAlg-\tau_\BAlg: G'_{\TAlg}(\PMC_0)\times_\ZZ G'_{\BAlg}(\PMC_1)\to \ZZ.
\end{equation}

\begin{lemma}
There is a group isomorphism $h: G''(\PMC) \to G'(\PMC)$ given by
\begin{equation}
\label{eq:f}
 h((j_0, \alpha_0), (j_1, \alpha_1)) = (j_0 + j_1 - \tau_{\TAlg}(\alpha_0)/2, \alpha_0 \cup \alpha_1).
 \end{equation}
\end{lemma}

\begin{proof}
The term $-\tau_{\TAlg}(\alpha_0)/2$ appears because gluing the two half-chord generators from $G'_{\TAlg}(\PMC_0)$ and $G'_{\BAlg}(\PMC_1)$ produces a single full chord segment, which should have Maslov component $1/2$ rather than $1/2 + 1/2$.
\end{proof}

We now consider the refined gradings.  Choose $i$-based grading refinement data $\psi^i_0$ for $\TAlg(\PMC_0)$, with associated $G(\PMC_0)$-grading $\gr^i_0$, and $j$-based grading refinement data $\psi^j_1$ for $\BAlg(\PMC_1)$, with associated $G(\PMC_1)$-grading $\gr^j_1$.  We can combine these gradings to obtain a grading
$$ \gr^{i, j} := (\gr_0^i, \gr_1^j)$$
on $\TAlg(\PMC_0) \odot_{\A}^v \BAlg(\PMC_1)$ (which by the pairing result is isomorphic to $\Alg(\PMC)$) with values in 
$$ G(\PMC_0) \times_\ZZ G(\PMC_1) \cong G(\PMC).$$
This grading cannot be obtained from grading refinement data for $\Alg(\PMC)$, and so is not suitable, as is, for a graded pairing result.  Note, however, that the algebra $\Alg(\PMC)$ decomposes as a direct sum
$$ \Alg(\PMC) = \bigoplus_{m=-k}^k \Alg(\PMC, m),$$
where $\Alg(\PMC,m)$ is spanned by elements of the form $I(\SetS)a$ with $|\SetS| = m+k$.  When restricted to the summand $\Alg(\PMC,i+j)$, the grading $\gr^{i,j}$ is associated to the grading refinement data
\begin{equation}
\label{eq:psiSum}
\psi^{i, j} (\SetS_0 \cup (\SetS_1 + 2k_0)) := h(\psi^i_0(\SetS_0), \psi^j_1(\SetS_1)),
\end{equation}
where $h$ is the isomorphism given in Formula~\eqref{eq:f}.

As described, the algebra $\Alg(\PMC)$ decomposes into a sum of the algebras $\Alg(\PMC,m)$ according to the number $m+k$ of strands.  The restricted algebra-module $\TAlg(\PMC_0)^v$ also decomposes into a sum of sub-modules (but not sub-algebra-modules) $\TAlg(\PMC_0,p)^v$ according to the number $p+k_0$ of incoming (that is, initial idempotent) strands.  The restricted algebra-module $\BAlg(\PMC_1)^v$ similarly decomposes into a sum of sub-modules $\BAlg(\PMC_1,q)^v$ according to the number $q+k_1$ of outgoing (that is, final idempotent) strands.  The grading refinement data $\psi^i_0$ defines the $G(\PMC_0)$-grading $\gr^i_0$ on $\TAlg(\PMC_0,p)^v$ for all $p$, and the data $\psi^j_1$ defines the $G(\PMC_1)$-grading $\gr^j_1$ on $\BAlg(\PMC_1,q)^v$ for all $q$.  These gradings together define the grading, denoted as above $\gr^{i,j}$ on $\TAlg(\PMC_0,p)^v \odot_{\A^v} \BAlg(\PMC_1,q)^v$.  Equipped with this notation, we can now state the graded structure of the pairing isomorphism:

\begin{proposition}
Fix $m$ with $|m| \leq k$.  Choose $i$ and $j$ with $|i| \leq k_0$, $|j| \leq k_1$, and $i+j=m$; also choose $i$-based grading refinement data $\psi^i_0$ for $\TAlg(\PMC_0)$ and $j$-based grading refinement data $\psi^j_1$ for $\BAlg(\PMC_1)$.  The pairing isomorphism in Proposition~\ref{prop:rtp01}, from $\TAlg(\PMC_0)^v \odot_{\A^v} \BAlg(\PMC_1)^v$ to $\Alg(\PMC)$ intertwines the grading $\gr^{i,j}$ on the summand $\TAlg(\PMC_0,p)^v \odot_{\A^v} \BAlg(\PMC_1,q)^v$ with the grading on the summand $\Alg(\PMC,m)$ associated to the grading refinement data $\psi^{i,j}$.
\end{proposition}

\noindent The proof is straightforward, and an analogous result holds for the tensor product of $\LAlg(\PMC_0)$ and $\RAlg(\PMC_1)$.

The bent tensor products $\TAlg(\PMC_0) \obent \RAlg(\PMC_1)$ and $\BAlg(\PMC_1) \obent \LAlg(\PMC_0)$ also have induced gradings by $G(\PMC)$, and these gradings can be restricted to the smoothed tensor products $\TAlg(\PMC_0) \osmooth \RAlg(\PMC_1)$ and $\BAlg(\PMC_1) \osmooth \LAlg(\PMC_0)$. These smoothed tensor products are, by Proposition~\ref{prop:bent01}, both isomorphic to $\Alg(\PMC)$, and those isomorphism respect the $G(\PMC)$-gradings, in the same sense as described in the above lemma.

\begin{remark}
Secretly, the algebra $\Alg(\PMC)$ is naturally graded not by the group $G(\PMC)$ but by a groupoid $\fG(\PMC)$ whose objects are the idempotents of $\Alg(\PMC)$ and 
where $\Hom(I(\SetS),I(\SetT))$ is the set of elements $(j,\alpha)\in G'(\PMC)$ satisfying Formula~\eqref{eq:gr-ref-dat}.
Grading refinement data is a noncanonical equivalence of groupoids from $\fG(\PMC)$ to $G(\PMC)$.  The above graded pairing lemma would become less tortured if presented in terms of gradings by fiber products of groupoids, but we do not pursue that perspective here.
\end{remark}

\subsubsection{Gradings on the cornered 2-modules}
In this subsection, we assume familiarity with some notation and results from~\cite[Chapter 10]{LOT1}.
Given a bordered Heegaard diagram $\widetilde{\HD}$ with boundary $\PMC$ the
module $\CFAa(\widetilde{\HD})$ is graded by a right $G(\PMC)$-set $S(\widetilde{\HD})$; the
grading is specified by a function $\gr\co \Gen(\widetilde{\HD})\to S(\widetilde{\HD})$,
where $\Gen(\widetilde{\HD})$ is the set of generators of $\CFAa(\widetilde{\HD})$---see~\cite[Definition 10.33]{LOT1}.

Let $\HD$ be a cornered Heegaard diagram with boundary  $\bdy\HD= \PMC = \PMC_0\cup\PMC_1$. Let $\widetilde{\HD}$ be the bordered Heegaard diagram obtained by smoothing the corner of $\HD$. By definition
\[
\CAA(\HD)=\CFAa(\widetilde{\HD})\otimes_{\Alg(\PMC)} \putaround{(\TAlg(\PMC_0)\obent \RAlg(\PMC_1))}{0}{0}{\sbull}{\sbull}.
\]
(Here we have suppressed the projective resolution of $\CFAa(\widetilde{\HD})$ for simplicity.)  Recall from~\cite[Remark 3.28]{LOT1} that $\Alg(\PMC,m)$ acts trivially on $\CFAa(\widetilde{\HD})$ if $m\neq 0$, so we could equally well have tensored over $\Alg(\PMC,0)$ instead of $\Alg(\PMC)$.

Fix $0$-based grading refinement data for $\TAlg(\PMC_0)$ and $\RAlg(\PMC_1)$.
As mentioned in the previous subsection, using this grading refinement data, the bent tensor product $\TAlg(\PMC_0)\obent \RAlg(\PMC_1)$ is graded by $G(\PMC)$. We can define a grading on $\CAA(\HD)$ by setting 
\[
\gr(\x\otimes a)=\gr(\x) \gr(a) \in S(\widetilde{\HD}),
\]
where $\x\in \CFAa(\widetilde{\HD})$ and $a\in \TAlg(\PMC_0)\obent \RAlg(\PMC_1)$ are homogeneous elements. 

Recall that $\CAD(\HD)$ is obtained by tensoring $\CFAa(\widetilde{\HD})$ with 
$\putaround{(\TAlg(\PMC_0)\obent \RAlg(\PMC_1))}{\sbull}{0}{0}{\sbull} \odot^v_{\rot{\: \TAlg \:}} \DDa(\Id)$, and it follows (compare the calculations in Section~\ref{sec:example}) that any element of $\CAD(\HD)$ can be written in the form $\x\otimes r$ where $\x\in \CFAa(\widetilde{\HD})$ and $r\in \RAlg(\PMC_0)$. Define
\[
\gr(\x\otimes r)=\gr(\x) \gr(r)\in S(\widetilde{\HD}).
\]
Here we used the inclusion $G(\PMC_0) \rightarrow G(\PMC_0) \times_\ZZ G(\PMC_1) \cong G(\PMC)$ to let $\gr(r)$ act on $S(\widetilde{\HD})$.  The grading on $\CDA(\HD)$ is defined similarly to the grading on $\CAD(\HD)$.

Finally, 
\[
\CDD(\HD)= \putaround{(\BAlg(\PMC_0)\obent\LAlg(\PMC_1))}{\sbull}{\sbull}{0}{0} \otimes_{\Alg(-\PMC)}\CFDa(\widetilde{\HD}).
\]
The module $\CFDa(\widetilde{\HD})$ is graded by the left $G(-\PMC)$-set $S(\widetilde{\HD})$.  (Note that $G(-\PMC) = G(\PMC)^\op$.) We define the grading on $\CDD(\HD)$ by 
\[
\gr(a\otimes \x)=\gr(a) \gr(\x)\in S(\widetilde{\HD}),
\]
where $\x\in \CFDa(\widetilde{\HD})$ and $a\in \BAlg(\PMC_0)\obent \LAlg(\PMC_1)$ are homogeneous elements.

We have the following graded addendum to Theorem~\ref{thm:invariance}:
\begin{proposition}\label{prop:invariance-gr}
  The above definitions make $\CAA(\HD)$, $\CAD(\HD)$, $\CDA(\HD)$, and $\CDD(\HD)$ into graded $2$-modules. If $\HD$ and $\HD'$ are cornered Heegaard diagrams representing the same cornered $3$-manifold then there exist identifications of grading sets so that the quasi-isomorphisms $\CAA(\HD)\simeq\CAA(\HD')$, $\CAD(\HD)\simeq\CAD(\HD')$, $\CDA(\HD)\simeq\CDA(\HD')$, and $\CDD(\HD)\simeq\CDD(\HD')$ are grading preserving.
\end{proposition}
\begin{proof}
  The first statement is clear from the definitions. The second follows from the same argument used to prove Theorem~\ref{thm:invariance}, using the fact that $\CFAa(\widetilde{\HD})$ is a $G$-set graded invariant of the bordered 3-manifold associated to $\widetilde{\HD}$~\cite[Theorem 10.39]{LOT1}. 
\end{proof}

\subsection{The graded pairing theorem}
In this section we describe the graded structure of the quasi-isomorphisms appearing in Theorem~\ref{thm:pairing}, focusing on the quasi-isomorphism between $\CAA(\HD_{12}) \odot^v_{\RAlg(\PMC_2)} \CAD(\HD_{23})$ and $\CFAa(\HD_{12} \cup_{\PMC_2} \HD_{23})$.  Here $\PMC_1$, $\PMC_2$, and $\PMC_3$ are matched intervals and $\HD_{12}$ and $\HD_{23}$ are cornered Heegaard diagrams with $\partial \HD_{12} = \PMC_1 \cup \PMC_2$ and $\partial \HD_{23} = (-\PMC_2) \cup \PMC_3$, as in Section~\ref{sec:trimods}.  We again assume familiarity with~\cite[Chapter 10]{LOT1}.

We begin by reviewing some ingredients used in constructing the pairing quasi-isomorphism.  Recall from Definition~\ref{def:THD} that the Heegaard diagram $\THD{\PMC_1}{\PMC_2}{\PMC_3}$ has boundary $\PMC_{12}\amalg \PMC_{23}\amalg \PMC_{31}$, where $\PMC_{ij}$ denotes a pairwise union as in Formula~\eqref{eq:PMC-ij}.
The Heegaard diagram $\THD{\PMC_1}{\PMC_2}{\PMC_3}$ has two kinds of
regions: $8$-sided regions $R_i$ and a single $12$-sided region $T$ (see Figure~\ref{fig:THD-region-labels}). Only the region $T$ is adjacent to all three boundary components.  The pairing quasi-isomorphisms make essential use of the trimodules $\CFDDDa(\THD{\PMC_1}{\PMC_2}{\PMC_3})$ and $\CFDDAa(\THD{\PMC_1}{\PMC_2}{\PMC_3})$, whose structure is described in Section~\ref{sec:describe-trimodules}.

When talking about the gradings on these trimodules, it is convenient to work one
$\SpinC$-structure at a time. The $\SpinC$-structures on
$\THD{\PMC_1}{\PMC_2}{\PMC_3}$ correspond to triples of integers
$(m,n,p)$ with $m+n+p=0$. A generator $\x$ of $\CFDDDa(\THD{\PMC_1}{\PMC_2}{\PMC_3})$, corresponding to a complementary idempotent triple $(I_{12},I_{23},I_{31})$,
has the $\SpinC$-structure $(m,n,p)$ if $I_{12}$ has weight
$m-g(\PMC_{12})$, $I_{23}$ has weight $n-g(\PMC_{23})$, and
$I_{31}$ has weight $p-g(\PMC_{31})$. (Here, $g(\PMC)$ denotes the genus of the
surface associated to $\PMC$, in other words one quarter of the number of marked points in
$\PMC$. The \emph{weight} of a generator $a$ of $\Alg(\PMC)$ is the number of solid strands in $a$ plus half the number of dashed strands in $a$.)  We write $\spinc(\x)$ for the $\SpinC$-structure of $\x$ and we say the idempotent triple $(I_{12},I_{23},I_{31})$ itself is ``in" the $\SpinC$-structure $\spinc(\x)$. There is a generator with the $\SpinC$-structure $(m,n,p)$ if and only if 
\[
-g(\PMC_1)\leq m\leq g(\PMC_1),\quad -g(\PMC_2)\leq n\leq g(\PMC_2), \quad\text{and}\quad -g(\PMC_3)\leq p\leq g(\PMC_3).
\]
From now on, by a $\SpinC$-structure we will mean a $\SpinC$-structure satisfying these conditions.

\begin{remark}
  Because we have restricted attention to $3$-manifolds with connected boundary, in
  fact only the central $\SpinC$ structure $(0,0,0)$ is relevant.
\end{remark}

We now introduce the notion of compatible grading refinement data.
Note that for a pointed matched circle $\PMC$ there is an identification $G'(\PMC)\cong G'(-\PMC)^{\op}\cong G'(-\PMC)$, where the second isomorphism sends $g$ to $g^{-1}$. Similarly, for a matched interval $\PMC$, $G'_\TAlg(\PMC)\cong G'_\BAlg(-\PMC)^{\op}\cong G'_\BAlg(-\PMC)$. In particular, given grading refinement data $\psi$ for $G'_\TAlg(\PMC)$ there is corresponding grading refinement data for $G'_{\BAlg}(-\PMC)$, denoted, somewhat abusively, by $\psi^{-1}$.  Let $h_{12}\co \ker(G'_{\TAlg}(\PMC_1) \times_\ZZ G'_{\BAlg}(\PMC_2) \rightarrow \ZZ) \xrightarrow{\cong} G'(\PMC_{12})$ be the isomorphism~\eqref{eq:f} associated to the decomposition $\PMC_{12} = \PMC_1 \cup \PMC_2$, and similarly let $h_{23}$ and $h_{31}$ be the corresponding isomorphisms for the decompositions $\PMC_{23} = \PMC_2 \cup \PMC_3$ and $\PMC_{31} = \PMC_3 \cup \PMC_1$.

\begin{definition}\label{def:compat-gr-data} Fix a $\SpinC$-structure $\spinc=(m,n,p)$ for the Heegaard diagram 
  $\THD{\PMC_1}{\PMC_2}{\PMC_3}$.
  Choose base subsets $\SetT_i$ for $\PMC_i$, with corresponding idempotents $I(\SetT_i)\in \TAlg(\PMC_i)$, and choose grading
  refinement data $\psi_i$ for $\TAlg(\PMC_i)$ relative to $\SetT_i$. Let
  $\SetT_i^c$ be the complementary subsets for $-\PMC_i$ and let $I(\SetT_i^c)\in \BAlg(-\PMC_i)$ be the corresponding idempotents, which are complementary to the idempotents
  $I(\SetT_i)$. Assume that $(I(\SetT_1)\vtp I(\SetT_2^c), I(\SetT_2)\vtp I(\SetT_3^c), I(\SetT_3)\vtp I(\SetT_1^c))$
  is a complementary idempotent triple in the $\SpinC$-structure
  $\spinc$.  The subsets $\SetT_i$ and data $\psi_i$ define grading refinement data
  for $\Alg(\PMC_{12})$, $\Alg(\PMC_{23})$, and $\Alg(\PMC_{31})$ as
  follows:
  \begin{itemize}
  \item The base idempotent for $\Alg(\PMC_{ij})$ is $I_{ij}=I(\SetT_i)\vtp I(\SetT_j^c)$.
  \item The grading refinement data for $\Alg(\PMC_{12})$ is given by $\psi_{12}(J_{12})=h_{12}(\psi_1(J_1), \psi_2^{-1}(J_{2}))$, where $J_{12}=J_1\vtp J_2$.
  \item The grading refinement data for $\Alg(\PMC_{23})$ is given by $\psi_{23}(J_{23})=h_{23}(\psi_2(J_2), \psi_3^{-1}(J_{3}))$, where $J_{23}=J_2\vtp J_3$.
  \item The grading refinement data for $\Alg(\PMC_{31})$ is given by $\psi_{31}(J_{31})=\lambda^a h_{31}(\psi_3(J_3), \psi_1^{-1}(J_{1}))$, where $J_{31}=J_3\vtp J_1$ and $a$ is the weight of $J_1$ minus the weight of $I(\SetT_1)$.
  \end{itemize}
  We call grading refinement data arising this way 
  \emph{compatible grading refinement data} for $\PMC_{12}$,
  $\PMC_{23}$ and $\PMC_{31}$ in the $\SpinC$-structure $\spinc$.
\end{definition}
\noindent Note the shift by $a$ in the Maslov component of $\psi_{31}(J_{31})$. This shift will be needed, in the proof of Proposition~\ref{prop:gr-on-TDDD} below, to cancel the Maslov contribution of the central region $T$ of the Heegaard diagram $\THD{\PMC_1}{\PMC_2}{\PMC_3}$; cf.~Corollary~\ref{cor:gr-on-diag}.

Let $\pi_k\co H_1(Z_{ij},\CircPts_{ij})\to H_1(Z_k,\CircPts_k)$, for $k=i$ or $k=j$, denote projection.

\begin{lemma}\label{lem:compat-gr-ref-dat}
  Let $(\psi_{12},\psi_{23},\psi_{31})$ be compatible grading refinement data (in the $\SpinC$-structure $\spinc$, with base subsets $\SetT_1$, $\SetT_2$, $\SetT_3$) for the triple of matched intervals $(\PMC_1,\PMC_2,\PMC_3)$. For any complementary idempotent triple
  $(J_{12},J_{23},J_{31})$ in the $\SpinC$-structure $\spinc$, we have
    \begin{align*}
    \pi_1([\psi_{12}(J_{12})])+\pi_1([\psi_{31}(J_{31})])&=0\\
    \pi_2([\psi_{12}(J_{12})])+\pi_2([\psi_{23}(J_{23})])&=0\\
    \pi_3([\psi_{23}(J_{23})])+\pi_3([\psi_{31}(J_{31})])&=0\\
    \mu(\psi_{12}(J_{12}))+\mu(\psi_{23}(J_{23}))+\mu(\psi_{31}(J_{31}))&=-a/2,
  \end{align*}
  where $a$ is the weight of $J_{12}\cap \PMC_1$ minus the weight of
  $I(\SetT_1)$, and $\mu$ denotes the Maslov index.  
\end{lemma}
\begin{proof}
  The first three relations are immediate from the definitions. For the fourth, observe that each of the three $\psi_{ij}$ terms uses an isomorphism $h_{ij}$, which by Formula~\eqref{eq:f} has a Maslov shift of $-a/2$, but $\psi_{31}$ has an additional Maslov shift of $a$, resulting in an overall $-a/2$ shift.
\end{proof}

\begin{proposition}\label{prop:gr-on-TDDD}
  Given compatible grading refinement data for $\PMC_{12}$, $\PMC_{23}$,
  and $\PMC_{31}$, the trimodule $\CFDDDa(\THD{\PMC_1}{\PMC_2}{\PMC_3})$ is
  graded by the $(G(\PMC_{12}),G(\PMC_{23}),G(\PMC_{31}))$-set 
  \[
  S_{\PMC_1,\PMC_2,\PMC_3}=G(\PMC_1)\times_\ZZ G(\PMC_2)\times_\ZZ G(\PMC_3).
  \]
  
  For each $\SpinC$-structure $\spinc$ on
  $\THD{\PMC_1}{\PMC_2}{\PMC_3}$ the grading function
  \[
  \gr\co \Gen(\THD{\PMC_1}{\PMC_2}{\PMC_3},\spinc)\to S_{\PMC_1,\PMC_2,\PMC_3}
  \]
is constant, that is takes the same value on all elements of $\Gen(\THD{\PMC_1}{\PMC_2}{\PMC_3},\spinc)$.

  The same statements hold with $\CFDDDa(\THD{\PMC_1}{\PMC_2}{\PMC_3})$ replaced by 
  $\CFDDAa(\THD{\PMC_1}{\PMC_2}{\PMC_3})$.
\end{proposition}
\begin{proof}
  Fix a generator $\x=\x(I_{12},I_{23},I_{31})$ for
  $\CFDDDa(\THD{\PMC_1}{\PMC_2}{\PMC_3})$ corresponding to the base idempotents $I_{ij}$ of the compatible grading refinement data.  The matched pairs $\{m\}$ in
  $\PMC_i$ give a basis $\{h_m\}$ for $H_1(F(\PMC_i))$. Moreover, for each
  matched pair $m$ there is a corresponding periodic domain
  $B_m\in\pi_2(\x,\x)$, such that $\bdy(\bdy^\bdy B_m)$ is the pair of
  points $m$. The domains $B_m$, for $m$ ranging over all matched pairs in $\PMC_1$, $\PMC_2$, and $\PMC_3$, form a basis for $\pi_2(\x,\x)$.

  The grading set $S$ for $\CFDDDa(\THD{\PMC_1}{\PMC_2}{\PMC_3})$ is given by
  \begin{align*}
  S&=\bigl(G(\PMC_{12})\times_\ZZ G(\PMC_{23})\times_\ZZ G(\PMC_{31})\bigr)/\langle \{g'(B_m)\}\rangle\\
  &=\bigl(G(\PMC_1)\times_\ZZ G(-\PMC_2)\times_\ZZ G(\PMC_2)\times_\ZZ G(-\PMC_3)\times_\ZZ G(\PMC_3)\times_\ZZ G(-\PMC_1)\bigr)/\langle \{g'(B_m)\}\rangle.
  \end{align*}
  (Cf.~\cite[Definition 10.36]{LOT1}.  In general, $g'(B)=(-e(B)-n_\x(B)-n_\y(B),\bdy^\bdy(B))$ though here $\y = \x$.  In both expressions for $S$, the index $m$ runs over all matched pairs in $\PMC_1$, $\PMC_2$, and $\PMC_3$.)
  For $m$ a matched pair in $\PMC_1$, for instance, we have
  \begin{align*}
  g'(B_m)&=(0,h(m),0,0,0,0,-h(m))\\
  &\qquad\in (\tfrac{1}{2}\ZZ)\times H_1(F(\PMC_1))\times H_1(F(-\PMC_2))\times H_1(F(\PMC_2))\\
  &\qquad\quad\times H_1(F(-\PMC_3))\times H_1(F(\PMC_3))\times H_1(F(-\PMC_1)).
  \end{align*}
  Because the elements $h(m)$, for $m$ a pair in $\PMC_i$, form a basis for $H_1(F(\PMC_i))$, it follows that
  \[
  S=G(\PMC_1)\times_\ZZ G(\PMC_2)\times_\ZZ G(\PMC_3),
  \]
  as desired.

  Next, given another generator $\y=(J_{12},J_{23},J_{31})$ in the same
  $\SpinC$-structure as $\x$, let $B\in \pi_2(\x,\y)$. Let $T$ be the
  region in the middle of $\THD{\PMC_1}{\PMC_2}{\PMC_3}$ and $a$ the
  multiplicity of $B$ on $T$. We have
  \begin{equation}\label{eq:gr-y-1}
  \gr(\y)=\psi_{12}(J_{12})^{-1}\, \psi_{23}(J_{23})^{-1}\, \psi_{31}(J_{31})^{-1}\, g'(B) \, \gr(\x)\in S.
  \end{equation}
  Because the grading refinement data $\psi_{ij}$ are compatible, the
  $\psi$ terms reduce to $(a/2,-a\cdot \bdy^\bdy(T))$. Now because the parts of $\bdy^\bdy(R_j)$ in $\PMC_i$ and $-\PMC_i$ cancel, we have,
  as in the proof of Corollary~\ref{cor:gr-on-diag}, that
  $g'(B)=(-a/2,a\cdot \bdy^\bdy(T))$. It follows that
  \[ 
  \gr(\y)=(a/2,-a\cdot \bdy^\bdy(T))(-a/2,a\cdot \bdy^\bdy(T))\gr(\x)=\gr(\x).
  \]

  For $\CFDDAa$, the grading set and the gradings of generators are defined identically to those for $\CFDDDa$.
\end{proof}

\noindent Note that because the grading on $\THD{\PMC_1}{\PMC_2}{\PMC_3}$ is constant, by translating the grading function we may assume $\gr$ always takes the value $(e_1,e_2,e_3)$, where $e_i\in G(\PMC_i)$ is the identity.

\begin{corollary}\label{cor:bord-glue-gradings}
  Let $\HD_{12}$ and $\HD_{23}$ be bordered Heegaard diagrams with
  boundary $\PMC_{12}$ and $\PMC_{23}$ respectively. Let $S_{ij}$ be the grading
  set of $\HD_{ij}$ and let $\gr_{ij}\co \Gen(\HD_{ij})\to S_{ij}$ be the
  grading function for $\HD_{ij}$. Set
  $\HD=\HD_{12}\cup_{\PMC_2}\HD_{23}$. The module $\CFAa(\HD)$ is
  graded by 
  \[
  S := S_{12}\times_{G(\PMC_2)}S_{23}=S_{12}\times S_{23}\big/(s_{12}g,s_{23}g)\sim (s_{12},s_{23})\text{ for }g\in G(\PMC_2),
  \]
  and the grading function $\Gen(\HD) \rightarrow S$ is given by
  \[
  \gr(\x_{12}\cup \x_{23})=(\gr_{12}(\x_{12}),\gr_{23}(\x_{23})).
  \]
(Here, we view $\Gen(\HD)$ as the subset of
$\Gen(\HD_{12})\times\Gen(\HD_{23})$ of pairs of generators occupying
complementary sets of $\alpha$-arcs.)  
  
  The corresponding statements for $\CFDa(\HD)$ also hold.
\end{corollary}
\begin{proof}
  In Corollary~\ref{cor:corn-pair-via-bord} we showed that $\CFAa(\HD)$ is quasi-isomorphic to the tensor product
  \[
  \CFAa(\HD_{12})\otimes_{\Alg(\PMC_{12})} \bigl(\CFAa(\HD_{23})\otimes_{\Alg(\PMC_{23})} \CFDDAa(\THD{\PMC_1}{\PMC_2}{\PMC_3})\bigr).
  \]
(Again we have suppressed the derivation.)  That $\CFAa(\HD)$ is graded by $S$ now follows from the first part of Proposition~\ref{prop:gr-on-TDDD}.  The expression for the grading follows from the second part of Proposition~\ref{prop:gr-on-TDDD}, utilizing a grading on $\THD{\PMC_1}{\PMC_2}{\PMC_3}$ that is constant at the identity.
\end{proof}

We can finally state and prove the graded addendum to Theorem~\ref{thm:pairing}:
\begin{proposition}\label{prop:pairing-gr}
Let $\PMC_1$, $\PMC_2$, and $\PMC_3$ be matched intervals, and let $\HD_{12}$ and $\HD_{23}$ be cornered Heegaard diagrams with $\bdy\HD_{12}=\PMC_1\cup\PMC_2$ and $\bdy\HD_{23}=(-\PMC_2)\cup\PMC_3$, and with smoothings denoted $\widetilde{\HD}_{12}$ and $\widetilde{\HD}_{23}$. Choose compatible grading refinement data for $\PMC_{12}$, $\PMC_{23}$, and $\PMC_{31}$, providing a grading $\gr_{12}$ of $\CFAa(\widetilde{\HD}_{12})$ by $S_{12}$, a grading $\gr_{23}$ of $\CFAa(\widetilde{\HD}_{23})$ by $S_{23}$, and a grading $\gr$ of $\CFAa(\HD_{12} \cup_{\PMC_2} \HD_{23})$ by $S_{12} \times_{G(\PMC_2)} S_{23}$, as in Corollary~\ref{cor:bord-glue-gradings}.  The quasi-isomorphism
  \[
  \rsmallvertme{\CAA(\HD_{12})}{\CAD(\HD_{23})}{\RAlg(\PMC_2)}{0pt}{0pt}\!\!\!\!\xrightarrow{\:F\:} \:\:\CFAa(\HD_{12}\cup_{\PMC_2}\HD_{23}),
  \]
constructed in the proof of Theorem~\ref{thm:pairing}, respects the grading.  That is, for any
  homogeneous elements $\x_{12}\in\CAA(\HD_{12})$ and $\x_{23}\in
  \CAD(\HD_{23})$, the element $F(\x_{12}\vtp
  \x_{23})\in\CFAa(\HD_{12}\cup_{\PMC_2}\HD_{23})$ is homogeneous and
  \[
  \gr(F(\x_{12}\vtp \x_{23}))=(\gr_{12}(\x_{12}),\gr_{23}(\x_{23})).  
  \]
  Corresponding statements hold for the other quasi-isomorphisms~\eqref{eq:pair1}, \eqref{eq:pair0'}, and \eqref{eq:pair1'} in Theorem~\ref{thm:pairing}.
\end{proposition}

\begin{proof}
  We focus on the statement for $\CAA$, $\CAD$, and $\CFAa$; the proofs in the other cases are similar.

  For each generator $\tilde{\x}$ of $\CFAa(\widetilde{\HD}_{12})$ there is a corresponding generator $\x=\tilde{\x}\otimes 1_{\TAlg\obent\RAlg}$ of $\CAA(\HD_{12})$. Similarly, for each generator $\tilde{\y}$ of $\CFAa(\widetilde{\HD}_{23})$ there is a corresponding generator $\y=\tilde{\y}\otimes (1_{\TAlg\obent\RAlg} \vtp 1_{\DDa(\Id)})$ of $\CAD(\HD_{23})$. The proof of Theorem~\ref{thm:pairing} shows that the elements $\x\vtp \y$ form a basis (over $\Field$) for  
  $\CAA(\HD_{12}) \odot^v_{\RAlg} \CAD(\HD_{23})$.
  
  By definition, the grading set for $\CAA(\HD_{12})$ is the grading set $S_{12}$ for $\CFAa(\widetilde{\HD}_{12})$, and $\gr(\x)=\gr(\tilde{\x})$. Similarly, the grading set for $\CAD(\HD_{23})$ is the grading set $S_{23}$ for $\CFAa(\widetilde{\HD}_{23})$ and $\gr(\y)=\gr(\tilde{\y})$. The result now follows from Corollary~\ref{cor:bord-glue-gradings}.
\end{proof}


\section{Practical computations}
\label{sec:various}
\subsection{Induction and restriction functors}\label{sec:ind-rest}
Let $\A$ and $\B$ be $2$-algebras. A \emph{(unital) $2$-algebra homomorphism}
$f\co \A \to \B$ is a family of chain maps
$\putaround{f}{m}{n}{p}{q}\co \putaround{\A}{m}{n}{p}{q}\to
\putaround{\B}{m}{n}{p}{q}$ commuting with the horizontal and vertical
multiplications, and respecting the horizontal and vertical units. 

Given a $2$-algebra homomorphism $f$ and a right algebra-module
$\RAlg$ over $\B$ there is an algebra-module  $f^*\RAlg$ over $\A$,
gotten by restriction of scalars. That is, as a chain complex,
$f^*\RAlg$ is the same as $\RAlg$, and $f^*\RAlg$ has the same
vertical multiplication as $\RAlg$. The right action of $\A$ on
$\RAlg$ is given by $\zeta\cdot a = \zeta\cdot f(a)$.

Similarly, fix a $2$-algebra homomorphism $g\co \B\to \A$.  Assume
that $\RAlg$ (respectively $\A$) satisfies the motility hypothesis
(Definition~\ref{def:motileAM}) as right (respectively left)
algebra-modules over $\B$. Then there is an algebra-module $g_*\RAlg$
over $\A$, obtained by extension of scalars. That is, $g_*\RAlg =
\RAlg\htp_\B \A$, with the obvious differential and horizontal
action. The vertical multiplication is defined by 
\[
\bmat
(\zeta\htp a)\\
\cdot\\
(\zeta'\htp a')
\emat
=
\pbmat \zeta \\ \cdot\\ \zeta'\pemat
\htp
\pbmat a\\ \cdot\\ a'\pemat
\]
where $\zeta \in \putaround{\RAlg}{0}{}{0}{n}$,
$a\in\putaround{\A}{m}{n}{p}{q}$,
$\zeta'\in\putaround{\RAlg}{0}{}{0}{n'}$ and
$a'\in\putaround{\A}{m'}{n'}{m}{q'}$. By Lemma~\ref{lem:AMweak}, this
defines a multiplication on $g_*\RAlg$; this is where we use the motility hypothesis.
There are obvious analogues for top, bottom and left algebra-modules.

There is a map of vector spaces $\iota\co \RAlg\to g_*\RAlg$ defined by $\iota(\zeta)=\zeta \htp e_n^h$ (where $\zeta \in\putaround{\RAlg}{\sbull}{}{\sbull}{n}$).

\begin{lemma} These definitions make $f^*\RAlg$ and $g_*\RAlg$ into
  right algebra-modules. Moreover, the map $\iota$ is a map of
  algebras.
\end{lemma}
\begin{proof}
  The statement for $f^*\RAlg$ is obvious.  For $g_*\RAlg$, the proof
  that $\cdot$ is well-defined and associative is a special case of
  Lemma~\ref{lem:AMweak}. Associativity of the right action of $\A$
  on $g_*\RAlg$ is clear. The local commutation relation between
  $*$ and $\cdot$ follows from the definition and the local
  commutation relation for $\A$.
  Finally, we check that $\iota$ is a map of algebras: for $\zeta \in \RAlga{k}{m}{q}$ and $\zeta' \in \RAlga{m}{n}{p}$ we have
  \[
  \iota\pbmat
  \zeta' \\\cdot\\ \zeta \pemat 
  =\pbmat
 \zeta' \\\cdot\\ \zeta \pemat\htp
  e_{p+q}^h
  =
  \pbmat
  \zeta' \\\cdot\\ \zeta \pemat\htp
   \pbmat e_p^h\\\cdot\\e_q^h\pemat
  =
  \bmat
  (\zeta' \htp e_p^h)\\
  \cdot\\
  (\zeta \htp e_q^h)\emat
  =
  \bmat \iota(\zeta')\\\cdot\\
  \iota(\zeta)\emat.\qedhere
  \]
\end{proof}

\begin{convention}
  When talking about induction functors (i.e., $g_*$) we will assume
  that $\RAlg$ and $\A$ satisfy the motility hypothesis as
  algebra-modules over $\B$.
\end{convention}

Now, fix also a top algebra-module $\TAlg$ over $\B$. Given a top-right
$2$-module $M$ over $\RAlg$ and $\TAlg$ there is an associated
$2$-module $f^*M$ over $f^*\RAlg$ and $f^*\TAlg$ given by $f^*M=M$. 

To define the induction functor for $2$-modules, we will resort to the
bent tensor product. Assume that $\A$ and $\B$ satisfy the bent
motility hypothesis (Definition~\ref{def:bentMH}). 
\begin{lemma}\label{lem:induced-bent}
  The map
  \begin{align*}
    \iota\co \TAlg \obent \RAlg = 
    \left( \mathcenter{ \bendme{\RAlga{\sbull}{\sbull}{\sbull}}{ \putaround{\B}{\sbull}{\sbull}{0}{0}}{\TAlga{\sbull}{\sbull}{\sbull}}{\putaround{\scriptstyle\B}{0}{\sbull}{0}{\sbull}} {\putaround{\scriptstyle\B}{\sbull}{0}{\sbull}{0}} } \!\! \right)
        &\to 
    \pbmat \putaround{\RAlg}{\sbull}{}{\sbull}{\sbull} 
    &{\mathlarger{\mathlarger{\circledast}}} & \putaround{\A}{\sbull}{\sbull}{\sbull}{\sbull}&
   {\mathlarger{\mathlarger{\circledast}}} &
    \putaround{\A}{\sbull}{\sbull}{0}{0} & \\
    & \putaround{\scriptstyle\B}{\sbull}{\sbull}{\sbull}{\sbull} & &\putaround{\scriptstyle\A}{0}{\sbull}{0}{\sbull} & \mathlarger{\mathlarger{\odot}} & \!\!\!\!\!\!{\putaround{\scriptstyle\A}{\sbull}{0}{\sbull}{0}}\\
    & & &  & \putaround{\A}{\sbull}{\sbull}{\sbull}{\sbull} & \\
    & & &  & \mathlarger{\mathlarger{\odot}} & \!\!\!\!\!\!{\putaround{\scriptstyle\B}{\sbull}{\sbull}{\sbull}{\sbull}}\\
    & & &  & \putaround{\TAlg}{}{\sbull}{\sbull}{\sbull} &
    \pemat
    = g_*(\TAlg)\obent g_*(\RAlg)\\
   \pbmat
   \zeta & b\\
   & \phi
   \pemat
   &\mapsto
   \pbmat
   \iota(\zeta) & g(b)\\
   & \iota(\phi)
   \pemat
  \end{align*}
  is a well-defined map of differential algebras.
\end{lemma}
\begin{proof}
  To see that the map is well-defined, observe that if $b\in
  \putaround{\B}{0}{m}{0}{n}$ and $\phi\in
  \putaround{\TAlg}{}{\sbull}{q}{\sbull}$ then 
  \[
  \iota\pbmat
  \zeta b  & b'\\
  &  \phi
  \pemat=
  \pbmat
  \zeta b & e_n^h & g(b')\\
  & & e_q^v\\
  & & \phi
  \pemat=
  \pbmat
  \zeta & g(b) & g(b')\\
  & & e_q^v\\
  & & \phi
  \pemat=
  \pbmat
  \zeta & e_m^h & g(b)g(b')\\
  & & e_q^v\\
  & & \phi
  \pemat
  =
  \iota\pbmat
  \zeta  & bb'\\
  &  \phi
  \pemat
  \]
  so the relations induced by the tensor product over
  $\putaround{\B}{0}{\sbull}{0}{\sbull}$ are respected. A similar
  computation shows that the relations imposed by the tensor product
  over $\putaround{\B}{\sbull}{0}{\sbull}{0}$ are respected.

  It is immediate from the definitions that $\iota$ is a chain map.

  Finally, we check that $\iota$ respects the multiplication:
  \begin{align*}
    \iota\left[\pbmat \zeta & b\\ & \phi \pemat \pbmat \zeta' & b'\\ & \phi'\pemat \right]
    &= \iota
      \pbmat
      \pbmat\zeta'\\ \zeta\pemat 
      & \pbmat e_{\B,n}^h & b' \\ b & e_{\B,q}^v \pemat\\
      & \pbmat \phi & \phi' \pemat
      \pemat\\
    &=
    \pbmat
    \pbmat \zeta' \\ \zeta \pemat & e_{\A,m+n}^h & g\pbmat e_{\B,n^h} & b' \\ b & e_{\B,q}^v\pemat\\
    & & e_{\A,p+q}^v\\
    & & \pbmat \phi & \phi' \pemat
    \pemat\\
    &= 
    \pbmat 
      \pbmat
        \pbmat \zeta' & e_{\A,n}^h\pemat\\
        \pbmat \zeta & e_{\A,m}^h\pemat
      \pemat
     & \pbmat
       g(e_{\B,n}^h) & g(b')\\
       g(b) & g(e_{\B,q}^v)
     \pemat\\
     & \pbmat
       \pbmat e_{\A,p}^v\\ \phi \pemat &
       \pbmat e_{\A,q}^v\\ \phi' \pemat
     \pemat
   \pemat\\
  &=\iota\pbmat \zeta & b\\ & \phi \pemat \iota \pbmat \zeta' & b' \\ & \phi' \pemat.
  \end{align*}
(The last equality uses the fact that $g(e_{\B,n}^h)=e_{\A,n}^h$, from the definition of a 2-algebra homomorphism.) This completes the proof.
\end{proof}

By Proposition~\ref{prop:2modsNew}, the $2$-module $M$ can be viewed
as an indexed module $M'$ over $\TAlg\obent \RAlg$. Using the
induction functor corresponding to the differential algebra map $\iota$
from Lemma~\ref{lem:induced-bent} we get an indexed module $\iota_*M$ over
$(g_*\TAlg)\obent (g_*\RAlg)$. We define $g_*M$ to be the
$2$-module corresponding to the module $\iota_*(M')$.

The restriction and induction operations $f^*$ and $g_*$ induce
functors of derived categories of $2$-modules. For induction (i.e.,
$g_*$), one replaces $M$ by a (categorically) projective
resolution before tensoring.

There are corresponding operations for the other variants of $2$-modules.

A $2$-algebra homomorphism $f$ is a \emph{quasi-isomorphism} if $\putaround{f}{m}{n}{p}{q}$ induces an isomorphism on homology for each $m,n,p,q$.

\begin{lemma}\label{lem:qi-equiv}	
  Fix $2$-algebras $\A$ and $\B$, a right algebra-module $\RAlg$ over
  $\B$ and a top algebra-module $\TAlg$ over $\B$. Assume that $\A$
  and $\B$ satisfy the bent motility hypothesis.
  \begin{enumerate}
  \item If $f\co \A\to \B$ is a $2$-algebra quasi-isomorphism then
    $f^*$ induces an equivalence of derived categories of $2$-modules.
  \item If $g\co \B\to \A$ is a $2$-algebra quasi-isomorphism such
    that $\iota_\RAlg\co \RAlg\to g_*\RAlg$ and $\iota_\TAlg\co
    \TAlg\to g_*\TAlg$ are quasi-isomorphisms then $g_*$ induces an
    equivalence of derived categories of $2$-modules.
  \end{enumerate}
\end{lemma}
\begin{proof}
  In both cases, there is a corresponding map of bent tensor
  products. For the induction functor, this was verified in
  Lemma~\ref{lem:induced-bent}, and for the restriction functor it is
  given by 
  \begin{align*}
    f^*\TAlg \obent_\A f^*\RAlg = 
    \left( \mathcenter{ \bendme{\RAlga{\sbull}{\sbull}{\sbull}}{ \putaround{\A}{\sbull}{\sbull}{0}{0}}{\TAlga{\sbull}{\sbull}{\sbull}}{\putaround{\scriptstyle\A}{0}{\sbull}{0}{\sbull}} {\putaround{\scriptstyle\A}{\sbull}{0}{\sbull}{0}} } \!\! \right)
      &\to 
   \left( \mathcenter{ \bendme{\RAlga{\sbull}{\sbull}{\sbull}}{ \putaround{\B}{\sbull}{\sbull}{0}{0}}{\TAlga{\sbull}{\sbull}{\sbull}}{\putaround{\scriptstyle\B}{0}{\sbull}{0}{\sbull}} {\putaround{\scriptstyle\B}{\sbull}{0}{\sbull}{0}} } \!\! \right)
   =\TAlg \obent_\B \RAlg\\
    \pbmat \zeta & a \\ & \phi \pemat &\mapsto \pbmat \zeta  & f(a) \\ & \phi \pemat.
  \end{align*}
  In both cases, the hypothesis guarantees that this map of bent
  tensor products is a quasi-isomorphism. So, the result follows from
  Propositions~\ref{prop:2modsNew} and~\ref{prop:bent-der-cat}
  and the corresponding result for differential categories (see,
  e.g.,~\cite[Lemma 3.10]{Keller06:DGCategories}), of which indexed
  modules are a special case. (See also~\cite[Theorem
  10.12.5.1]{BernsteinLunts94:EquivariantSheaves} for the result for
  ordinary differential algebras.)
\end{proof}

\begin{lemma}\label{lem:flat-iota}
  If $\RAlg$ is flat as a $\B$-module and the map $g$ is a quasi-isomorphism
  then $\iota$ is a quasi-isomorphism as well.
\end{lemma}
\begin{proof}
  The map $\iota$ can be viewed as 
  \[
  \RAlg\htp_\B\B \stackrel{\Id\htp g}{\longrightarrow} \RAlg\htp_\B\A
  \]
  under the obvious identification $\RAlg\cong
  \RAlg\htp_\B\B$. Since tensoring with a flat module preserves
  quasi-isomorphisms, the result follows.
\end{proof}

\subsection{The multiplicity-one 2-algebra}\label{sec:mult-1-alg}
\begin{lemma}\label{lem:acyclic}
  If $m+n+p+q>2$ then $\Dnila{m}{n}{p}{q}$ is acyclic.
\end{lemma}
\begin{proof}
  $\Dnila{m}{n}{p}{q}=0$ unless $m+q=n+p$, so assume that $m+q=n+p$.
  As a chain complex, $\Dnila{m}{n}{p}{q}$ is isomorphic to
  $\Dnila{m+q}{0}{m+q}{0}$, so it suffices to show that
  $\Dnila{m+q}{0}{m+q}{0}$ is acyclic if $m+q>1$. To see that $\Dnila{m+q}{0}{m+q}{0}$ is
  acyclic, it suffices to verify that the vertical unit is in the
  image of the boundary map; but if $\sigma_1$ denotes the element of
  $\Dnila{m+q}{0}{m+q}{0}$ consisting of a single crossing between the
  first two strands then $\bdy(\sigma_1)=e^v_{m+q}$.
\end{proof}

Note that 
\[
\Ideal = \bigoplus_{m,n,p,q\mid m+n+p+q>2}\Dnila{m}{n}{p}{q}\subset \Dnil.
\]
is closed under the differential and forms an ideal with
respect to both the horizontal and vertical multiplications. So,
$\sDnil=\Dnil/\Ideal$ is again a $2$-algebra. There is a projection
map $\pi\co \Dnil\to \sDnil$ which, by Lemma~\ref{lem:acyclic}, is a
quasi-isomorphism.

\begin{lemma}\label{lem:RAlg-proj}
  The algebra $\RAlg(\PMC)$ is projective over $\Dnil$.
\end{lemma}
\begin{proof}
  Note that, as a right module over itself, $\Dnil$ decomposes as a direct sum $\Dnil=\oplus_{n}e_h^n*\Dnil=\oplus_{n}\putaround{\Dnil}{\sbull}{n}{\sbull}{\sbull}$.
  A basis for $\putaround{\RAlg(\PMC)}{\sbull}{}{\sbull}{n}$ over $\putaround{\Dnil}{\sbull}{n}{\sbull}{\sbull}$ is given by the basic
  elements of $\putaround{\RAlg}{0}{}{0}{n}$ (that is, the strand diagrams) such that
  there are no crossings between the $n$ strands ending in the
  nilCoxeter region. The set of basic elements is filtered by the
  number of crossings, and the differential strictly decreases this
  filtration; it follows that $\RAlg$ is (categorically) projective (compare the proof of~[Proposition~10.12.2.6]\cite{BernsteinLunts94:EquivariantSheaves}).
\end{proof}

\begin{corollary}\label{cor:iota-is-qi} For any matched interval $\PMC$, the map $\iota\co\RAlg(\PMC)\to \pi_*\RAlg(\PMC)$ induced by $\pi$ is a quasi-isomorphism. Corresponding statements hold for $\TAlg$, $\LAlg$, and $\BAlg$.
\end{corollary}
\begin{proof}
This is immediate from Lemmas~\ref{lem:flat-iota} and~\ref{lem:RAlg-proj}.
\end{proof}

It follows from Lemma~\ref{lem:qi-equiv} and Corollary~\ref{cor:iota-is-qi} that there is no real loss of information in passing from $\Dnil$ to $\Dnil/\Ideal$.
We can be somewhat more precise about this, as follows.
The tensor product $\pi_*\RAlg(\PMC_1)\htp^h_{\sDnil}\pi_*\LAlg(\PMC_2)$ admits an explicit description. Let $p$ denote the (non-basepoint) point in $\PMC=\PMC_1\cup\PMC_2$ where the two matched intervals are glued together. Recall that a basic element $a\in \Alg(\PMC)$ has a support $[a]\in H_1(Z,\CircPts)$. The bordered algebra
$\Alg(\PMC)$ has an acyclic differential ideal $\otherIdeal$ generated by strand diagrams with support $\geq 2$ at $p$, and it is clear from the definitions that
\[
\pi_*\RAlg(\PMC_1)\htp^h_{\sDnil}\pi_*\LAlg(\PMC_2)\cong \Alg(\PMC)/\otherIdeal.
\]
There is a projection map $\pi\co \Alg(\PMC)\to \Alg(\PMC)/\otherIdeal$ and, since $\otherIdeal$ is acyclic, $\pi$ is a quasi-isomorphism.

We call $\Dnil/\Ideal$, $\pi_*\RAlg$, $\pi_*\LAlg$, and so on, the \emph{multiplicity-one} versions of $\Dnil$, $\RAlg$, $\LAlg$, and so on.

In this multiplicity-one setting, we have the following version of the pairing theorem:
\begin{proposition} \label{prop:paired} Let $\HD_1$ (respectively $\HD_2$) be a cornered Heegaard diagram with boundaries $\PMC_1$ and $\PMC_2$ (respectively $-\PMC_2$ and $\PMC_3$).  Let $\HD=\HD_1\cup_{\PMC_2}\HD_2$ and $\PMC=\PMC_1\cup\PMC_3$. Then 
\[
\rvertme{ \pi_*\CAA(\HD_1)}{\pi_*\CAD(\HD_2)}{{\sDnil}}{0pt}{0pt} \simeq \pi_*\CFAa(\HD) \ \ \ \text{and} \ \ \ 
\rvertme{ \pi_*\CDA(\HD_1)}{\pi_*\CDD(\HD_2)}{{\sDnil}}{0pt}{0pt} \simeq \pi_*\CFDa(\HD)
\]
as modules over $\Alg(\PMC)/\otherIdeal$ and $\Alg(-\PMC)/\otherIdeal$, respectively.
\end{proposition}
\begin{proof}
  This follows from the same argument used to prove
  Theorem~\ref{thm:pairing}, setting to zero any algebra element with
  multiplicity bigger than one at the corner.
\end{proof}

Unlike $\Dnil$, the $2$-algebra $\sDnil$ is finite-dimensional, as are the cornered algebra-modules and $2$-modules over $\sDnil$. So, for practical computations it is often better to work over $\sDnil$; and this is what we shall do in the next section.

\subsection{An example} \label{sec:example}
We end with an example, computed using $\sDnil$ (see
Section~\ref{sec:mult-1-alg}). Consider the Heegaard diagram on the right of 
Figure~\ref{fig:example}. In the terminology of~\cite[Section
9.5]{LOT4}, this is a union of four copies of the \emph{self-gluing
  handlebody}. In this section we will discuss two of the four
cornered invariants associated to the decomposition in
Figure~\ref{fig:example}---the type \AAm\ and \AD\ ones---and see how
the pairing theorem gives the corresponding bordered invariant $\pi_*\CFAa$.
(In parts of the example we will restrict to summands for which the computation is less cumbersome.)

\begin{figure}
  \centering
  \includegraphics{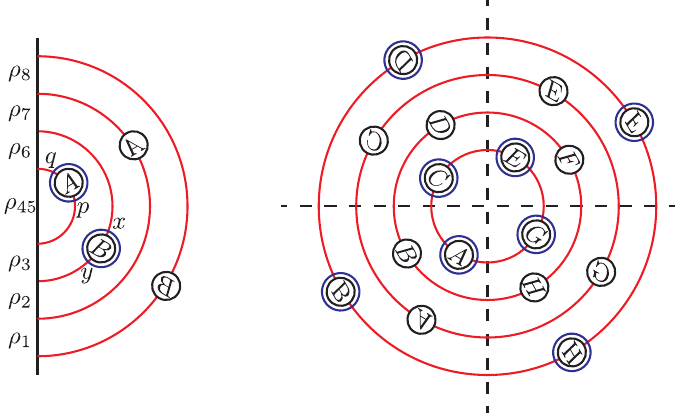}
  \caption{\textbf{The self-gluing handlebody.} Left: The self-gluing
    handlebody. Right: A
    Heegaard diagram for $(S^1\times S^2)\# (S^1\times S^2)$ built
    from $4$ copies of the self-gluing handlebody. The dashed lines
    decompose the diagram into four cornered Heegaard
    diagrams.\label{fig:example}}
\end{figure}

We start with some more notation. Let $\PMC$ be the genus $1$ matched
interval. Number the matched points in $\PMC$ as $1,2,3,4$, so $1$ and
$3$ are matched, as are $2$ and $4$. View $\PMC$ as vertical, so there
is a basepoint just below $1$ and the corner just above $4$. The
algebra $\pi_*\TAlg(\PMC)$ has $8$ left-right idempotents:
\[
\includegraphics{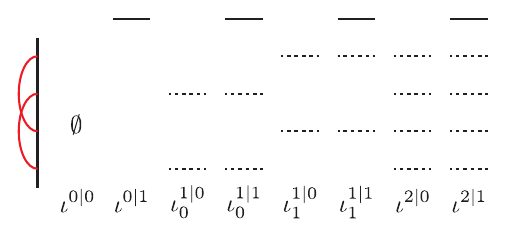}.
\]
(The solid line at the top of $\iota^{0|1}$ and other elements is in the nilCoxeter ($\sDnil$) part.)
It has many non-idempotent elements, some of which we label using the
following conventions:
\[
\includegraphics{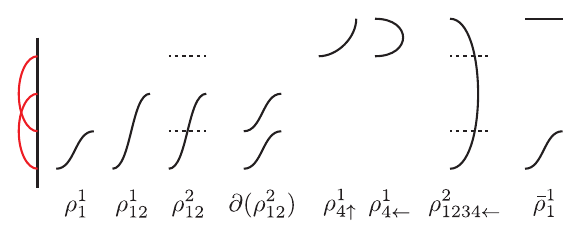}.
\]
In particular, the superscript indicates the weight (number of solid lines plus half the number of dashed lines) in $\PMC$ on the left.

We number elements of $\RAlg$ similarly, except that the numbers run
from $5$ to $8$, with $5$ at the corner (so $\rho_{\leftarrow56}^1$ is
an element of $\RAlg$). Superscripts indicate the number of occupied
positions in $\PMC$ on the bottom. We will use $j$ to denote the
idempotents on $\RAlg$ (instead of $\iota$).

Let $\HD$ be the cornered, genus $2$ handlebody occurring four times
in Figure~\ref{fig:example} (right), and $\HD^\circ$ the bordered
Heegaard diagram gotten by smoothing the corner in $\HD$. To compute
$\pi_*\CAA(\HD)$ we must compute a projective resolution of
$\pi_*\CFAa(\HD^\circ)$, and then tensor that resolution with the cornering
module--$2$-module. By~\cite[Theorem 3]{LOTHomPair}, $\CFDa(\HD^\circ)$ is
quasi-isomorphic to $\CFAa(\HD^\circ)$, and by~\cite[Corollary
2.3.25]{LOT2}, $\CFDa(\HD^\circ)$ is projective. (Here, we are viewing
$\CFDa(\HD^\circ)$, which is a left module over $\Alg(-(\PMC\cup\PMC))$,
as a right module over $\Alg(\PMC\cup\PMC)=\Alg(-(\PMC\cup\PMC))^{\op}$.)
The induction functor takes projective modules to projective modules,
so we can (and will)
use $\pi_*\CFDa(\HD^\circ)$ as our projective resolution.

According to~\cite[Theorem 9.13]{LOT4}, the module
$\pi_*\CFDa(\HD^\circ)$ is given by
\[
\pi_*\CFDa(\HD^\circ)=
\mathcenter{
\begin{tikzpicture}
  \node at (0,0) (xp) {$xp$};
  \node at (0,-2) (yq) {$yq$};
  \node at (-6,-1) (xq) {$xq$};
  \node at (6,-1) (yp) {$yp.$};
  \draw[->, bend left=5] (xq) to node[above]{$\rho_{45}+\rho_{234567}$} (xp);
  \draw[->, bend right=5] (xq) to node[below]{$\rho_{3456}+\rho_{12345678}$} (yq);
  \draw[->, bend right=5] (yq) to node[below]{$\rho_{45}+\rho_{234567}$} (yp);
  \draw[->, bend left=5] (xp) to node[above]{$\rho_{3456}+\rho_{12345678}$} (yp);
  \draw[->, bend right=15] (yq) to node[right]{$\rho_2\rho_7$} (xp);
  \draw[->, bend right=15] (xp) to
  node[left]{$\rho_1\rho_8+\rho_3\rho_6+\rho_{123}\rho_{678}$} (yq);
\end{tikzpicture}}
\]
The vertex names correspond to the generator names in Figure~\ref{fig:example}.
Each vertex corresponds to an elementary projective right module over 
\[
\Alg:=\pi_*\Alg(-\PMC\cup\PMC).
\]
For example, the vertex $xq$ corresponds to the module
$I\cdot \Alg$ where $I$ is the idempotent corresponding to
$\{5,6,7,8\}$ in the pointed matched circle. The arrows correspond to
differentials; for instance, $\{xp\}\rho_2\rho_7$ occurs in
$\bdy\{yq\}$. Note that the arc in the middle of $-\PMC\cup\PMC$ is labeled $45$, so the indecomposable chords are $\rho_1$, $\rho_2$, $\rho_3$, $\rho_{45}$, $\rho_6$, $\rho_7$ and $\rho_8$. The differential of any other element, say
$\{yq\}\rho_{12}$, is determined by these differentials and the
Leibniz rule.

The cornered 2-module $\pi_*\CAA(\HD)$ is obtained by replacing each copy
of $\Alg$ with the bent tensor product
$\putaround{\TAlg\obent\RAlg}{0}{0}{\sbull}{\sbull}$ (where $\TAlg:=\pi_*\TAlg(\PMC)$ and $\RAlg:=\pi_*\RAlg(\PMC)$). The names for the algebra elements
decorating the differentials change; for example, the $\rho_{45}$ from
$\{xq\}$ to $\{xp\}$ becomes
$\rho_{4\leftarrow}^2\obent\rho_{\leftarrow5}^0
=\rho_{4\uparrow}^2\obent\rho_{\uparrow5}^0$. With the new labels, we
have
\[
\pi_*\CAA(\HD)=
\mathcenter{
\begin{tikzpicture}
  \node at (0,0) (xp) {$xp$};
  \node at (0,-2) (yq) {$yq$};
  \node at (-7,-1) (xq) {$xq$};
  \node at (7,-1) (yp) {$yp.$};
  \draw[->, bend left=5] (xq) to
  node[above, sloped]{$\rho_{4\leftarrow}^2\obent\rho_{\leftarrow5}^0+\rho_{234\leftarrow}^2\obent\rho_{\leftarrow567}^0$} (xp);
  \draw[->, bend right=5] (xq) to node[below,sloped]{$\rho^2_{34\leftarrow}\obent\rho^0_{\leftarrow56}+\rho^2_{1234\leftarrow}\obent\rho^0_{\leftarrow5678}$} (yq);
  \draw[->, bend right=5] (yq) to node[below,sloped]{$\rho^1_{4\leftarrow}\obent\rho^1_{\leftarrow5}+\rho^1_{234\leftarrow}\obent\rho^1_{\leftarrow567}$} (yp);
  \draw[->, bend left=5] (xp) to node[above,sloped]{$\rho^1_{34\leftarrow}\obent\rho^1_{\leftarrow56}+\rho^1_{1234\leftarrow}\obent\rho^1_{\leftarrow5678}$} (yp);
  \draw[->, bend right=15] (yq) to node[right]{$\rho_2^1\obent\rho_7^1$} (xp);
  \draw[->, bend right=15] (xp) to
  node[left]{$\rho_1^1\obent\rho_8^1+\rho_3^1\obent\rho_6^1+\rho_{123}^1\obent\rho_{678}^1$} (yq);
\end{tikzpicture}}
\]

To construct $\pi_*\CAD(\HD)$ we also need the \DD\ identity module for
$\PMC$. First, some notation for elements:
\[
\includegraphics{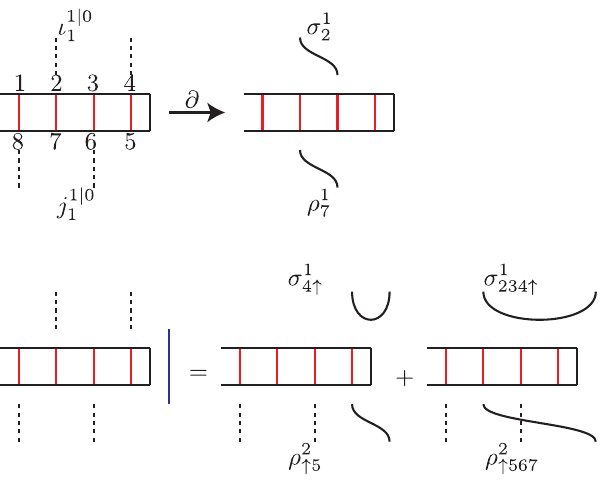}
\]
In this notation, the \DD\ identity module is given by
\[
\pi_*\DDa(\Id)=\mathcenter{
  \begin{tikzpicture}
    \node at (10,0) (two) {$\bmat \iota^{0|0}\phantom{.}\\\otimes\\j^{2|0}.\emat$};
    \node at (5,2) (one1) {$\bmat
      \iota_1^{1|0}\\ \otimes\\j_1^{1|0}\emat$};
    \node at (5,-2) (one0) {$\bmat
      \iota_0^{1|0}\\ \otimes\\j_0^{1|0}\emat$};
    \node at (0,0) (zero) {$\bmat \iota^{2|0}\\ \otimes\\j^{0|0}\emat$};
    \draw[->, bend right=15] (one1) to node[left]{$\bmat\sigma_2^1\\ \otimes\\ \rho_7^1 \emat$} (one0); 
    \draw[->, bend right=15] (one0) to node[right]{$\bmat\sigma_1^1\\
      \otimes\\ \rho_8^1 \emat+\bmat\sigma_3^1\\ \otimes\\ \rho_6^1
      \emat+\bmat\sigma_{123}^1\\ \otimes\\ \rho_{678}^1 \emat$}
    (one1);
    \draw[->, color=blue, bend right=5] (zero) to
    node[below,sloped]{$\bmat\sigma_{4\uparrow}^1\\\otimes\\\rho_{\uparrow5}^2\emat+\bmat\sigma_{234\uparrow}^1\\\otimes\\\rho_{\uparrow567}^2\emat$}
    (one0);
    \draw[->, color=blue, bend left=5] (zero) to
    node[above,sloped]{$\bmat\sigma_{34\uparrow}^1\\\otimes\\\rho_{\uparrow56}^2\emat+\bmat\sigma_{1234\uparrow}^1\\\otimes\\\rho_{\uparrow5678}^2\emat$}
    (one1);
    \draw[->, color=blue, bend right=5] (one0) to node[below,sloped]{$\bmat\sigma_{34\uparrow}^1\\\otimes\\\rho_{\uparrow56}^2\emat+\bmat\sigma_{1234\uparrow}^1\\\otimes\\\rho_{\uparrow5678}^2\emat$}
 (two);
    \draw[->, color=blue, bend left=5] (one1) to node[above,sloped]{$\bmat\sigma_{4\uparrow}^1\\\otimes\\\rho_{\uparrow5}^2\emat+\bmat\sigma_{234\uparrow}^1\\\otimes\\\rho_{\uparrow567}^2\emat$} (two);
  \end{tikzpicture}
}
\]
The notation needs some explanation. The nodes are generators of
$\pi_*\DDa(\Id)$ as a biprojective (top-bottom) ($\rot{\TAlg}$,$\RAlg$)-bimodule. So, for
instance, $\bmat \iota^{0|0}\\\otimes\\j^{2|0}\emat$ stands for 
$\bmat\rot{\TAlg}\\\cdot\\ \iota^{0|0}\\\otimes\\j^{2|0}\\\cdot\\
\RAlg\emat$. The black arrows are differentials; for example, 
\[
\bdy \pbmat \iota_1^{1|0}\\ \otimes\\j_1^{1|0}\pemat 
=
\pbmat \sigma_2^1\\ \iota_0^{1|0}\\ \otimes\\j_0^{1|0}\\\rho_7^1\pemat.
\]
The blue arrows correspond to the effect of multiplying by a single
vertical strand in the barbell algebra; together with the associativity relations, this determines the action of all elements of the barbell algebra.

The next step in computing $\pi_*\CAD(\HD)$ is to take the restricted
tensor product, over $\rot{\TAlg}$, of $\pi_*\DDa(\Id)$ and a rotated copy $\rot{\pi_*\CAA(\HD)}$
of $\pi_*\CAA(\HD)$. The idempotents of $\rot{\TAlg}$ corresponding to the
generators of $\rot{\pi_*\CAA(\HD)}$ are
\begin{equation}\label{eq:gens-corresp}
xq\leftrightarrow \iota^{0|0}\qquad xp \leftrightarrow \iota_1^{1|0}\qquad 
yq \leftrightarrow \iota_0^{1|0} \qquad yp \leftrightarrow \iota^{2|0}.
\end{equation}
So, the generators of the restricted tensor product are as in Figure~\ref{fig:eg-cdad-nodiff}.
Again, the
notation needs some explanation. For example, the top-right entry
stands for a copy of
\[
\bmat I(xq) \\ \Alg\\ \iota^{0|0}\\
    \otimes\\ j^{2|0}\\ \putaround{\RAlg}{\sbull}{}{0}{\sbull}\emat,
\]
where $I(xq)$ denotes the idempotent of $\Alg(\PMC\cup\PMC)$ associated to the generator $xq$,
via the identification
\begin{equation}\label{eq:big-with-brackets}
\begin{split}
\rvertme{\DDa(\Id)}{\putaround{\rot{\TAlg\obent\RAlg}}{\sbull}{0}{0}{\sbull}}{\rot{\:\TAlg\:}}{6pt}{4pt}
&= \bigoplus_{i,j}
\left[\rvertme{
\pbmat
\rot{\TAlg}\\[1pt]
i\\
\otimes\\
j\\[1pt]
\RAlg
\pemat
}
{
\putaround{\pbmat
\bendhim{\rot{\RAlg}}{\rot{\TAlg}}{\rot{\Dnil}}{\rot{\:\Dnil\:}}{\rot{\:\Dnil\:}}
\pemat}{\sbull}{0}{0}{\sbull}
}{\rot{\:\TAlg\:}}{24pt}{6pt}
\!\!\!\!\!\!\!\right]
=\bigoplus_{i,j}
\left[
\begin{tabular}{l}
$\bendhim{\putaround{\rot{\RAlg}}{\sbull}{0}{}{0}\;\;\,}{\putaround{\rot{\TAlg}}{\sbull}{}{0}{\sbull}}{\putaround{\rot{\Dnil}}{0}{\sbull}{\sbull}{0}\;\;\,}{\!\!\raisebox{1pt}{\putaround{\scriptscriptstyle\rot{\:\Dnil\:}}{\sbull}{0}{\sbull}{0}}}{\putaround{\scriptscriptstyle\rot{\:\Dnil\:}}{0}{\sbull}{0}{\sbull}}$\\
$\;\;\vtp{\scriptscriptstyle\putaround{\scriptscriptstyle\rot{\:\TAlg\:}}{\sbull}{}{\sbull}{0}}$\\[6pt]
$\bmat
\putaround{\rot{\TAlg}}{0}{}{\sbull}{0}\\[9pt]
i\\
\otimes\\
j
\\[4pt]
\putaround{\RAlg}{\sbull}{}{0}{\sbull}
\emat$
\end{tabular}
\!\right]\\
&=\bigoplus_{i,j}
\left[\begin{tabular}{l}
        $\bendhim{\putaround{\rot{\RAlg}}{\sbull}{0}{}{0}\;\;\,}{\putaround{\rot{\TAlg}}{0}{}{0}{\sbull}}{\putaround{\rot{\Dnil}}{0}{\sbull}{\sbull}{0}\;\;\,}{\!\!\raisebox{1pt}{\putaround{\scriptscriptstyle\rot{\;\Dnil\;}}{\sbull}{0}{\sbull}{0}}}{\putaround{\scriptscriptstyle\rot{\;\Dnil\;}}{0}{\sbull}{0}{\sbull}}$\\[-7pt]
$\bmat
\:\:i\\
\:\:\otimes\\
\:\:j
\\[4pt]
\:\:\putaround{\RAlg}{\sbull}{}{0}{\sbull}
\emat$
\end{tabular}
\!\right]
=\bigoplus_{i,j}
\left[\bmat \\ \Alg\\[2pt] i\\
    \otimes\\ j\\[4pt] \putaround{\RAlg}{\sbull}{}{0}{\sbull}\emat\right]
,
\end{split}
\end{equation}
where in the last equality we identify $\Alg$ with the smoothed tensor
product of $\rot{\TAlg}$ and $\rot{\RAlg}$. (Again, we have suppressed the $\pi_*$'s to keep the notation cleaner.) In particular, in the rest of this section we will repeatedly make use of the identification between $\Alg$ and this smoothed tensor product. See Figure~\ref{fig:Alg-is-rot} for some examples of this identification. With respect to this identification, the left action on $\Alg$ corresponds to multiplying on the top (or top-left) of the smoothed tensor product, and we will draw this action as on the top. Similarly, the right action on $\Alg$ appears on the bottom (or bottom-right) of the smoothed tensor product. The idempotent $i$ in Formula~\eqref{eq:big-with-brackets} is a restriction on the bottom part of the right idempotent of the element of $\Alg$.

\begin{figure}
\centering
\[
\rvertme{\pi_*\DDa(\Id)}{\rot{\pi_*\CAA(\HD)}}{\rot{\:\TAlg\:}}{0pt}{0pt}\hspace{-1ex}=\hspace{3.25ex}
\mathcenter{
\begin{tikzpicture}
  \node at (0,0) (xq0) {\fbox{$\bmat xq\\ \Alg\\ \iota^{2|0}\\
    \otimes\\ j^{0|0}\emat$}};
  \node at (2,0) (xq11) {\fbox{$\bmat xq\\ \Alg\\ \iota^{1|0}_1\\
    \otimes\\ j^{1|0}_1\emat$}};
  \node at (4,0) (xq10) {\fbox{$\bmat xq\\ \Alg\\ \iota^{1|0}_0\\
    \otimes\\ j^{1|0}_0\emat$}};
  \node at (6,0) (xq2) {\fbox{$\bmat xq\\ \Alg\\ \iota^{0|0}\\
    \otimes\\ j^{2|0}\emat$}};
  \node at (0,-3) (xp0) {\fbox{$\bmat xp\\ \Alg\\ \iota^{2|0}\\
    \otimes\\ j^{0|0}\emat$}};
  \node at (2,-3) (xp11) {\fbox{$\bmat xp\\ \Alg\\ \iota^{1|0}_1\\
    \otimes\\ j^{1|0}_1\emat$}};
  \node at (4,-3) (xp10) {\fbox{$\bmat xp\\ \Alg\\ \iota^{1|0}_0\\
    \otimes\\ j^{1|0}_0\emat$}};
  \node at (6,-3) (xp2) {\fbox{$\bmat xp\\ \Alg\\ \iota^{0|0}\\
    \otimes\\ j^{2|0}\emat$}};
  \node at (0,-6) (yq0) {\fbox{$\bmat yq\\ \Alg\\ \iota^{2|0}\\
    \otimes\\ j^{0|0}\emat$}};
  \node at (2,-6) (yq11) {\fbox{$\bmat yq\\ \Alg\\ \iota^{1|0}_1\\
    \otimes\\ j^{1|0}_1\emat$}};
  \node at (4,-6) (yq10) {\fbox{$\bmat yq\\ \Alg\\ \iota^{1|0}_0\\
    \otimes\\ j^{1|0}_0\emat$}};
  \node at (6,-6) (yq2) {\fbox{$\bmat yq\\ \Alg\\ \iota^{0|0}\\
    \otimes\\ j^{2|0}\emat$}};
  \node at (0,-9) (yp0) {\fbox{$\bmat yp\\ \Alg\\ \iota^{2|0}\\
    \otimes\\ j^{0|0}\emat$}};
  \node at (2,-9) (yp11) {\fbox{$\bmat yp\\ \Alg\\ \iota^{1|0}_1\\
    \otimes\\ j^{1|0}_1\emat$}};
  \node at (4,-9) (yp10) {\fbox{$\bmat yp\\ \Alg\\ \iota^{1|0}_0\\
    \otimes\\ j^{1|0}_0\emat$}};
  \node at (6,-9) (yp2) {\fbox{$\bmat yp\\ \Alg\\ \iota^{0|0}\\
    \otimes\\ j^{2|0}\emat$}};
\end{tikzpicture}}
\]
\caption{\textbf{Decomposition into idempotents of the 2-module $\pi_* \CAD(\HD)$ 
for the cornered, genus 2 handlebody $\HD$.} Note this figure does not indicate the differential or module structure, or the dimensions of the pieces of the decomposition.}
\label{fig:eg-cdad-nodiff}
\end{figure}

\begin{figure}
  \centering
  \begin{overpic}[tics=10]{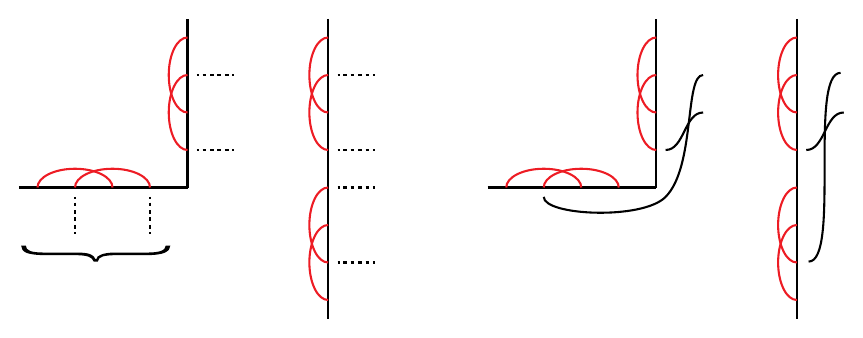}
    \put(10.5,6){$i$}
    \put(8.5,21){$\rot{\TAlg}$}
    \put(14,27){$\rot{\RAlg}$}
    \put(19,18.5){$\Dnil$}
    \put(34,18.5){$\Alg$}
  \end{overpic}
  \caption{\textbf{Identification of the algebra $\Alg$ with the smoothed tensor product of $\rot{\TAlg}$ and $\rot{\RAlg}$.} Left: an idempotent. The part corresponding to $i$ in Figure~\eqref{eq:big-with-brackets} is indicated. Right: a non-idempotent element.}
  \label{fig:Alg-is-rot}
\end{figure}

Recall that in the construction of $\CAA(\HD)$, the algebra $\Alg$ was effectively replaced with the bent tensor product $\putaround{\TAlg\obent\RAlg}{0}{0}{\sbull}{\sbull}$; roughly speaking, the above identification uses the $\rot{\TAlg}$ portion of $\DDa(\Id)$ to undo this replacement, though this process leaves behind a copy of $\RAlg$ on the bottom.
We have not yet included the differential (not to mention the right
algebra action) in the diagram.

Many of the summands are trivial. For example, there are no elements
of $\Alg$ with bottom idempotent $\iota^{2|0}$ and top idempotent
$I(xq)$, $I(xp)$, or $I(yp)$ (cf.~Formula~\eqref{eq:gens-corresp}). Similarly, because of the multiplicity-one condition, there are no elements of $\Alg$ with bottom idempotent $\iota^{0|0}$ and top idempotent $I(yp)$. So, the tensor product reduces to the 2-module in Figure~\ref{fig:eg-cdad}.%
\begin{figure}
\centering
\vspace*{50pt}
\[
\mathcenter{
\begin{tikzpicture}
  \node at (12,0) (xq2) {\fbox{$\bmat xq\\ \Alg\\ \iota^{0|0}\\
    \otimes\\ j^{2|0}\emat$}};
  \node at (4,-5) (xp11) {\fbox{$\bmat xp\\ \Alg\\ \iota^{1|0}_1\\
    \otimes\\ j^{1|0}_1\emat$}};
  \node at (8,-5) (xp10) {\fbox{$\bmat xp\\ \Alg\\ \iota^{1|0}_0\\
    \otimes\\ j^{1|0}_0\emat$}};
  \node at (12,-5) (xp2) {\fbox{$\bmat xp\\ \Alg\\ \iota^{0|0}\\
    \otimes\\ j^{2|0}\emat$}};
  \node at (4,-10) (yq11) {\fbox{$\bmat yq\\ \Alg\\ \iota^{1|0}_1\\
    \otimes\\ j^{1|0}_1\emat$}};
  \node at (8,-10) (yq10) {\fbox{$\bmat yq\\ \Alg\\ \iota^{1|0}_0\\
    \otimes\\ j^{1|0}_0\emat$}};
  \node at (12,-10) (yq2) {\fbox{$\bmat yq\\ \Alg\\ \iota^{0|0}\\
    \otimes\\ j^{2|0}\emat$}};
  \node at (0,-15) (yp0) {\fbox{$\bmat yp\\ \Alg\\ \iota^{2|0}\\
    \otimes\\ j^{0|0}\emat$}};
  \node at (4,-15) (yp11) {\fbox{$\bmat yp\\ \Alg\\ \iota^{1|0}_1\\
    \otimes\\ j^{1|0}_1\emat$}};
  \node at (8,-15) (yp10) {\fbox{$\bmat yp\\ \Alg\\ \iota^{1|0}_0\\
    \otimes\\ j^{1|0}_0\emat$}};
  \draw[->, bend right=15] (xp11) to node[below, sloped]{\lab{\sigma_1\sigma_8+\sigma_3\sigma_6+\sigma_{123}\sigma_{678}}} (yq11);
  \draw[->, bend right=15] (xp10) to node[below, sloped]{\lab{\sigma_1\sigma_8+\sigma_3\sigma_6+\sigma_{123}\sigma_{678}}} (yq10);
  \draw[->, bend right=15] (xp2) to node[below, sloped]{\lab{\sigma_1\sigma_8+\sigma_3\sigma_6+\sigma_{123}\sigma_{678}}} (yq2);
  \draw[->, bend right=15] (yq11) to node[below, sloped]{\lab{\sigma_2\sigma_7}} (xp11);
  \draw[->, bend right=15] (yq10) to node[below, sloped]{\lab{\sigma_2\sigma_7}} (xp10);
  \draw[->, bend right=15] (yq2) to node[below, sloped]{\lab{\sigma_2\sigma_7}} (xp2);
  \draw[->] (xq2) to node[below, sloped]{\lab{\sigma_{45}+\sigma_{234567}}} (xp2);
  \draw[->, bend left=25] (xq2) to node[above, sloped]{\lab{\sigma_{3456}+\sigma_{12345678}}} (yq2);
  \draw[->] (yq11) to node[below, sloped]{\lab{\sigma_{45}+\sigma_{234567}}} (yp11);
  \draw[->] (yq10) to node[below, sloped]{\lab{\sigma_{45}+\sigma_{234567}}} (yp10);
  \draw[->, bend right=25] (xp11) to node[below, sloped]{\lab{\sigma_{3456}+\sigma_{12345678}}} (yp11);
  \draw[->, bend left=25] (xp10) to node[above, sloped]{\lab{\sigma_{3456}+\sigma_{12345678}}} (yp10);
  \draw[->, bend left=15] (xp11) to node[above]{\lab{\bmat\sigma_2\\ \otimes\\ \rho_7^1 \emat}} (xp10); 
  \draw[->, bend left=15] (yq11) to node[above]{\lab{\bmat\sigma_2\\ \otimes\\ \rho_7^1 \emat}} (yq10); 
  \draw[->, bend left=15] (yp11) to node[above]{\lab{\bmat\sigma_2\\ \otimes\\ \rho_7^1 \emat}} (yp10); 
  \draw[->, bend left=15] (xp10) to node[below]{\lab{\bmat\sigma_1\\
    \otimes\\ \rho_8^1 \emat+\bmat\sigma_3\\ \otimes\\ \rho_6^1
    \emat+\bmat\sigma_{123}\\ \otimes\\ \rho_{678}^1 \emat}}  (xp11);
  \draw[->, bend left=15] (yq10) to node[below]{\lab{\bmat\sigma_1\\
    \otimes\\ \rho_8^1 \emat+\bmat\sigma_3\\ \otimes\\ \rho_6^1
    \emat+\bmat\sigma_{123}\\ \otimes\\ \rho_{678}^1 \emat}}
  (yq11);
  \draw[->, bend left=15] (yp10) to node[below]{\lab{\bmat\sigma_1\\
    \otimes\\ \rho_8^1 \emat+\bmat\sigma_3\\ \otimes\\ \rho_6^1
    \emat+\bmat\sigma_{123}\\ \otimes\\ \rho_{678}^1 \emat}}
  (yp11);
\end{tikzpicture}}.
\]
\caption{\textbf{The 2-module $\pi_* \CAD(\HD)$, with differentials.}}
\label{fig:eg-cdad}
\end{figure}
In that figure we have included the differential, so again the notation needs
some explanation. There are two kinds of differentials: the horizontal arrows and the vertical arrows. Recall that we are drawing the left action on $\Alg$ as on the top, and the right action on $\Alg$ as on the bottom. The vertical arrows correspond to multiplications on the left (top) of the element of $\Alg$; for example, the $\sigma_1\sigma_8$ terms mean that the element of
$\Alg$  is multiplied on the left by $\sigma_1\sigma_8$. (These are the
differentials coming from the differential on $\pi_*\CAA(\HD)$.)  The horizontal arrows correspond to multiplications on the right (bottom) of $\Alg$ and the top of the $j$ idempotent. For example, the $\bmat\sigma_2\\ \otimes\\
\rho_7^1\emat$ terms mean that the element of $\Alg$ is multiplied
on the bottom by $\sigma_2$ and the $j$ idempotent on the bottom is
multiplied by the element $\rho_7^1$ (i.e., we pick up a $\rho_7^1$ in
the bottom $\putaround{\RAlg}{\sbull}{}{0}{\sbull}$). (These are the
differentials coming from the differential on $\pi_*\DDa(\Id)$.)

The complex has three connected components. (The algebra action
intertwines these components, but is filtered: nontrivial products
stay in the same component or move one component to the right.) To
avoid a long and tedious computation, we will focus on the two smaller
components, on the left and right, ignoring the more complicated
component in the middle.

As a next step in our computation, we simplify the two components
under consideration. Let us start with the smaller one, $\bmat
yp\\ \Alg\\ \iota^{2|0}\\ \otimes\\ j^{0|0}\emat$. The
complex $\bmat yp\\ \Alg\\ \iota^{2|0}\emat$ is
$7$-dimensional:
\[
\includegraphics{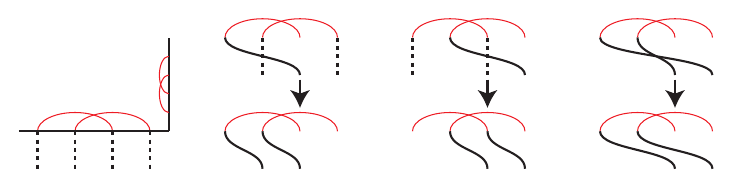}
\]
All of the non-idempotent basis elements cancel in pairs, as
indicated, so the module $\bmat yp\\ \Alg\\ \iota^{2|0}\\
\otimes\\ j^{0|0}\emat$ is homotopy equivalent (at least as a bottom $\RAlg$-module) to $\bmat\iota^{2|0}\\
\otimes\\ j^{0|0}\\ \putaround{\RAlg}{\sbull}{}{0}{\sbull}
\emat$. Note that this latter, simpler module is still projective over $\putaround{\RAlg}{\sbull}{}{\sbull}{\sbull}$.

The other component under consideration is more complicated, and is depicted in Figure~\ref{fig:awesome}.
\begin{figure}
\centering
\[
\includegraphics[width=\textwidth]{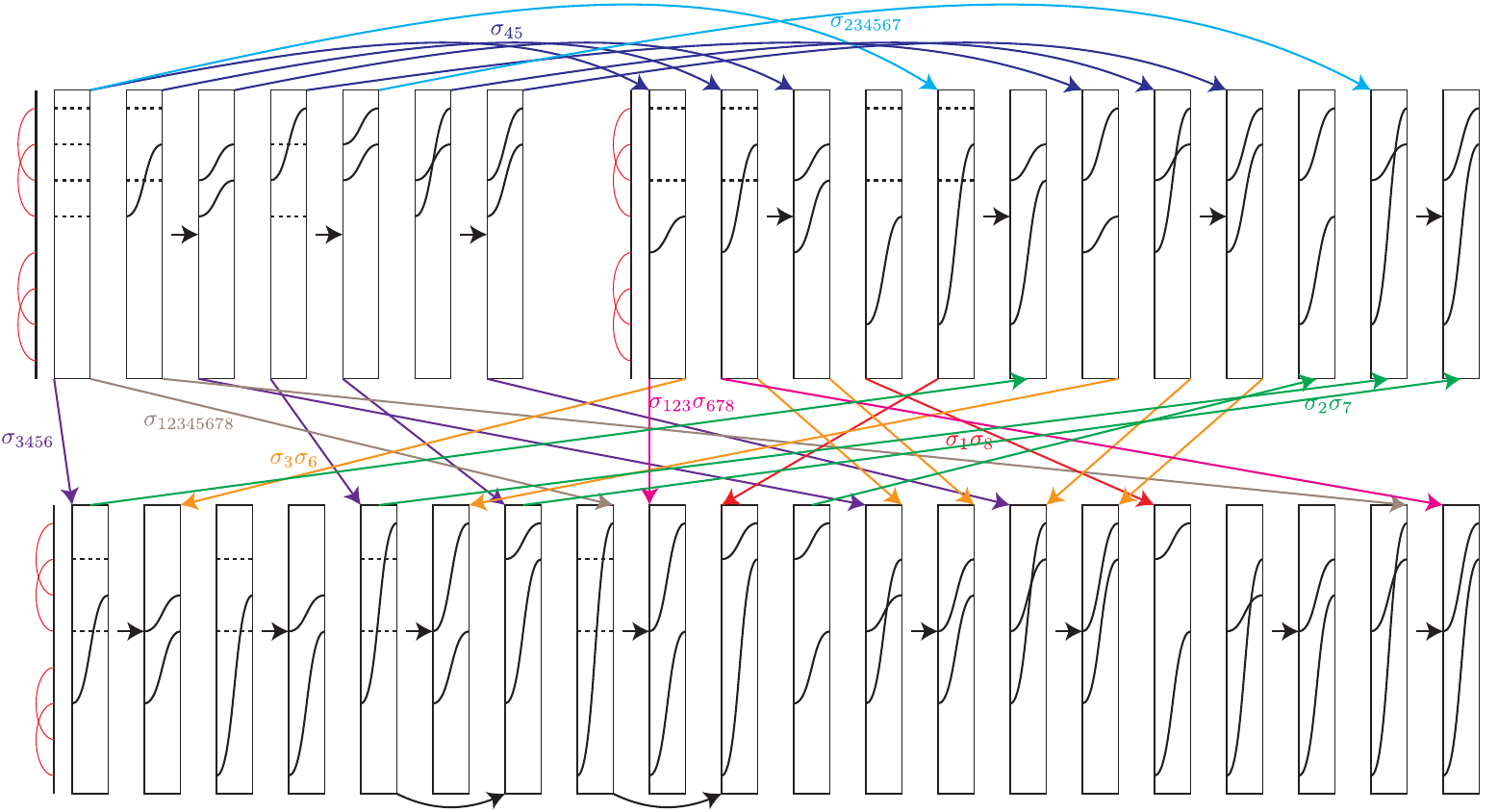}.
\]
\caption{\textbf{The third component of the 2-module $\pi_* \CAD(\HD)$.}  The colors and labels indicate which term in the differential on
$\CFDa(\HD^\circ)$ contributes a given arrow.}
\label{fig:awesome}
\end{figure}
Canceling acyclic subcomplexes quickly reduces this complex to the one depicted in Figure~\ref{fig:smawes}.
\begin{figure}
\centering
\[
\includegraphics[width=\textwidth]{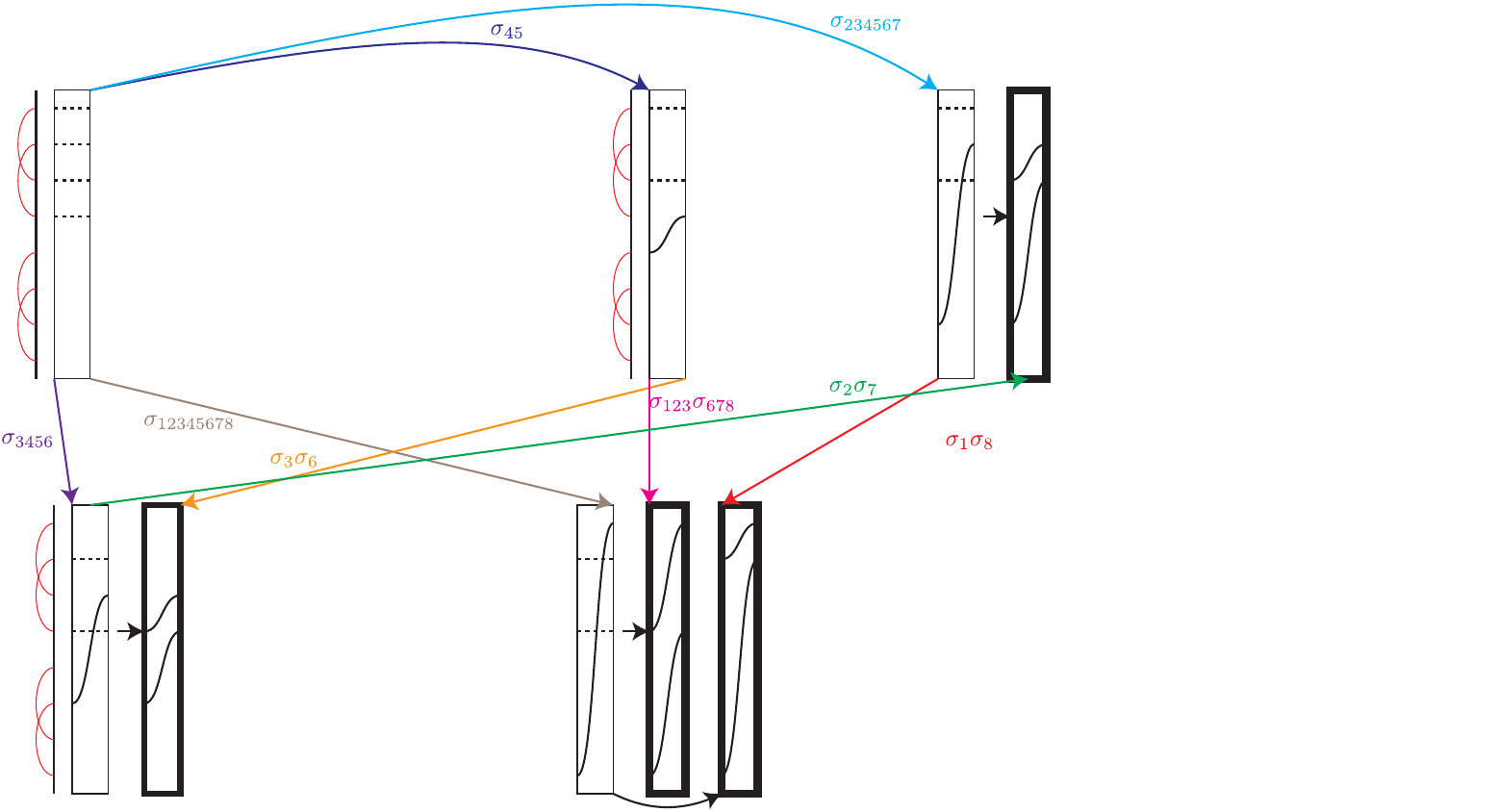}.
\]
\caption{\textbf{A reduction of the third component of the 2-module $\pi_* \CAD(\HD)$.}}
\label{fig:smawes}
\end{figure}
Let $\eta$ denote any of the four darkly boxed elements in that figure, all of which
are homologous. The component under consideration is homotopy
equivalent (again, as a bottom $\RAlg$-module) to~$\bmat \eta\\ \otimes
\\j^{2|0}\\\putaround{\RAlg}{\sbull}{}{0}{\sbull}\emat$. Again, this latter,
simplified module is still projective over
$\putaround{\RAlg}{\sbull}{}{\sbull}{\sbull}$.

Next we take the restricted tensor product over $\RAlg$ 
of $\pi_*\CAD(\HD)$ 
(which we have just been discussing)
with $\pi_*\CAA(\HD)$. 
The idempotents of $\TAlg$ corresponding to the
generators of $\pi_*\CAA(\HD)$ are
\[
xq\leftrightarrow j^{2|0}\qquad xp\leftrightarrow j_1^{1|0}\qquad
yq\leftrightarrow j^{1|0}_0\qquad yp\leftrightarrow j^{0|0}.
\]
So, tensoring 
$\bmat\iota^{2|0}\\ \otimes\\ j^{0|0}\\
  \putaround{\RAlg}{\sbull}{}{0}{\sbull} \emat$
with $\pi_*\CAA(\HD)$ gives
$\bmat\iota^{2|0}\\ \otimes\\ j^{0|0}\\
  \Alg\\ j^{0|0}\emat,$
coming from $yp\in\pi_*\CAA(\HD)$ (because $\bmat j^{0|0}\\ \Alg\\
j\emat=0$ for $j\neq j^{0|0}$). This complex is $7$-dimensional with
$1$-dimensional homology: it is essentially the same as $\bmat
\iota^{2|0}\\ \Alg\\ \iota^{2|0}\emat$ which we considered earlier.

The complex obtained by tensoring $ \bmat \eta\\
  \otimes
  \\j^{2|0}\\\putaround{\RAlg}{\sbull}{}{0}{\sbull}\emat$
with $\pi_*\CAA(\HD)$ is more complicated: it consists of
\[
\bmat\eta\\ \otimes\\ j^{2|0}\\
  \Alg\\ xp\emat\oplus 
\bmat\eta\\ \otimes\\ j^{2|0}\\
  \Alg\\ yq\emat
\oplus
\bmat\eta\\ \otimes\\ j^{2|0}\\
  \Alg\\ xq\emat,
\]
where the first summand corresponds to $xp$, the second to $yq$, and
the third to $yp$. (The differential, which is induced by the
differentials on $\pi_*\CAA(\HD)$ and $\Alg$, does not respect this direct
sum decomposition.) Indeed, this complex is rather similar to the
right-hand component of the complex in
Figure~\ref{fig:eg-cdad}. The resulting homology is again (by another
tedious computation) one-dimensional.

We have now computed the parts of 
\[
\rsmallvertme{\pi_*\CAA(\HD)}{\pi_*\CAD(\HD)}{\RAlg}{0pt}{0pt}
\]
corresponding to two of the three summands of the complex $\pi_*\CAD(\HD)$
and obtained two copies of $\Field$. This is in agreement with the pairing theorem (Proposition~\ref{prop:paired}): the two copies of $\Field$ correspond to the generators $yp$ and $xq$ of $\pi_*\CFAa(\HD^\circ)$. The (big) third summand of $\pi_*\CAD(\HD)$ would give the other two generators $xp$ and $yq$ of $\pi_*\CFAa(\HD^\circ)$.

We could proceed to verify that $\pi_*\CFDa(\HD^\circ)$ is quasi-isomorphic to the tensor product of 
$\pi_*\CDA(\HD)$ and $\pi_*\CDD(\HD)$; we leave this computation as an exercise
to the energetic reader. Tensoring $\pi_*\CFAa(\HD^\circ)$ and $\pi_*\CFDa(\HD^\circ)$
will give $\CFa$ of the $3$-manifold on the right of
Figure~\ref{fig:example} (up to chain homotopy equivalence).

The example discussed in this section extends in an obvious way to
compute $\HFa$ of $3$-manifolds given by (genus-one)
open-books. Consider the cornered Heegaard diagrams shown in
Figure~\ref{fig:Dehn-twist}. Gluing one these diagrams to a cornered
Heegaard diagram has the effect of changing the parametrization of a
boundary component by a Dehn twist.

\begin{figure}
  \centering
  \includegraphics{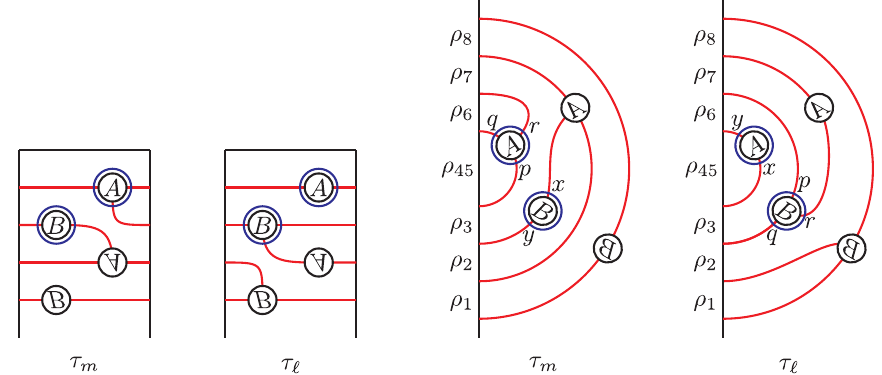}
  \caption{\textbf{Heegaard diagrams for Dehn twists.} Left: cornered
    Heegaard diagrams for two positive Dehn twists on a
    torus. Right: the bordered Heegaard diagrams obtained by smoothing
    these cornered Heegaard diagrams. Diagrams for negative Dehn
    twists can be obtained by reflecting these pictures (horizontally for the cornered diagrams, vertically for the bordered diagrams).}
  \label{fig:Dehn-twist}
\end{figure}

The Heegaard diagram in Figure~\ref{fig:example} corresponds to a
genus one, one boundary component open book with trivial
monodromy. Inserting copies of the diagrams from
Figure~\ref{fig:Dehn-twist} at the dashed lines in
Figure~\ref{fig:example} changes the monodromy by Dehn twists.

We have not defined invariants of Heegaard diagrams with more than one
corner in this paper. But we can still use these diagrams to compute
$\HFa$, by smoothing and using the trimodule $\TDDD$ or $\TDDA$. That
is, let $\CFDa(\tau_\ell)$ and $\CFDa(\tau_m)$ denote the bordered
invariants associated to the Heegaard diagrams on the right of
Figure~\ref{fig:Dehn-twist} and $Y$ a bordered $3$-manifold with
boundary parametrized by the split, genus $2$ pointed matched circle
$\PMC^2$. Consider the
bordered $3$-manifold $\tau_\ell(Y)$ gotten by twisting the
parametrization of (half of) $\bdy Y$ by $\tau_\ell$. The (type $D$) invariant
of this bordered manifold is given by
\begin{equation}\label{eq:twist-param}
\CFDa(\tau_\ell(Y))\simeq
\CFDa(Y)\otimes_{\Alg(\PMC^2)}\bigl(\CFDa(\tau_\ell)\otimes_{\Alg(\PMC^2)} \TDDD(\PMC,\PMC,\PMC)\bigr).
\end{equation}
(Here, $\PMC$ is the genus $1$ matched interval.  We are again using~\cite[Theorem 3]{LOTHomPair} to replace
$\CFAa$'s with $\CFDa$'s.) See Figure~\ref{fig:GlueDehnTwist} for an illustration at the level of Heegaard diagrams.
Applying this formula repeatedly, with $\HD$ being the self-gluing handlebody shown on the left of Figure~\ref{fig:example}, allows one to describe $\CFa$ of any genus one, one
boundary component open book. Working with higher-genus self-gluing
handlebodies (whose invariants are computed in~\cite[Section
9.5]{LOT4}), and modules associated to generators of the higher genus
mapping class group(oid)s, gives a corresponding description of $\CFa$
of open books with arbitrary genus pages and connected
bindings. Extending further to Zarev's bordered-sutured
theory~\cite{Zarev09:BorSut,Zarev:JoinGlue} allows one to understand
open books with disconnected bindings this way, as well.

\begin{figure}
  \centering
  \includegraphics{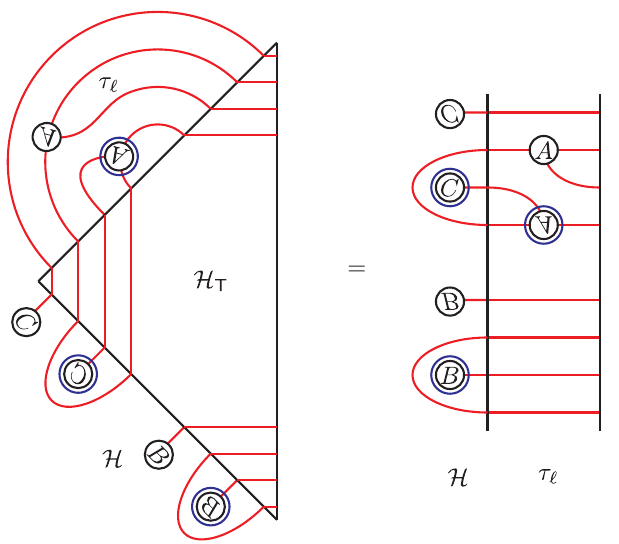}
  \caption{\textbf{Changing parameterization using trimodules.} Left: the result of gluing a Heegaard diagram $\HD$ and a diagram for $\tau_\ell$ to the Heegaard diagram $\THD{\PMC}{\PMC}{\PMC}$, where $\PMC$ is the genus $1$ matched interval. Some handles (and $\beta$-circles) have been omitted. Right: a simplification of the glued diagram.}
  \label{fig:GlueDehnTwist}
\end{figure}

To round out the discussion in the genus $1$ case, the bordered
modules $\CFDa$ associated to the smoothed Heegaard diagrams for Dehn
twists $\tau_\ell$ and $\tau_m$ are shown in
Figure~\ref{fig:Dehn-twist-mods}. (These computations follow from the
computations for the bimodules associated to arcslides in~\cite{LOT4},
together with the existence of the holomorphic disk giving the $\rho_4$ terms and the fact that $\bdy^2=0$.)

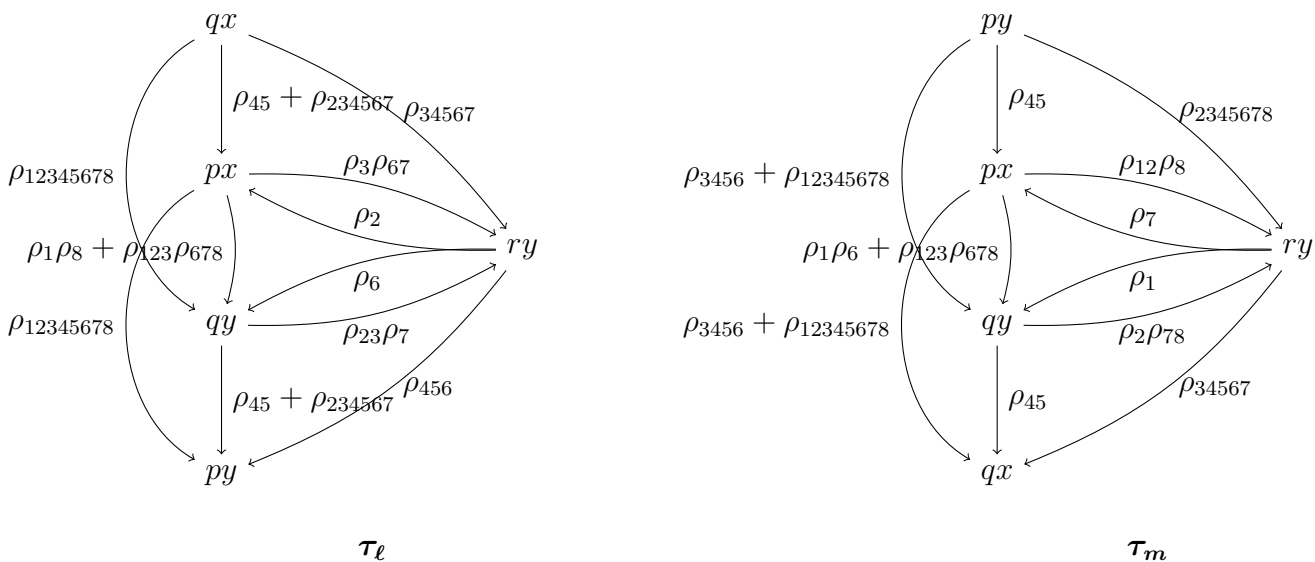
\begin{figure}
  \centering
  \[
  \begin{tikzpicture}
    \node at (0,0) (qx) {$qx$};
    \node at (0,-2) (px) {$px$};
    \node at (0,-4) (qy) {$qy$};
    \node at (0,-6) (py) {$py$};
    \node at (4,-3) (ry) {$ry$};
    \node at (2,-7) (label) {$\boldsymbol{\tau_\ell}$};
    \draw[->] (qx) to node[right]{$\rho_{45}+\rho_{234567}$} (px);
    \draw[->, bend left=15] (px) to node[left]{$\rho_1\rho_8+\rho_{123}\rho_{678}$} (qy);
    \draw[->] (qy) to node[right]{$\rho_{45}+\rho_{234567}$} (py);
    \draw[->, bend left=15] (px) to node[above]{$\rho_3\rho_{67}$} (ry);
    \draw[->, bend right=15] (qy) to node[below]{$\rho_{23}\rho_{7}$} (ry);
    \draw[->, bend left=15] (ry) to node[above]{$\rho_2$} (px);
    \draw[->, bend right=15] (ry) to node[below]{$\rho_6$} (qy);
    \draw[->, bend left=15] (ry) to node[right]{$\rho_{456}$} (py);
    \draw[->, bend left=15] (qx) to node[right]{$\rho_{34567}$} (ry);
    \draw[->, bend right=60] (qx) to node[left]{$\rho_{12345678}$} (qy);
    \draw[->, bend right=60] (px) to node[left]{$\rho_{12345678}$} (py);
  \end{tikzpicture}
  \qquad\qquad
  \begin{tikzpicture}
    \node at (0,0) (py) {$py$};
    \node at (0,-2) (px) {$px$};
    \node at (0,-4) (qy) {$qy$};
    \node at (0,-6) (qx) {$qx$};
    \node at (4,-3) (ry) {$ry$};
    \node at (2,-7) (label) {$\boldsymbol{\tau_m}$};
    \draw[->] (py) to node[right]{$\rho_{45}$} (px);
    \draw[->, bend left=15] (px) to node[left]{$\rho_1\rho_6+\rho_{123}\rho_{678}$} (qy);
    \draw[->] (qy) to node[right]{$\rho_{45}$} (qx);
    \draw[->, bend left=15] (px) to node[above]{$\rho_{12}\rho_{8}$} (ry);
    \draw[->, bend right=15] (qy) to node[below]{$\rho_2\rho_{78}$} (ry);
    \draw[->, bend left=15] (ry) to node[above]{$\rho_7$} (px);
    \draw[->, bend right=15] (ry) to node[below]{$\rho_1$} (qy);
    \draw[->, bend left=15] (ry) to node[right]{$\rho_{34567}$} (qx);
    \draw[->, bend left=15] (py) to node[right]{$\rho_{2345678}$} (ry);
    \draw[->, bend right=60] (py) to node[left]{$\rho_{3456}+\rho_{12345678}$} (qy);
    \draw[->, bend right=60] (px) to node[left]{$\rho_{3456}+\rho_{12345678}$} (qx);
  \end{tikzpicture}
  \]
  \caption{\textbf{The invariant $\CFDa$ for Dehn twists in the torus.}}
  \label{fig:Dehn-twist-mods}
\end{figure}

Note that another algorithm for computing $\HFa$ of open books is
given in~\cite{LOT4}, again based on the self-gluing handlebody. The
complexes given by the two techniques have similar numbers of
generators.


\section{The nilCoxeter planar algebra}
\label{sec:conclusion}

The manifolds with corners considered in this paper are of a special type: Their boundaries are split into two distinguished kinds of surfaces: vertical ($F_0$ and $F_1$) and horizontal ($F_0'$ and $F_1'$). This fits into the framework of $\langle n \rangle$-manifolds with corners developed by J\"anich \cite{Janich}. This framework precludes, for example, a manifold whose boundary is decomposed into three pieces $F_0, F_1, F_2$, glued to each other along three codimension two corners $F_0 \cap F_1, F_1 \cap F_2$ and $F_2 \cap F_0$. 

The vertical/horizontal distinction that appears in this paper is a natural consequence of the use of rectangular $2$-algebras. If we are interested in general $3$-manifolds with codimension two corners, we expect that one could define cornered Floer invariants by using planar algebras instead of rectangular $2$-algebras. 

Specifically, we envision the following replacement for the $2$-algebra $\Dnil$. Let $D$ be a disk in the plane with $2n$ distinguished points of the boundary, $n$ of which are marked $+$ and the other $n$ are marked $-$. For any such marked disk $D$, we consider the $\Field$-vector space generated by pictures (up to isotopy) of $n$ oriented strands inside $D$ such that: each strand starts at a point on the boundary marked $-$ and ends at a point marked $+$; no strand self-intersects; any two strands intersect in at most one point; and no three strands intersect at the same point. Here is an example:
$$\includegraphics[scale=0.8]{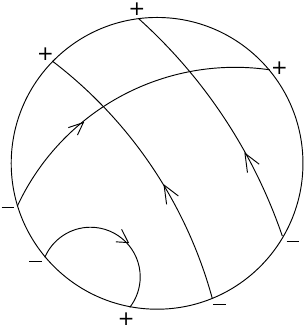}$$
Define $\Planar(D)$ to be the quotient of the vector space of such pictures by relations of the form:
$$\includegraphics[scale=0.8]{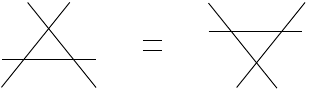}$$
We equip $\Planar(D)$ with a differential $\del$ given by adding up all ways of taking the oriented resolution at a crossing. 

Further, we have compositions corresponding to planar tangles. A planar tangle $T$ consists of a disk $D_0$ in the plane and several disjoint disks $D_1, \dots, D_k$ in the interior of $D_0$, such that $D_i$ has $2n_i$ markings on the boundary; further, we have oriented non-intersecting strands in $D_0 \setminus (D_1 \cup \dots   \cup D_k)$ that connect all the marked points. Here is an example for $k=3$:
$$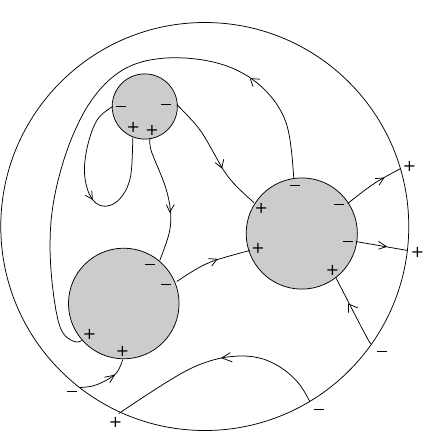$$
Associated to any planar tangle $T$ is an operation:  
$$ *_T : \Planar(D_1) \otimes \dots \otimes \Planar(D_k) \to \Planar(D_0)$$
given by plugging in pictures of strands into $D_1, \dots, D_k$, and combining them with the tangle $T$ to obtain a picture in $D_0$. We impose the rules that whenever we obtain a self-intersection of a strand, a closed circle, or a double intersection of two strands by this process, we set the result to equal zero.

We call the structure $\Planar = \{\Planar_D, \del, *_T \}$ the {\em nilCoxeter planar algebra}\footnote{Our structure differs slightly from the usual notion of planar algebra defined by Jones \cite{JonesPlanar}.}. This is a close relative of the diagonal nilCoxeter $2$-algebra used in this paper, with the advantage that it no longer has the vertical/horizontal distinction. The operations $*_T$ take the place of the vertical and horizontal multiplications on the $2$-algebra.

Note that the infinity-version of a planar algebra (where we consider diagrams and families of diagrams without dividing by isotopy) can be viewed as a kind of colored $E_2$-algebra. This may be compared   with the appearance of $E_2$-algebras in other extended TQFTs \cite[Section 4.1]{LurieTQFT}.


\bibliographystyle{hamsalpha}
\bibliography{heegaardfloer}
\end{document}